\documentclass[11pt,letterpaper]{amsart}
\usepackage[foot]{amsaddr}

\usepackage{mathtools}
\usepackage{amsmath}
\usepackage[T1]{fontenc}
\usepackage[latin9]{inputenc}
\usepackage{geometry}
\usepackage{wrapfig}
\usepackage{multicol}
\usepackage{graphicx}
\usepackage{soul}
\usepackage{xcolor}
\usepackage{amssymb}
\usepackage{placeins}
\usepackage{bbm}
\usepackage{cite}
\usepackage{caption}
\usepackage{enumerate}
\usepackage{afterpage}
\usepackage{enumitem}
\usepackage{bmpsize}
\usepackage{hyperref}
\usepackage{tabu}
\usepackage{enumitem}
\numberwithin{equation}{section}
\usepackage{stmaryrd}
\usepackage{tikz}
\usetikzlibrary{matrix,graphs,arrows,positioning,calc,decorations.markings,shapes.symbols}
\usepackage{ifthen}
\usetikzlibrary{math}
\usetikzlibrary{matrix,graphs,arrows,positioning,calc,decorations.markings,shapes.symbols}

\makeatletter
\renewcommand{\email}[2][]{%
  \ifx\emails\@empty\relax\else{\g@addto@macro\emails{,\space}}\fi%
  \@ifnotempty{#1}{\g@addto@macro\emails{\textrm{(#1)}\space}}%
  \g@addto@macro\emails{#2}%
}
\makeatother

\newtheorem{theorem}{Theorem}[section]
\newtheorem{lemma}[theorem]{Lemma}
\newtheorem{proposition}[theorem]{Proposition}

{ \theoremstyle{definition}
\newtheorem{definition}[theorem]{Definition}}
{ \theoremstyle{remark}
\newtheorem{remark}[theorem]{Remark}}

\newcommand{\ex}{\mathbb{E}}

\newcommand{\weyl}{W^\circ}

\newcommand{\im}{\mathsf{i}}
\newcommand{\Real}{\mathsf{Re}\hspace{0.5mm}}
\newcommand{\Imag}{\mathsf{Im}\hspace{0.5mm}}
\newcommand{\Arg}{\mathsf{Arg}}
\newcommand{\cev}[1]{\reflectbox{\ensuremath{\vec{\reflectbox{\ensuremath{#1}}}}}}

\newcommand{\kcr}{K^{\mathrm{cross}}}
\newcommand{\icr}{I^{\mathrm{cross}}}
\newcommand{\rcr}{R^{\mathrm{cross}}}
\newcommand{\hsa}{\mathcal{A}^{\mathrm{hs}}}
\newcommand{\ap}{\mathsf{a}}

\newcommand{\kgeo}{K^{\mathrm{geo}}}
\newcommand{\pfbm}{\mathbb{P}_{\operatorname{free}}}
\newcommand{\efbm}{\mathbb{E}_{\operatorname{free}}}
\newcommand{\pabm}{\mathbb{P}_{\operatorname{avoid}}}
\newcommand{\eabm}{\mathbb{E}_{\operatorname{avoid}}}
\newcommand{\ice}{\mathsf{Inter}}

\newcommand{\cb}{\mathsf{C}}

\usepackage{ifthen}

\title{Half-space Airy line ensembles}
\date{\today}
\author{Evgeni Dimitrov} 
\email[Evgeni Dimitrov]{edimitro@usc.edu}
\author{Zongrui Yang} 
\email[Zongrui Yang]{zy2417@columbia.edu}
\begin{document}

\begin{abstract}
We construct a one-parameter family of infinite line ensembles on $[0, \infty)$ that are natural half-space analogues of the Airy line ensemble. Away from the origin these ensembles are locally described by avoiding Brownian bridges, and near the origin they are described by a sequence of avoiding reverse Brownian motions with alternating drifts, that depend on the parameter of the model. In addition, the restrictions of our ensembles to finitely many vertical lines form Pfaffian point processes with the crossover kernels obtained by Baik-Barraquand-Corwin-Suidan \cite{BBCS}.
\end{abstract}

\maketitle

\tableofcontents

%
%
\section{Introduction and main results}\label{Section1} Over the last two decades there has been a marked interest in the asymptotic analysis of {\em half-space} models in the {\em Kardar-Parisi-Zhang (KPZ)} universality class. The earliest works in this direction are due to Baik and Rains, who studied the asymptotics of the longest increasing subsequence of random involutions and symmetrized last passage percolation (LPP) with geometric weights \cite{BR01a, BR01b, BR01c}. Other half-space models which have been substantially investigated include: the polynuclear growth model \cite{SI04}, Schur processes \cite{BR05, BBNV}, LPP with exponential weights \cite{BBCS}, the facilitated (totally) asymmetric simple exclusion process or (T)ASEP \cite{BBCS2}, the KPZ equation through ASEP \cite{CS18, Par19} and through directed polymers \cite{Wu20}, the stochastic six-vertex model \cite{BBCW18}, Macdonald processes \cite{BBC20}, and the log-gamma polymer \cite{IMS22, BCD24}.

The local description of the half-space models in the KPZ class is similar to the full-space ones, except that the origin (which should be interpreted as a left boundary of the interval $[0, \infty)$ where the models are defined) creates a boundary effect that non-trivially affects their asymptotic behavior. For example, in interacting particle models the origin is a reservoir of particles that stochastically injects or absorbs particles from the system, and in LPP models the weights corresponding to the origin are different from those away from it. There is a general goal of understanding the boundary effect in various settings, and at least in some instances, showing that the asymptotic behavior remains the same across models. In other words, we are interested in showing that these models exhibit universal behavior and belong to a half-space KPZ universality class. In order to formally define such a class, one needs to find suitable analogues that replace the universal scaling limits in the full-space KPZ class, and the {\bf main goal of the present paper is to give a formal construction of the {\em half-space Airy line ensembles}}. Our construction is based on taking the weak scaling limit of a sequence of discrete line ensembles arising from Pfaffian Schur processes, or equivalently from symmetrized geometric LPP. In the next section we give an informal description of our new ensembles, and compare them to the usual Airy line ensemble.

%
%
\subsection{The Airy line ensemble and its half-space analogues}\label{Section1.1} The Airy line ensemble $\mathcal{A}= \{\mathcal{A}_{i}\}_{i \geq 1}$ is a sequence of real-valued random continuous functions, which are defined on $\mathbb{R}$, and are strictly ordered in the sense that $\mathcal{A}_{i}(t) > \mathcal{A}_{i+1}(t)$ for all $i \geq 1$ and $t \in \mathbb{R}$. It arises as the edge scaling limit of Wigner matrices \cite{Sod15}, lozenge tilings \cite{AH21}, avoiding Brownian bridges (also known as {\em Brownian watermelons}) \cite{CorHamA}, as well as various integrable models of non-intersecting random walkers and last passage percolation \cite{DNV19}. The list of models becomes quite vast if one further includes those that converge to the various projections of $\mathcal{A}$ such as $(\mathcal{A}_{1}(t): t \in \mathbb{R})$ (called the {\em Airy process}), $(\mathcal{A}_{i}(t_0): i \geq 1)$ for a fixed $t_0 \in \mathbb{R}$ (called the {\em Airy point process}) and $\mathcal{A}_{1}(t_0)$ for a fixed $t_0$ (called the {\em Tracy-Widom distribution}). Due to its appearance as a universal scaling limit, and its role in the construction of the {\em Airy sheet} in \cite{DOV22}, the Airy line ensemble is regarded as a central object in the KPZ universality class \cite{CU2}.

One of the salient features of the Airy line ensemble is that it has the structure of a {\em determinantal point process}, see Definition \ref{def: DefDPP}. More specifically, if one fixes a finite set $\mathsf{S} = \{s_1, \dots, s_m\} \subset \mathbb{R}$ with $s_1 < \cdots < s_m$, then the random measure on $\mathbb{R}^2$, defined by
\begin{equation}\label{RMS1}
M(A) = \sum_{i \geq 1} \sum_{j = 1}^m {\bf 1} \left\{\left(s_j, \mathcal{A}_{i}(s_j) \right) \in A \right\},
\end{equation}
is a determinantal point process on $\mathbb{R}^2$ with reference measure $\mu_{\mathsf{S}} \times \lambda$, where $\mu_{\mathsf{S}}$ is the counting measure on $\mathsf{S}$, and $\lambda$ is the usual Lebesgue measure on $\mathbb{R}$, and whose correlation kernel is given by the {\em extended Airy kernel}, defined for $x_1, x_2 \in \mathbb{R}$ and $t_1, t_2 \in \mathsf{S}$ by 
\begin{equation}\label{S1AiryKer}
\begin{split}
K^{\mathrm{Airy}}(t_1,x_1; t_2,x_2) = & -  \frac{{\bf 1}\{ t_2 > t_1\} }{\sqrt{4\pi (t_2 - t_1)}} \cdot e^{ - \frac{(x_2 - x_1)^2}{4(t_2 - t_1)} - \frac{(t_2 - t_1)(x_2 + x_1)}{2} + \frac{(t_2 - t_1)^3}{12} } \\
& + \frac{1}{(2\pi \im)^2} \int_{\mathcal{C}_{\alpha}^{\pi/3}} d z \int_{\mathcal{C}_{\beta}^{2\pi/3}} dw \frac{e^{z^3/3 -x_1z - w^3/3 + x_2w}}{z + t_1 - w - t_2}.
\end{split}
\end{equation}
In (\ref{S1AiryKer}) we have that $\alpha, \beta \in \mathbb{R}$ are arbitrary subject to $\alpha + t_1 > \beta + t_2$, and $\mathcal{C}_{z}^{\varphi}=\{z+|s|e^{\mathrm{sgn}(s)\im\varphi}, s\in \mathbb{R}\}$ is the infinite curve oriented from $z+\infty e^{-\im\varphi}$ to $z+\infty e^{\im\varphi}$. We mention that there are various formulas for the extended Airy kernel, and the one in (\ref{S1AiryKer}) comes from \cite[Proposition 4.7 and (11)]{BK08} under the change of variables $u \rightarrow z + t_1$ and $w \rightarrow w + t_2$. 

The determinantal structure of the Airy line ensemble allows one to express its finite-dimensional distributions as {\em Fredholm determinants} involving $K^{\mathrm{Airy}}$. For example, if $t_1 = t_2 = t_0$ in (\ref{S1AiryKer}), then the extended Airy kernel reduces to the ordinary Airy kernel \cite{TWPaper}, and one obtains the classical Fredholm determinant formula for the Tracy-Widom distribution $F_2(t) = \mathbb{P}(\mathcal{A}_1(t_0) \leq t)$. The finite-dimensional distribution of $\mathcal{A}$ was obtained in \cite{Spohn} as a limit of the polynuclear growth model, although it appeared earlier in the random matrix literature, see \cite{Mac94} and \cite{FNH99}. The construction of $\mathcal{A}$ as a sequence of random continuous functions was achieved in \cite{CorHamA} by taking a weak limit of Brownian watermelons, and in the same paper the authors showed that $\mathcal{A}$ has a local description in terms of avoiding Brownian bridges. Specifically, define the {\em parabolic Airy line ensemble} $\mathcal{L}^{\mathrm{pAiry}}$ via
\begin{equation}\label{PALE}
\mathcal{L}^{\mathrm{pAiry}}_i(t) = 2^{-1/2} \cdot \mathcal{A}_i(t) - 2^{-1/2} \cdot t^2 \mbox{ for $i \geq 1$ and $t \in \mathbb{R}$}.
\end{equation}
Then, for any $k \in \mathbb{N}$, and real $a < b$, one has that the conditional law of $( \mathcal{L}^{\mathrm{pAiry}}_i(t): i = 1, \dots, k, t \in [a,b] )$, given that $\mathcal{L}^{\mathrm{pAiry}}_1(a) = x_1, \dots, \mathcal{L}^{\mathrm{pAiry}}_k(a) = x_k$, $\mathcal{L}^{\mathrm{pAiry}}_1(b) = y_1, \dots, \mathcal{L}^{\mathrm{pAiry}}_k(b) = y_k$, and $\mathcal{L}^{\mathrm{pAiry}}_{k+1}[a,b] = g$ (for a fixed continuous $g$ on $[a,b]$), is the same as that of $k$ independent Brownian bridges $B_i$ from $(a,x_i)$ to $(b,y_i)$ that have been conditioned to avoid each other and also the graph of the function $g$, see the left side of Figure \ref{S11}. This property is now referred to as the {\em Brownian Gibbs property}. It was recently established in \cite{AH23} that $\mathcal{L}^{\mathrm{pAiry}}$ is (up to an independent affine shift) the unique line ensemble that possesses this property, and whose top curve globally looks like $-2^{-1/2} t^2$. In this sense, the Brownian Gibbs property is not just an exotic feature of the Airy line ensemble, but essentially a defining one.\\

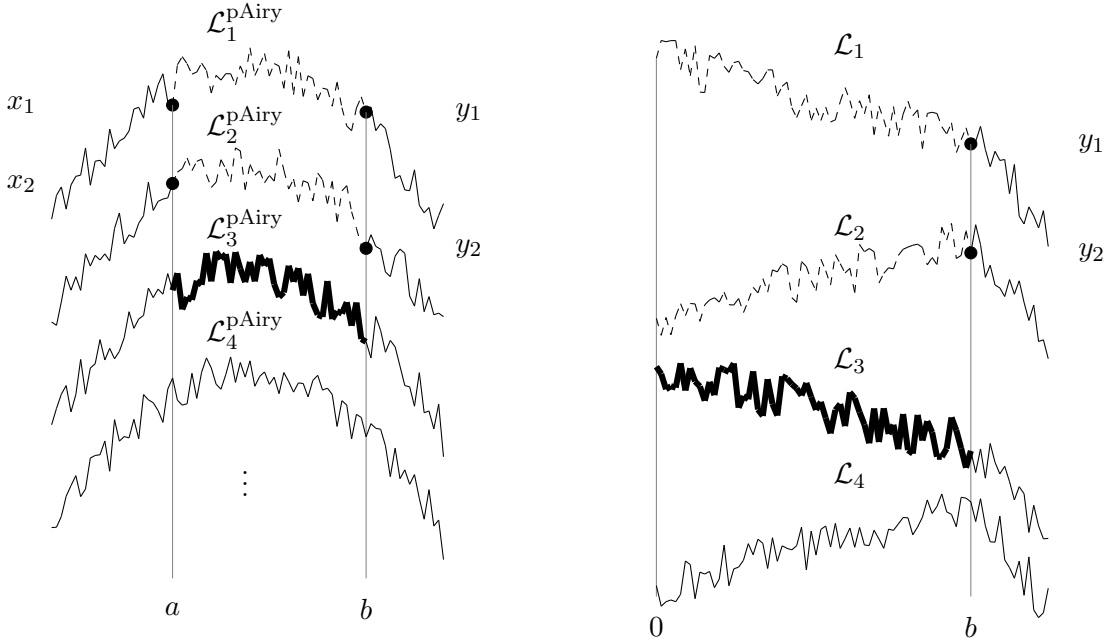
\begin{figure}[h] 
    \centering
    \begin{tikzpicture}[scale=0.8]

  \def\r{0}
  \def\s{0.08}
  \def\v{1.7}
  \def\h{10}

    \foreach \x [remember=\r as \rlast (initially 0)] in {-40,...,40}
    {
        \tikzmath{\r = 5*rand*\s; \y = \x*\s; \z = (\x+1)*\s;}
        \ifthenelse{\x < -15}
        {
            \draw[-,thin][black] (\y, -0.25*\y*\y + \rlast ) -- (\z, -0.25*\z*\z + \r );
        }
        {

            \ifthenelse{\x < 25}{
                \ifthenelse{\x = -15}{
                    \draw[black,fill=black] (\y, -0.25*\y*\y + \rlast ) circle (0.6ex); 
                    \draw (\y -2.5, -0.25*\y*\y + \rlast  ) node{$x_1$};
                    \draw[-,very thin][gray] (\y, -0.25*\y*\y + \rlast ) -- (\y, -5*\v );
                    \draw (\y, -5*\v  - 0.5 ) node{$a$};
                }{}
                \draw[-,dashed][black] (\y, -0.25*\y*\y + \rlast ) -- (\z, -0.25*\z*\z + \r );
            }
            {
                \ifthenelse{\x = 25}{
                    \draw[black,fill=black] (\y, -0.25*\y*\y + \rlast ) circle (0.6ex); 
                    \draw (\y + 1.7, -0.25*\y*\y + \rlast ) node{$y_1$};
                    \draw[-,very thin][gray] (\y, -0.25*\y*\y + \rlast ) -- (\y, -5*\v );
                    \draw (\y, -5*\v  - 0.5 ) node{$b$};
                }{}
                \draw[-,thin][black] (\y, -0.25*\y*\y + \rlast ) -- (\z, -0.25*\z*\z + \r );
            }
        }
    }

    \foreach \x [remember=\r as \rlast (initially 0)] in {-40,...,40}
    {
        \tikzmath{\r = 5*rand*\s; \y = \x*\s; \z = (\x+1)*\s;}
        \ifthenelse{\x < -15}
        {
            \draw[-,thin][black] (\y, -0.25*\y*\y + \rlast -\v) -- (\z, -0.25*\z*\z + \r -\v);
        }
        {
                \ifthenelse{\x = -15}{
                    \draw[black,fill=black] (\y, -0.25*\y*\y + \rlast -\v) circle (0.6ex); 
                    \draw (\y - 2.5, -0.25*\y*\y + \rlast -\v ) node{$x_2$};
                }{}
        
            \ifthenelse{\x < 25}{
                \draw[-,dashed][black] (\y, -0.25*\y*\y + \rlast -\v) -- (\z, -0.25*\z*\z + \r -\v);
            }
            {
                \ifthenelse{\x = 25}{
                    \draw[black,fill=black] (\y, -0.25*\y*\y + \rlast - \v) circle (0.6ex); 
                    \draw (\y + 1.7, -0.25*\y*\y + \rlast -\v) node{$y_2$};
                }{}
                \draw[-,thin][black] (\y, -0.25*\y*\y + \rlast -\v) -- (\z, -0.25*\z*\z + \r -\v);
            }
        }
    }

    \foreach \x [remember=\r as \rlast (initially 0)] in {-40,...,40}
    {
        \tikzmath{\r = 5*rand*\s; \y = \x*\s; \z = (\x+1)*\s;}
        \ifthenelse{\x < -15}
        {
            \draw[-,thin][black] (\y, -0.25*\y*\y + \rlast -2*\v) -- (\z, -0.25*\z*\z + \r - 2*\v);
        }
        {
            \ifthenelse{\x < 25}{
                \draw[line width = 0.8mm, black] (\y, -0.25*\y*\y + \rlast - 2*\v) -- (\z, -0.25*\z*\z + \r - 2*\v);
            }
            {
                \draw[-,thin][black] (\y, -0.25*\y*\y + \rlast - 2*\v) -- (\z, -0.25*\z*\z + \r - 2*\v);
            }
        }
    }

    \foreach \x [remember=\r as \rlast (initially 0)] in {-40,...,40}
    {
        \tikzmath{\r = 5*rand*\s; \y = \x*\s; \z = (\x+1)*\s;}
        \ifthenelse{\x < 81}
        {
            \draw[-,thin][black] (\y, -0.25*\y*\y + \rlast -3*\v ) -- (\z, -0.25*\z*\z + \r -3*\v );
        }
    }

    \draw (0, 0.7) node{$\mathcal{L}^{\mathrm{pAiry}}_1$};
    \draw (0, 0.7 - \v) node{$\mathcal{L}^{\mathrm{pAiry}}_2$};
    \draw (0, 0.7 - 2*\v) node{$\mathcal{L}^{\mathrm{pAiry}}_3$};
    \draw (0, 0.7 - 3*\v) node{$\mathcal{L}^{\mathrm{pAiry}}_4$};

    \draw (0, -4*\v ) node{$\vdots$};

    \foreach \x [remember=\r as \rlast (initially 0)] in {-40,...,40}
    {
        \tikzmath{\r = 5*rand*\s; \y = \x*\s; \z = (\x+1)*\s;}
        {
            \ifthenelse{\x = -40}{
                    \draw[-,very thin][gray]  (\y + \h, 1-0.25*\y + \rlast -\v )  -- (\y + \h, -5*\v -0.3);
                    \draw (\y + \h, -5*\v  - 0.8 ) node{$0$};
            }{}
            \ifthenelse{\x < 25}{
                \draw[-,dashed][black] (\y + \h, 1-0.25*\y + \rlast -\v ) -- (\z + \h, 1-0.25*\z + \r -\v );
            }
            {
                \ifthenelse{\x = 25}{
                    \draw[black,fill=black] (\y + \h, 1.5-0.25*\y*\y + \rlast -\v) circle (0.6ex); 
                    \draw (\y + 2 + \h, 1.5 -0.25*\y*\y + \rlast -\v ) node{$y_1$};
                    \draw[-,very thin][gray] (\y + \h,  1.5-0.25*\y*\y + \rlast - \v) -- (\y + \h, -5*\v -0.3 );
                    \draw (\y + \h, -5*\v  - 0.8 ) node{$b$};
                }{}
                \draw[-,thin][black] (\y + \h, 1.5-0.25*\y*\y + \rlast -\v ) -- (\z + \h, 1.5-0.25*\z*\z + \r -\v);
            }
        }
    }

    \foreach \x [remember=\r as \rlast (initially 0)] in {-40,...,40}
    {
        \tikzmath{\r = 5*rand*\s; \y = \x*\s; \z = (\x+1)*\s;}
        {
            \ifthenelse{\x < 25}{
                \draw[-,dashed][black] (\y + \h, 0.25*\y + \rlast -2*\v ) -- (\z + \h, 0.25*\z + \r -2*\v );
            }
            {
                \ifthenelse{\x = 25}{
                    \draw[black,fill=black] (\y + \h, 1.5-0.25*\y*\y + \rlast -2*\v) circle (0.6ex); 
                    \draw (\y + 2 + \h, 1.5 -0.25*\y*\y + \rlast -2*\v ) node{$y_2$};
                }{}
                \draw[-,thin][black] (\y + \h, 1.5-0.25*\y*\y + \rlast -2*\v ) -- (\z + \h, 1.5-0.25*\z*\z + \r -2*\v);
            }
        }
    }

    \foreach \x [remember=\r as \rlast (initially 0)] in {-40,...,40}
    {
        \tikzmath{\r = 5*rand*\s; \y = \x*\s; \z = (\x+1)*\s;}
        {
            \ifthenelse{\x < 25}{
                \draw[line width = 0.8mm, black](\y + \h, 1-0.25*\y + \rlast -4*\v ) -- (\z + \h, 1-0.25*\z + \r -4*\v );
            }
            {
                \draw[-,thin][black] (\y + \h, 1.5-0.25*\y*\y + \rlast -4*\v ) -- (\z + \h, 1.5-0.25*\z*\z + \r -4*\v);
            }
        }
    }

    \foreach \x [remember=\r as \rlast (initially 0)] in {-40,...,40}
    {
        \tikzmath{\r = 5*rand*\s; \y = \x*\s; \z = (\x+1)*\s;}
        {
            \ifthenelse{\x < 25}{
                \draw[-,thin][black] (\y + \h, 0.25*\y + \rlast -4.6*\v ) -- (\z + \h, 0.25*\z + \r -4.6*\v );
            }
            {
                \draw[-,thin][black] (\y + \h, 1.5-0.25*\y*\y + \rlast -4.6*\v ) -- (\z + \h, 1.5-0.25*\z*\z + \r -4.6*\v);
            }
        }
    }

    \draw (\h, 0.3) node{$\mathcal{L}_1$};
    \draw (\h, 0.7 - 2*\v) node{$\mathcal{L}_2$};
    \draw (\h, 0.2 -3*\v ) node{$\mathcal{L}_3$};
    \draw (\h,  - 4*\v) node{$\mathcal{L}_4$};

\end{tikzpicture}

\caption{The left side depicts the parabolic Airy line ensemble $\mathcal{L}^{\mathrm{pAiry}}$. The Brownian Gibbs property implies that the conditional law of the dashed lines is that of independent Brownian bridges connecting $(a,x_i)$ to $(b,y_i)$ that are conditioned to avoid each other and the thick curve. The right side depicts the line ensemble $\mathcal{L} = \{\mathcal{L}_i\}_{i \geq 1}$, where $\mathcal{L}_i(t) = 2^{-1/2} \cdot \hsa_i(t) - 2^{-1/2} \cdot t^2$, and $\hsa$ is the half-space Airy line ensemble. The half-space Brownian Gibbs property implies that the conditional law of the dashed lines is that of two independent reverse Brownian motions -- one started from $(b,y_1)$ with drift $-\sqrt{2} \varpi$, and the other started from $(b,y_2)$ with drift $\sqrt{2} \varpi$, which are conditioned to avoid each other and the thick curve.}
    \label{S11}
\end{figure}

We next turn to an informal description of our new models, and refer to Section \ref{Section1.2} for the precise statements. We let $\hsa = \{\hsa_i\}_{i \geq 1}$ denote the half-space Airy line ensemble, which is a sequence of strictly ordered real-valued random continuous functions on $[0, \infty)$. The distribution of $\hsa$ depends on a parameter $\varpi \in \mathbb{R}$, which reflects the boundary effect of the origin in the model. To simplify our exposition we will suppress the parameter $\varpi$ from our notation. If one considers the random measures as in (\ref{RMS1}) with $\mathcal{A}$ replaced with $\hsa$ and $\mathsf{S} \subset [0,\infty)$ (as opposed to $\mathsf{S} \subset \mathbb{R}$), then they are no longer determinantal, but rather {\em Pfaffian point processes}, see Definition \ref{def:def of Pfaffian point process}. They have the same reference measure $\mu_{\mathsf{S}} \times \lambda$ as above, and a correlation kernel which we denote by $\kcr$. The exact form of $\kcr$ is given in Definition \ref{def:kcr}, but we mention that similarly to $K^{\mathrm{Airy}}$ in (\ref{S1AiryKer}), the kernel is given by a double-contour integral of cubic exponential functions. Our formula for $\kcr$ agrees with the one obtained in \cite[Section 2.5]{BBCS} up to a few minor corrections, see Remark \ref{S1Correction}. The Pfaffian point process structure of $\hsa$ allows one to express its finite-dimensional distributions as {\em Fredholm Pfaffians}, and in particular we have that $F_{\mathrm{cross}}(t; \varpi, t_0) = \mathbb{P}(\hsa_1(t_0) \leq t)$, where $F_{\mathrm{cross}}$ is the two-parameter family of {\em crossover distributions} from \cite[Definition 2.9]{BBCS}. We mention that due to stationarity we have that $\mathbb{P}(\mathcal{A}_1(t_0) \leq t)$ does not depend on $t_0$. On the other hand, $\mathbb{P}(\hsa_1(t_0) \leq t)$ depends both on $\varpi$ (the boundary parameter) and $t_0$ (the spatial location).

Suppose now that $\mathcal{L} = \{\mathcal{L}_i\}_{i \geq 1}$ is the line ensemble as in the right side of (\ref{PALE}) with $\mathcal{A}_i$ replaced with $\hsa_i$. Then, this new ensemble satisfies the same Brownian Gibbs property as $\mathcal{L}^{\mathrm{pAiry}}$, whenever $[a,b] \subset [0,\infty)$. However, near the origin the ensemble $\mathcal{L}$ satisfies a {\em new} kind of Gibbs property, which we refer to as the {\em half-space Brownian Gibbs property}. We explain this property informally in the next several lines, and refer the interested reader to Definition \ref{def:BGP} for the precise formulation. For any $k \in \mathbb{N}$ and real $b > 0$, one has that the conditional law of $\left( \mathcal{L}_i(t): i = 1, \dots, k, t \in [0,b] \right),$ given that $\mathcal{L}_1(b) = y_1, \dots, \mathcal{L}_k(b) = y_k$, and $\mathcal{L}_{k+1}[0,b] = g$ (for a fixed continuous $g$ on $[0,b]$), is the same as that of $k$ independent reverse Brownian motions $\cev{B}_i$ from $(b,y_i)$ with drifts $\mu_i = (-1)^{i} \sqrt{2} \cdot \varpi$, which have been conditioned to avoid each other and also the graph of the function $g$, see the right side of Figure \ref{S11}. Here, if $W_t$ is a standard Brownian motion, we have that $B(t) = y + W_t + \mu t$ is a Brownian motion with drift $\mu$ that is started from $y$, and the reverse Brownian motion from $(b,y)$ with drift $\mu$ is $\cev{B}(t) = B(b-t)$. If $\varpi = 0$, we note that all the drifts are equal to zero, and so consecutive curves do not interact, except for avoiding each other. In general, we have that $\mu_1 = \mu_3 = \mu_5 = \cdots = -\sqrt{2} \varpi$ and $\mu_2 = \mu_4 = \mu_6 = \cdots = \sqrt{2} \varpi$, so that curves with odd index have one drift, and curves with even index have the opposite drift. If $\varpi > 0$, then odd curves have a negative drift, while even ones have a positive drift causing an attraction between curves $2i - 1$ and $2i$. Conversely, if $\varpi < 0$, then odd curves have a positive drift, while even ones have a negative drift causing a repulsion between curves $2i - 1$ and $2i$.\\

In view of the preceding discussion, the half-space Airy line ensemble naturally fits into the broader framework of half-space Gibbsian line ensembles. Roughly speaking, these are collections of random walks or Brownian motions with model-dependent local interactions, influenced by boundary effects. In terms of previous works on half-space Gibbsian line ensembles, our work is closest to \cite{BCD24}, and here we explain how our results compare to the ones in that paper. Firstly, in \cite{BCD24} the authors consider a sequence, indexed by $N$, of discrete Gibbsian line ensembles that arise in the half-space log-gamma polymer model. The weights of their model depend on a real parameter $\theta > 0$ away from the origin, and a parameter $\alpha > -\theta$ at the origin. The authors then proceed to show that the lowest-indexed curves of their ensembles are tight when $\alpha > 0$ is fixed and when $\alpha = \mu N^{-1/3}$ is critically scaled to converge to zero. Here, the parameter $\mu \in \mathbb{R}$ is fixed and corresponds to our parameter $\varpi$. The results in \cite{BCD24} are the first ones to obtain the transversal $2/3$ critical exponent for a positive temperature half-space model in the KPZ universality class, and the $1/3$ vertical critical exponent away from the origin. We mention also that \cite{DS24} goes beyond \cite{BCD24} in that the authors leverage the Gibbsian line ensemble structure to study the behavior of the lowest-indexed three curves of the log-gamma line ensemble, but the limit they consider is very different from that in \cite{BCD24} and the present paper, and the comparison to their work less direct. 

In the present paper we consider instead the half-space LPP with geometric weights. This model has the distinct advantage of being exactly solvable. In particular, it has an interpretation as a {\em Pfaffian (half-space) Schur process} from \cite{BR05}, which in turn is a Pfaffian analogue of the determinantal Schur processes in \cite{OR03}. This interpretation gives us access to exact formulas, which are suitable for asymptotic analysis, and allow us to show that the finite-dimensional distributions of our model converge to those of the limiting ensemble. In addition, the structure of skew Schur polynomials naturally give an interpretation of our model as a discrete Gibbsian line ensemble, whose curves behave as {\em interlacing} geometric random walkers. The latter allows us to control the modulus of continuity of our ensembles by comparing them to avoiding Brownian motions, which ultimately yields a functional limit for the full ensembles. Our results go substantially beyond \cite{BCD24} in that we show that all the curves (as opposed to just the lowest-indexed one) have subsequential limits. In addition, from the finite-dimensional convergence we can conclude convergence (as opposed to tightness) of the full ensembles.

The formulas available for Pfaffian Schur processes provide much more detailed information than what is currently available for the half-space log-gamma polymer, which is why our results are considerably sharper than those in \cite{BCD24}. Despite having access to exact formulas there are serious challenges that we needed to overcome to establish our results. We will discuss our methods in more detail in Section \ref{Section1.3} after we have stated our main theorem in the next section, but here we point out one important challenge that substantially altered our approach compared to what is usually done in the full-space setting. When showing finite-dimensional convergence for a sequence of determinantal full-space models to the Airy line ensemble, one usually shows that the joint cumulative distribution functions (cdfs) converge. The latter is possible as the joint cdfs can be written as Fredholm determinants, whose convergence is established by showing that the correlation kernels of the processes converge, together with various technical estimates. In our setup, the joint cdfs are expressible in terms of Fredholm Pfaffians, which one can also show converge, with substantial care. One immediate issue we are facing is that such a convergence {\em a priori} only establishes {\em vague} as opposed to {\em weak convergence}. In the full space setting this issue is resolved by directly showing that the limits are actually probability distributions (i.e. have total mass one). For the Tracy-Widom $F_2$ distribution this is shown by finding an alternative expression of $F_2$ in terms of the solution to the Painlev{\'e} equation of type II, see \cite{TWPaper}. In our context we do not have such an alternative expression for the limiting Fredholm Pfaffians and so we need to proceed differently.

The approach to establishing the finite-dimensional convergence in our setup boils down to showing three key statements. The first is that our correlation kernels converge -- this ensures our models converge {\em on the level of point processes}. The second is to show that our kernels have well-behaved upper-tail behavior -- this ensures tightness from above. Both of these statements are established using detailed asymptotic analysis and the {\em method of steepest descent}. The third (and trickiest) statement we require is that our limiting point processes have almost surely infinitely many atoms. Having infinitely many atoms in the limit can be used to ensure tightness from below, which together with tightness from above improves the point process convergence to a finite-dimensional one. This third statement (i.e. that the Airy point process is almost surely infinite) is well-known in the full-space setting, but not in our setup. To show this statement for the half-space setting, we utilize the Gibbsian line ensemble structure, and effectively couple our ensembles at finite time with time ``infinity''. As we increase time, we move away from the boundary $0$, and our half-space models start looking like full-space ones, which allows us to conclude the presence of infinitely many atoms. The Gibbs property then allows us to transfer that information from time ``infinity'' to any finite time, and conclude presence of infinitely many atoms for each fixed time. The use of the Gibbsian line ensemble structure to transfer information from one time to another is not new; however, to our knowledge this is the first time where one can actually recover information all the way from infinity. From a technical perspective, this argument helps us circumvent the vague vs weak convergence issue mentioned earlier, and avoid doing lengthy computations involving Fredholm Pfaffians.

We finally mention that the existence of the half-space Airy line ensembles was discussed near the end of \cite[Section 1.2]{BCD24}, and our construction is very much in the spirit of that section. The ensembles $\hsa$ should correspond to the subsequential limits in \cite[Theorem 1.1(2)]{BCD24}, i.e. when $\alpha = \mu N^{-1/3}$ with the parameter $\mu$ corresponding to $\varpi$ in our setting. The subsequential limit in \cite[Theorem 1.1(1)]{BCD24}, i.e. when $\alpha > 0$ is fixed, should correspond to a limiting version of our ensembles when $\varpi \rightarrow \infty$. Recalling that in our model odd curves have drifts $-\sqrt{2} \varpi$ and even ones have drifts $\sqrt{2} \varpi$ near the origin, we expect that as $\varpi \rightarrow \infty$ the first and second curve will collide at the origin, as will the third and fourth etc. This should lead to a Gibbsian line ensemble whose curves of index $2i-1$ and $2i$ are pinched at the origin, see \cite[Figure 2(B)]{BCD24}. We seek to formally construct this ensemble in a subsequent work and describe its finite-dimensional distribution as well as the nature of its half-space Gibbs property.

%
%
\subsection{Main result}\label{Section1.2} We start by formally introducing the crossover kernels $\kcr$ that are the half-space analogues of the extended Airy kernel from (\ref{S1AiryKer}).

\begin{definition}\label{S1Contours} For a fixed $z \in \mathbb{C}$ and $\varphi \in (0, \pi)$, we denote by $\mathcal{C}_{z}^{\varphi}=\{z+|s|e^{\mathrm{sgn}(s)\im\varphi}, s\in \mathbb{R}\}$ the infinite contour oriented from $z+\infty e^{-\im\varphi}$ to $z+\infty e^{\im\varphi}$. 
\end{definition}

\begin{definition}\label{def:kcr}
We fix $\varpi \in \mathbb{R}$, $m \in \mathbb{N}$ and $\mathcal{T} = \{t_1, \dots, t_m\}$, where $0 \leq t_1 < t_2 < \cdots < t_m$, and set $\ap_i = |\varpi| + 3i$ for $i = 1,2,3$.

For $s,t \in \mathcal{T}$ and $x, y \in \mathbb{R}$ we define the {\em crossover kernel} with parameter $\varpi$ via
\begin{equation}\label{eq:S1DefKcross}
\begin{split}
&\kcr(s,x; t,y) = \begin{bmatrix}
    \kcr_{11}(s,x;t,y) & \kcr_{12}(s,x;t,y)\\
    \kcr_{21}(s,x;t,y) & \kcr_{22}(s,x;t,y) 
\end{bmatrix} \\
&= \begin{bmatrix}
    \icr_{11}(s,x;t,y) & \icr_{12}(s,x;t,y) + \rcr_{12}(s,x;t,y) \\
    -\icr_{12}(t,y;s,x) - \rcr_{12}(t,y;s,x) & \icr_{22}(s,x;t,y) + \rcr_{22}(s,x;t,y) 
\end{bmatrix} ,
\end{split}
\end{equation}
where the kernels $\icr_{ij}, \rcr_{ij}$ are defined as follows. We have
\begin{equation}\label{S1DefIcross}
\begin{split}
\icr_{11}(s,x;t,y) = &\frac{1}{(2\pi \im)^2} \int_{\mathcal{C}_{\ap_1 - s}^{\pi/3}}dz \int_{\mathcal{C}_{\ap_3 - t}^{\pi/3}} dw \frac{z + s - w - t}{(z + s + w + t)(z+s)(w+ t)} \\
& \times (z + s+ \varpi )(w + t + \varpi) \cdot e^{z^3/3 + w^3/3 - xz - y w}, \\
\icr_{12}(s,x;t,y) = &\frac{1}{(2\pi \im)^2} \int_{\mathcal{C}_{\ap_3 - s}^{\pi/3}}dz \int_{\mathcal{C}_{\ap_1 - t}^{2\pi/3}} dw \frac{z + s + w + t}{2(z+ s)(z + s - w -t)} \\
& \times \frac{z + \varpi + s}{w + \varpi + t} \cdot e^{z^3/3 - w^3/3 - xz + y w},\\
\icr_{22}(s,x;t,y) = &\frac{1}{(2\pi \im)^2} \int_{\mathcal{C}_{-\ap_2 - s}^{2\pi/3}}dz \int_{\mathcal{C}_{-\ap_2 - t}^{2\pi/3}} dw \frac{z + s - w - t}{4 (z + s + w + t)} \\
& \times \frac{1}{(z + s  + \varpi )(w + t + \varpi)} \cdot e^{-z^3/3 - w^3/3 + xz + y w}.
\end{split}
\end{equation}
We also have
\begin{equation}\label{S1DefRcross}
\begin{split}
\rcr_{12}(s,x;t,y) = & - \frac{{\bf 1}\{s < t\} }{\sqrt{4\pi (t-s)}} \cdot \exp \left(\frac{- (s-t)^4 + 6 (x+ y)(s-t)^2 + 3 (x-y)^2}{12 (s-t)} \right),
\end{split}
\end{equation}
$\rcr_{22}(s,x; t,y) = - \rcr_{22}(t,y; s,x) $, and when $x- s^2 > y -t^2$ we have
\begin{equation}\label{S1DefRcross2}
\begin{split}
\rcr_{22}(s,x;t,y) = & \frac{1}{2\pi \im} \int_{\mathcal{C}_{\ap_1 }^{2\pi/3}}dz \frac{e^{(-z + s)^3/3 + (\varpi + t)^3/3 - x (-z+s) - y (\varpi + t)}}{4(z - \varpi)}\\
&- \frac{1}{2\pi \im}\int_{\mathcal{C}_{-\ap_2 }^{2\pi/3}}dw \frac{e^{(-w + t)^3/3 + (\varpi + s)^3/3 - y (-w+t) - x (\varpi + s)}}{4(w - \varpi)}\\
& + \frac{{\bf 1}\{s + t > 0\} }{2\pi \im} \int_{\mathcal{C}^{2\pi/3}_{-\ap_{2}}} dw \frac{we^{ (-w + t)^3/3 + (w + s)^3/3   - y (-w + t) - x (w+s)}   }{2(w- \varpi)(w + \varpi) }.
\end{split}
\end{equation}
\end{definition}
\begin{remark}\label{S1Correction} The formula for $\kcr$ in Definition \ref{def:kcr} was previously obtained in \cite[Section 2.5]{BBCS} and \cite[Section 5.1]{BBCS2}. As we discovered in the process of writing the present paper, the formulas in \cite{BBCS,BBCS2} have a few typos. These have since been corrected in the arXiv versions of these papers -- see \cite[Section 2.5]{BBCSArxiv} and \cite[Section 5.1]{BBCS2Arxiv}. We mention that our formulas for $I_{11}, I_{12}$ and $R_{12}$ completely agree (after a simple change of variables) with the corrected ones in \cite[Section 2.5]{BBCSArxiv}. As the authors chose slightly different contours in the definitions of $I_{22}$ and $R_{22}$ than ours, these terms do not match; however, their sum (which is precisely $K_{22}$) is the same as ours. In other words, our kernel $\kcr$ matches that in \cite[Section 2.5]{BBCSArxiv}. We also mention that in \cite[Equation (4.10)]{BBNV} the authors introduce a kernel $K^v$, which also matches with our kernel after several (tedious) changes of variables, including setting $\varpi = 2 v$, $s = u_a$, $t = u_b$, $x = \xi - u_a^2$, $y = \xi' - u_b^2$, $z = Z -s$, $w = W-t$, and conjugating the kernel.
\end{remark}

The main result of the paper is as follows.
\begin{theorem}\label{thm:MainThm1} Fix $\varpi \in \mathbb{R}$. There exists a line ensemble $\hsa = \{\hsa_i\}_{i \geq 1}$ on $[0, \infty)$ that satisfies the following properties. Firstly, the ensemble is non-intersecting, meaning that almost surely
\begin{equation}\label{Eq.OrdHSA}
\hsa_i(t) > \hsa_{i+1}(t) \mbox{ for all } i \in \mathbb{N}, t \in [0,\infty).
\end{equation}
In addition, for each $m \in \mathbb{N}$, $s_1, \dots, s_m \in [0,\infty)$ with $s_1 < s_2 < \cdots < s_m$, we have that the random measure
\begin{equation}\label{eq:HSA point process}
M(\omega, A) = \sum_{i \geq 1} \sum_{j = 1}^m {\bf 1}\{(s_j, \hsa_i(s_j,\omega) ) \in A\}
\end{equation}
is a Pfaffian point process on $\mathbb{R}^2$, with correlation kernel $\kcr$ as in (\ref{eq:S1DefKcross}) and reference measure $\mu_{\mathsf{S}} \times \lambda$, where $\mu_{\mathsf{S}}$ is the counting measure on $\mathsf{S} = \{s_1, \dots, s_m\}$, and $\lambda$ is the Lebesgue measure on $\mathbb{R}$. In addition, if $\mathcal{L} = \{\mathcal{L}_i\}_{i \geq 1}$ is the line ensemble defined through 
\begin{equation}\label{eq:Parabolic HSA}
\sqrt{2} \cdot \mathcal{L}_i(t) + t^2 = \hsa_i(t) \mbox{ for } i \geq 1, t \in [0,\infty),
\end{equation}
then $\mathcal{L}$ satisfies the half-space Brownian Gibbs property from Definition \ref{def:BGP} with parameters $\mu_i = (-1)^{i} \sqrt{2} \cdot \varpi$.
\end{theorem}
\begin{remark}\label{S1Unique} Theorem \ref{thm:MainThm1} states that there {\em exists} a non-intersecting line ensemble $\hsa$, whose associated point processes $M$ are Pfaffian with correlation kernel $\kcr$ and reference measures $\mu_{\mathsf{S}} \times \lambda$. However, there is at most one non-intersecting line ensemble that satisfies this property as we explain next, i.e. the line ensemble $\hsa$ in Theorem \ref{thm:MainThm1} is {\em unique}. Indeed, suppose $\tilde{\mathcal{A}}$ is another such ensemble and let $\tilde{M}$ be as in (\ref{eq:HSA point process}) with $\hsa$ replaced with $\tilde{\mathcal{A}}$. Then both $M$ and $\tilde{M}$ are Pfaffian point processes with the same kernel and reference measure; hence $M \overset{d}{=} \tilde{M}$ (as random locally bounded measures on $\mathbb{R}^2$) by Proposition \ref{prop:basic properties Pfaffian point process}(3). From \cite[Corollary 2.20]{dimitrov2024airy}, we conclude $(\hsa_i(s_j): i \geq 1, j = 1, \dots, m) \overset{d}{=} (\tilde{\mathcal{A}}_i(s_j): i \geq 1, j = 1, \dots, m)$ (as random vectors in $\mathbb{R}^{\infty}$). As finite-dimensional sets form a separating class, see \cite[Lemma 3.1]{DimMat} with $\Lambda=[0,\infty)$, we conclude $\hsa \overset{d}{=} \tilde{\mathcal{A}}$ (as line ensembles or equivalently random elements in $C(\mathbb{N}\times[0,\infty)$).
\end{remark}

\begin{remark}\label{S1BGP} It follows from Lemma \ref{LemmaConsistent} that $\mathcal{L}$ also satisfies the Brownian Gibbs property as in \cite[Definition 2.2]{CorHamA} on the interval $[0,\infty)$, see also \cite[Definition 2.7]{DimMat}.
\end{remark}

\begin{remark}\label{S1MainRes} While we view Theorem \ref{thm:MainThm1} as the main result of the paper, it is worth mentioning that in its proof we construct $\hsa$ as the weak limit of a sequence of discrete line ensembles that arise in the Pfaffian Schur processes introduced in \cite{BR05}, which are in turn intimately related to geometric LPP in a half-space. Consequently, an important contribution of the paper is establishing not just the existence of the limiting object but showing that half-space geometric LPP converges to it, which can be thought as a strengthening of the results in \cite{BBNV}. We believe that the arguments of the present paper can also be adapted to exponential LPP in a half-space, corresponding to a strengthening of \cite{BBCS}, but we leave this for future work.
\end{remark}

\begin{remark}\label{S1Characterize} As mentioned in Section \ref{Section1.1}, \cite{AH23} shows that $\mathcal{L}^{\mathrm{pAiry}}$ is (up to an independent affine shift) the unique line ensemble that possesses the Brownian Gibbs property, and whose top curve globally looks like $-2^{-1/2} t^2$. Although we do not show this directly, Lemma \ref{lem:ConvToAiryKernel} in the text can be used to show that the curves $\mathcal{L}_i(t)$ from Theorem \ref{thm:MainThm1} for each fixed index $i$ look like $-2^{-1/2} t^2$ as $t \rightarrow \infty$. It would be interesting to see if one can establish an analogous characterization of $\mathcal{L}$ as the unique ensemble which is globally parabolic and satisfies the half-space Brownian Gibbs property.
\end{remark}

%
%
\subsection{Ideas behind the proof and paper outline}\label{Section1.3} The way we construct the half-space Airy line ensembles $\hsa$ from Theorem \ref{thm:MainThm1} is as weak limits of certain discrete line ensembles that arise in the Pfaffian Schur processes from \cite{BR05}. Recall that a {\em partition} $\lambda = (\lambda_1, \lambda_2, \dots)$ is a sequence of decreasing non-negative integers that are eventually zero. The Pfaffian Schur processes we consider are measures $\mathbb{P}_N$ on sequences of partitions $\lambda^0, \lambda^1, \lambda^2, \dots $. By defining $L_i(t) = \lambda_i^t$ and linearly interpolating these points for non-integer $t \geq 0$, one obtains a discrete line ensemble $\mathfrak{L} = \{L_i\}_{i = 1}^{\infty}$, see Definition \ref{DefDLE}, where $L_i: \mathbb{Z}_{\geq 0} \rightarrow \mathbb{Z}$ satisfy
\begin{equation}\label{S1E1}
L_i(s) \leq L_{i} (s+1) \mbox{ for all } s \in \mathbb{Z}_{\geq 0}.
\end{equation}
In words, equation (\ref{S1E1}) says that one can think of $L_i$ as the trajectories of geometric random walkers. In addition, our ensemble $\mathfrak{L}$ satisfies the property that for each $s \in \mathbb{Z}_{\geq 0}$ the vectors $\mathfrak{L}(s) = (L_1(s), L_2(s), \dots)$ and $\mathfrak{L}(s+1) = (L_1(s+1), L_2(s+1), \dots)$ {\em interlace} in the sense that 
\begin{equation}\label{S1E2}
L_1(s+1) \geq L_1(s) \geq L_2(s+1) \geq L_2(s) \geq L_3(s+1) \geq L_3(s) \geq \cdots.
\end{equation}
A visual depiction of $\mathfrak{L}$ is given in Figure \ref{S1_3}.
\begin{figure}[ht]
	\begin{center}
		\begin{tikzpicture}[scale=0.7]
		\begin{scope}
        \def\r{0.1}
		\draw[dotted, gray] (0,0) grid (8,6);

        \draw[fill = black] (0,1) circle [radius=\r];
        \draw[fill = black] (1,4) circle [radius=\r];
        \draw[fill = black] (2,5) circle [radius=\r];
        \draw[fill = black] (3,5) circle [radius=\r];
        \draw[fill = black] (4,5) circle [radius=\r];
        \draw[fill = black] (5,5) circle [radius=\r];
        \draw[fill = black] (6,6) circle [radius=\r];
        \draw[fill = black] (7,6) circle [radius=\r];
        \draw[fill = black] (8,6) circle [radius=\r];
        \draw[-][black] (0,1) -- (1,4);
        \draw[-][black] (1,4) -- (2,5);
        \draw[-][black] (2,5) -- (3,5);
        \draw[-][black] (3,5) -- (4,5);
        \draw[-][black] (4,5) -- (5,5);
        \draw[-][black] (5,5) -- (6,6);
        \draw[-][black] (6,6) -- (7,6);
        \draw[-][black] (7,6) -- (8,6);

        \draw[fill = black] (0,1) circle [radius=\r];
        \draw[fill = black] (1,1) circle [radius=\r];
        \draw[fill = black] (2,2) circle [radius=\r];
        \draw[fill = black] (3,3) circle [radius=\r];
        \draw[fill = black] (4,3) circle [radius=\r];
        \draw[fill = black] (5,3) circle [radius=\r];
        \draw[fill = black] (6,3) circle [radius=\r];
        \draw[fill = black] (7,4) circle [radius=\r];
        \draw[fill = black] (8,4) circle [radius=\r];
        \draw[-][black] (0,1) -- (1,1);
        \draw[-][black] (1,1) -- (2,2);
        \draw[-][black] (2,2) -- (3,3);
        \draw[-][black] (3,3) -- (4,3);
        \draw[-][black] (4,3) -- (5,3);
        \draw[-][black] (5,3) -- (6,3);
        \draw[-][black] (6,3) -- (7,4);
        \draw[-][black] (7,4) -- (8,4);

        \draw[fill = black] (0,0) circle [radius=\r];
        \draw[fill = black] (1,0) circle [radius=\r];
        \draw[fill = black] (2,1) circle [radius=\r];
        \draw[fill = black] (3,1) circle [radius=\r];
        \draw[fill = black] (4,3) circle [radius=\r];
        \draw[fill = black] (5,3) circle [radius=\r];
        \draw[fill = black] (6,3) circle [radius=\r];
        \draw[fill = black] (7,3) circle [radius=\r];
        \draw[fill = black] (8,3) circle [radius=\r];
        \draw[-][black] (0,0) -- (1,0);
        \draw[-][black] (1,0) -- (2,1);
        \draw[-][black] (2,1) -- (3,1);
        \draw[-][black] (3,1) -- (4,3);
        \draw[-][black] (4,3) -- (5,3);
        \draw[-][black] (5,3) -- (6,3);
        \draw[-][black] (6,3) -- (7,3);
        \draw[-][black] (7,3) -- (8,3);

        \draw (8.5, 6) node{$L_1$};
        \draw (8.5, 4) node{$L_2$};
        \draw (8.5, 3) node{$L_3$};

		\end{scope}

		\end{tikzpicture}
	\end{center}
	\caption{The figure depicts the top three curves in $\mathfrak{L}$, satisfying (\ref{S1E1}) and (\ref{S1E2}).}
	\label{S1_3}
\end{figure}
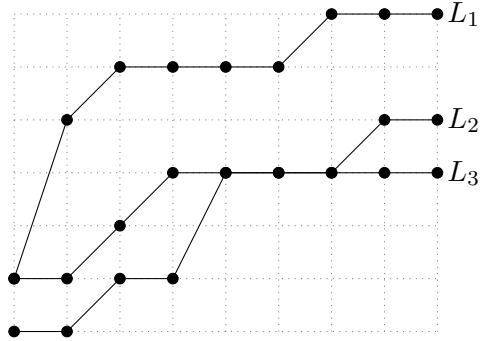

The measures $\mathbb{P}_N$ we consider depend on two parameters $q$ (which can be thought of as the weight of the underlying geometric walkers), and $c$ (which reflects the boundary interaction of the curves in $\mathfrak{L}$ at the origin). The scaling we consider involves sending $N \rightarrow \infty$, while keeping $q \in (0,1)$ fixed and $c = 1 - \mu N^{-1/3} + o(N^{-1/3})$ for some $\mu$ that is proportional to the boundary parameter $\varpi$ in the definition of $\hsa$ in Theorem \ref{thm:MainThm1}. If one appropriately vertically shifts the curves in $\mathfrak{L}$, and performs a horizontal scaling by $N^{2/3}$ and a vertical one by $N^{1/3}$, one obtains a sequence of line ensembles $\mathcal{L}^N$ that weakly converge to the ensemble $\mathcal{L}$ from Theorem \ref{thm:MainThm1}. In order to establish this convergence and deduce all the properties in the theorem, we need to show that:
\begin{enumerate}
    \item[I.] the ensembles $\mathcal{L}^N$ converge in the finite-dimensional sense;
    \item[II.] the ensembles $\mathcal{L}^N$ have tight curves in the space of continuous functions, and all subsequential limits satisfy the half-space Brownian Gibbs property.
\end{enumerate}

Let us first focus on point II above. As mentioned earlier, the $L_i(s)$ for increasing $s$ can be thought of as geometric random walkers that are conditioned to interlace. If instead one looks at $L_i(T-s)$ for $s = 0, 1, \dots, T$, then these curves have an interpretation as reverse geometric walkers with weights $q_i = c^{(-1)^i} \cdot q$. In other words, the interaction at the origin for the different curves can be absorbed into the weights of the different walkers with all even walkers having weight $cq$ and all odd ones having weight $c^{-1} q$. This interpretation of the boundary interaction follows directly from the formula for Pfaffian Schur processes in terms of Schur symmetric functions, see Lemma \ref{lem:SchurGibbs}. In Section \ref{Section2} we make a formal definition of this structural property, and refer to it as the {\em half-space interlacing Gibbs property}. We mention that our definition is made in considerable generality allowing arbitrary weights $\{q_i\}_{i \geq 1}$ (as opposed to $q_i = c^{(-1)^i} \cdot q$). 

In Section \ref{Section2} we formally define half-space continuous and discrete line ensembles, and their Gibbs properties. In Section \ref{Section3} we derive a few structural properties of our discrete ensembles that are based on a monotone coupling of these ensembles in their boundary data from Lemma \ref{lem:monotone coupling} and on a strong comparison between geometric random walks and Brownian motions from Lemma \ref{prop:ThmA Shao} (the latter is a version of the celebrated Koml{\' o}s-Major-Tusn{\' a}dy (KMT) coupling, \cite{KMT1,KMT2}). The main technical result we establish about our discrete line ensembles is contained in Theorem \ref{thm:main thm tightness} in Section \ref{Section4}, which shows that a sequence of discrete ensembles is tight (in the space of continuous functions) under the assumption of one-point tightness of its curves. Theorem \ref{thm:main thm tightness} further shows that, under $N^{2/3}$ horizontal and $N^{1/3}$ vertical scalings of the ensembles, together with our scalings of the parameters $q$ and $c$, the half-space interlacing Gibbs property becomes the half-space Brownian Gibbs property. We mention that our proof of Theorem \ref{thm:main thm tightness} relies considerably on the results from \cite{dimitrov2024tightness}, which allow us to deal with our ensemble away from the origin where they behave exactly as full-space models. The main contribution of the present work is to demonstrate how to deal with the boundary effect in the half-space setting.

If we can establish the finite-dimensional convergence of $\mathcal{L}^N$, then that in particular implies the one-point tightness required to apply Theorem \ref{thm:main thm tightness}, and so we can conclude point II above from point I. It is worth mentioning that the formulation of Theorem \ref{thm:main thm tightness} is quite general -- it allows for considerable parameter freedom and can be applied to finite/infinite ensembles on finite/semi-infinite intervals. There are a few reasons for this. First of all, this theorem should directly be applicable to the setting of Schur processes with two open boundaries, see \cite{BBNV}, or to the construction of half-space analogues of the {\em Airy wanderer line ensembles}, see \cite{dimitrov2024airy}. Secondly, one can try to analyze the limits of generic half-space discrete Gibbsian line ensembles with arbitrary parameters $\{q_i\}_{i \geq 1}$ (and not just $q_i = c^{(-1)^i} \cdot q$), and construct large families of half-space continuous ensembles that away from the origin satisfy the Brownian Gibbs property, but whose behavior at the origin is much more complex than that of $\mathcal{L}$ from Theorem \ref{thm:MainThm1}. Lastly, we hope that the results in Sections \ref{Section2}, \ref{Section3} and \ref{Section4} can serve as a template for the asymptotic analysis of various other models in the half-space KPZ class such as the half-space log-gamma polymer and the half-space exponential LPP.\\

In the remainder of this section we discuss how we prove the finite-dimensional convergence of $\mathcal{L}^N$ in point I above. The key ingredient here is the Pfaffian point process structure of the Pfaffian Schur process, which was shown in \cite{BR05}, see Lemma \ref{lem:PSP as PPP} for a precise statement. In words, we have that for appropriately chosen $s_1, \dots, s_m \in [0,\infty)$ with $s_1 < s_2 < \cdots < s_m$ the random measure
\begin{equation*}
M^N(\omega, A) = \sum_{i \geq 1} \sum_{j = 1}^m {\bf 1}\{(s_j, \mathcal{L}^N_i(s_j,\omega) ) \in A\}
\end{equation*}
is a Pfaffian point process, and its correlation kernel $K^N$ is given by a double contour integral. The finite-dimensional convergence then goes through showing the following statements. 
\begin{enumerate}
    \item[IA.] The kernels $K^N$ converge uniformly to a limiting kernel $K^{\infty}$. This step ensures that the $M^N$ converge weakly to a point process $M^{\infty}$, which is also Pfaffian with kernel $K^{\infty}$.
    \item[IB.] For a fixed $t \geq 0$ the sequence $(\mathcal{L}^N_1(t))^+ = \max (0, \mathcal{L}_1^N(t))$ is tight.
    \item[IC.] The measures $M^{\infty}$ from IA above satisfies for each $i =1, \dots, m$
\begin{equation}\label{S1InfMass}
M^{\infty}(\{ s_i \} \times \mathbb{R}) = \infty \mbox{ a.s.}.
\end{equation}
\end{enumerate}
The fact that the above three statements suffice for establishing finite-dimensional convergence follows from \cite[Proposition 2.21]{dimitrov2024airy}, see Section \ref{Section6.4} for the details. Here, we mention that statements IB and IC improve the weak convergence for the point processes from statement IA to a finite-dimensional one by together implying tightness.

We rewrite the formula for $K^N$ in a way that is suitable for asymptotic analysis in Lemma \ref{lem:PrelimitKernel} and show that the latter converges uniformly in Lemma \ref{lem:kernelLimits} using the method of steepest descent. This establishes point IA above. We mention that for a slightly different choice of specializations of the Schur process a similar analysis was carried out in \cite{BBNV} although our choice of contours is considerably simpler. Point IB above is established in Lemma \ref{lem:TightFromAbove} and is based on detailed upper-tail estimates for the correlation kernels $K^N$.

In order to show IC above we utilize the half-space interlacing Gibbs property of our discrete ensembles. Specifically, we demonstrate that the number of atoms in any subsequential limit, i.e. $M^{\infty}(\{s\} \times \mathbb{R})$ as a function of $s$, needs to stochastically decrease. The latter is accomplished in Steps 2 and 3 of the proof of Lemma \ref{lem:InfiniteAtoms}. On the other hand, in the first step of the proof of Lemma \ref{lem:InfiniteAtoms} we show that $M^{\infty}(\{s\} \times \mathbb{R})$ converge to the Airy point process as $s \rightarrow \infty$. As the Airy point process contains almost surely infinitely many atoms, we conclude the same for $M^{\infty}(\{s\} \times \mathbb{R})$ for each fixed $s \in [0, \infty)$. At a high level, we utilize the half-space interlacing Gibbs property to couple our {\em finite-time} ensemble with its value at ``infinity'', and leverage some of the known properties of the Airy point process at this ``infinite'' time to deduce some finite-time properties for our ensembles. The precise monotonicity statements we establish to ensure (\ref{S1InfMass}) are Lemmas \ref{lem:technical lemma fdd 1} and \ref{lem:technical lemma fdd 2}.

We mention that point IC above is somewhat special to our present setup as we are constructing the half-space Airy line ensemble for the first time. Future works can use this statement freely, as it is merely a statement about the limiting object rather than the Pfaffian Schur processes that we study. In general, we hope that the framework we develop in this paper can be adapted to large classes of half-space models or other Pfaffian point processes to study their asymptotic behavior.

%
%
\subsection*{Acknowledgments}\label{Section1.4} The authors would like to thank Ivan Corwin and Sayan Das for many useful remarks on earlier drafts of the paper. ED was partially supported by the NSF grant DMS:2230262.  ZY was partially supported by Ivan Corwin's NSF grants DMS:1811143, DMS:2246576, Simons Foundation Grant 929852, and the Fernholz Foundation's `Summer Minerva Fellows' program.

%
%
\section{Half-space line ensembles}\label{Section2} In this section we introduce several core definitions related to line ensembles and their Gibbs properties. In Section \ref{Section2.1} we introduce continuous line ensembles and the {\em half-space Brownian Gibbs property}. In Section \ref{Section2.2} we introduce geometric line ensembles and their corresponding {\em half-space interlacing Gibbs property}. Finally, in Section \ref{Section2.3} we list several properties for geometric line ensembles, including a monotone coupling lemma, see Lemma \ref{lem:monotone coupling}, and a strong coupling lemma, see Lemma \ref{prop:ThmA Shao}. Some of our notation in this section is based on \cite[Section 2]{dimitrov2024tightness} and \cite[Section 2]{DEA21}, which in turn goes back to \cite{CorHamA}.

%
%
\subsection{The half-space Brownian Gibbs property}\label{Section2.1} Given two integers $a \leq b$, we let $\llbracket a, b \rrbracket$ denote the set $\{a, a+1, \dots, b\}$. We also set $\llbracket a,b \rrbracket = \emptyset$ when $a > b$, $\llbracket a, \infty \rrbracket = \{a, a+1, a+2 , \dots \}$, $\llbracket - \infty, b\rrbracket = \{b, b-1, b-2, \dots\}$ and $\llbracket - \infty, \infty \rrbracket = \mathbb{Z}$. Given an interval $\Lambda \subseteq \mathbb{R}$, we endow it with the subspace topology of the usual topology on $\mathbb{R}$. We let $(C(\Lambda), \mathcal{C})$ denote the space of continuous functions $f: \Lambda \rightarrow \mathbb{R}$ with the topology of uniform convergence over compact sets, see \cite[Chapter 7, Section 46]{Munkres}, and Borel $\sigma$-algebra $\mathcal{C}$. Given a set $\Sigma \subseteq \mathbb{Z}$, we endow it with the discrete topology and denote by $\Sigma \times \Lambda$ the set of all pairs $(i,x)$ with $i \in \Sigma$ and $x \in \Lambda$ with the product topology. We also denote by $\left(C (\Sigma \times \Lambda), \mathcal{C}_{\Sigma}\right)$ the space of real-valued continuous functions on $\Sigma \times \Lambda$ with the topology of uniform convergence over compact sets and Borel $\sigma$-algebra $\mathcal{C}_{\Sigma}$. We typically take $\Sigma = \llbracket 1, N \rrbracket$ with $N \in \mathbb{N} \cup \{\infty\}$. The following defines the notion of a line ensemble.
\begin{definition}\label{CLEDef}
Let $\Sigma \subseteq \mathbb{Z}$ and $\Lambda \subseteq \mathbb{R}$ be an interval. A {\em $\Sigma$-indexed line ensemble $\mathcal{L}$} is a random variable defined on a probability space $(\Omega, \mathcal{F}, \mathbb{P})$ that takes values in $\left(C (\Sigma \times \Lambda), \mathcal{C}_{\Sigma}\right)$. Intuitively, $\mathcal{L}$ is a collection of random continuous curves (sometimes referred to as {\em lines}), indexed by $\Sigma$,  each of which maps $\Lambda$ in $\mathbb{R}$. We will often slightly abuse notation and write $\mathcal{L}: \Sigma \times \Lambda \rightarrow \mathbb{R}$, even though it is not $\mathcal{L}$ which is such a function, but $\mathcal{L}(\omega)$ for every $\omega \in \Omega$. For $i \in \Sigma$ we write $\mathcal{L}_i(\omega) = (\mathcal{L}(\omega))(i, \cdot)$ for the curve of index $i$ and note that the latter is a map $\mathcal{L}_i: \Omega \rightarrow C(\Lambda)$, which is $\mathcal{F}/\mathcal{C}$ measurable. If $a,b \in \Lambda$ satisfy $a \leq b$, we let $\mathcal{L}_i[a,b]$ denote the restriction of $\mathcal{L}_i$ to $[a,b]$. We call a line ensemble {\em non-intersecting} if $\mathbb{P}$-almost surely $\mathcal{L}_i(r) > \mathcal{L}_j(r)$  for all $i < j$ and $r \in \Lambda$.
\end{definition}
\begin{remark}\label{RemPolish} As shown in \cite[Lemma 2.2]{DEA21}, we have that $C(\Sigma \times \Lambda)$ is a Polish space, and so a line ensemble $\mathcal{L}$ is just a random element in $C(\Sigma \times \Lambda)$ in the sense of \cite[Section 3]{Billing}.
\end{remark}

We let $W_t$ denote a standard one-dimensional Brownian motion, and if $y, \mu \in \mathbb{R}$, we define the {\em Brownian motion with drift $\mu$ from $B(0) = y$} by
\begin{equation}\label{eq:DefBrownianMotionDrift}
B(t) = y + W_t + \mu t.
\end{equation}
If $b \in (0,\infty)$ we also define the {\em reverse Brownian motion with drift $\mu$ from $\cev{B}(b) = y$} by
\begin{equation}\label{eq:RevDefBrownianMotionDrift}
\cev{B}(t) = B(b-t) \mbox{ for } 0 \leq t \leq b.
\end{equation}
We next state two lemmas about reverse Brownian motions, whose proofs follow directly from those of \cite[Corollaries 2.9 and 2.10]{CorHamA}, and are given in Sections \ref{SectionA2} and \ref{SectionA3}. 
\begin{lemma}\label{lem: BB touch intersect} 
Suppose $\mu, y \in \mathbb{R}, b > 0$ and $\cev{B}(t)$ is as in (\ref{eq:RevDefBrownianMotionDrift}). Suppose that $f \in C([0,b])$ is such that $f(b) < y$ and set $C=\{ \cev{B}(t)<f(t) \mbox{ for some } t\in[0,b] \}$ and $D=\{\cev{B}(t)=f(t) \mbox{ for some } t\in [0,b]\}$. Then, $\mathbb{P}(D\cap C^c)=0$.
\end{lemma}

\begin{lemma}\label{lem: BB positive measure}
Suppose $\mu, y \in \mathbb{R}, b > 0$ and $\cev{B}(t)$ is as in (\ref{eq:RevDefBrownianMotionDrift}). Suppose $U$ is an open subset of $C([0,b])$ which contains a function $f$ such that $f(b) = y$. Then, $\mathbb{P}(\cev{B} \in U) > 0$.
\end{lemma}

For $k\in\mathbb{N}$ we denote by $\weyl_k$ the open Weyl chamber in $\mathbb{R}^{k}$, i.e.
\begin{equation}\label{DefWeyl}
\weyl_k=\{\vec{x}=(x_1,\dots,x_k)\in\mathbb{R}^k: x_1>x_2>\cdots>x_k\}.
\end{equation}
We next define the $g$-avoiding reverse Brownian line ensembles.
\begin{definition}\label{def: avoidBLE} 
Suppose $\vec{y},\vec{\mu} \in \mathbb{R}^k$, $b > 0$ and $g:[0,b]\rightarrow[-\infty,\infty)$ is a continuous function (i.e., either $g \in C([0,b])$  or $g= -\infty$ everywhere) such that $g(b) < y_k$. We denote the law of $k$ independent reverse Brownian motions $\{\cev{B}_i\}_{i = 1}^k$ on $[0,b]$, such that $\cev{B}_i$ has drift $\mu_i$ and $\cev{B}_i(b) = y_i$ by $\pfbm^{b, \vec{y}, \vec{\mu}}$, and write $\efbm^{b, \vec{y}, \vec{\mu}}$ for the expectation with respect to this measure. If $\vec{y} \in \weyl_k$, we also let $\pabm^{b, \vec{y}, \vec{\mu}, g}$ be the law $\pfbm^{b, \vec{y}, \vec{\mu}}$, conditioned on the event 
$$E_{\operatorname{avoid}}=\left\{\cev{B}_1(r)>\cev{B}_2(r)>\cdots>\cev{B}_k(r)>g(r)\mbox{ for all } r\in[0,b] \right\},$$
and write $\eabm^{b, \vec{y}, \vec{\mu}, g}$ for the expectation with respect to this measure.
\end{definition}
\begin{remark}\label{rem:WellDAvoid} We mention that Lemma \ref{lem: BB positive measure} shows that $\pfbm^{b, \vec{y}, \vec{\mu}}(E_{\operatorname{avoid}}) > 0$, and hence $\pabm^{b, \vec{y}, \vec{\mu}, g}$ is well-defined, see \cite[Definition 2.4]{DimMat} for a similar argument.
\end{remark}

We next introduce the main definition of the section -- the half-space Brownian Gibbs property.
\begin{definition}\label{def:BGP}
Fix a set $\Sigma = \llbracket 1, N \rrbracket$ with $N \in \mathbb{N} \cup \{\infty\}$, an interval $\Lambda= [0,T]$ or $\Lambda = [0,T)$ with $T \in (0, \infty]$, and $\mu_i \in \mathbb{R}$ for $i \in \llbracket 1, N-1 \rrbracket$. For $k\in\llbracket 1,N-1\rrbracket$ and $b \in \Lambda \cap (0,\infty)$ we write $\vec{\mu}_k = (\mu_1, \dots, \mu_k)$, $D_{\llbracket 1,k\rrbracket,b} = \llbracket 1,k\rrbracket \times [0,b)$ and $D_{\llbracket 1,k\rrbracket,b}^c = (\Sigma \times \Lambda) \setminus D_{\llbracket 1,k\rrbracket,b}$. 
 
A $\Sigma$-indexed line ensemble $\mathcal{L}$ on $\Lambda$ satisfies the {\em half-space Brownian Gibbs property with parameters $\{\mu_i\}_{i \in \llbracket 1, N- 1 \rrbracket}$} if it is non-intersecting and for any $b \in \Lambda \cap (0,\infty)$, any $k\in\llbracket 1,N-1\rrbracket$, and any bounded Borel-measurable function $F: C(\llbracket 1,k\rrbracket \times [0,b]) \rightarrow \mathbb{R}$
\begin{equation}\label{eq:HSBGP}
\mathbb{E} \left[ F\left(\mathcal{L}|_{\llbracket 1,k\rrbracket \times [0,b]} \right) \,\big|\, \mathcal{F}_{\operatorname{ext}} (\llbracket 1,k\rrbracket \times [0,b))  \right] =\eabm^{b, \vec{y}, \vec{\mu}_k,g} \bigl[ F( \mathcal{Q} ) \bigr],\quad \mathbb{P}\text{-almost surely,} 
\end{equation}
where $\mathcal{L}|_{\llbracket 1,k\rrbracket \times [0,b]} $ is the restriction of $\mathcal{L}$ to $\llbracket 1, k \rrbracket \times [0,b]$, $g = \mathcal{L}_{k + 1}[0,b]$, $\vec y=(\mathcal{L}_1(b), \dots, \mathcal{L}_k(b))$, and
\[\mathcal{F}_{\operatorname{ext}} (\llbracket 1,k\rrbracket \times [0,b)) := \sigma \left ( \mathcal{L}_i(s): (i,s) \in D_{\llbracket 1,k\rrbracket, b}^c \right).
\] 
\end{definition}
\begin{remark}\label{RemMeas} It is perhaps worth explaining why equation (\ref{eq:HSBGP}) makes sense. Firstly, since $\Sigma \times \Lambda$ is locally compact, we know by \cite[Lemma 46.4]{Munkres} that $\mathcal{L} \rightarrow \mathcal{L}|_{\llbracket 1, k \rrbracket \times [0,b]}$ is a continuous map from $C(\Sigma \times \Lambda)$ to $C(\llbracket 1, k \rrbracket \times [0,b])$, so that the left side of (\ref{eq:HSBGP}) is the conditional expectation of a bounded measurable function, and is thus well-defined. A more subtle question is why the right side of (\ref{eq:HSBGP}) is $\mathcal{F}_{\operatorname{ext}} (\llbracket 1, k \rrbracket \times [0,b))$-measurable. In fact, we will show in Lemma \ref{LemmaMeasExp} in Section \ref{SectionA35} that the right side is a measurable function of $\vec{y}$ and $g$, i.e., it is measurable with respect to the $\sigma$-algebra 
$$ \sigma \left\{ \mathcal{L}_i(s) : \mbox{  $i \in \llbracket 1, k \rrbracket$ and $s = b$, or $i = k + 1$ and $s \in [0,b]$} \right\}.$$
\end{remark}

We end this section with the following lemma, which shows that the half-space Brownian Gibbs property is consistent with the {\em partial} Brownian Gibbs property from \cite[Definition 2.7]{DimMat}, which is a modification of the Brownian Gibbs property in \cite{CorHamA}, and is recalled later as Definition \ref{DefPBGP}. Specifically, the Brownian Gibbs property is slightly stronger than the partial version: when $\Sigma=\llbracket1,N\rrbracket$ and $N\in\mathbb{N}$ in Definition \ref{DefPBGP}, the Brownian Gibbs property also allows the possibility $k_2=N$, in which case the convention is $g=-\infty$. These two definitions are equivalent when $\Sigma=\mathbb{N}$. See also \cite[Remark 2.9]{DimMat} for a more detailed explanation. The lemma is proved in Section \ref{SectionA4}.

\begin{lemma}\label{LemmaConsistent} Suppose that $\mathcal{L}$ is a line ensemble that satisfies the half-space Brownian Gibbs property from Definition \ref{def:BGP} for some choice of $\Lambda, N, \{\mu_i\}_{i \in \llbracket 1, N-1 \rrbracket}$. Then, $\mathcal{L}$ satisfies the partial Brownian Gibbs property from Definition \ref{DefPBGP} for the same choice of $\Lambda$ and $N$.
\end{lemma}

%
%
\subsection{The half-space interlacing Gibbs property}\label{Section2.2} In this section we introduce discrete analogues of the concepts in Section \ref{Section2.1}. Some of our notation and exposition comes from \cite[Section 2]{dimitrov2024tightness}.

\begin{definition}\label{DefDLE}
Let $\Sigma \subseteq \mathbb{Z}$ and $\llbracket T_0, T_1 \rrbracket$ be a non-empty integer interval in $\mathbb{Z}$. Consider the set $Y$ of functions $f: \Sigma \times \llbracket T_0, T_1 \rrbracket \rightarrow \mathbb{Z}$ such that $f(i, j+1) - f(i,j) \in \mathbb{Z}_{\geq 0}$ when $i \in \Sigma$ and $j \in\llbracket T_0, T_1 -1 \rrbracket$ and let $\mathcal{D}$ denote the discrete $\sigma$-algebra on $Y$. We call a function $g: \llbracket T_0, T_1 \rrbracket \rightarrow \mathbb{Z}$ such that $g( j+1) - g(j) \in \mathbb{Z}_{\geq 0}$ when $j \in\llbracket T_0, T_1 -1 \rrbracket$  an {\em increasing path} and elements in $Y$ {\em collections of increasing paths}. A $\Sigma$-{\em indexed geometric line ensemble $\mathfrak{L}$ on $\llbracket T_0, T_1 \rrbracket$}  is a random variable defined on a probability space $(\Omega, \mathcal{B}, \mathbb{P})$, taking values in $Y$ such that $\mathfrak{L}$ is a $\mathcal{B}/\mathcal{D}$ measurable function. Unless otherwise specified, we will assume that $T_0 \leq T_1$ are both integers, although the above definition makes sense if $T_0 = -\infty$, or $T_1 = \infty$, or both.
\end{definition}

\begin{remark} The condition $f(i, j+1) - f(i,j) \in \mathbb{Z}_{\geq 0}$ when $i \in \Sigma$ and $j \in\llbracket T_0, T_1 -1 \rrbracket$ essentially means that for each $i \in \Sigma$ the function $f(i, \cdot)$ can be thought of as a sample trajectory of a geometric random walk from time $T_0$ to time $T_1$. Here, and throughout the paper a geometric random variable $X$ with parameter $q \in [0,1)$ has probability mass function $\mathbb{P}(X = k) = (1-q)q^k$ for $k \in \mathbb{Z}_{\geq 0}$.
\end{remark}

We think of geometric line ensembles as collections of random increasing paths on the integer lattice, indexed by $\Sigma$. Observe that one can view an increasing path $L$ on $\llbracket T_0, T_1 \rrbracket$ as a continuous curve by linearly interpolating the points $(i, L(i))$, see Figure \ref{S1_3}. This allows us to define $ (\mathfrak{L}(\omega)) (i, s)$ for non-integer $s \in [T_0,T_1]$ and to view geometric line ensembles as line ensembles in the sense of Definition \ref{CLEDef}. In particular, we can think of $\mathfrak{L}$ as a random element in $C (\Sigma \times \Lambda)$ with $\Lambda = [T_0, T_1]$. We write $L_i = (\mathfrak{L}(\omega)) (i, \cdot)$ for the index $i \in \Sigma$ path. If $L$ is an increasing path on $\llbracket T_0, T_1 \rrbracket$ and $a, b \in \llbracket T_0, T_1 \rrbracket$ satisfy $a \leq b$, we let $L\llbracket a, b \rrbracket$ denote the restriction of $L$ to $\llbracket a,b\rrbracket$. \\

Let $t_i, z_i \in \mathbb{Z}$ for $i = 1,2$ be given such that $t_1 \leq t_2$ and $z_1 \leq z_2$. We denote by $\Omega(t_1,t_2,z_1,z_2)$ the collection of increasing paths that start from $(t_1,z_1)$ and end at $(t_2,z_2)$, by $\mathbb{P}_{\operatorname{Geom}}^{t_1,t_2, z_1, z_2}$ the uniform distribution on $\Omega(t_1,t_2,z_1,z_2)$ and write $\mathbb{E}^{t_1,t_2,z_1,z_2}_{\operatorname{Geom}}$ for the expectation with respect to this measure. One thinks of the distribution $\mathbb{P}_{\operatorname{Geom}}^{t_1,t_2, z_1, z_2}$ as the law of a random walk with i.i.d. geometric increments with parameter $q \in (0,1)$ that starts from $z_1$ at time $t_1$ and is conditioned to end in $z_2$ at time $t_2$ -- this interpretation does not depend on the choice of $q \in (0,1)$. Notice that by our assumptions on the parameters the state space $\Omega(t_1,t_2,z_1,z_2)$ is non-empty.  

If $T_1 \in \mathbb{N}$ and $y \in \mathbb{Z}$, we let $\Omega(T_1,y) = \cup_{x \leq y} \Omega(0, T_1, x ,y)$. Given $q \in [0,1)$, we define $\mathbb{P}_{\operatorname{Geom}}^{T_1, y, q}$ to be the measure on increasing paths such that 
\begin{equation}\label{eq:BackGeomRW}
\mathbb{P}_{\operatorname{Geom}}^{T_1, y, q}(L)  \propto q^{y - L(0)} \cdot {\bf 1}\{L \in \Omega(T_1,y)\}.
\end{equation}
Note that since $q \in [0,1)$, we have that the sum of the terms on the right side of (\ref{eq:BackGeomRW}) is finite and hence $\mathbb{P}_{\operatorname{Geom}}^{T_1, y, q}$ is an honest probability measure. In words, $\mathbb{P}_{\operatorname{Geom}}^{T_1, y, q}$ is the law of a random walk that starts from $y$ at time $T_1$, is being run backwards, and at each step makes a {\em negative} geometric jump with parameter $q$. We refer to this law as a {\em reverse geometric random walk with parameter $q$}. We write $\mathbb{E}_{\operatorname{Geom}}^{T_1, y, q}$ for the expectation with respect to $\mathbb{P}_{\operatorname{Geom}}^{T_1, y, q}$.

Given $k \in \mathbb{N}$, $\vec{q} = (q_1, \dots, q_k) \in [0,1)^k$, $T_1 \in \mathbb{N}$ and $ \vec{y} \in \mathbb{Z}^k$, we let $\mathbb{P}^{T_1,\vec{y},\vec{q}}_{\operatorname{Geom}}$ denote the law of $k$ independent reverse geometric random walks $\{B_i: \llbracket 0, T_1 \rrbracket  \rightarrow \mathbb{Z} \}_{i = 1}^k$ with $B_i(T_1) = y_i$. We also write $\mathbb{E}_{\operatorname{Geom}}^{T_1, \vec{y}, \vec{q}}$ for the expectation with respect to $\mathbb{P}^{T_1,\vec{y},\vec{q}}_{\operatorname{Geom}}$.\\

The following definition introduces the notion of a reverse $g$-interlacing geometric line ensemble, which in simple terms is a collection of $k$ independent reverse geometric random walks, conditioned on interlacing with each other and the graph of $g$.
\begin{definition}\label{def:interlaceGeom}
Let $k \in \mathbb{N}$ and $\mathfrak{W}_k$ denote the set of signatures of length $k$, i.e.
\begin{equation}\label{DefSig}
\mathfrak{W}_k = \{ \vec{x} = (x_1, \dots, x_k) \in \mathbb{Z}^k: x_1 \geq  x_2 \geq  \cdots \geq  x_k \}.
\end{equation}
Let $\vec{q} = (q_1, \dots, q_k) \in [0,1)^k$, $ \vec{y} \in \mathfrak{W}_k$, $T_1 \in \mathbb{Z}_{\geq 0}$, and $g: \llbracket 0, T_1 \rrbracket \rightarrow [-\infty, \infty)$ be any function. With the above data we define the {\em reverse $g$-interlacing geometric line ensemble on the interval $\llbracket 0, T_1 \rrbracket$ with jump parameters $\vec{q}$ and exit data $\vec{y}$} to be the $\Sigma$-indexed geometric line ensemble $\mathfrak{Q} = \{Q_i\}_{i \in \Sigma}$ with $\Sigma = \llbracket 1, k\rrbracket$ on $\llbracket 0, T_1 \rrbracket$ and with the law of $\mathfrak{Q}$ equal to $\mathbb{P}^{T_1,\vec{y} ,\vec{q}}_{\operatorname{Geom}}$ (the law of $k$ independent reverse geometric random walks $\{B_i: \llbracket 0, T_1 \rrbracket \rightarrow \mathbb{Z} \}_{i = 1}^k$ from $B_i(T_1) = y_i$), conditioned on 
\begin{equation}\label{EventInter}
\begin{split}
\ice  = &\left\{ B_i(r-1) \geq B_{i+1}(r)  \mbox{ for all $r \in \llbracket  1, T_1 \rrbracket$ and $i \in \llbracket 0 , k \rrbracket$} \right\},
\end{split}
\end{equation}
with the convention that $B_0(x) = \infty$ and $B_{k+1}(x) = g(x)$.

The above definition is well-posed if there exist $B_i \in \Omega(T_1,y_i)$ for $i \in \llbracket 1, k \rrbracket$ that satisfy the conditions in $\ice$. We denote by $\Omega_{\ice}(T_1, \vec{y},g)$ the set of collections of $k$ increasing paths that satisfy the conditions in $\ice$, the probability distribution of $\mathfrak{Q}$ by $\mathbb{P}_{\ice, \operatorname{Geom}}^{T_1, \vec{y}, \vec{q}, g}$ and write $\mathbb{E}_{\ice, \operatorname{Geom}}^{T_1,\vec{y}, \vec{q}, g}$ for the expectation with respect to this measure. 
\end{definition}
\begin{remark}\label{rem:well-posed} We mention that if $g:\llbracket 0, T_1 \rrbracket \rightarrow [-\infty, \infty)$ is increasing and $g(T_1) \leq y_k$, then $\Omega_{\ice}(T_1, \vec{y},g) \neq \emptyset$ as it contains the element $\{B_i\}_{i = 1}^k$ with $B_i(r) = y_i$ for $i \in \llbracket 1, k \rrbracket$ and $r \in \llbracket 0, T_1 \rrbracket$. 
\end{remark}

\begin{remark}\label{rem:accept} From Definition \ref{def:interlaceGeom}, we see that for any set $E$ we have
\begin{equation}\label{eq:accept}
\mathbb{P}_{\ice, \operatorname{Geom}}^{T_1, \vec{y}, \vec{q}, g}(E) = \frac{\mathbb{P}_{ \operatorname{Geom}}^{T_1, \vec{y}, \vec{q}}(E \cap \Omega_{\ice}(T_1, \vec{y},g))}{\mathbb{P}_{ \operatorname{Geom}}^{T_1, \vec{y}, \vec{q}}(\Omega_{\ice}(T_1, \vec{y},g))} ,
\end{equation}
and we refer to $\mathbb{P}_{ \operatorname{Geom}}^{T_1, \vec{y}, \vec{q}}(\Omega_{\ice}(T_1, \vec{y},g))$ as an {\em acceptance probability}.
\end{remark}

The following definition introduces the notion of the half-space interlacing Gibbs property.
\begin{definition}\label{def:HSIGP}
Fix a set $\Sigma = \llbracket 1, N \rrbracket$ with $N \in \mathbb{N}$ or $N = \infty$, $q_i \in [0,1)$ for $i \in \llbracket 1, N-1 \rrbracket$, and $ T_1\in \mathbb{Z}_{\geq 0}$. A $\Sigma$-indexed geometric line ensemble $\mathfrak{L} : \Sigma \times \llbracket 0, T_1 \rrbracket \rightarrow \mathbb{Z}$ is said to satisfy the {\em half-space interlacing Gibbs property with parameters $\{q_i\}_{i \in \llbracket 1, N-1\rrbracket}$} if it is interlacing, meaning that 
$$ L_i(j-1) \geq L_{i+1}(j) \mbox{ for all $i \in \llbracket 1, N - 1 \rrbracket$ and $j \in \llbracket 1, T_1 \rrbracket$},$$
and for any $k \in \llbracket 1, N - 1 \rrbracket$ and $b \in \llbracket 0, T_1 \rrbracket$ the following holds.  Suppose that $ g$  (an increasing path drawn in $\{ (r,z) \in \mathbb{Z}^2 : 0 \leq r \leq b\}$) and $ \vec{y} \in \mathfrak{W}_k$ satisfy $\mathbb{P}(A) > 0$, where $A$ denotes the event
$$A =\{ \vec{y} = ({L}_{1}(b), \dots, {L}_{k}(b)), L_{k+1} \llbracket 0,b \rrbracket = g \}.$$
Then, setting $\vec{q}_k = (q_1, \dots, q_k)$ we have for any $\{ B_i \in \Omega(b, y_i) \}_{i = 1}^k$ that
\begin{equation}\label{eq:HSIGP}
\mathbb{P}\left( L_{i}\llbracket 0,b \rrbracket = B_{i} \mbox{ for $i \in \llbracket 1, k \rrbracket$} \, \vert  A \, \right) = \mathbb{P}_{\ice, \operatorname{Geom}}^{b, \vec{y}, \vec{q}_k, g} \left( \cap_{i = 1}^k\{ Q_i = B_i \} \right).
\end{equation}
\end{definition}
\begin{remark}\label{rem:HSIGP} In simple words, a geometric line ensemble is said to satisfy the half-space interlacing Gibbs property if the distribution of the top $k$ paths, conditioned on their end-points and the $(k+1)$-st path, is simply that of $k$ independent reverse random walks conditioned on interlacing with each other and the $(k+1)$-st path. 
\end{remark}
\begin{remark}\label{rem:HSIGP2} Observe that in Definition \ref{def:HSIGP} the index $k$ is assumed to be less than or equal to $N-1$, so that if $N < \infty$ the $N$-th path is special and is not conditionally a reverse geometric random walk. We mention that since $\mathfrak{L}$ is interlacing we automatically have for each $j \in \llbracket 0, T_1\rrbracket$ that $L_1(j) \geq L_2(j) \geq  \cdots$. In addition, the well-posedness of $\mathbb{P}_{\ice, \operatorname{Geom}}^{b, \vec{y}, \vec{q}_k, g}$ in (\ref{eq:HSIGP}) is a consequence of our assumption that $\mathbb{P}(A) > 0$.
\end{remark}
\begin{remark}\label{rem:HSIGP3} Note that if $g$ is a fixed increasing path, and $\mathfrak{L}: \llbracket 1, k+1 \rrbracket \times \llbracket 0, T_1 \rrbracket \rightarrow \mathbb{Z}$ is such that $L_{k+1} = g$, and the law of $\{L_i\}_{i = 1}^k$ is $\mathbb{P}_{\ice, \operatorname{Geom}}^{T_1, \vec{y}, \vec{q}, g}$ as in Definition \ref{def:interlaceGeom}, then $\mathfrak{L}$ is a geometric line ensemble that satisfies the half-space interlacing Gibbs property from Definition \ref{def:HSIGP} with $N = k+1$.
\end{remark}

In \cite{dimitrov2024tightness} one of the authors introduced a similar definition to Definition \ref{def:HSIGP}, where reverse geometric random walks are replaced with geometric random walk {\em bridges}. The following lemma shows that the two notions are consistent, and should be viewed as a discrete analogue of Lemma \ref{LemmaConsistent}.

\begin{lemma}\label{LemmaConsistentGeom} Fix a set $\Sigma = \llbracket 1, N \rrbracket$ with $N \in \mathbb{N}$ or $N = \infty$, $q_i \in [0,1)$ for $i \in \llbracket 1, N-1 \rrbracket$, and $ T_1\in \mathbb{Z}_{\geq 0}$. Suppose that $\mathfrak{L}: \Sigma \times \llbracket 0, T_1 \rrbracket \rightarrow \mathbb{Z}$ is a geometric line ensemble that satisfies the half-space interlacing Gibbs property as in Definition \ref{def:HSIGP}. Then, $\mathfrak{L}$ satisfies the interlacing Gibbs property from \cite[Definition 2.6]{dimitrov2024tightness}. Specifically, for any finite $K = \llbracket k_1, k_2 \rrbracket \subseteq \llbracket 1, N - 1 \rrbracket$ and $a,b \in \llbracket 0, T_1 \rrbracket$ with $a < b$ the following holds.  Suppose that $f, g$ are two increasing paths drawn in $\{ (r,z) \in \mathbb{Z}^2 : a \leq r \leq b\}$ and $\vec{x}, \vec{y} \in \mathfrak{W}_k$ with $k=k_2-k_1+1$ altogether satisfy $\mathbb{P}(A) > 0$ with  
$$A =\{ \vec{x} = ({L}_{k_1}(a), \dots, {L}_{k_2}(a)), \vec{y} = ({L}_{k_1}(b), \dots, {L}_{k_2}(b)), L_{k_1-1} \llbracket a,b \rrbracket = f, L_{k_2+1} \llbracket a,b \rrbracket = g \},$$
where if $k_1 = 1$ we adopt the convention $f = \infty = L_0$. Then, we have for any $\{ B_i \in \Omega(a, b, x_i , y_i) \}_{i = 1}^k$
\begin{equation}\label{eq:IGP}
\mathbb{P}\left( L_{i + k_1-1}\llbracket a,b \rrbracket = B_{i} \mbox{ for $i \in \llbracket 1, k \rrbracket$} \, \vert  A \, \right) =  \mathbb{P}^{a, b, \vec{x}, \vec{y},f,g}_{\ice,\operatorname{Geom}} \left(\cap_{i = 1}^k\{ Q_i = B_i \} \right),
\end{equation}
where $\mathbb{P}^{a, b, \vec{x}, \vec{y},f,g}_{\ice,\operatorname{Geom}}$ is the uniform measure on
\begin{equation*}
\begin{split}
&\ice(a,b,\vec{x}, \vec{y},f,g) = \{ Q_i \in \Omega(a,b,x_i,y_i) \mbox{ for } i \in \llbracket 1, k \rrbracket : Q_i(a) = x_i, Q_i(b) = y_i \mbox{ for } i \in \llbracket 1, k \rrbracket  \mbox{ and }\\
&Q_i(r-1) \geq Q_{i+1}(r)  \mbox{ for all $r \in \llbracket  a+1, b \rrbracket$ and $i \in \llbracket 0 , k \rrbracket$} \},
\end{split}
\end{equation*}
with the convention that $Q_{0}(x) = f(x)$ and $Q_{k+1}(x) = g(x)$ for $x \in \llbracket a,b \rrbracket$.
\end{lemma}
\begin{proof} We define for any increasing path $h$ on $[a,b]$ and $\vec{y} \in \mathfrak{W}_{k}$
\begin{equation*}
\begin{split}
&\mathfrak{W}_{k_2}(\vec{y}) = \{\vec{z} \in \mathfrak{W}_{k_2}: z_{i + k_1-1} = y_i \mbox{ for } i \in \llbracket 1, k \rrbracket \} \mbox{,  }\\
& \Omega(a,b, h) = \{\tilde{h} \in \Omega(b,h(b)): \tilde{h}(s) = h(s) \mbox{ for } s \in \llbracket a, b \rrbracket \}.
\end{split}
\end{equation*}
If $D = (D_1, \dots, D_k) \in \ice(a,b,\vec{x}, \vec{y},f,g)$, $\vec{z} \in \mathfrak{W}_{k_2}(\vec{y})$, and $\tilde{g} \in \Omega(a,b, g)$, we define 
\begin{equation*}
\begin{split}
&\ice(D, \vec{z}, \tilde{g}) = \big{\{} Q_i \in \cup_{y \in \mathbb{Z}} \Omega(b,y) \mbox{ for } i \in \llbracket 1, k_2 + 1 \rrbracket : Q_{k_1 + i - 1}\llbracket a, b \rrbracket = D_i \mbox{ for $i \in \llbracket 0, k\rrbracket$ },  Q_i(b) = z_i  \\
&\mbox{ for } i \in \llbracket 1, k_2 \rrbracket, Q_{k_2+1} \llbracket 0, b \rrbracket = \tilde{g},\mbox{ and } Q_i(r-1) \geq Q_{i+1}(r)  \mbox{ for all $r \in \llbracket  1, b \rrbracket$ and $i \in \llbracket 0 , k_2 \rrbracket$} \big{\}},
\end{split}
\end{equation*}
where as usual we adopt the convention that $D_0 = f$, $Q_0 = \infty$. In words, $\ice(D, \vec{z}, \tilde{g})$ is the set of $(k_2 + 1)$-tuples of up-right paths drawn in $\{ (r,z) \in \mathbb{Z}^2 : 0 \leq r \leq b\}$, which interlace, and match $\vec{z}$ on the right, $\tilde{g}$ on the bottom and $D$ in the rectangle $\llbracket k_1-1, k_2 \rrbracket \times \llbracket a, b \rrbracket$.

From (\ref{eq:HSIGP}) we have that if $\tilde{A}(\vec{z}, \tilde{g}) = \{ \vec{z} = ({L}_{1}(b), \dots, {L}_{k_2}(b)), L_{k_2+1} \llbracket 0,b \rrbracket = \tilde{g} \}$, then for any $\{ \tilde{B}_i \in \Omega(b, z_i) \}_{i = 1}^{k_2}$
\begin{equation}\label{YU1}
\mathbb{P}\left( \{ L_{i}\llbracket 0,b \rrbracket = \tilde{B}_{i} \mbox{ for $i \in \llbracket 1, k_2 \rrbracket$} \} \cap   \tilde{A}(\vec{z}, \tilde{g})\right) = \mathbb{P}_{\ice, \operatorname{Geom}}^{b, \vec{z}, \vec{q}_{k_2}, \tilde{g}} \left( \cap_{i = 1}^k\{ Q_i = \tilde{B}_i \} \right) \cdot \mathbb{P} (\tilde{A}(\vec{z}, \tilde{g}) ).
\end{equation}
The latter shows that for any $B = (B_1, \dots, B_k) \in \ice(a,b,\vec{x}, \vec{y},f,g)$  
\begin{equation}\label{YU2}
\begin{split}
&\mathbb{P}\left( \{ L_{i + k_1-1}\llbracket a,b \rrbracket = B_{i} \mbox{ for $i \in \llbracket 1, k \rrbracket$} \} \cap  A \right)  =  \sum_{\vec{z} \in \mathfrak{W}_{k_2}(\vec{y})} \sum_{\tilde{g} \in  \Omega(a,b, g) }   \mathbb{P} (\tilde{A}(\vec{z}, \tilde{g}) )     \\
& \sum_{(\tilde{B}_1, \dots, \tilde{B}_{k_2+1}) \in\ice(B, \vec{z}, \tilde{g}) }  \mathbb{P}_{\ice, \operatorname{Geom}}^{b, \vec{z}, \vec{q}_{k_2}, \tilde{g}}  \left( \cap_{i = 1}^{k_2}\{ Q_i = \tilde{B}_i \} \right), 
\end{split}
\end{equation}
and also
\begin{equation}\label{YU3}
\begin{split}
&\mathbb{P}\left( A \right)  = \sum_{\vec{z} \in \mathfrak{W}_{k_2}(\vec{y})} \sum_{\tilde{g} \in  \Omega(a,b, g) }   \mathbb{P} (\tilde{A}(\vec{z}, \tilde{g}) ) \sum_{D \in \ice(a,b,\vec{x}, \vec{y},f,g)} \\ 
& \sum_{(\tilde{B}_1, \dots, \tilde{B}_{k_2+1}) \in\ice(D, \vec{z}, \tilde{g}) }  \mathbb{P}_{\ice, \operatorname{Geom}}^{b, \vec{z}, \vec{q}_{k_2}, \tilde{g}} \left( \cap_{i = 1}^{k_2}\{ Q_i = \tilde{B}_i \} \right).
\end{split}
\end{equation}
Observe that for any $D = (D_1, \dots, D_k) \in \ice(a,b,\vec{x}, \vec{y},f,g)$ the second lines in (\ref{YU2}) and (\ref{YU3}) are the same. Indeed, the sets $\ice(D, \vec{z}, \tilde{g})$ and $\ice(B, \vec{z}, \tilde{g})$ are in a clear bijection, where to go from the first to the second we modify $\tilde{B}_i(s)$ for $(i,s) \in \llbracket k_1, k_2 \rrbracket \times \llbracket a, b \rrbracket$ from matching $D$ to matching $B$. In addition, from (\ref{eq:BackGeomRW}) and (\ref{eq:accept}) we have for some $C > 0$ that depends on $\vec{q}_{k_2}$, $\vec{z}$ and $b$
$$\mathbb{P}_{\ice, \operatorname{Geom}}^{b, \vec{z}, \vec{q}_{k_2}, \tilde{g}} \left( \cap_{i = 1}^{k_2}\{ Q_i = \tilde{B}_i \} \right) = C \cdot \prod_{i = 1}^{k_2} q_i^{z_i - \tilde{B}_i(0)},$$
which shows that the corresponding summands on the second lines of (\ref{YU2}) and (\ref{YU3}) agree.

From the last observation, we conclude that for $B = (B_1, \dots, B_k) \in \ice(a,b,\vec{x}, \vec{y},f,g)$, we have
\begin{equation*}
\begin{split}
&\mathbb{P}\left( \{ L_{i + k_1-1}\llbracket a,b \rrbracket = B_{i} \mbox{ for $i \in \llbracket 1, k \rrbracket$} \} \vert  A \right)  =  \frac{1}{|\ice(a,b,\vec{x}, \vec{y},f,g)|},
\end{split}
\end{equation*}
which is precisely (\ref{eq:IGP}). When $B = (B_1, \dots, B_k) \not\in \ice(a,b,\vec{x}, \vec{y},f,g)$, then both sides of (\ref{eq:IGP}) are equal to zero, and so the equation holds in this case as well.
\end{proof}

%
%
\subsection{Some properties}\label{Section2.3} In this section we state some of the properties of line ensembles with law $\mathbb{P}_{\ice, \operatorname{Geom}}^{T_1, \vec{y}, \vec{q}, g}$ as in Definition \ref{def:interlaceGeom}. The first result is a lemma that shows that these ensembles can be monotonically coupled in their boundary data. Its proof is in Section \ref{SectionA5}. 

\begin{lemma}\label{lem:monotone coupling}
Fix $k, T \in \mathbb{N}$, $\vec{q}=(q_1,\dots,q_k)\in[0,1)^k$ and two increasing functions $g^{\mathrm{b}}, g^{\mathrm{t}}: \llbracket 0, T \rrbracket  \rightarrow \mathbb{Z}\cup\{-\infty\}$ with $g^{\mathrm{b}}\leq g^{\mathrm{t}}$ on $\llbracket0,T\rrbracket$. Let $\vec{y}\,^{\mathrm{b}},\vec{y}\,^{\mathrm{t}} \in \mathfrak{W}_k$ with $y_i^{\mathrm{b}}\leq y_i^{\mathrm{t}}$ for $i\in\llbracket1,k\rrbracket$ and $g^{\mathrm{b}}(T)\leq y_k^{\mathrm{b}}$ and  $g^{\mathrm{t}}(T)\leq y_k^{\mathrm{t}}$. Then, there exists a probability space $(\Omega, \mathcal{F}, \mathbb{P})$, which supports two $\llbracket 1, k \rrbracket$-indexed geometric line ensembles $\mathfrak{L}^{\mathrm{b}}$ and $\mathfrak{L}^{\mathrm{t}}$ on $\llbracket 0,T \rrbracket$ such that the law of $\mathfrak{L}^{\mathrm{b}}$ {\big (}resp. $\mathfrak{L}^{\mathrm{t}}${\big )} under $\mathbb{P}$ is given by $\mathbb{P}_{\ice,\operatorname{Geom} }^{ T, \vec{y}\,^{\mathrm{b}},\vec{q}, g^{\mathrm{b}}}$ {\big (}resp. $\mathbb{P}_{\ice,\operatorname{Geom} }^{ T, \vec{y}\,^{\mathrm{t}},\vec{q} ,g^{\mathrm{t}}}${\big )} and such that $\mathbb{P}$-a.s. ${L}_i^{\mathrm{b}}(r) \leq {L}^{\mathrm{t}}_i(r)$ for $i \in \llbracket 1, k \rrbracket$, $r \in \llbracket 0,T \rrbracket$.
\end{lemma}

The next statement shows that we can strongly couple a (centered) geometric random walk and a Brownian motion. The latter is an immediate consequence of \cite[Theorem A]{shao1995strong}, which is a version of the celebrated Koml{\' o}s-Major-Tusn{\' a}dy (KMT) coupling \cite{KMT1,KMT2}.
\begin{lemma}\label{prop:ThmA Shao}
Fix $\varepsilon \in (0,1/2), q \in [\varepsilon,1-\varepsilon]$ and let $\{G_j\}_{j=1}^{\infty}$ be a sequence of i.i.d. geometric random variables with parameter $q$. Then, there is a probability space $(\Omega, \mathcal{F}, \mathbb{P})$ that supports $\{G_j\}_{j=1}^{\infty}$ and a sequence of i.i.d. normal random variables $\{H_j\}_{j=1}^{\infty}$ with 
$$\mathbb{E}\left[H_1\right]=0,\quad
    \operatorname{Var}(H_1)=\operatorname{Var}(G_1)=\frac{q}{(1-q)^2},
$$
such that for any $m\in\mathbb{N}$,
\begin{equation}\label{Eq.ShaoRestate}
\mathbb{E}\left[\exp\left(\lambda A\max_{i\leq m}\left| \sum_{j=1}^i\mathring{G}_j-\sum_{j=1}^iH_j\right| \right)\right]\leq 1+ \frac{\lambda mq}{(1-q)^2}, 
\end{equation}
where $\mathring{G}_j=G_j-\mathbb{E}\left[G_j\right]$, $A > 0$ is an absolute constant and $\lambda>0$ depends only on $\varepsilon$.  
\end{lemma} 
\begin{proof} We can find $\lambda > 0$ sufficiently small, depending on $\varepsilon$, such that if $q \in [\varepsilon,1-\varepsilon]$ we have
$$\lambda \mathbb{E}[e^{\lambda |\mathring{G}_1|} |\mathring{G}_1|^3] \leq \frac{\varepsilon}{(1-\varepsilon)^2} \leq \frac{q}{(1-q)^2} = \mathbb{E}[\mathring{G}_1^2].$$
The latter satisfies \cite[(1.1)]{shao1995strong} and then the lemma follows from \cite[Theorem A]{shao1995strong}. Specifically, (\ref{Eq.ShaoRestate}) is precisely \cite[(1.2)]{shao1995strong} with $X_j = \mathring{G}_j$, and $Y_j = H_j$.
\end{proof}

The final result of the section shows that under an appropriate scaling the reverse $g$-interlacing geometric walks from Definition \ref{def:interlaceGeom} converge to the reverse $g$-avoiding Brownian motions from Definition \ref{def: avoidBLE}. Its proof is given in Section \ref{SectionA6}. 
\begin{lemma}\label{lem:RW} Let $k, \vec{y}, \vec{\mu}, b, g$ be as in Definition \ref{def: avoidBLE}. Suppose that $d_n$ is a sequence of positive reals such that $d_n \rightarrow \infty$ as $n \rightarrow \infty$, and set $B_n = \lceil b d_n \rceil$. Let $g_n: [0, B_n/d_n] \rightarrow [-\infty, \infty)$ be continuous functions such that $g_n \rightarrow g$ uniformly on $[0,b]$. In the case when $g = -\infty$, the last statement means that $g_n = -\infty$ for all large $n$. We also suppose that 
\begin{equation}\label{EdgeLim}
\lim_{n \rightarrow \infty} |g_n(B_n/d_n) - g_n(b)| = 0 \mbox{ if } g \neq -\infty.
\end{equation}

Fix $p \in (0,1)$, $u = p/(1-p)$, $\sigma = \sqrt{p}/(1-p)$, define $G_n: [0, B_n] \rightarrow [-\infty, \infty)$ through
$$G_n(s) = \sigma d_n^{1/2} \cdot g_n(s/d_n) + u s,$$
and suppose that $\vec{Y}^n \in \mathfrak{W}_k$ is a sequence, such that 
\begin{equation}\label{SideLim}
 \lim_n \sigma^{-1} d_n^{-1/2} \cdot (Y_i^n - u B_n)   = y_i \mbox{ for } i \in \llbracket 1, k \rrbracket.
\end{equation}
We further suppose that $\vec{q}^{\,n} = (q_1^n, \dots, q_k^n) \in (0,1)^k$ is a sequence that satisfies 
\begin{equation}\label{eq:asym of q}
q_i^n=p-\mu_i\sqrt{p}(1-p) \cdot d_n^{-1/2}+o\left(d_n^{-1/2}\right) \mbox{ for } i \in \llbracket 1, k \rrbracket.
\end{equation}

Then, we have the following statements.
\begin{enumerate}
\item There exists $N_0 \in \mathbb{N}$ such that for $n \geq N_0$ the laws $\mathbb{P}_{\ice, \operatorname{Geom}}^{B_n, \vec{Y}^n, \vec{q}^{\,n}, G_n}$ are well-defined, i.e. the sets $\Omega_{\ice}( B_n, \vec{Y}^n, G_n)$ are non-empty.
\item If $\mathfrak{Q}^n$ has law $\mathbb{P}_{\ice, \operatorname{Geom}}^{B_n, \vec{Y}^n, \vec{q}^{\,n}, G_n}$, then the sequence of $\llbracket 1,k \rrbracket$-indexed line ensembles $\mathcal{Q}^n$ on $[0,b]$ defined by
\begin{equation}\label{S22E5}
\mathcal{Q}_i^n(t) = \sigma^{-1} d_n^{-1/2} \cdot \left( Q_i(t d_n) - utd_n\right) \mbox{ for } n \geq N_0, \hspace{2mm} t \in [0,b] \mbox{ and } i \in \llbracket 1, k \rrbracket,
\end{equation}
converges weakly to $\pabm^{b, \vec{y}, \vec{\mu}, g}$ from Definition \ref{def: avoidBLE} as $n \rightarrow \infty$.
\end{enumerate}
\end{lemma}

%
%
\section{Estimates for half-space line ensembles}\label{Section3} In this section we derive several statements about the reverse geometric random walks from (\ref{eq:BackGeomRW}) and the reverse $g$-interlacing geometric walks from Definition \ref{def:interlaceGeom}. The main tools we use in deriving the results of this section are the monotone coupling from Lemma \ref{lem:monotone coupling} and the strong coupling from Lemma \ref{prop:ThmA Shao}. We continue with the same notation as in Section \ref{Section2}.

%
%
\subsection{Modulus of continuity estimates}\label{Section3.1}
Fix $a,b \in \mathbb{R}$ with $a < b$. For a function $f\in C([a,b])$ we define its {\em modulus of continuity} for $\delta>0$ by
\begin{equation}\label{eq:def of modulus of continuity}
w(f,\delta)=\sup_{\substack{x,y\in[a,b] \\ |x-y|\leq\delta}}|f(x)-f(y)|.
\end{equation}
Our next result states that if $Q_1$ has law $\mathbb{P}_{ \operatorname{Geom}}^{T, y, q}$, then $Q_1$ has a well-behaved modulus of continuity. The latter is expected, since by Lemma \ref{prop:ThmA Shao} the curve $Q_1$ behaves like a Brownian motion, which has a well-behaved modulus of continuity.
\begin{lemma}\label{lem:modulus of continuity bound}
Fix $q_0\in(0,1)$ and $M,\varepsilon,\eta>0$. Then, there exist $\mathsf{T}_1\in\mathbb{N}$ and $\delta>0$, depending on $q_0,M,\varepsilon,\eta$, such that for any $T\geq \mathsf{T}_1$, $y\in\mathbb{Z}$ and $q\in(0,1)$ satisfying $|q-q_0|\leq MT^{-1/2}$, we have 
\begin{equation}\label{eq:modulus of continuity estimate}
\mathbb{P}_{\operatorname{Geom}}^{T,y,q}\left(w(\mathcal{Q}_1,\delta)> \eta\right)<\varepsilon,
\end{equation}
where $u =q_0/(1-q_0)$, $\sigma=\sqrt{q_0}/(1-q_0)$ and $\mathcal{Q}_1\in C([0,1])$ is $\mathcal{Q}_1(t)=\sigma^{-1}T^{-1/2}(Q_1(tT)- utT)$.
\end{lemma} 
\begin{proof} All constants in the proof depend on $q_0,M,\varepsilon,\eta$, and we will not mention this further. Let $G_1, \dots, G_T$ be i.i.d. geometric random variables with parameter $q$, i.e. $\mathbb{P}(G_1 = k) = (1-q)q^k$ for $k \in \mathbb{Z}_{\geq 0}$, and put
\begin{equation*}  
Q_1(j)=y-\sum_{\ell=j+1}^TG_{\ell}, \quad j\in\llbracket0,T\rrbracket.
\end{equation*}   
Then, $Q_1$ has law $\mathbb{P}_{\operatorname{Geom}}^{T,y,q}$. We denote $\tilde{u} = q/(1-q)$ and $\tilde{\sigma} = \sqrt{q}/(1-q)$. 

Let $T_1 \in \mathbb{N}$ be large enough and $\delta_0 > 0$ small enough so that for $T \geq T_1$ we have $q \in [\delta_0, 1 - \delta_0]$. By the strong coupling Lemma \ref{prop:ThmA Shao}, we can couple $Q_1$ with $H_1,\dots,H_T\sim N(0,\tilde{\sigma}^2)$, so that if $\mathring{G}_j = G_j - \mathbb{E}[G_j] = G_j - q/(1-q)$, we have for some $A, \lambda > 0$ 
$$\mathbb{E}\left[\exp\left(\lambda A \Delta_T \right)\right]\leq1+ \frac{\lambda Tq}{(1-q)^2}, \mbox{ where } 
\Delta_T = \max_{i\in\llbracket0,T\rrbracket}\left| \sum_{j=i+1}^T \mathring{G}_j-\sum_{j=i+1}^T H_j\right|.$$
In addition, by possibly enlarging the space, we may assume that there is a standard Brownian motion $B$ on $[0,1]$ and $H_i = \tilde{\sigma} T^{1/2}  (B(i/T) - B((i-1)/T))$ for $i \in \llbracket 1, T\rrbracket$. By the exponential Chebyshev's inequality there exists $T_2\in\mathbb{N}$, $T_2 \geq T_1$, such that for $T\geq T_2$
\begin{equation}\label{eq:intermediate bound modulus continuity 1}
\mathbb{P}\left(\sigma^{-1}T^{-1/2} \cdot \Delta_T >\frac{\eta}{4}\right) \leq \exp\left(- \lambda A \eta \sigma T^{1/2}/4 \right) \cdot \left(1+ \frac{\lambda Tq}{(1-q)^2} \right)<\frac{\varepsilon}{4}.
\end{equation}

For any $0\leq t_2\leq t_1\leq1$ such that $t_1T,t_2T\in\mathbb{Z}$, we have
\begin{equation}\label{eq:intermediate bound modulus continuity 2}
\begin{split}
&|\mathcal{Q}_1(t_1)-\mathcal{Q}_1(t_2)|-\sigma^{-1}T^{1/2}|u-\tilde{u}||t_1-t_2|\\ 
&=\sigma^{-1}T^{-1/2}|Q_1(t_1T)-Q_1(t_2T)- u(t_1T-t_2T)|-\sigma^{-1}T^{1/2}|u-\tilde{u}||t_1-t_2|\\
&\leq\sigma^{-1}T^{-1/2}|Q_1(t_1T)-Q_1(t_2T)-\tilde{u}(t_1T-t_2T)|  
=\sigma^{-1}T^{-1/2}\left|\sum_{j=t_2T+1}^{t_1T}\mathring{G}_j\right| \\
&\leq\sigma^{-1}T^{-1/2} \left|\sum_{j=t_2T+1}^{t_1T}H_j\right| +  \sum_{r\in\{1,2\}} \sigma^{-1}T^{-1/2} \left|\sum_{j=t_rT+1}^T \mathring{G}_j- \sum_{j=t_rT+1}^T H_j \right| \\
& \leq \tilde{\sigma} \sigma^{-1}|B(t_1)-B(t_2)|+ 2\sigma^{-1}T^{-1/2} \Delta_T.
\end{split}
\end{equation}
Since $u=q_0/(1-q_0)$, $\tilde{u}=q/(1-q)$ and $|q-q_0|\leq MT^{-1/2}$, there exist $T_3\in\mathbb{N}$ and $M_1>0$, such that for $T\geq T_3$ we have $\sigma^{-1}T^{1/2}|u-\tilde{u}|\leq M_1$ and $\tilde{\sigma} \sigma^{-1} \in [1/2,2]$.
From the continuity of $B$, we can find $\delta>0$, such that $8M_1\delta<\eta$ and
\begin{equation}\label{eq:intermediate bound modulus continuity 3}
\mathbb{P}\left(w(B,2\delta)>\frac{\eta}{8}\right)<\frac{\varepsilon}{4}.
\end{equation}
We finally choose $T_4\in\mathbb{N}$ such that for $T\geq T_4$ we have $2/T<\delta$, and set $\mathsf{T}_1 =\max(T_1,T_2,T_3,T_4)$. 

We observe that for $T \geq \mathsf{T}_1$ we have
\begin{equation*}
\begin{split}
&\mathbb{P}\left(w(\mathcal{Q}_1,\delta)> \eta\right)\leq\mathbb{P}\left(\sup_{\substack{0\leq t_2\leq t_1\leq1, \mbox{ } t_1-t_2\leq2\delta, \mbox{ } t_1T,t_2T\in\mathbb{Z}}} |\mathcal{Q}_1(t_1)-\mathcal{Q}_1(t_2)|>\eta\right)\\
&\leq\mathbb{P}\left(w(B,2\delta)> \eta/8\right)+
\mathbb{P}\left(2\sigma^{-1}T^{-1/2} \cdot \Delta_T > \eta/2 \right)<2 \cdot ( \varepsilon/4)<\varepsilon,
\end{split}
\end{equation*}
where in the first inequality we used $2/T<\delta$ and the fact that $\mathcal{Q}_1$ is a piece-wise linear function interpolating the points $\mathcal{Q}_1(i/T)$ for $i\in\llbracket0,T\rrbracket$; in the second inequality we used \eqref{eq:intermediate bound modulus continuity 2}, the triangle inequality, $\sigma^{-1}T^{1/2}|u-\tilde{u}||t_1-t_2|\leq 2M_1\delta< \eta/4$ and $\tilde{\sigma} \sigma^{-1}\leq 2$;
in the third inequality we used \eqref{eq:intermediate bound modulus continuity 1} and \eqref{eq:intermediate bound modulus continuity 3}. The last equation implies (\ref{eq:modulus of continuity estimate}).
\end{proof}

%
%
\subsection{Staying in the corridor}\label{Section3.2}
Let $A,B\in\mathbb{R}$, $y\in\mathbb{R}$ and $\delta\in(0,1)$. We define $f(t|A,B,y,\delta)$ in $C([0,1])$ to be the function such that
\begin{equation}\label{eq:defining the corridor}
 f(0|A,B,y,\delta)=f(1-\delta|A,B,y,\delta)=B+A,\quad f(1|A,B,y,\delta)=y+A,
\end{equation}
and $f(t|A,B,y,\delta)$ is linear on the intervals $[0,1-\delta]$ and $[1-\delta,1]$, see Figure \ref{S3_1}.
\begin{figure}[ht]
    \centering
     \begin{tikzpicture}[scale=2.7]

        \def\tra{3} 
        \draw[->, thick, gray] (-1.4,0)--(1.4,0);
        \draw[->, thick, gray] (-1,-1)--(-1,1.2);
        
        \draw[black, fill = black] (-1,0) circle (0.02);
        \draw[black, fill = black] (1,0) circle (0.02);        
        \draw (-1.1,-0.1) node{$0$};
        \draw (1, -0.1) node{$1$};

        \draw[black, fill = black] (1,-0.8) circle (0.02);
        \draw (1.2,-0.8) node{$(1,y)$};

        \draw[black, fill = black] (1,-0.6) circle (0.02);
        \draw (1.34,-0.6) node{$(1,y + A)$};

        \draw[black, fill = black] (-1,0.8) circle (0.02);
        \draw (-1.34,0.8) node{$(0, B + A)$};

        \draw[-] (-1,0.8)--(0.5, 0.8);
        \draw[-] (0.5,0.8)--(1, -0.6);

        \draw[black, fill = black] (0.5,0.8) circle (0.02);
        \draw (0.5,0.9) node{$(1-\delta, B+A)$};


        \draw[black, fill = black] (0.5,0.4) circle (0.02);

        \draw[->] (0,0.8)--(0, 0.4);
        \draw[->] (0,0.4)--(0, 0.8);
        \draw (0.1,0.6) node{$2A$};

        \draw[black, fill = black] (1,-1) circle (0.02);
        \draw (1.34,-1) node{$(1,y - A)$};

        \draw[black, fill = black] (-1,0.4) circle (0.02);
        \draw (-1.34,0.4) node{$(0,B - A)$};

        \draw[-] (-1,0.4)--(0.5, 0.4);
        \draw[-] (0.5,0.4)--(1, -1);

        \draw[black, fill = black] (0.5,0) circle (0.02);
        \draw (0.48, -0.1) node{$1 -\delta$};

    \end{tikzpicture} 

    \caption{The figure depicts $f(\cdot |A,B,y, \delta)$, $f(\cdot |-A,B,y, \delta)$ from (\ref{eq:defining the corridor}) for $A > 0$.}
    \label{S3_1}
\end{figure}
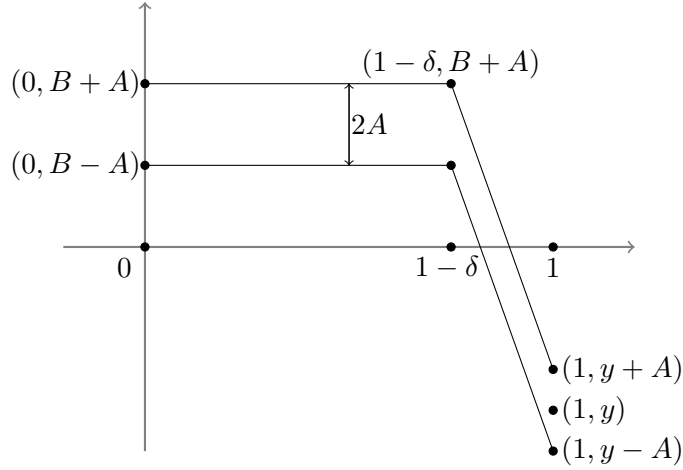

The main result in this section, Lemma \ref{lem:staying in corridor} below, shows that if $Q_1$ has law $\mathbb{P}_{ \operatorname{Geom}}^{T,y,q}$, then with some non-vanishing probability its scaled version $\mathcal{Q}_1$ stays in the corridor determined by two functions of the kind in (\ref{eq:defining the corridor}).

\begin{lemma} \label{lem:staying in corridor}
Fix $q_0\in(0,1)$, $M,M_1,M_2,A>0$ and $\delta_1\in(0,1)$. Denote $u=q_0/(1-q_0)$ and $\sigma=\sqrt{q_0}/(1-q_0)$. Then, there exist $\mathsf{T}_2\in\mathbb{N}$ and $\varepsilon>0$, depending on $q_0,M,M_1,M_2,A$ and $\delta_1$, such that the following holds for all $T\geq \mathsf{T}_2$
and $q\in(0,1)$ satisfying $|q-q_0|\leq MT^{-1/2}$.
Suppose $y\in\mathbb{Z}$ is such that $|y-u T|\leq M_1T^{1/2}$ and $B\in[-M_2,M_2]$. Let $f^{t}, f^{b}\in C([0,1])$ be defined by
\begin{equation}
\label{eq:define ft fb statement}
\begin{split}
& f^{t}(t)=f(t|A,B,\sigma^{-1}T^{-1/2}(y-u T),\delta_1), \hspace{2mm} f^{b}(t)=f(t|-A,B,\sigma^{-1}T^{-1/2}(y-u T),\delta_1).
\end{split}
\end{equation}
If $\mathcal{Q}_1\in C([0,1])$ is given by $\mathcal{Q}_1(t)=\sigma^{-1}T^{-1/2} (Q_1(tT)-u tT)$, then 
\begin{equation}
\mathbb{P}_{\operatorname{Geom}}^{T,y,q}\left(f^{t}(t)\geq\mathcal{Q}_1(t)\geq f^{b}(t)\mbox{ for all }t\in[0,1]\right)>\varepsilon.
\end{equation}
\end{lemma}
\begin{remark} Note that in the lemma $y = uT + O(T^{1/2})$. Consequently, the scaling inside (\ref{eq:define ft fb statement}) is made so that the third argument of $f^{t}$ and $f^b$ is of unit order. 
\end{remark}
\begin{proof} Suppose for the sake of contradiction that no such $\mathsf{T}_2$ and $\varepsilon$ exist. Then, we can find sequences $T_k$ increasing to infinity, 
$q_k\in(0,1)$ with $|q_k-q_0|\leq MT_k^{-1/2}$,
$y_k\in\mathbb{Z}$ with $|y_k-u T_k|\leq M_1T_k^{1/2}$ and $B_k\in[-M_2,M_2]$, such that
\begin{equation}\label{eq:assume contradition in the proof staying in corridor}
\lim_{k\rightarrow\infty}\mathbb{P}_{\operatorname{Geom}}^{T_k,y_k,q_k}\left(f_k^{t}(t)\geq\mathcal{Q}_1(t)\geq f_k^{b}(t)\mbox{ for all }t\in[0,1]\right)=0, 
\end{equation}
where $f_k^{t}$ and $f_k^{b}$ are defined as \eqref{eq:define ft fb statement} but with $B,y,T$ replaced by $B_k,y_k,T_k$, respectively. By possibly passing to a subsequence, which we also refer to as $k$, we may assume that
\begin{equation}\label{eq:scaling constants in proof stay in corridor}
\lim_{k\rightarrow\infty}\sigma^{-1}T_k^{-1/2}(y_k-u T_k)=y,\quad \lim_{k\rightarrow\infty}B_k=B,
\end{equation}
and that for a fixed $u\in\mathbb{R}$,
\begin{equation}\label{eq:scaling constants in proof stay in corridor 2}
q_k=q_0-u\sqrt{q_0}(1-q_0)T_k^{-1/2}+o(T_k^{-1/2}).
\end{equation}

Using Lemma \ref{lem:RW}, we know that if $Q_1^k$ have laws $\mathbb{P}_{\operatorname{Geom}}^{T_k,y_k,q_k}$, then $\mathcal{Q}_1^k(t)=\sigma^{-1}T_k^{-1/2}(Q_1^k(tT_k)-u tT_k)$ converge weakly to a reverse Brownian motion $\cev{B}$ on $[0,1]$ with drift $u$ from $\cev{B}(1)=y$. Consider 
$$U:=\left\{g\in C([0,1]): \sup_{t\in[0,1]} |g(t)-f(t|0,B,y,\delta_1)|<A/2 \right\}.$$
Note that $U\subset C([0,1])$ is an open subset that contains the function $f(s|0,B,y,\delta_1)$. By Lemma \ref{lem: BB positive measure} we have $\mathbb{P}(\cev{B}\in U)>0$. On the other hand, using \eqref{eq:scaling constants in proof stay in corridor} we have that for all large enough $k$,
$$U\subset\left\{g\in C([0,1]): f_k^{t}(t)\geq g(t)\geq f_k^{b}(t)\mbox{ for }t\in[0,1] \right\}.$$
Therefore, we have
    \begin{equation*}
        \begin{split}
            \liminf_{k\rightarrow\infty}\mathbb{P}_{\operatorname{Geom}}^{T_k,y_k,q_k}\left(f_k^{t}(t)\geq\mathcal{Q}_1(t)\geq f_k^{b}(t)\mbox{ for }t\in[0,1]\right) \geq \liminf_{k\rightarrow\infty}\mathbb{P}_{\operatorname{Geom}}^{T_k,y_k,q_k}\left(\mathcal{Q}_1\in U\right)\geq\mathbb{P}(\cev{B}\in U)>0,
        \end{split}
    \end{equation*}
    which contradicts \eqref{eq:assume contradition in the proof staying in corridor}. This concludes the proof.
\end{proof}

%
%
\subsection{No low minima}\label{Section3.3} The goal of this section is to establish the following lemma.

\begin{lemma}\label{lem:lower bound on curve canal}
Fix $\delta\in(0,1/2)$ and $\varepsilon\in(0,1)$. Let $k \in \mathbb{N}$, $\vec{q}=(q_1,\dots,q_k)\in [\delta,1-\delta]^k$, $\vec{y}\in\mathfrak{W}_k$ and $g:\llbracket0,T\rrbracket\rightarrow\mathbb{Z}\cup\{-\infty\}$ be an increasing function satisfying $g(T)\leq y_k$. Then, there exist $\mathsf{T}_3\in\mathbb{N}$ and  $M>0$, depending on $\delta$, $\varepsilon$ and $k$, such that for $T\geq \mathsf{T}_3$ and $t\in \llbracket 0,T \rrbracket$, we have
\begin{equation}\label{eq:ververgre}
\mathbb{P}_{\ice, \operatorname{Geom}}^{T, \vec{y}, \vec{q}, g}\left(Q_k(t)\geq y_k + \frac{q_k (t-T)}{1 - q_k} -T\sum_{i=1}^{k-1}\left|\frac{q_i}{1 - q_i} -\frac{q_{i+1}}{1 - q_{i+1}}\right|-M T^{1/2}\right)\geq 1-\varepsilon.
\end{equation}
\end{lemma}
\begin{remark} Notice that by the monotone coupling Lemma \ref{lem:monotone coupling}, the probability on the left side of (\ref{eq:ververgre}) is minimized when $g = -\infty$, and so it suffices to show (\ref{eq:ververgre}) for this choice of $g$. When $g = -\infty$, the lemma roughly states that $k$ interlacing reverse geometric walkers are unlikely to go much lower than their starting point.
\end{remark}

\begin{proof} For brevity we write $\mu_i = \frac{q_i}{1-q_i}$ for $i \in \llbracket 1, k \rrbracket$. By translating the ensemble, we may assume that $y_k = 0$, which we do in the sequel. Let $C > 0$ be large enough, depending on $\delta, \varepsilon, k$, so that 
\begin{equation}\label{S3DefC}
C > 1 + \delta^{-1} \mbox{ and } \mathbb{P}\left(\max_{t\in[0,1]} \left|B(t)\right|\geq \delta C\right) < \varepsilon \cdot (2k)^{-1},
\end{equation}
where $B(t)$ is a standard Brownian motion on $[0,1]$. In addition, let $\lambda, A$ be as in the strong coupling Lemma \ref{prop:ThmA Shao} for $\varepsilon = \delta$, and let $\mathsf{T}_3$ be sufficiently large, depending on $\delta$ and $C$, so that for $T \geq \mathsf{T}_3$
\begin{equation}\label{S3DefT}
\exp\left(-\lambda ACT^{1/2} \right)\left(1+\frac{\lambda T}{\delta^2}\right) < \varepsilon \cdot (2k)^{-1}.
\end{equation}
We proceed to prove the lemma with the above choice of $\mathsf{T}_3$ and $M = (5k+2)C$.\\

We define $\vec{y}\,^{\operatorname{new}}=(y_1^{\operatorname{new}},\dots,y_k^{\operatorname{new}})\in\mathfrak{W}_k$ by
\begin{equation}\label{eq:choice of yi pushdown}
y_j^{\operatorname{new}}:=\left\lceil-T\sum_{i =1}^{j-1}\left| \mu_i - \mu_{i+1} \right|-5jC T^{1/2}\right\rceil,\quad j\in\llbracket1,k\rrbracket.\end{equation}
Using the monotone coupling Lemma \ref{lem:monotone coupling}, by pushing $g$ to $-\infty$ and $\vec{y}$ to $\vec{y}\,^{\operatorname{new}}$ (note that $y_1 \geq \cdots \geq y_k = 0$), the left side of \eqref{eq:ververgre} is lower bounded by the probability of the same event under  $\mathbb{P}_{\ice,\operatorname{Geom}}^{T,\vec{y}\,^{\operatorname{new}},\vec{q}, -\infty}$, which together with (\ref{eq:accept}) implies
\begin{equation}\label{eq:ergrg}
\begin{split}
&[\mbox{left side of (\ref{eq:ververgre})}] \geq 
\mathbb{P}_{\operatorname{Geom}}^{T,\vec{y}\,^{\operatorname{new}},\vec{q}}\Big{(}\Omega_{\ice}(T, \vec{y}\,^{\operatorname{new}},-\infty)\\
&\cap\Big{\{} Q_k(t) \geq \mu_k (t-T) -T\sum_{i=1}^{k-1}\left|\mu_i -\mu_{i+1}\right|-(5k+2)C T^{1/2}\Big{\}}\Big{)}.
\end{split}
\end{equation}

We next consider the events
$$E_i:=\left\{\max_{t\in\llbracket0,T\rrbracket}\left|Q_i(t)-y_i^{\operatorname{new}}+\mu_i(T-t)\right|\leq 2C T^{1/2} \right\},\quad i\in\llbracket1,k\rrbracket, $$
and let $E = E_1 \cap \cdots \cap E_k$. Note that on the event $E$ we have for any $t\in\llbracket0,T\rrbracket$
\begin{equation}\label{eq:btetg}
Q_k(t)\geq y_k^{\operatorname{new}}-\mu_k(T-t)-2CT^{1/2}\geq \mu_k(t-T)-T\sum_{i=1}^{k-1}\left|\mu_i -\mu_{i+1}\right|-(5k+2)CT^{1/2},
\end{equation}
where the last inequality uses \eqref{eq:choice of yi pushdown}. Note also that on $E$, for $i\in\llbracket 1,k-1\rrbracket$ and $t\in\llbracket0,T-1\rrbracket$, 
\begin{equation}\label{eq:brtbrtbr}
\begin{split}
Q_i(t)-Q_{i+1}(t+1) \,&\geq y_i^{\operatorname{new}}-y_{i+1}^{\operatorname{new}}-(T-t)(\mu_i-\mu_{i+1})-\mu_{i+1}-4C T^{1/2}\\
&\geq T|\mu_i-\mu_{i+1}|+5C T^{1/2}-1-(T-t)\left(\mu_i - \mu_{i+1}\right)-\mu_{i+1}-4C T^{1/2}\\
&\geq  C T^{1/2}-1-\mu_{i+1}\geq C-1-1/\delta>0,
\end{split}
\end{equation}
where we used (\ref{S3DefC}) and (\ref{eq:choice of yi pushdown}). Combining (\ref{eq:ergrg}), \eqref{eq:btetg} and \eqref{eq:brtbrtbr}, we conclude
\begin{equation}\label{eq:ergrg2}
\begin{split}
&[\mbox{left side of (\ref{eq:ververgre})}] \geq \prod_{i = 1}^k \mathbb{P}_{\operatorname{Geom}}^{T,y_i^{\operatorname{new}},q_i}\left(\max_{t\in\llbracket0,T\rrbracket}\left|Q_i(t)-y_i^{\operatorname{new}}  +\mu_i(T-t)\right|\leq 2C T^{1/2} \right).
\end{split}
\end{equation}

In the remainder of this proof we show that if $T \geq \mathsf{T}_3$, $y \in \mathbb{Z}$, $q \in [\delta, 1-\delta]$, $\mu = \frac{q}{1-q}$ and $\sigma = \sqrt{q}/(1-q)$, then
\begin{equation}\label{eq:ergrg3}
\begin{split}
&\mathbb{P}_{\operatorname{Geom}}^{T,y,q}\left(\max_{t\in\llbracket0,T\rrbracket}\left|Q_1(t) - y+ \mu(T-t)\right|\leq 2C T^{1/2} \right) \geq 1- \varepsilon/k.
\end{split}
\end{equation}
If true, then (\ref{eq:ergrg2}) and (\ref{eq:ergrg3}) would imply the statement of the lemma with $M = (5k+2)C$. \\

Let $G_j$ be i.i.d. geometric random variables with parameter $q$. Setting 
\begin{equation*}
Q_1(j) = y - \sum_{\ell = j+1}^T G_j, \quad j \in \llbracket 0, T\rrbracket,
\end{equation*}
we have that $Q_1$ has law $\mathbb{P}_{\operatorname{Geom}}^{T,y,q}$. By the strong coupling Lemma \ref{prop:ThmA Shao} with $\varepsilon = \delta$, we can couple $Q_1$ with $H_1,\dots,H_T\sim N(0,\sigma^2)$, so that
$$\mathbb{E}\left[\exp\left(\lambda A \Delta_T \right)\right]\leq1+ \frac{\lambda Tq}{(1-q)^2} \leq 1 + \frac{\lambda T}{\delta^2}, \mbox{ where } 
\Delta_T = \max_{i\in\llbracket0,T\rrbracket}\left| \sum_{j=i+1}^T {G}_j - \mu(T-i)-\sum_{j=i+1}^T H_j\right|.$$
In addition, by enlarging the space, we may assume that there is a standard Brownian motion $B$ on $[0,1]$ and $H_{T-i+1} = \sigma T^{1/2}  (B(i/T) - B((i-1)/T))$ for $i \in \llbracket 1, T\rrbracket$. 

From the above displayed equation and Chebyshev's inequality we obtain
\begin{equation}\label{eq:rtbsbrgbsr}
\mathbb{P}\left(\Delta_T \geq C T^{1/2} \right)\leq\exp\left(-\lambda ACT^{1/2} \right)\left(1+\frac{\lambda T}{\delta^2}\right).
\end{equation}
On the other hand, notice that since $\sigma\leq \delta^{-1}$, we have
\begin{equation}\label{eq:vsvtsrtbsrtb}
\mathbb{P}\left(\max_{i\in\llbracket0,T\rrbracket}\left|\sum_{j = i+1}^TH_{j}\right|\geq CT^{1/2}\right)
\leq \mathbb{P}\left(\max_{t\in[0,1]} \left|B(t)\right|\geq \sigma^{-1} C\right) \leq \mathbb{P}\left(\max_{t\in[0,1]} \left|B(t)\right|\geq \delta C\right).
\end{equation}
Combining \eqref{eq:rtbsbrgbsr} and \eqref{eq:vsvtsrtbsrtb}, we conclude
\begin{equation*}
    \begin{split}
        &\mathbb{P}_{\operatorname{Geom}}^{T,y,q}\left(\max_{t\in\llbracket0,T\rrbracket}\left|Q_1(t) - y+ \mu(T-t)\right|\leq 2C T^{1/2} \right)
        \geq\\
        &1-\exp\left(-\lambda ACT^{1/2} \right)\left(1+\frac{\lambda T}{\delta^2}\right) - \mathbb{P}\left(\max_{t\in[0,1]} \left|B(t)\right|\geq \delta C\right).
    \end{split}
\end{equation*}
The last inequality, combined with (\ref{S3DefC}) and (\ref{S3DefT}), implies (\ref{eq:ergrg3}). 
\end{proof}

%
%
\section{General conditions for tightness}\label{Section4} In this section we state the general tightness criterion for geometric line ensembles that satisfy the half-space interlacing Gibbs property -- this is Theorem \ref{thm:main thm tightness} below. We present the proof of this result in Sections \ref{Section4.2} and \ref{Section4.3} after we deduce a few preliminary results from \cite{dimitrov2024tightness} in Section \ref{Section4.1}. We continue with the same notation as in Section \ref{Section2}.

\begin{theorem}\label{thm:main thm tightness}
Let $p\in(0,1)$, $u=p/(1-p)$,  $\sigma=\sqrt{p}/(1-p)$, $K\in\mathbb{N}\cup\{\infty\}$, $\Sigma=\llbracket1,K+1\rrbracket$, $\mu_i \in \mathbb{R}$ for $i \in \llbracket 1, K \rrbracket$, and $\Lambda=[0,\beta)\subset\mathbb{R}$ be any non-empty interval. Let $d_n\in(0,\infty)$, $T_n\in\mathbb{Z}_{\geq 0}$ and $q_i^n\in(0,1)$ for $i\in \llbracket 1, K \rrbracket$ be sequences (in $n$) that satisfy the following conditions:
\begin{equation}\label{conditions}
d_n\rightarrow\infty, \hspace{2mM} T_n/d_n\rightarrow\beta, \hspace{2mm} q_i^n=p-\mu_i\sqrt{p}(1-p)\cdot d_n^{-1/2}+o\left(d_n^{-1/2}\right) \mbox{ for $i \in \llbracket 1, K \rrbracket$ as } n \rightarrow \infty.
\end{equation}
We further suppose that we have a sequence of $\Sigma$-indexed geometric line ensembles $\mathfrak{L}^n = \{L^n_i\}_{i \in \Sigma}$ on $\mathbb{Z}$ that satisfy the following conditions:
\begin{enumerate}
\item [$\bullet$] for each $t\in\Lambda$ and $i\in\llbracket1,K\rrbracket$, the random variables $\sigma^{-1}d_n^{-1/2} (L_i^n(\lfloor td_n\rfloor)-utd_n)$ are tight;
\item [$\bullet$] the restriction $\{L_i^n(s):i\in\Sigma \mbox{ and } s\in\llbracket0,T_n\rrbracket\}$, satisfies the half-space interlacing Gibbs property with parameters $\{q_i^n\}_{i \in \llbracket 1, K \rrbracket}$ from Definition \ref{def:HSIGP}.
\end{enumerate}
Then, the sequence of line ensembles $\mathcal{L}^n=\{\mathcal{L}_i^n\}_{i=1}^K\in C(\llbracket1,K\rrbracket\times\Lambda)$, defined through $\mathcal{L}_i^n(t)=\sigma^{-1}d_n^{-1/2}(L_i^n(td_n)-utd_n)$, is tight. Moreover, any subsequential limit satisfies the half-space Brownian Gibbs property with parameters $\{\mu_i\}_{i \in \llbracket 1, K-1 \rrbracket}$ from Definition \ref{def:BGP}.
\end{theorem}
\begin{remark}\label{rem:Tight1} In plain words Theorem \ref{thm:main thm tightness} says that in the presence of the half-space interlacing Gibbs property, one point tightness for a sequence of line ensembles implies tightness of the sequence as random elements in $C(\llbracket1,K\rrbracket\times\Lambda)$. Moreover, the half-space interlacing Gibbs property in the limit becomes the half-space Brownian Gibbs property.
\end{remark}
\begin{remark}\label{rem:Tight2} From Lemma \ref{LemmaConsistentGeom} we know that the restriction $\{L_i^n(s):i\in\Sigma \mbox{ and } s\in\llbracket0,T_n\rrbracket\}$, satisfies the interlacing Gibbs property from \cite[Definition 2.6]{dimitrov2024tightness}. In particular, we see that $\mathfrak{L}^n$ satisfy the conditions of \cite[Theorem 5.1]{dimitrov2024tightness} with $N = n$, $\alpha = \hat{A}_N = 0$, $\hat{B}_N = T_n$, $p = u$, and $\beta, d_n$ as in the present setup. From \cite[Theorem 5.1]{dimitrov2024tightness} we conclude that the restriction of $\mathcal{L}^n =\{\mathcal{L}_i^n\}_{i=1}^K$ to $(0, \beta)$, i.e. the interior of $\Lambda$, is a tight sequence in $C(\llbracket1,K\rrbracket\times (0, \beta))$, and any subsequential limit satisfies the partial Brownian Gibbs property from \cite[Definition 2.7]{DimMat}, recalled later as Definition \ref{DefPBGP}. The additional information contained in Theorem \ref{thm:main thm tightness} is that the tightness of the ensembles can be extended to the full interval $[0,\beta)$, and also that the asymptotic interaction of the curves at the origin is governed by the half-space Brownian Gibbs property.
\end{remark}

The proof of Theorem \ref{thm:main thm tightness} is the content of the remainder of this section, and in all statements we make the same assumptions as in the statement of the theorem. We mention that the tightness assumption in Theorem \ref{thm:main thm tightness} ensures the existence of functions $\psi_1(\cdot |i, t): (0,\infty) \rightarrow (0, \infty)$ and (\ref{conditions}) ensures the existence of functions $\psi_2(\cdot|i): \mathbb{N} \rightarrow (0,\infty)$, such that for each $i \in \Sigma$, $t \in \Lambda$, $\epsilon \in (0, \infty)$ we have for all $n \in \mathbb{N}$
\begin{equation}\label{S5E1}
\begin{split}
&\mathbb{P}\left( \left| L_i^n(\lfloor t d_n \rfloor) - utd_n \right| > d_n^{1/2} \cdot \psi_1(\epsilon| i, t) \right) \leq \epsilon.\\
& \left|q_i^n - p + \mu_i\sqrt{p}(1-p) \cdot d_n^{-1/2} \right| \leq d_n^{-1/2} \cdot \psi_2(n|i), \mbox{ and } \lim_n\psi_2(n|i) = 0.
\end{split}
\end{equation}
Throughout our proofs in the following sections we will encounter various constants that depend on $p$, the sequences $d_n, T_n$ in the statement of Theorem \ref{thm:main thm tightness} and also the functions $\psi_1(\cdot|i,t), \psi_2(\cdot|i)$. We will not list this dependence explicitly. 

%
%
\subsection{Results from \cite{dimitrov2024tightness}}\label{Section4.1} As explained in Remark \ref{rem:Tight2}, we have that $\mathfrak{L}^n$ satisfies the conditions of \cite[Theorem 5.1]{dimitrov2024tightness}. In particular, the results from \cite[Section 5]{dimitrov2024tightness} all hold in our setup, and in this section we recall those that will be used in the proof of Theorem \ref{thm:main thm tightness} in the next section. 

\begin{lemma}\label{lem:no big max} For any $b\in\Lambda$ with $b>0$ and $\varepsilon>0$, we can find $W_1\in\mathbb{N}$ and $M^{\mathrm{top}}>0$, depending on $b$ and $\varepsilon$, such that for $n\geq W_1$ we have
\begin{equation}
\mathbb{P}\left(\max_{t\in[0,b]}(L_1^n(td_n)-utd_n)\geq d_n^{1/2}\cdot M^{\mathrm{top}}\right)<\varepsilon.
\end{equation}
\end{lemma}
\begin{proof}
One can repeat verbatim the proof of \cite[Lemma 5.2]{dimitrov2024tightness} with $a= 0$.
\end{proof}

\begin{lemma}\label{S52L} For any $k \in \llbracket 1, K \rrbracket$, $a, b \in \Lambda$ with $0 < a < b$, and $\varepsilon \in (0,1)$ we can find $W_2 \in \mathbb{N}$, $\delta^{\mathsf{sep}}, \Delta^{\mathsf{sep}} > 0$, depending on $a,b, k$ and $\varepsilon$, such that for $n \geq W_2$, and $ s_0 \in \mathbb{Z} \cap [a \cdot d_n,b \cdot d_n]$ we have 
\begin{equation}\label{S52E1}
\begin{split}
&\mathbb{P}\left( \cap_{m = 1}^k E^{\mathsf{sep}}_m  \right) > 1- \varepsilon, \mbox{ where } E^{\mathsf{sep}}_m = \Big\{ L_m^n(s_0) - ps_0 \geq L_{m+1}^n(s) - ps + \delta^{\mathsf{sep}} \cdot d_n^{1/2} \\
& \mbox{ for all } m \in \llbracket 1, k \rrbracket \mbox{, } s \in \mathbb{Z} \cap [ s_0 - \Delta^{\mathsf{sep}} \cdot d_n,s_0 + \Delta^{\mathsf{sep}} \cdot d_n] \Big\}.
\end{split}
\end{equation}
\end{lemma}
\begin{remark} Equation (\ref{S52E1}) implies that $L_m^n(s_0) - ps_0$ is likely well-separated (on scale $d_n^{1/2}$) from $L_{m+1}^n(s_0) - ps_0$ for all $m \in \llbracket 1, k \rrbracket$, but in fact says something stronger. Namely, it shows that $L_m^n(s_0) - ps_0$ is likely well-separated from the whole curve $L_{m+1}^n(s) - ps$ on a small interval (on scale $d_n$) around $s_0$. 
\end{remark}
\begin{proof}
The statement follows directly from \cite[Lemma 5.3]{dimitrov2024tightness}.
\end{proof}

%
%
\subsection{Tightness of ensembles}\label{Section4.2} The goal of this section is to prove that $\mathcal{L}^n$ in Theorem \ref{thm:main thm tightness} is a tight sequence in $C(\llbracket 1,K \rrbracket \times \Lambda)$. By \cite[Lemma 2.4]{DEA21} it suffices to show that for each $d \in \Lambda \cap (0,\infty)$ and $k\in\llbracket1, K\rrbracket$, we have
\begin{equation}\label{eq:only need to show tightness}
\begin{split}
&\lim_{a\rightarrow\infty}\limsup_{n\rightarrow\infty}\mathbb{P}\left(|\mathcal{L}_k^n(d_n^{-1}D_n)|\geq a\right)=0,\mbox{ and for each } \eta>0,\\
&\lim_{\delta\rightarrow0}\limsup_{n\rightarrow\infty}\mathbb{P}\left(\sup_{x,y\in[0, d_n^{-1}D_n], |x-y|\leq\delta}\left|\mathcal{L}^n_k(x)-\mathcal{L}^n_k(y) \right|\geq \eta\right)=0,
\end{split}
\end{equation}
where we have set $D_n = \lfloor d \cdot d_n \rfloor$. The first line in \eqref{eq:only need to show tightness} follows from (\ref{S5E1}). To show the second line in \eqref{eq:only need to show tightness} it suffices to show that for any $\eta,\varepsilon>0$, there exist $W_3\in\mathbb{N}$ and $\delta>0$, such that for $n\geq W_3$, we have
\begin{equation}\label{eq:only need to show tightness 2}
\begin{split}
&\mathbb{P}\left(w_n(L^n_k, \delta)> \eta\right)<\varepsilon, \mbox{ where } \\
&w_n(L^n_k,\delta)=\sup_{x,y\in[0, D_n], |x-y|\leq\delta D_n}\left|\sigma^{-1} D_n^{-1/2}(L^n_k(x)-ux)-\sigma^{-1}D_n^{-1/2}(L^n_k(y)-uy)\right|. 
\end{split}
\end{equation} 
For clarity, we split the remainder of the proof into three steps.\\

{\bf \raggedleft Step 1.} We first utilize (\ref{S5E1}) and the results of Section \ref{Section4.1}, which all hold for $L^n$ as in the statement of the theorem. From Lemma \ref{lem:no big max} we can find $W_{3,1} \in \mathbb{N}$, $M^{\operatorname{top}} > 0$, such that for $n \geq W_{3,1}$ we have
\begin{equation*} 
\begin{split}
&\mathbb{P}\left(\max_{s\in[0, D_n] }\left(L_1^n(s)-us\right)\leq M^{\operatorname{top}}\cdot D_n^{1/2}\right)>1-\varepsilon/4.
\end{split}
\end{equation*}
In view of $L^n_i(s)\geq L^{n}_{i+1}(s)$ for all $s\in [0, D_n]$ and $i\in\llbracket1, k\rrbracket$, we have for $n \geq W_{3,1}$
\begin{equation}\label{eq:top in proof}
\mathbb{P}\left( E^{\operatorname{top}}\right)>1- \varepsilon/4 \mbox{, where }
E^{\operatorname{top}}=\left\{ 
\max_{s\in[0, D_n] }\left(L_{k+1}^n(s)-us\right)\leq M^{\operatorname{top}}\cdot D_n^{1/2} 
\right\}.
\end{equation} 
From (\ref{S5E1}) we can find $W_{3,2} \in \mathbb{N}$, $M^{\operatorname{side}} >0$, depending on $k$, such that for $n\geq W_{3,2}$
\begin{equation}\label{eq:side in proof}
\begin{split}
&\mathbb{P}\left(E^{\operatorname{side}}\right)\geq1-\varepsilon/4, \mbox{ where } E^{\operatorname{side}}=\left\{\left|L_i^{n}(D_n)- uD_n \right|\leq M^{\operatorname{side}}\cdot D_n^{1/2} \mbox{ for } i\in\llbracket1, k\rrbracket\right\}.  
\end{split}
\end{equation}
From Lemma \ref{S52L} we can find $W_{3,3} \in \mathbb{N}$, $\delta^{\operatorname{sep}}>0$, $\Delta^{\operatorname{sep}} \in(0,1/2)$, such that for $n\geq W_{3,3}$
\begin{equation}\label{eq:separation in proof}
\begin{split}
&\mathbb{P}\left(E^{\operatorname{sep}}\right)>1-\varepsilon/4, \mbox{ where } E^{\operatorname{sep}} =\Big\{L_m^n(D_n)-uD_n\geq L_{m+1}^n(s)-us+\delta^{\operatorname{sep}}\cdot D_n^{1/2}  \\
& \mbox{ for } m\in\llbracket1, k\rrbracket \mbox{, } s\in\mathbb{Z}\cap[(1-\Delta^{\operatorname{sep}})D_n, D_n]\Big\}. 
\end{split}
\end{equation}

We claim that there exist $W_4 \in \mathbb{N}$, $\delta>0$, such that for $n\geq W_4$
\begin{equation}\label{S4Claim1}
{\bf 1}\{E^{\operatorname{top}}\cap E^{\operatorname{side}}\cap E^{\operatorname{sep}} \} \cdot \mathbb{P}_{\ice,\operatorname{Geom}}^{D_n,\vec{y},\vec{q}^{\,n}_k, g}(w_n(Q_k^n,\delta)> \eta)<{\bf 1}\{E^{\operatorname{top}}\cap E^{\operatorname{side}}\cap E^{\operatorname{sep}} \} \cdot \varepsilon/4,
\end{equation}
where $\vec{q}^{\,n}_k = (q_1^n, \dots, q_k^n)$, $\vec{y}=\left(L^n_1(D_n), \dots, L^n_k(D_n)\right)$ and $g(s)= L^{n}_{k+1}(s)$ for $s\in\llbracket0,D_n\rrbracket$. 
We prove (\ref{S4Claim1}) in the next step. Here, we assume its validity and show (\ref{eq:only need to show tightness 2}).\\

We choose $W_3=\max(W_{3,1},W_{3,2}, W_{3,3},W_{4})$. By the half-space interlacing Gibbs property
\begin{equation*}\label{eq:revfagvear}
\mathbb{P}\left(E^{\operatorname{top}}\cap E^{\operatorname{side}}\cap E^{\operatorname{sep}}\cap\{w_n(L_k^n,\delta)> \eta \}\right)=
\mathbb{E}\left[\mathbf{1}_{E^{\operatorname{top}}}\cdot\mathbf{1}_{E^{\operatorname{side}}}\cdot\mathbf{1}_{E^{\operatorname{sep}}}\cdot\mathbb{P}_{\ice,\operatorname{Geom}}^{D_n,\vec{y},\vec{q}^{\,n}_k, g}(w_n(Q_k^n,\delta)> \eta)\right],
\end{equation*}
where $\vec{q}^{\,n}_k,\vec{y}, g$ are as above. From (\ref{S4Claim1}) we conclude
\[
\mathbb{P}\left(E^{\operatorname{top}}\cap E^{\operatorname{side}}\cap E^{\operatorname{sep}}\cap\{w_n(L_k^n,\delta)> \eta\}\right)
\leq\mathbb{P}\left(E^{\operatorname{top}}\cap E^{\operatorname{side}}\cap E^{\operatorname{sep}} \right)\cdot\varepsilon/4.
\]
Using the last equation, \eqref{eq:top in proof}, \eqref{eq:side in proof} and \eqref{eq:separation in proof}, we get for $n \geq W_3$
\[
\mathbb{P}\left(w_n(L_k^n,\delta)> \eta\right)< \mathbb{P}\left(E^{\operatorname{top}}\cap E^{\operatorname{side}}\cap E^{\operatorname{sep}}\cap\{w_n(L_k^n,\delta)>\eta\}\right)+3\varepsilon/4\leq\varepsilon,
\]
which proves \eqref{eq:only need to show tightness 2} and hence the tightness part in Theorem \ref{thm:main thm tightness}.\\

{\bf \raggedleft Step 2.} In this step we prove (\ref{S4Claim1}). Note that on the event $E^{\operatorname{top}}\cap E^{\operatorname{side}}\cap E^{\operatorname{sep}}$ we have:
\begin{enumerate}
    \item [$\bullet$] $\vec{y}\in\mathfrak{W}_k$, and $\left|y_i- uD_n \right|\leq M^{\operatorname{side}}\cdot D_n^{1/2}$ for $i\in\llbracket1,k\rrbracket$;
    \item [$\bullet$] $y_i-y_{i+1}\geq\delta^{\operatorname{sep}}\cdot D_n^{1/2}$ for $i\in\llbracket1,k-1\rrbracket$;
    \item [$\bullet$] $g:\llbracket0,D_n\rrbracket\rightarrow[-\infty,\infty)$ satisfies $g(s)-us\leq M^{\operatorname{top}}\cdot D_n^{1/2}$ for $s\in\llbracket0,D_n\rrbracket$;
    \item [$\bullet$] $y_k-uD_n\geq g(s)-us+\delta^{\operatorname{sep}}\cdot D_n^{1/2}$ for all $s\in\llbracket0,D_n\rrbracket$ with $s\geq D_n-\Delta^{\operatorname{sep}}\cdot D_n$;
\end{enumerate}

We set $\tilde{A}=\delta^{\operatorname{sep}}/4$ and for $i\in\llbracket1,k\rrbracket$ define the functions
\begin{align*}
& f^{t}_i(s)=f(s|\sigma^{-1}\tilde{A},\sigma^{-1}(M^{\operatorname{top}}+M^{\operatorname{side}}+4(k-i+1)\tilde{A}),\sigma^{-1}D_n^{-1/2}(y_i- uD_n), \Delta^{\operatorname{sep}}/2),\\
& f^{b}_i(s)=f(s|-\sigma^{-1}\tilde{A},\sigma^{-1}(M^{\operatorname{top}}+M^{\operatorname{side}}+4(k-i+1)\tilde{A}),\sigma^{-1}D_n^{-1/2}(y_i- uD_n), \Delta^{\operatorname{sep}}/2),
\end{align*} 
where we recall $f(s|A,B,y,\delta)$ is defined by \eqref{eq:defining the corridor}.
Using Lemma \ref{lem:staying in corridor} for $q_0=p$, $\delta_1=\Delta^{\operatorname{sep}}/2$, $A=\sigma^{-1}\tilde{A}$, $M_1=M^{\operatorname{side}}$, $M_2=\sigma^{-1}(M^{\operatorname{top}}+M^{\mathsf{side}}+4k\tilde{A})$ and
$M=1 + \sqrt{dp}(1-p) \cdot \max_{i \in \llbracket 1, k \rrbracket}|\mu_i| $,
we conclude there exists $W_{4,1}\in\mathbb{N}$ and $\epsilon_0>0$ such that for $n\geq W_{4,1}$ and $i\in\llbracket1, k\rrbracket$
\begin{equation} \label{eq:eargarerw}
\mathbb{P}_{\operatorname{Geom}}^{D_n,y_i,q_i^n}\left(f^{t}_i(s)\geq\mathcal{Q}_1(s) \geq f^{b}_i(s) \mbox{ for all } s\in [0,1] \right)>\epsilon_0.
\end{equation} 
We let $W_{4,2}\in\mathbb{N}$ be sufficiently large so that for $n\geq W_{4,2}$, we have $n\geq W_{4,1}$, $\tilde{A}D_n^{1/2}\geq u$, $\Delta^{\operatorname{sep}}/2\geq1/D_n$, and 
\begin{equation} \label{eq:begvgver}
f_i^{b}\left(\frac{s}{D_n}\right)\geq f_{i+1}^{t}\left(\frac{s+1}{D_n}\right)+\frac{u}{\sigma D_n^{1/2}}
\quad\mbox{ for }s\in\llbracket0,D_n-1\rrbracket\mbox{, }i\in\llbracket1,k-1\rrbracket.
\end{equation}
We mention that our choice of $W_{4,2}$ is possible since the functions $f_{i+1}^{t}$ and $f_i^{b}$ are Lipschitz continuous with constant $(\Delta^{\operatorname{sep}}/2)^{-1}\cdot\sigma^{-1}\left(M^{\operatorname{side}} + M^{\operatorname{top}} + 4k \tilde{A} \right)$, while by construction
$$\inf_{s \in [0,1]} \left(f_i^{b}(s) - f^{t}_{i+1}(s) \right) = 2\sigma^{-1} \tilde{A}.$$

We claim that  
\begin{equation}\label{S4Claim2}
F \subseteq \Omega_{\ice}(D_n, \vec{y}, g), \mbox{ where }F = \left\{  \{Q_i\}_{i = 1}^k: Q_i \in \Omega(D_n, y_i) \mbox{ and }   f^{t}_i \geq\mathcal{Q}_i \geq f^{b}_i \mbox{ on }[0,1]\right\},
\end{equation}
where we recall that $\Omega_{\ice}(D_n, \vec{y}, g)$ was defined below (\ref{EventInter}) and we set $\mathcal{Q}_i(t)=\sigma^{-1}D_n^{-1/2}(Q_k^n(tD_n)-utD_n)$ for $t\in[0,1]$. We prove (\ref{S4Claim2}) in the next step. Here, we assume its validity and show (\ref{S4Claim1}).\\

From (\ref{eq:eargarerw}) and \eqref{S4Claim2} we have
$$ 
\mathbb{P}_{\operatorname{Geom}}^{D_n,\vec{y},\vec{q}^{\,n}_k}\left(\Omega_{\ice}(D_n, \vec{y}, g)\right)
 \geq \mathbb{P}_{\operatorname{Geom}}^{D_n,\vec{y},\vec{q}^{\,n}_k}(F) =\prod_{i=1}^k\mathbb{P}_{ \operatorname{Geom}}^{D_n,y_i,q_i^n}\left(f^{t}_i\geq\mathcal{Q}_1\geq f^{b}_i \mbox{ on }[0,1] \right)
 >\epsilon_0^k.$$
Using Lemma \ref{lem:modulus of continuity bound} (with $\varepsilon = \epsilon_0^k \cdot \varepsilon/4$), there exist $W_{4,3}\in\mathbb{N}$ and $\delta>0$ such that for $n\geq W_{4,3}$
$$\mathbb{P}_{\operatorname{Geom}}^{D_n,\vec{y},\vec{q}_k^{\,n}}\left(w(\mathcal{Q}_k,\delta)> \eta\right) =\mathbb{P}_{\operatorname{Geom}}^{D_n,y_k,q_k^n}\left(w(\mathcal{Q}_1,\delta)> \eta\right) <\epsilon_0^k\cdot \varepsilon/4,$$
where $w(f,\delta)$ is the modulus of continuity on $[0,1]$ from \eqref{eq:def of modulus of continuity}. 

By combining the last two displayed equations and (\ref{eq:accept}), we conclude for $n \geq \max(W_{4,2},W_{4,3})$
$$\mathbb{P}_{\ice,\operatorname{Geom}}^{D_n,\vec{y},\vec{q}_k^{\,n},g}(w(\mathcal{Q}_k,\delta)> \eta)
=\frac{ \mathbb{P}_{\operatorname{Geom}}^{D_n,\vec{y},\vec{q}_k^{\,n}}\left(\{w(\mathcal{Q}_k,\delta)> \eta\}\cap\Omega_{\ice}(D_n, \vec{y}, g)\right)}{\mathbb{P}_{\operatorname{Geom}}^{D_n,\vec{y},\vec{q}^{\,n}_k}\left(\Omega_{\ice}(D_n, \vec{y}, g) \right)}\leq\epsilon_0^k\cdot \varepsilon/4\cdot\epsilon_0^{-k}=\varepsilon/4.
$$
The last equation and the identity $w_n(Q_k^n,\delta)=w(\mathcal{Q}_k,\delta)$ give (\ref{S4Claim1}) with $W_4 = \max(W_{4,2},W_{4,3})$.\\

{\bf \raggedleft Step 3.} In this final step we verify (\ref{S4Claim2}). We only need to check that for $\{Q_i\}_{i=1}^k \in F$
\begin{equation} \label{eq:verify interlace}
\begin{split}
&Q_i(s)\geq Q_{i+1}(s+1)\mbox{ for } i \in \llbracket 1, k-1 \rrbracket \mbox{ and } s \in \llbracket 0, D_n -1\rrbracket\mbox{, and } \\
& Q_k(s) \geq g(s+1) \mbox{ for } s \in \llbracket 0, D_n -1\rrbracket.
\end{split}
\end{equation}
The first line in \eqref{eq:verify interlace} is equivalent to 
$$\mathcal{Q}_i\left( \frac{s}{D_n} \right) \geq \mathcal{Q}_{i+1}\left( \frac{s+1}{D_n} \right) + \frac{u}{\sigma D_n^{1/2}},$$
which follows from \eqref{eq:begvgver} and the definition of $F$. We next verify the second line of \eqref{eq:verify interlace}.

Suppose that $s \in [(1-\Delta^{\operatorname{sep}}/2) \cdot D_n, D_n-1]$. We have 
\begin{equation*}
\begin{split}
&Q_k(s) = \sigma D_n^{1/2} \cdot \mathcal{Q}_k(s/D_n)+us \geq \sigma D_n^{1/2} \cdot f_k^{b}(s/D_n)+us\geq \sigma D_n^{1/2} \cdot f_k^{b}(1)+us \\
&=y_k-uD_n-\tilde{A}\cdot D_n^{1/2}+us\geq[g(s+1)-u(s+1)+\delta^{\operatorname{sep}}\cdot D_n^{1/2}]-\tilde{A}\cdot D_n^{1/2}+us \geq g(s + 1),
\end{split}
\end{equation*}
where in the first inequality we used $\mathcal{Q}_k\geq f^{b}_k$, in the second one and in going from the first to the second line we used the definition of $f_k^{b}$, in the first inequality on the second line we used the lower bound of $y_k$, and in the last inequality we used that $\delta^{\operatorname{sep}} = 4 \tilde{A}$.

Suppose that $s \in [0, (1-\Delta^{\operatorname{sep}}/2)\cdot D_n]$. We have
\begin{equation*}
\begin{split}
&Q_k(s)=\sigma D_n^{1/2}\cdot\mathcal{Q}_k(s/D_n)+us\geq\sigma D_n^{1/2}\cdot f_k^{b}(s/D_n)+us=D_n^{1/2}\cdot (M^{\operatorname{top}} + M^{\operatorname{side}} +3\tilde{A})+us \\
& \geq [g(s+1)-u(s+1)] +D_n^{1/2} \cdot(M^{\operatorname{side}}+3\tilde{A})+us\geq g(s + 1),
\end{split}
\end{equation*}
where in the first inequality we used $\mathcal{Q}_k\geq f^{b}_k$, in the following equality we used the definition of $f_k^{b}$, and in going from the first to the second line we used the upper bound for $g$. 
The last two displayed equations imply the second line of \eqref{eq:verify interlace}, which completes the proof of (\ref{S4Claim2}).

%
%
\subsection{Brownian Gibbs property in the limit}\label{Section4.3} In this section we complete the proof of Theorem \ref{thm:main thm tightness} by showing that any subsequential limit of $\mathcal{L}^n$ satisfies the half-space Brownian Gibbs property with parameters $\{\mu_i\}_{i \in \llbracket 1, K-1 \rrbracket}$ from Definition \ref{def:BGP}. Our arguments are quite similar to those in \cite[Section 5.4]{dimitrov2024tightness}, and so we will be brief.

We assume that $\mathcal{L}^{\infty}$ is any subsequential limit and by possibly passing to a subsequence, which we continue to call $\mathcal{L}^n$, we assume that $\mathcal{L}^n \Rightarrow \mathcal{L}^{\infty}$. By Skorohod's Representation Theorem, see \cite[Theorem 6.7]{Billing}, we may assume that the sequence $\{\mathcal{L}^n\}_{n \geq 1}$ and $\mathcal{L}^{\infty}$ are all defined on the same probability space $(\Omega, \mathcal{F}, \mathbb{P})$ and $\lim_n \mathcal{L}^n(\omega) = \mathcal{L}^{\infty}(\omega)$ for each $\omega \in \Omega$. We mention that in applying \cite[Theorem 6.7]{Billing} we implicitly used that $C(\llbracket 1, K \rrbracket \times \Lambda)$ is a Polish space, cf. Remark \ref{RemPolish}. Note that by Lemma \ref{S52L} we have for each fixed $t \in (0, \beta)$ that $\mathbb{P}$-a.s.
\begin{equation}\label{S54E1}
\mathcal{L}^{\infty}_i(t) > \mathcal{L}^{\infty}_{i+1}(t) \mbox{ for all } i \in \llbracket 1, K -1 \rrbracket.
\end{equation}

When $K = 1$ there is nothing to check in Definition \ref{def:BGP}, and so we assume that $K \geq 2$. Fix $k \in \llbracket 1, K - 1 \rrbracket$, $m \in \mathbb{N}$, $n_1, \dots, n_m \in \llbracket 1 , K \rrbracket$, $b, t_1, \dots, t_m \in \Lambda$ and bounded continuous $h_1, \dots, h_m : \mathbb{R} \rightarrow \mathbb{R}$. Define $R = \{i \in \llbracket 1, m \rrbracket: n_i \in \llbracket 1 , k \rrbracket, t_i \in [0,b]\}$. We claim that 
\begin{equation}\label{S54E2}
\mathbb{E}\left[ \prod_{i = 1}^m h_i(\mathcal{L}^{\infty}_{n_i}(t_i)) \right] = \mathbb{E}\left[ \prod_{i \not \in R} h_i(\mathcal{L}^{\infty}_{n_i}(t_i))  \cdot \mathbb{E}_{\operatorname{avoid}}^{b,\vec{y}, \vec{\mu}_{k},g} \left[ \prod_{i  \in R} h_i(\mathcal{Q}_{n_i}(t_i))   \right] \right], 
\end{equation}
where $\vec{y} = (\mathcal{L}^{\infty}_{1} (b), \dots, \mathcal{L}^{\infty}_{k} (b))$, $g = \mathcal{L}^{\infty}_{k+1}[0,b]$. Assuming (\ref{S54E2}) for the moment, we have by a standard monotone class argument, see Lemma \ref{MCA}, that for any bounded Borel-measurable $F: C(\llbracket 1, k \rrbracket \times [0,b]) \rightarrow \mathbb{R}$ and any bounded $\mathcal{F}_{\operatorname{ext}} (\llbracket 1,k\rrbracket \times [0,b))$-measurable $Y$ (recall that $\mathcal{F}_{\operatorname{ext}} (\llbracket 1,k\rrbracket \times [0,b))$ was given in Definition \ref{def:BGP}), we have that
\begin{equation}\label{S54E3}
\mathbb{E}\left[ F\left( \mathcal{L}^{\infty} \vert_{\llbracket 1, k \rrbracket \times [0,b]} \right) \cdot Y \right] = \mathbb{E}\left[ \mathbb{E}_{\operatorname{avoid}}^{b,\vec{y}, \vec{\mu}_{k}, g} \left[ F(\mathcal{Q}) \right]\cdot Y \right],
\end{equation}
and for any bounded Borel-measurable $G: C(\llbracket 1, k +1 \rrbracket \times [0,b]) \rightarrow \mathbb{R}$
\begin{equation}\label{S54E4}
\mathbb{E}\left[ G\left( \mathcal{L}^{\infty}_1[0,b], \dots, \mathcal{L}_{k+1}^{\infty}[0,b] \right)\right] = \mathbb{E}\left[ \mathbb{E}_{\operatorname{avoid}}^{b,\vec{y}, \vec{\mu}_{k}, g} \left[ G(\mathcal{Q}_1, \dots, \mathcal{Q}_k, g) \right] \right],
\end{equation}
where $\mathcal{L}^{\infty} \vert_{\llbracket 1, k \rrbracket \times [0,b]}$ is the restriction of $\mathcal{L}^{\infty}$ to $\llbracket 1, k \rrbracket \times [0,b]$ and $\vec{y}, g$ are as above.

If we set in (\ref{S54E4})
$$G(f_1, f_2, \dots,f_{k+1}) = {\bf 1}\{f_1(s) > f_2(s) > \cdots > f_{k+1}(s) \mbox{ for all } s \in [0,b]\},$$
and use that by Definition \ref{def: avoidBLE} we have $\mathbb{E}_{\operatorname{avoid}}^{b,\vec{y}, \vec{\mu}_{k}, g} \left[G(\mathcal{Q}_1, \dots, \mathcal{Q}_k, g) \right] = 1$, we conclude 
\begin{equation}\label{S54E5}
\mathbb{P}\left(\mathcal{L}^{\infty}_i(t) >\mathcal{L}^{\infty}_{i+1}(t) \mbox{ for all }  t\in [0,b] , i \in \llbracket 1, k \rrbracket \right) = 1.
\end{equation}
Taking a countable sequence of intervals $[0,b]$ that exhausts $\Lambda$ and an increasing sequence of $k$'s converging to $K-1$, and taking intersections of the events in (\ref{S54E5}) shows that $\mathcal{L}^{\infty}$ is non-intersecting. In addition,  using the defining properties of conditional expectation, we see that (\ref{S54E3}) implies (\ref{eq:HSBGP}) and so we have reduced our proof to establishing (\ref{S54E2}).\\

From the convergence of $\mathcal{L}^n$ to $\mathcal{L}^{\infty}$ we get
\begin{equation}\label{S54E8}
\lim_{n \rightarrow \infty} h_i(\mathcal{L}^{n}_{n_i}(t_i)) = h_i(\mathcal{L}^{\infty}_{n_i}(t_i))
\end{equation}
for each $i \in \llbracket 1, m \rrbracket$, and so by the bounded convergence theorem
\begin{equation}\label{S54E9}
\lim_{n \rightarrow \infty} \mathbb{E}\left[ \prod_{i = 1}^m h_i(\mathcal{L}^{n}_{n_i}(t_i)) \right] = \mathbb{E}\left[ \prod_{i = 1}^m h_i(\mathcal{L}^{\infty}_{n_i}(t_i)) \right].
\end{equation}
In addition, if we set $B_n = \lceil b \cdot d_n \rceil$, $\vec{Y}^n = (L_{1}^n(B_n), \dots, L_{k}^n(B_n))$, $G_n(t) = L_{k+1}^n(t)$ for $t \in [0,B_n]$, then from the convergence of $\mathcal{L}^n$ to $\mathcal{L}^{\infty}$ and (\ref{S54E1}), we know that $\mathbb{P}$-a.s. the sequences $B_n, d_n,\vec{Y}^n, G_n, \vec{q}^{\,n}_k = (q_1^n, \dots, q_k^n)$ satisfy the conditions of Lemma \ref{lem:RW} and so $\mathbb{P}$-a.s.
\begin{equation}\label{S54E10}
\lim_{n \rightarrow \infty} \mathbb{E}^{B_n,\vec{Y}^n,\vec{q}^{\,n}_k, G_n}_{\ice, \operatorname{Geom}} \left[ \prod_{i  \in R} h_i(\mathcal{Q}^n_{n_i}(t_i))   \right] = \mathbb{E}_{\operatorname{avoid}}^{b,\vec{y},\vec{\mu}_k,g} \left[ \prod_{i  \in R} h_i(\mathcal{Q}_{n_i}(t_i))   \right],
\end{equation}
where on the left side of (\ref{S54E10}) we have $ \mathcal{Q}^n_{i}(t) = \sigma^{-1} d_n^{-1/2}  (Q^n_{i}(td_n) - utd_n)$ for $i \in \llbracket 1, k \rrbracket$ with $\mathfrak{Q}^n = \{Q_{i}^n\}_{i = 1}^{k}$ having law $\mathbb{P}^{ B_n,\vec{Y}^n,\vec{q}_k^{\,n},G_n}_{\ice, \operatorname{Geom}}$. Combining (\ref{S54E8}) with (\ref{S54E10}) and the bounded convergence theorem gives
\begin{equation}\label{S54E11}
\begin{split}
&\lim_{n \rightarrow \infty} \mathbb{E}\left[ \prod_{i \not \in R} h_i(\mathcal{L}^{n}_{n_i}(t_i))  \cdot \mathbb{E}^{B_n,\vec{Y}^n,\vec{q}^{\,n}_k,G_n}_{\ice, \operatorname{Geom}} \left[ \prod_{i  \in R} h_i(\mathcal{Q}^n_{n_i}(t_i))   \right] \right]  \\
&= \mathbb{E}\left[ \prod_{i \not \in R} h_i(\mathcal{L}^{\infty}_{n_i}(t_i))  \cdot \mathbb{E}_{\operatorname{avoid}}^{b,\vec{y},\vec{\mu}_k,g} \left[ \prod_{i  \in R} h_i(\mathcal{Q}_{n_i}(t_i))   \right] \right].
\end{split}
\end{equation}
Finally, if $n$ is large enough so that $[0, B_n] \subseteq [0, T_n]$, we have by the half-space interlacing Gibbs property that the terms on the first line in (\ref{S54E11}) agree with those on the left in (\ref{S54E9}), and so the limits agree, which is precisely (\ref{S54E2}).

%
%
\section{Pfaffian Schur processes}\label{Section5} The goal of this section is to define the Pfaffian Schur processes that we study. We do this in Section \ref{Section5.3} below, after we introduce some of the general theory for (Pfaffian) point processes in Sections \ref{Section5.1} and \ref{Section5.2}. In Section \ref{Section5.4} we state the main result we establish about the convergence of Pfaffian Schur processes, see Proposition \ref{prop: FinDimConv}, and then use the latter in conjunction with our general tightness result, see Theorem \ref{thm:main thm tightness}, to prove Theorem \ref{thm:MainThm1}.

%
%
\subsection{Point processes and correlation functions}\label{Section5.1}
In this section we give some background on point processes and correlation functions. Our exposition largely follows \cite[Section 2]{dimitrov2024airy}, which in turn goes back to \cite{johansson2006random}. Throughout this section, we fix $k\in\mathbb{N}$ and denote the space $\mathbb{R}^k$ with its Borel $\sigma$-algebra by $(E,\mathcal{E}) = (\mathbb{R}^k, \mathcal{R}^k)$. We also fix a probability space $(\Omega,\mathcal{F},\mathbb{P})$.

 We say that $M:\Omega\times\mathcal{E} \rightarrow [0,\infty]$ is a {\em random measure} on $(E,\mathcal{E})$ if $M_{\omega}=M(\omega,\cdot)$ is a measure on $(E,\mathcal{E})$ for each $\omega \in \Omega$ and $M(\cdot,A) = M(A)$ is an extended random variable for each $A\in\mathcal{E}$. We define the {\em mean} of $M$ to be the measure $\mu$ on $(E,\mathcal{E})$ given by $\mu(A)=\mathbb{E}[M(A)]$ for $A\in\mathcal{E}$. We say that $M$ is {\em locally finite} if for each bounded Borel set $B\in\mathcal{E}$ and $\omega\in\Omega$, we have $M(\omega,B) < \infty$. We denote by $\mathcal{M}_{E}$ the space of locally bounded measures on $E$, and equip it with the vague topology, see \cite[Chapter 4]{kallenberg2017random} for background on the latter. For brevity we denote $\mathcal{M}_E$ by $S$ and its Borel $\sigma$-algebra by $\mathcal{S}$. From  \cite[Chapter 4]{kallenberg2017random} we have that $S$ is a Polish space and so a locally finite random measure $M$ can be viewed as a random element in $(S,\mathcal{S})$ in the sense of \cite[Section 3]{Billing}.

We say that a random element $M$ in $(S, \mathcal{S})$ is a {\em point process} if $M_{\omega}(B)\in\mathbb{Z}_{\geq 0}$ for each $\omega\in\Omega$ and each bounded Borel set $B\in\mathcal{E}$. We call $M$ a {\em simple point process} if it is a point process and $M(\omega,\{x\})\in\{0,1\}$ for all $\omega\in\Omega$ and $x\in E$. 

Suppose that $(\bar{E}, \bar{\mathcal{E}})$ is a measurable space such that $E \subseteq \bar{E}$ and $\mathcal{E} \subseteq \bar{\mathcal{E}}$. We say that a sequence of random elements $\{X_n \}_{n \geq 1}$ in $(\bar{E}, \bar{\mathcal{E}})$, all defined on $(\Omega, \mathcal{F}, \mathbb{P})$, {\em forms} the random measure $M$ if 
\begin{equation}\label{FormX}
M(\omega, A) = \sum_{n \geq 1} {\bf 1} \{X_n(\omega) \in A\}
\end{equation}
for each $\omega \in \Omega$ and $A \in \mathcal{E}$. Unless otherwise specified, we will assume that $\bar{E} = E \cup \{\partial\}$ and $\bar{\mathcal{E}} = \sigma(\mathcal{E}, \{\partial\})$ where $\partial$ is an extra point we add to the space $E$. We mention that the extra point $\partial$ is sometimes called a {\em cemetery}, and its inclusion is standard in the theory of random measures, see \cite[Chapter 6]{Cinlar}. As shown in \cite[Lemma 2.1]{dimitrov2024airy}, if $M$ is a point process on $(E,\mathcal{E})$, then there exists a sequence of random elements $\{X_n \}_{n \geq 1}$ in $(\bar{E},\bar{\mathcal{E}})$, all defined on $(\Omega,\mathcal{F},\mathbb{P})$, which satisfy (\ref{FormX}).

Suppose that $M$ is a point process on $(E, \mathcal{E})$ that is formed by $\{X_n \}_{n \geq 1}$ as in (\ref{FormX}). With this data we can construct for any $n \in \mathbb{N}$ a new point process $M_n$ on $(E^n, \mathcal{E}^{\otimes n})$ through
\begin{equation}\label{eq: Mn}
M_n(\omega, A) = \sum_{\substack{i_1, \dots, i_n = 1 \\ |\{i_1, \dots, i_n\}| = n} }^\infty {\bf 1} \{ (X_{i_1}(\omega), \dots, X_{i_n} (\omega)) \in A) \} \mbox{ for } \omega \in \Omega \mbox{ and } A \in \mathcal{E}^{\otimes n}.
\end{equation}
We also denote the mean of $M_n$ by $\mu_n$ -- this is a measure on $(E^n, \mathcal{E}^{\otimes n})$. From (\ref{eq: Mn}) we have that $M_n$ depends only on the measure $M$ (and not the particular definition and order of the variables $X_n$) and for each $\omega \in \Omega$ the measure $M_n(\omega, \cdot)$ is symmetric in the sense that 
$$M_n(\omega, B_1 \times \cdots \times B_n) = M_n(\omega, B_{\sigma(1)} \times \cdots \times B_{\sigma(n)})$$
for any permutation $\sigma \in S_n$ and Borel sets $B_1, \dots, B_n \in \mathcal{E}$. The mean $\mu_n$ is also symmetric. \\

Suppose that $\lambda$ is a locally finite measure on $(E, \mathcal{E})$, called a {\em reference measure}, and denote by $\lambda^n$ the product measure of $n$ copies of $\lambda$ on $E^n$. Suppose that $M$ is a point process on $(E, \mathcal{E})$ and let $\mu_n$ be the mean measures of $M_n$ as in (\ref{eq: Mn}). If $\mu_n$ is absolutely continuous with respect to $\lambda^n$, then we call its density $\rho_n$ the {\em $n$-th correlation function} of $M$ (with respect to the reference measure $\lambda$). Specifically, $\rho_n : E^n \rightarrow [0, \infty]$ is $\mathcal{E}^{\otimes n}$-measurable and satisfies  
\begin{equation}\label{eq:RND determinantal}
\mu_n(A)=\int_A\rho_n(x_1,\dots,x_n)\lambda^n(dx)\quad\mbox{ for all }A\in\mathcal{E}^{\otimes n}.
\end{equation}
We mention that given a point process $M$ it is not necessary that functions $\rho_n$ satisfying \eqref{eq:RND determinantal} exist. However, if they exist, then they are unique $\lambda^n$-a.e. and are also symmetric $\lambda^n$-a.e.\\

We recall the following result for future use.
\begin{proposition}\label{prop:Last Particle Cdf}\cite[Proposition 2.4]{johansson2006random} Assume that $M$ is a point process on $E=\mathbb{R}$ that has correlation functions $\rho_n$ for each $n \geq 1$. We define 
\[ 
X_{\mathsf{max}}(\omega) = \inf\{ t \in \mathbb{R} : M(\omega, (t,\infty)) = 0 \},
\]
and note $X_{\mathsf{max}}$ is an extended random variable in $[-\infty, \infty]$.
Suppose that for each $t \in \mathbb{R}$
\begin{equation}\label{eq: finite series correlation}
1+\sum_{n=1}^{\infty}\frac{1}{n!}\int_{(t,\infty)^n}\rho_n(x_1,\dots,x_n)\lambda^n(dx)<\infty.
\end{equation}
Then, for each $t \in \mathbb{R}$ we have 
\begin{equation}\label{eq: last particle cdf}
\mathbb{P}\left(X_{\mathsf{max}}\leq t\right)=1+\sum_{n=1}^{\infty}\frac{(-1)^n}{n!}\int_{(t,\infty)^n}\rho_n(x_1, \dots, x_n)\lambda^n(dx). 
\end{equation}
\end{proposition}
\begin{remark}\label{rem: last particle} We mention that equation (\ref{eq: finite series correlation}) ensures that $\mathbb{P}(X_{\mathsf{max}} < \infty) = 1$; however, it is possible to have $\mathbb{P}(X_{\mathsf{max}} = - \infty) > 0$, see \cite[Remark 2.10]{dimitrov2024airy}.
\end{remark}

We end this section with two technical statements about limits of point processes. The first, Lemma \ref{lem:technical lemma fdd 1} below, provides sufficient conditions that ensure that the random variables $M^N([t, \infty))$ converge weakly to $M^{\infty}([t, \infty))$ when $M^N$ converge weakly to $M^{\infty}$ (as random elements in $\mathcal{M}_{\mathbb{R}}$). The second, Lemma \ref{lem:technical lemma fdd 2} below, roughly states that if $M^N$ converge weakly to $M^{\infty}$, the number of atoms of $M^N$ stochastically decreases, while $M^{\infty}$ has infinitely many atoms, then $M^N$ all have infinitely many atoms. The two lemmas are proved in Sections \ref{SectionB2} and \ref{SectionB3}.

\begin{lemma}\label{lem:technical lemma fdd 1} Suppose that $X^N = \left(X_i^{N}: i \geq 1\right)$ is a sequence of random elements in $(\mathbb{R}^{\infty}, \mathcal{R}^{\infty})$, see \cite[Example 1.2]{Billing}, such that 
\begin{equation}\label{eq:OrderedX}
X_i^{N}(\omega) \geq X_{i+1}^{N}(\omega) \mbox{ for each } \omega \in \Omega,  i \geq 1.
\end{equation}
Let $M^N$ be the random measures formed by $\{X_i^N\}_{i \geq 1}$ as in (\ref{FormX}). We assume the following.
\begin{enumerate}
\item The sequence $(X_1^N)^+ = \max(0, X_1^N)$ is tight.
\item $M^N$ is a point process for each $N \in \mathbb{N}$, and the sequence $M^N$ converges weakly to $M^{\infty}$. 
\item  For each $t\in\mathbb{R}$ we have $\mathbb{P}(M^{\infty}(\{t\})=0)  =1$.
\end{enumerate}
Then, for all $t\in\mathbb{R}$ and $a \in\mathbb{Z}_{\geq0}$, we have 
\begin{equation}\label{eq:conv lem1 fdd}
\lim_{N\rightarrow\infty}\mathbb{P}\left(M^N([t,\infty))\leq a \right)=\mathbb{P}\left(M^{\infty}([t,\infty))\leq a\right).
\end{equation}
\end{lemma}

\begin{lemma}\label{lem:technical lemma fdd 2}
Suppose that $\{M^N\}_{N\geq 1}$ is a sequence of point processes that converge weakly to a point process $M^{\infty}$. Assume that for any fixed $a \in\mathbb{Z}_{\geq 0}$, the sequence $p_N^a:=\mathbb{P}\left(M^N(\mathbb{R})\geq a \right)$ is decreasing in $N$, and $\mathbb{P}(M^{\infty}(\mathbb{R}) = \infty) = 1$. Then, $\mathbb{P}(M^{N}(\mathbb{R}) = \infty) = 1$ for all $N \in \mathbb{N}$.
\end{lemma}

%
%
\subsection{Pfaffian point processes}\label{Section5.2} In this section we introduce Pfaffian point processes and state some of their properties. We refer the interested reader to \cite[Appendix B]{OQR17} and \cite{R00} for more background on Pfaffian point processes. We continue with the notation from Section \ref{Section5.1}. Throughout this section we denote by $\mathrm{Mat}_n(\mathbb{C})$ the space of $n \times n$ matrices with complex entries, and $\mathrm{Skew}_n(\mathbb{C})$ the subset of skew-symmetric matrices. 

\begin{definition}\label{def: DefDPP} Suppose that $M$ is a point process on $(E,\mathcal{E})$ and $\lambda$ is a locally finite measure on $(E,\mathcal{E})$. Suppose that $K: E^2\rightarrow \mathbb{C}$ is a locally bounded measurable function, called the {\em correlation kernel}. We say that $M$ is a {\em determinantal point process with correlation kernel $K$ and reference measure $\lambda$} if the $n$-th correlation function $\rho_n$ exists for each $n \geq 1$ and 
\begin{equation}\label{eq:def of DPP}
\rho_n(x_1,\dots,x_n)=\det\left[K(x_i, x_j) \right]_{i,j = 1}^n\quad\mbox{ for all }x_1,\dots,x_n\in E.
\end{equation}
\end{definition}

If $A \in \mathrm{Skew}_{2n}(\mathbb{C})$, we define its {\em Pfaffian} by
\begin{equation}\label{eq: def of Pf}
\mathrm{Pf}(A) =\frac{1}{2^{n}n!}\sum_{\sigma \in S_{2n}}\mathrm{sgn}(\sigma)A_{\sigma(1)\sigma(2)}A_{\sigma(3)\sigma(4)}\cdots A_{\sigma(2n-1)\sigma(2n)}.
\end{equation}
From \cite[Equation (B.2)]{OQR17} we have  
\begin{equation}\label{eq: det to Pfaf}
\mathrm{Pf}(A)^2 = \det(A).
\end{equation}

With the above notation in place we can define Pfaffian point processes.
\begin{definition}\label{def:def of Pfaffian point process} Suppose that $K:E^2\rightarrow\mathrm{Mat}_2(\mathbb{C})$ is a skew-symmetric locally bounded measurable function. The latter means that
\begin{equation}\label{eq:in def of Pfaffian point process}
K(x,y)=\begin{bmatrix}
    K_{11}(x,y) & K_{12}(x,y)\\
    K_{21}(x,y) & K_{22}(x,y) 
\end{bmatrix}
\quad\mbox{ for }x,y\in E,
\end{equation}
where $K_{ij}:E^2\rightarrow\mathbb{C}$ are locally bounded measurable functions such that  $K_{ij}(x,y)=-K_{ji}(y,x)$ for $i,j \in \{1,2\}$. We refer to $K$ as the {\em (Pfaffian) correlation kernel}.

Suppose that $M$ is a point process on $(E,\mathcal{E})$ and $\lambda$ is a locally finite measure on $(E,\mathcal{E})$. We say that $M$ is a {\em Pfaffian point process with correlation kernel $K$ and reference measure $\lambda$} if the $n$-th correlation function $\rho_n$ exists for each $n \geq 1$ and 
\begin{equation}\label{eq:def of PfaffianPP}
\rho_n(x_1, \dots, x_n) = \operatorname{Pf} \left[ K(x_i, x_j) \right]_{i, j = 1}^n\quad\mbox{ for all }x_1,\dots,x_n\in E,
\end{equation}
where $\left[ K(x_i, x_j) \right]_{i, j = 1}^n$ on the right hand side above is the $2n\times 2n$ skew-symmetric matrix formed by the $2\times2$ blocks $K(x_i,x_j)$ for $1\leq i,j\leq n$.
\end{definition}
\begin{remark}\label{def: Pfaffian PP}
We note that the equations (\ref{eq:RND determinantal}) define the correlation functions $\rho_n$ on the support of $\lambda^n$. In particular, the value of $K$ outside of the support of $\lambda$ is immaterial. In some of our applications, we have that $\lambda$ is supported on some lattice in $E$ and we only define $K$ on that set.
\end{remark}

The following proposition summarizes some of the basic properties of Pfaffian point processes. Its proof is given in Section \ref{SectionB4}.
\begin{proposition}\label{prop:basic properties Pfaffian point process}
Suppose that $M$ is a Pfaffian point process on $(E,\mathcal{E})$ with correlation kernel $K$ and reference measure $\lambda$. We have the following statements.
\begin{enumerate}
\item $M$ is a simple point process $\mathbb{P}$-almost surely.
\item Suppose $D \in \mathcal{E}$ is any Borel set and define 
\begin{equation}\label{eq:RestrictEqn}
N(\omega, A)= M(\omega, A \cap D) \mbox{ for each $A \in \mathcal{E}$ and $\omega \in \Omega$}.
\end{equation}
Then, $N$ is a Pfaffian point process with kernel $\tilde{K}(x,y) = {\bf 1}\{x \in D\} \cdot K(x,y) \cdot {\bf 1}\{y \in D\}$ and reference measure $\lambda$. 
\item The law of $M$ (as a random element in the space $(S, \mathcal{S})$ of locally bounded measures) is uniquely determined by the correlation kernel $K$ and the reference measure $\lambda$.
\item Suppose that $f:E\rightarrow\mathbb{C}\setminus \{0\}$ is such that $f(x), 1/f(x)$ are locally bounded measurable functions. Then, $M$ is also a Pfaffian point process with the correlation kernel $\tilde{K}$, given by
\[
\tilde{K}(x,y)=\begin{bmatrix}
        f(x) f(y) \cdot K_{11}(x,y) & \frac{f(x)}{f(y)} \cdot K_{12}(x,y)\\
        \frac{f(y)}{f(x)} \cdot K_{21}(x,y) & \frac{1}{f(x)f(y)} \cdot K_{22}(x,y) 
    \end{bmatrix},
    \]
and the same reference measure $\lambda$.
\item Let $\phi: E \rightarrow E$ be a measurable bijection with a measurable inverse such that $\phi$ and $\phi^{-1}$ are locally bounded. Then, the pushforward measure $M \phi^{-1}$ is a Pfaffian point process with correlation kernel $\tilde{K}(x,y) = K(\phi^{-1}(x), \phi^{-1}(y))$ and reference measure $\lambda \phi^{-1}$. 
\item Fix $c_1,c_2 > 0$ and denote $\tilde{\lambda} = (c_1c_2)^{-1} \cdot \lambda$. Then, $M$ is also a Pfaffian point process with reference measure $\tilde{\lambda}$ and  with correlation kernel given by
\[
\tilde{K}(x,y)=\begin{bmatrix}
        c_1^2 \cdot  K_{11}(x,y) & c_1c_2 \cdot K_{12}(x,y)\\
        c_1c_2 \cdot K_{21}(x,y) & c_2^2 \cdot K_{22}(x,y) 
\end{bmatrix}. 
\]  
\end{enumerate}
\end{proposition}

The next result shows that any determinantal point process is also a Pfaffian point process.
\begin{lemma}\label{lem:determinantal point process as Pfaffian}
Let $M$ be a determinantal point process with correlation kernel $K: E^2\rightarrow\mathbb{C}$ and reference measure $\lambda$. Then, $M$ is also a Pfaffian point process with the same reference measure and correlation kernel 
$\tilde{K}$, given by
\[
    \tilde{K}(x,y)=\begin{bmatrix}
        0 & K(x,y)\\
        -K(y,x) & 0
    \end{bmatrix}\quad\mbox{ for }x,y\in E.
\] 
\end{lemma}
\begin{proof}
The result follows from the Definitions \ref{def: DefDPP} and \ref{def:def of Pfaffian point process}, and the identity
\begin{equation}\label{eq:Pfaffian equal to determinant}
\operatorname{Pf}(B) = \det (A), \mbox{ where } B = \begin{bmatrix}
0 & a_{11} & \dots & 0 & a_{1n}\\
-a_{11} & 0 & \dots & -a_{n1} & 0\\
\vdots &  & \ddots & & \vdots\\
0 & a_{n1} & \dots & 0 & a_{nn}\\
-a_{1n} & 0 & \dots & -a_{nn} & 0\\
\end{bmatrix} \mbox{ and } A = \begin{bmatrix} 
   a_{11} & \dots  & a_{1n}\\
    \vdots & \ddots & \vdots\\
    a_{n1} & \dots  & a_{nn} 
    \end{bmatrix}.
\end{equation}
To see why (\ref{eq:Pfaffian equal to determinant}) holds, note that by swapping rows and columns we have
$$\det(B) = \det \begin{bmatrix} 0 & A \\ -A^T & 0 \end{bmatrix} = \det(A)^2.$$
The latter and (\ref{eq: det to Pfaf}) show $\operatorname{Pf}(B) = \pm \det (A)$, and comparing both sides when $A $ is the identity matrix gives the correct sign.
\end{proof}

The next result provides sufficient conditions for a sequence of Pfaffian point processes to converge weakly. The proof is given in Section \ref{SectionB5}.
\begin{proposition}\label{prop: conv of Pfaffian point processes 0} Let $\lambda_N$ be a sequence of locally finite measures on $E$ that converge vaguely to $\lambda$. Let $M^N$ be a sequence of Pfaffian point processes with correlation kernels $K^N$ and reference measures $\lambda_N$. Suppose that there is a continuous function $K:  E^2\rightarrow\mathrm{Mat}_2(\mathbb{C})$ such that for any compact set $\mathcal{V}\subset E$ and $i,j \in \{1,2\}$ we have
$$\lim_{N\rightarrow\infty}\sup_{x,y\in\mathcal{V}} | K^N_{ij}(x,y)-K_{ij}(x,y)|=0.$$
Then, there exists a Pfaffian point process $M$ on $E$ with correlation kernel $K$ and reference measure $\lambda$. Moreover, $M^N$ converge weakly to $M$ as random elements in $(S ,\mathcal{S})$.
\end{proposition}

For future use we recall the following result, which provides a bound of the Pfaffian of a kernel in terms of entry-wise bounds.
\begin{lemma}\label{lem:HadamardPf}\cite[Lemma 2.5]{BBCS} Let $K: \mathbb{R}^2 \rightarrow \mathrm{Mat}_2(\mathbb{C})$ be a skew-symmetric kernel. Assume that there exist constants $C> 0$ and $a > b \geq 0$ such that 
$$|K_{11}(x,y)| \leq Ce^{-ax - ay}, \hspace{2mm} |K_{12}(x,y)| \leq Ce^{-ax +by}, \hspace{2mm} |K_{22}(x,y)| \leq C e^{bx + by}.$$
Then, for all $n \in \mathbb{N}$
$$\left|\operatorname{Pf} \left[ K(x_i, x_j) \right]_{i,j = 1}^n \right| \leq (2n)^{n/2} C^n \prod_{i = 1}^ne^{-(a-b)x_i}.$$
\end{lemma}

In the remainder of this section we consider a specific setup, which is suitable for the applications we have in mind, and turn to explaining that first.
\begin{definition} \label{def:setup slices}
Suppose $E =\mathbb{R}^2$. We fix $t_1 < \cdots < t_r$ and set $\mathcal{T} = \{t_1, \dots, t_r\}$. If $\nu = (\nu_{t_1}, \dots, \nu_{t_r})$ is an $r$-tuple of locally finite measures on $\mathbb{R}$, we define the (necessarily locally finite) measure $\mu_{\mathcal{T},\nu}$ on $\mathbb{R}^2$ by
\begin{equation}\label{eq: muTnu}
\mu_{\mathcal{T},\nu}(A) = \sum_{t \in \mathcal{T}} \nu_t(A_{t}), \mbox{ where } A_{t} = \{ y \in \mathbb{R}: (t,y) \in A\}.
\end{equation}
Below we consider Pfaffian point processes on $E$ with reference measures of the form $\mu_{\mathcal{T},\nu}$. Note that only the values of the kernels on $ (\mathcal{T} \times \mathbb{R})^2$ matter as $\mathcal{T} \times \mathbb{R}$ contains the support of $\mu_{\mathcal{T},\nu}$.
\end{definition}

The following result shows that the measures in Definition \ref{def:setup slices} behave well under projections to a fixed $t \in \mathcal{T}$. Its proof is given in Section \ref{SectionB6}.
\begin{lemma}\label{lem:LemmaSlice} Assume the same notation as in Definition \ref{def:setup slices}, and let $M$ be a Pfaffian point process with correlation kernel $K$ and reference measure $\mu_{\mathcal{T},\nu}$. Fix $t \in \mathcal{T}$ and consider the random measure on $\mathbb{R}$, given by $M^t(A) := M(\{t\} \times A)$. Then, $M^t$ is a Pfaffian point process on $\mathbb{R}$ with correlation kernel $K^t(x,y) = K(t,x; t,y)$ and reference measure $\nu_t$.
\end{lemma}

The next result provides a sufficient condition for a sequence of point processes in the setup of Definition \ref{def:setup slices} to converge weakly. Its proof is given in Section \ref{SectionB7}.
\begin{proposition}\label{prop: conv of Pfaffian point processes 1}
Assume the same notation as in Definition \ref{def:setup slices}. Let $a_{t_i}(N) > 0$, $b_{t_i}(N) \in \mathbb{R}$ for $i = 1, \dots, r$ be sequences such that $\lim_N a_{t_i}(N) = 0$. Let $\nu_{t_i}(N)$ be the counting measure on $a_{t_i}(N) \cdot \mathbb{Z} + b_{t_i}(N)$ multiplied by $a_{t_i}(N)$, and set $\nu(N) = (\nu_{t_1}(N), \dots, \nu_{t_r}(N))$. In addition, let $M^N$ be a sequence of Pfaffian point processes with correlation kernels $K^N$ and reference measures $\mu_{\mathcal{T}, \nu(N)}$. Finally, let $U \subset \mathbb{R}^2$ be an open set such that $\mathbb{R}^2 \setminus U$ has zero Lebesgue measure.

Suppose that there exists a skew-symmetric locally bounded measurable function $K:(\mathcal{T}\times\mathbb{R})^2\rightarrow\mathrm{Mat}_2(\mathbb{C})$ such that for each $i,j \in \{1,2\}$, $(x,y) \in U$ and $s,t \in \mathcal{T}$, and any sequences $x_N \in a_s(N) \cdot \mathbb{Z} + b_s(N)$, $y_N \in a_t(N) \cdot \mathbb{Z} + b_t(N)$ with $\lim_N x_N = x$ and $\lim_N y_N = y$ we have
\begin{equation}\label{eq:limitProp}
\lim_{N \rightarrow \infty}K^N_{ij}(s,x_N;t,y_N) = K_{ij}(s,x;t,y).
\end{equation}
Suppose also that for each $A>0$, $i,j \in \{1,2\}$ and $s,t \in \mathcal{T}$
\begin{equation}\label{eq:limitProp2}
\limsup_{N \rightarrow \infty}\sup_{x\in[-A,A] \cap a_s(N) \cdot \mathbb{Z} + b_s(N)} \sup_{y\in[-A,A] \cap a_t(N) \cdot \mathbb{Z} + b_t(N)} \left|K^N_{ij}(s,x;t,y)\right| < \infty.
\end{equation}
Then, there exists a Pfaffian point process $M$ on $E$ with correlation kernel $K$ and reference measure $\mu_{\mathcal{T}} \times \lambda$ (here $\mu_{\mathcal{T}}$ is the counting measure on $\mathcal{T}$ and $\lambda$ is the Lebesgue measure on $\mathbb{R}$). Moreover, $M^N$ converge weakly to $M$ as random elements in $(S, \mathcal{S})$.
\end{proposition}
\begin{remark}\label{rem: Conv Prop} Roughly, Proposition \ref{prop: conv of Pfaffian point processes 1} says that pointwise convergence of the correlation kernels, and vague convergence of the reference measures implies weak convergence of the point processes; however, under fairly strong assumptions. The most restrictive assumption of the proposition is that $\nu_t^N$ are rescaled counting measures on a scaled lattice, and consequently converge quite regularly to the Lebesgue measure $\lambda$. We believe that without any additional assumptions on $K$ one needs to impose some restrictions on $\nu_t^N$, and we have made the above choice as it meets our later needs. 

An analogue of the proposition for determinantal point processes was recently proved in \cite[Proposition 2.18]{dimitrov2024airy}. The argument there works under the assumption that the limiting kernel $K$ is {\em continuous} and the convergence of the kernels is {\em uniform over compacts}; however, does not assume any special structure for the $\nu_t^N$. Later in the paper we will analyze a certain sequence of kernels $K^N$, for which the limiting kernel $K^{\infty}$ is {\em not} continuous, and for which the convergence is {\em not} uniform over compact sets. One of the advantages of Proposition \ref{prop: conv of Pfaffian point processes 1} is that we only need to check (\ref{eq:limitProp}) on an open set $U$ of full Lebesgue measure, as opposed to all of $\mathbb{R}^2$, which allows us to avoid verifying the statement for certain pathological points $(x,y) \in \mathbb{R}^2$.
\end{remark}

%
%
\subsection{Definition and properties}\label{Section5.3} In this section we introduce a class of {\em Pfaffian Schur processes} and discuss some of their properties. Our models are special cases of the ones introduced in \cite{BR05}, which in turn are Pfaffian analogs of the determinantal Schur processes in \cite{OR03}. For more background on Pfaffian Schur processes we refer to \cite{BBCS} and \cite{BR05}.

A {\em partition} is a non-increasing sequence of non-negative integers $\lambda=(\lambda_1\geq\lambda_2\geq\dots\geq\lambda_k \geq \cdots)$ with finitely many non-zero elements. For any partition $\lambda$, we denote its {\em weight} by $\vert\lambda\vert=\sum_{i=1}^{\infty}\lambda_i$. There is a single partition of weight $0$, which we denote by $\emptyset$.  Given two partitions $\lambda$ and $\mu$, we write $\mu\preceq\lambda$ or $\lambda \succeq \mu$ and say that $\mu$ and $\lambda$ {\em interlace} if $\lambda_1\geq\mu_1\geq\lambda_2\geq\mu_2\geq\dots$.

Given finitely many variables $x_1, \dots, x_n$, we define the {\em skew Schur polynomials} via
\begin{equation}\label{def: skew Schur}
s_{\lambda/ \mu}(x_1, \dots, x_n) = \sum_{\mu = \lambda^{0} \preceq  \lambda^{1} \preceq \cdots \preceq \lambda^{n} = \lambda} \prod_{i = 1}^n x_i^{|\lambda^{i}| - |\lambda^{i-1}|}.
\end{equation}
When $\mu = \emptyset$ we drop it from the notation, and write $s_{\lambda}$, which is then the Schur polynomial indexed by $\lambda$. We refer the interested reader to \cite[Section 2]{BG16} for a friendly introduction to Schur symmetric polynomials and \cite[Chapter I]{Mac} for a comprehensive textbook treatment. We also define the {\em boundary monomial} in a single variable $c$ by
\begin{equation}\label{def:DefTau}
\tau_\lambda(c) =c^{\sum_{j=1}^{\infty}(-1)^{j-1}\lambda_j}=c^{\lambda_1-\lambda_2+\lambda_3-\lambda_4+\dots}.
\end{equation}

With the above notation in place, we can define our main object of interest.
\begin{definition}\label{def: Pfaffian Schur} Fix $M, N \in \mathbb{N}$, $c,a_1,\dots,a_M,b_1,\dots,b_N\in [0,\infty)$ such that $cb_j,a_ib_j,b_jb_k<1$ for all $i\in\llbracket1,M\rrbracket$ and $j,k\in\llbracket1,N\rrbracket$, $j\neq k$. With these parameters we define the {\em Pfaffian Schur process} to be the probability distribution on sequences of partitions $\lambda^0,\dots,\lambda^M$, given by
\begin{equation}\label{eq:def of Pfaffian Schur}
\begin{split}
&\mathbb{P}\left(\lambda^{0},\dots,\lambda^{M}\right)=\frac{1}{Z}\cdot \tau_{\lambda^{0}}(c) \cdot s_{\lambda^{1}/\lambda^{0}}(a_1)\cdots s_{\lambda^M/\lambda^{M-1}}(a_M) \cdot s_{\lambda^{M}}(b_1,\dots,b_N),
\end{split}
\end{equation}
where the normalization constant was computed explicitly in \cite[Proposition 2.1]{BR05}:
$$ Z =\prod_{i\in\llbracket1,N\rrbracket}\frac{1}{1-cb_i} \prod_{i\in\llbracket1,M\rrbracket,j\in\llbracket1,N\rrbracket}\frac{1}{1-a_ib_j} \prod_{1\leq i<j\leq N}\frac{1}{1-b_ib_j}.$$
In this paper we exclusively work with the above measures when
\begin{equation}\label{HomPar}
a_1=\dots=a_M=b_1=\dots=b_N=q \in (0,1) \mbox{ and } c \in (q, q^{-1}).
\end{equation}
\end{definition}

For fixed $m \in \mathbb{N}$ and $0\leq M_1<\dots<M_m\leq M$, we consider the point process $\mathfrak{S}(\lambda)$ on $\{1, \dots,m\} \times \mathbb{Z} \subset \mathbb{R}^2$ 
\begin{equation}\label{eq:point process}
\mathfrak{S}(\lambda)(A) =\sum_{i \geq 1} \sum_{j=1}^m{\bf 1} \{(j,\lambda^{M_j}_i-i) \in A \}.
\end{equation} 
Our first key result, a special case of \cite[Theorem 3.3]{BR05}, states that $\mathfrak{S}(\lambda)$ in (\ref{eq:point process}) is a Pfaffian point process on $\mathbb{R}^2$. Below we let $C_r$ be the positively oriented zero-centered circle of radius $r > 0$.
\begin{lemma}\label{lem:PSP as PPP}
Assume the same notation as in Definition \ref{def: Pfaffian Schur}, and let $\mathfrak{S}(\lambda)$ be as in (\ref{eq:point process}). Then, $\mathfrak{S}(\lambda)$ is a Pfaffian point process on $\mathbb{R}^2$ with reference measure given by the counting measure on $\llbracket1,m\rrbracket\times\mathbb{Z}$, and with correlation kernel 
$\kgeo:(\llbracket1,m\rrbracket\times\mathbb{Z})^2\rightarrow\operatorname{Mat}_2(\mathbb{C})$ given as follows. For each $u,v\in\llbracket1,m\rrbracket$ and $x,y \in \mathbb{Z}$
\begin{equation*}
\begin{split}
\kgeo_{11}(u,x ; v,y) =& \frac{1}{(2\pi \im)^{2}}\oint_{C_{r_{1}}} dz \oint_{C_{r_{1}}} d w\frac{z-w}{(z^{2}-1)(w^{2}-1)(zw-1)} \cdot 
(1-c/z)(1-c/w) \cdot z^{-x }w^{-y }  \\
& \times (1-q/z)^{M_u+N} (1-q/w)^{M_v+N}  (1-q z)^{-N}(1-q w)^{-N}
\end{split}
\end{equation*}
where $r_{1} \in (1, q^{-1})$. In addition,
\begin{equation*}
\begin{split}
\kgeo_{12}(u,x ; v,y) = - \kgeo_{21}(v,y; u,x) =& \frac{1}{(2\pi \im)^{2}}\oint_{C_{r^z_{12}}} d  z \oint_{C_{r^w_{12}}}  dw  \frac{zw-1}{z(z-w)(z^{2}-1)} \cdot \frac{z-c}{w-c} \cdot  z^{-x } w^{y }  \\
&  \times     (1-q/z)^{M_u+N} (1-q/w)^{-M_v-N}(1-qz)^{-N}   (1-qw)^{N},
\end{split}
\end{equation*}
where $r^z_{12} \in (1, q^{-1})$, $r^w_{12} > \max(c, q)$, and $r^w_{12} < r^z_{12}$ when $u \geq v$, while $r^z_{12} < r^w_{12}$ when $u < v$. Finally,
\begin{equation*}
\begin{split}
\kgeo_{22}(u,x ;v,y) =& \frac{1}{(2\pi \im)^{2}}\oint_{C_{r_{2}}}  dz \oint_{C_{r_{2}}} dw \frac{z-w}{zw-1} \cdot \frac{1}{(z-c)(w-c)} \cdot z^{x }w^{y} \\
& 
\times  (1-q/z)^{-M_u-N}(1-q/w)^{-M_v-N} (1-qz)^N (1-qw)^N,
\end{split}
\end{equation*}
where $r_{2} > \max(c,q,1)$.
\end{lemma}
\begin{remark}
Lemma \ref{lem:PSP as PPP} is originally due to \cite[Theorem 3.3]{BR05}. The above formulation follows from \cite[Section 4.2]{BBCS} after the change of variables $w\mapsto1/w$ within $K_{12}$ and $z\mapsto1/z$, $w\mapsto1/w$ within $K_{22}$.
\end{remark}

The second key statement we require is that the line ensemble formed by the partitions of a Pfaffian Schur process satisfies the half-space interlacing Gibbs property.
\begin{lemma}\label{lem:SchurGibbs} Assume the same notation as in Definition \ref{def: Pfaffian Schur} and let $\mathfrak{L} = \{L_i\}_{i \in \mathbb{N}}$ be given by $L_i(j) = \lambda_i^j$ for $i \in \mathbb{N}, j \in \llbracket 0, M \rrbracket$. Then, $\mathfrak{L}$ is an $\mathbb{N}$-indexed geometric line ensemble on $\llbracket 0, M \rrbracket$ that satisfies the half-space interlacing Gibbs property with parameters $q_i = c^{(-1)^{i}} \cdot q$ for $i \in \mathbb{N}$ from Definition \ref{def:HSIGP}.
\end{lemma}
\begin{proof} Since $s_{\lambda}(x_1, \dots, x_n) = 0$ if $\lambda_{n+1} > 0$, we see from (\ref{eq:def of Pfaffian Schur}) that $\mathbb{P}$-a.s. $\lambda^j_{k} = 0$ for $j \in \llbracket 0, M \rrbracket$ and $k \geq N+1$. Fixing $M+1$ partitions $\mu^0, \dots, \mu^M$ with $\mu^M_{N+1} = 0$, we see from (\ref{def: skew Schur}), (\ref{def:DefTau}) and (\ref{eq:def of Pfaffian Schur}) that for any $k \geq N$ we have
\begin{equation*}
\begin{split}
&\mathbb{P}(L_i(j) = \mu_i^j \mbox{ for } i \in \llbracket 1, k \rrbracket, j \in \llbracket 0, M \rrbracket \vert L_i(M) = \mu^M_i \mbox{ for } i \in \llbracket 1, k \rrbracket) \\
& \propto c^{\sum_{i \geq 1} \mu^0_{2i - 1} - \mu^0_{2i}} \cdot \prod_{j = 1}^M {\bf 1}\{ \mu^{j-1} \preceq \mu^j \}  q^{|\mu^j| - |\mu^{j-1}|} = \prod_{i \geq 1} c^{(-1)^{i-1}\mu^0_{i}} q^{\mu_i^M - \mu_i^0} \cdot \prod_{j = 1}^M {\bf 1}\{ \mu^{j-1} \preceq \mu^j \},
\end{split}
\end{equation*}
Comparing the last equation with (\ref{eq:BackGeomRW}), we see that the conditional law of $\{L_i\}_{i = 1}^k$, given $L_i(M) = \mu^M_i \mbox{ for } i \in \llbracket 1, k \rrbracket$ is $\mathbb{P}_{\ice, \operatorname{Geom}}^{M, \vec{y}, \vec{q}, g}$ as in Definition \ref{def:interlaceGeom}, where $\vec{q} = (q_1, \dots, q_k)$ with $q_i = c^{(-1)^{i}} \cdot q$, $g(r) = 0$, and $\vec{y} = (y_1, \dots, y_k)$ with $y_i = \mu^M_i$. Here, we implicitly used that $c \in (q,q^{-1})$ so that $q_i \in (0,1)$. As $\mathbb{P}_{\ice, \operatorname{Geom}}^{M, \vec{y}, \vec{q}, g}$ satisfies the half-space interlacing Gibbs property, cf. Remark \ref{rem:HSIGP3}, we conclude the same for $\mathfrak{L}$. 
\end{proof}

%
%
\subsection{Proof of Theorem \ref{thm:MainThm1}}\label{Section5.4} In this section we state the main result we prove about the Pfaffian Schur processes from Definition \ref{def: Pfaffian Schur} -- this is Proposition \ref{prop: FinDimConv} below. Afterwards we use that result together with Theorem \ref{thm:main thm tightness} to prove Theorem \ref{thm:MainThm1}.

We summarize how we scale parameters in the following definition.
\begin{definition}\label{def:ParScale} Fix $q \in (0,1)$ and set
\begin{equation}\label{def:SigmaQ and FQ}
\sigma_q = \frac{q^{1/3} (1 + q)^{1/3}}{1- q} \mbox{ and } f_q = \frac{q^{1/3}}{2 (1 + q)^{2/3}}.
\end{equation}
We fix $\varpi \in \mathbb{R}$ and let $c_N \in (q, q^{-1})$ be a sequence such that $c_N = 1 - \varpi \sigma_q^{-1} N^{-1/3} + o(N^{-1/3})$. We further fix $m \in \mathbb{N}$, $t_1, \dots ,t_m \in [0, \infty) $ with $t_1 < \cdots < t_m$ and set $\mathcal{T} = \{t_1, \dots, t_m\}$. We let $\mu_{\mathcal{T}}$ be the counting measure on $\mathcal{T}$ and $\lambda$ be the Lebesgue measure on $\mathbb{R}$. We also define for $t \in \mathcal{T}$ the quantity $T_t = T_t(N) = \lfloor t N^{2/3} \rfloor$ and the lattice $\Lambda_{t}(N) = a_{t}(N) \cdot \mathbb{Z} + b_t(N)$, where 
\begin{equation}\label{eq:Lattices}
a_t(N) = \sigma_q^{-1} N^{-1/3} \mbox{ and } b_t(N) =  \sigma_q^{-1} N^{-1/3}\cdot \left( - \frac{2q N}{1-q} - \frac{q T_t(N)}{1-q}\right).
\end{equation}

We let $\mathbb{P}_N$ be the Pfaffian Schur process from Definition \ref{def: Pfaffian Schur} with parameters $a_i = b_j = q$ for $i \in \llbracket 1, M \rrbracket$, $j \in \llbracket 1, N \rrbracket$, and $c = c_N$, where $M(N)$ is sufficiently large so that $T_{t_m}(N) = \lfloor t_m N^{2/3} \rfloor \leq M$. If $(\lambda^0, \dots, \lambda^M)$ have law $\mathbb{P}_N$, we define the random variables
\begin{equation}\label{eq:DefXi's}
X_i^{j,N} = \sigma_q^{-1} N^{-1/3} \cdot \left( \lambda_i^{T_{t_j}(N)} - \frac{2q N}{1-q} - \frac{q T_{t_j}(N)}{1-q} - i \right) \mbox{ for } i \in \mathbb{N} \mbox{ and } j \in \llbracket 1, m \rrbracket.
\end{equation}
\end{definition}

With the above notation in place we can state our main result about the Pfaffian Schur process, whose proof is given in Section \ref{Section6.4}.
\begin{proposition}\label{prop: FinDimConv} Assume the same notation as in Definition \ref{def:ParScale}. Then, the sequence of random vectors $(X_i^{j,N}: i \geq 1, j \in \llbracket 1, m \rrbracket)$ converges in the finite-dimensional sense to a vector $(X_i^{j,\infty}: i \geq 1, j \in \llbracket 1, m \rrbracket)$. Moreover, the random measure
\begin{equation}\label{eq:Point process formed}
\hat{M}^{\infty}(\omega, A) = \sum_{i \geq 1} \sum_{j =1}^m {\bf 1}\{(t_j, X_i^{j, \infty}(\omega))\in A)\} 
\end{equation}
is a Pfaffian point process on $\mathbb{R}^2$ with reference measure $\mu_{\mathcal{T}} \times \lambda$ and correlation kernel 
\begin{equation}\label{eq:KX}
\hat{K}^{\infty}(s,x;t,y) = \kcr(f_q s, x + f_q^2 s^2; f_q t, y + f_q^2 t^2),
\end{equation}
where we recall that $\kcr$ is as in Definition \ref{def:kcr}.
\end{proposition}

We end this section by giving the proof of Theorem \ref{thm:MainThm1}.
\begin{proof}[Proof of Theorem \ref{thm:MainThm1}] For clarity we split the proof into three steps. In the first step we construct a certain sequence of line ensembles related to the Pfaffian Schur processes in Definition \ref{def:ParScale}, and we show that it is tight by verifying the conditions of Theorem \ref{thm:main thm tightness}. In the second step we show that in fact these ensembles converge to a limiting line ensemble $\mathcal{L}^{\infty}$, which satisfies similar conditions to those of $\mathcal{L}$ from Theorem \ref{thm:MainThm1}. In the last step we check that a suitable change of variables transforms $\mathcal{L}^{\infty}$ into $\mathcal{L}$. We assume throughout the notation in Definition \ref{def:ParScale} and suppose that $M = N$.\\

{\bf \raggedleft Step 1.} We define the sequence of $\mathbb{N}$-indexed geometric line ensembles $\mathfrak{L}^N = \{L_i^N\}_{i \geq 1}$ on $\mathbb{Z}$ via 
\begin{equation}\label{TQ1}
L_i^N(s) = \begin{cases} \lambda_i^{ s} -\mathsf{C}_N &\mbox{ if } s \in \llbracket 0, N \rrbracket  \\ \lambda_i^{0} - \mathsf{C}_N  &\mbox{ if } s < 0 \\ \lambda_i^{N} - \mathsf{C}_N &\mbox{ if } s > N,
\end{cases}
\end{equation}
where $\mathsf{C}_N = \lfloor \frac{2qN}{1-q} \rfloor$. In words, $\mathfrak{L}^N$ is just a translation of $\{\lambda^j_i: i \in \mathbb{N}, j \in \llbracket 0, N \rrbracket \}$ with a constant extension outside of the integer interval $\llbracket 0, N \rrbracket$. We also define $\mathcal{L}^N = \{\mathcal{L}_i^N\}_{i \geq 1}$ via
\begin{equation}\label{S62E2}
\mathcal{L}_i^N(s) = \sigma^{-1} N^{-1/3} \cdot \left( L_i^N(s N^{2/3}) - u s N^{2/3}\right), \mbox{ where $u = \frac{q}{1-q}$ and $\sigma = \frac{q^{1/2}}{1-q}$.}
\end{equation}

From Proposition \ref{prop: FinDimConv} we know that the random variables $\sigma^{-1}N^{-1/3}(L_i^N(\lfloor tN^{2/3}\rfloor)-utN^{2/3})$ converge for each $t \geq 0$ and so are in particular tight. In addition, we have from Lemma \ref{lem:SchurGibbs} that the restriction of $\mathfrak{L}^N$ to $\llbracket 0, N\rrbracket$ satisfies the half-space interlacing Gibbs property with parameters $$q^N_i = c_N^{(-1)^{i}} \cdot q = q - (-1)^{i} q \varpi \sigma_q^{-1} \cdot  N^{-1/3} + o(N^{-1/3})\mbox{ for $i \in \mathbb{N}$.}$$
The last two observations show that the conditions of Theorem \ref{thm:main thm tightness} are satisfied with $p = q$, $K = \infty$, $n = N$, $d_n = N^{2/3}$, $T_n = N$, $\beta = \infty$, $\Lambda = [0, \infty)$
and 
\begin{equation}\label{drifts}
\mu_i = (-1)^{i} \cdot \varpi \cdot \frac{q^{1/2}  \sigma_q^{-1}}{1- q} =   (-1)^{i} \cdot \varpi \cdot \frac{q^{1/6} }{(1 + q)^{1/3}} = (-1)^{i} \cdot \varpi  \cdot (2f_q)^{1/2}.
\end{equation}
We conclude that the sequence of line ensembles $\mathcal{L}^N\vert_{[0,\infty)}$ is tight and also any subsequential limit satisfies the half-space Brownian Gibbs property with parameters $\{\mu_i\}_{i \geq 1}$ as in (\ref{drifts}).\\

{\bf \raggedleft Step 2.} Let $\mathcal{L}^{\infty}$ be any subsequential limit of $\mathcal{L}^N\vert_{[0,\infty)}$. From Proposition \ref{prop: FinDimConv} we know that $\{\mathcal{L}^{\infty}_{i}(t_j): i \geq 1, j \in \llbracket 1, m \rrbracket\}$ has the same distribution as $\{\alpha^{-1} X_i^{j,\infty}: i \geq 1, j \in \llbracket 1, m \rrbracket\}$, where $\alpha = \sigma/\sigma_q = (2f_q)^{1/2}$. In particular, we see that any two subsequential limits of $\mathcal{L}^N\vert_{[0,\infty)}$ have the same finite-dimensional distributions and so they have the same law, see \cite[Lemma 3.1]{DimMat}. We conclude that $\mathcal{L}^{\infty}$ is the unique subsequential limit of $\mathcal{L}^N\vert_{[0,\infty)}$ and the latter converge weakly to it. In addition, we have from Proposition \ref{prop: FinDimConv} and parts (5) and (6) of Proposition \ref{prop:basic properties Pfaffian point process} that 
\begin{equation}\label{PPL}
M^{\mathcal{L}^{\infty}}(\omega, A) = \sum_{i \geq 1} \sum_{j = 1}^m {\bf 1}\{\mathcal{L}^{\infty}_i(t_j) \in A \}
\end{equation}
is a Pfaffian point process on $\mathbb{R}^2$ with reference measure $\mu_{\mathcal{T}} \times \lambda $ and correlation kernel 
\begin{equation}\label{KernelLInf}
\begin{bmatrix}
    \alpha \kcr_{11}(f_q s, \alpha x + f_q^2 s^2; f_q t, \alpha y + f_q^2 t^2) & \alpha \kcr_{12}(f_q s, \alpha x + f_q^2 s^2; f_q t, \alpha y + f_q^2 t^2)\\
    \alpha \kcr_{21}(f_q s, \alpha x + f_q^2 s^2; f_q t, \alpha y + f_q^2 t^2) & \alpha \kcr_{22}(f_q s, \alpha x + f_q^2 s^2; f_q t, \alpha y + f_q^2 t^2) 
\end{bmatrix}.
\end{equation}

{\bf \raggedleft Step 3.} In this final step we construct $\mathcal{L}$ and $\hsa$ from $\mathcal{L}^{\infty}$ and show that they satisfy the conditions of the theorem. For $i \geq 1$ and $t \in [0, \infty)$, we define 
\begin{equation}\label{eq:COVLE}
\mathcal{L}_i(t) = 2^{-1/2}\alpha \mathcal{L}^{\infty}_i(f_q^{-1} t) \mbox{ and } \hsa_i(t) = \alpha \mathcal{L}^{\infty}_i(f_q^{-1} t) +  t^2
\end{equation}
As mentioned in Step 1, we know that $\mathcal{L}^{\infty}$ satisfies the half-space Brownian Gibbs property. This implies in particular that $\mathcal{L}^{\infty}$ is non-intersecting, which together with (\ref{eq:COVLE}) shows (\ref{Eq.OrdHSA}).

From part (5) of Proposition \ref{prop:basic properties Pfaffian point process} with $\phi(s,x) = (f_qs, \alpha x+ f_q^2 s^2)$ and part (6) of the same proposition with $c_1 = c_2 = \alpha^{-1/2}$ we have that the measure in (\ref{eq:HSA point process}) is indeed a Pfaffian point process on $\mathbb{R}^2$, with correlation kernel $\kcr$ and reference measure $\mu_{\mathsf{S}} \times \lambda$. From (\ref{eq:COVLE}), we also see that (\ref{eq:Parabolic HSA}) holds. Lastly, since $\mathcal{L}^{\infty}$ satisfies the half-space Brownian Gibbs property with parameters $\mu_i$ as in (\ref{drifts}), we see that $2^{-1/2}\alpha \mathcal{L}^{\infty}_i(f_q^{-1} t) = f_q^{1/2}\mathcal{L}^{\infty}_i(f_q^{-1} t) $ satisfies the half-space Brownian Gibbs property with parameters $\mu_i = (-1)^{i} \cdot \varpi  \cdot (2f_q)^{1/2} \cdot f_q^{-1/2} = (-1)^{i} \sqrt{2} \cdot \varpi$. The latter follows from the fact that if $\cev{B}(t)$ is a reverse Brownian motion on $[0,b]$ with drift $\mu$, then $\cev{B}'(s) = f_q^{1/2}B(f_q^{-1}s)$ is a reverse Brownian motion on $[0, f_q b]$ with drift $f_q^{-1/2} \mu$.
\end{proof}

%
%
\section{Finite-dimensional convergence}\label{Section6} The goal of this section is to prove Proposition \ref{prop: FinDimConv}. In Section \ref{Section6.1} we derive a formula for the correlation kernel $K^N$ of the point process formed by $\{(t_j, X_i^{j,N}): i \geq 1, j \in \llbracket 1, m\rrbracket \}$, which is suitable for asymptotic analysis -- this is Lemma \ref{lem:PrelimitKernel}. In Section \ref{Section6.2} we state two asymptotic results for the kernels $K^N$, these are Lemmas \ref{lem:kernelLimits} and \ref{lem:kernelUpperTail}, whose proofs are contained in Section \ref{Section7}, and deduce a few statements from them. In Section \ref{Section6.3} we use the results from Section \ref{Section6.2} to show that the weak limits of the point processes on $\mathbb{R}$ formed by $\{X_i^{j,N}\}_{ i \geq 1}$ contain infinitely many atoms almost surely -- this is Lemma \ref{lem:InfiniteAtoms}. Finally, in Section \ref{Section6.4} we give the proof of Proposition \ref{prop: FinDimConv}.

%
%
\subsection{Alternative kernel}\label{Section6.1} Using the change of variables formula from part (5) in Proposition \ref{prop:basic properties Pfaffian point process} and Lemma \ref{lem:PSP as PPP}, one has that the point process formed by $\{(t_j, X_i^{j,N}): i \geq 1, j \in \llbracket 1, m\rrbracket \}$ as in Proposition \ref{prop: FinDimConv} is Pfaffian. The goal of this section is to find a form for its correlation kernel that is suitable for asymptotic analysis. The idea is to start from the formulas in Lemma \ref{lem:PSP as PPP}, deform the contours appropriately, and apply a gauge transformation as in part (4) in Proposition \ref{prop:basic properties Pfaffian point process}. 

We specify the types of contours we use in the following definition.
\begin{definition}\label{def:contours} Fix $a \in \mathbb{R}$ and $N \in \mathbb{N}$. We define the contours $$\gamma^+_N(a,0) = \{1 + aN^{-1/3} + |s|N^{-1/3} e^{\mathrm{sgn}(s)\im(\pi/3)}: s \in [-N^{1/4}, N^{1/4}] \}, \mbox{and }$$
$$\gamma^-_N(a,0) = \{1 + aN^{-1/3} + |s|N^{-1/3} e^{\mathrm{sgn}(s) \im(2\pi/3)}: s \in [-N^{1/4}, N^{1/4}] \},$$
both oriented in the direction of increasing imaginary part. We further let $\gamma^+_N(a,1)$ be the arc of the $0$-centered circle that connects the points $1 + aN^{-1/3} + N^{-1/12} e^{\im(\pi/3)}$ to $1 + aN^{-1/3} + N^{-1/12} e^{-\im(\pi/3)}$, and $\gamma^-_N(a,1)$ be the arc of the $0$-centered circle that connects the points $1 + aN^{-1/3} + N^{-1/12} e^{\im(2\pi/3)}$ to $1 + aN^{-1/3} + N^{-1/12} e^{-2\im(\pi/3)}$. Both $\gamma^+(a,1)$ and $\gamma^-(a,1)$ are oriented counter-clockwise. Finally, we set $\gamma_N^+(a) = \gamma_N^+(a,0) \cup \gamma_N^+(a,1)$ and $\gamma_N^-(a) = \gamma_N^-(a,0) \cup \gamma_N^-(a,1)$. See Figure \ref{S61} for an illustration of these contours. We also recall from Section \ref{Section5.3} that $C_r$ is the positively oriented zero-centered circle of radius $r > 0$.
\end{definition}

\begin{figure}[ht]
\scalebox{0.75}{
    \centering
     \begin{tikzpicture}[scale=2.7]

        \draw[->, thick, gray] (-1.4,0)--(1.8,0) node[right]{$\Real$};
        \draw[->, thick, gray] (0,-1.4)--(0,1.5) node[above]{$\Imag$};
        \def\radA{1.344} 
        \def\radB{0.690} 

        \draw[->,thick][black] (1.1,0) -- (1.2,0.17);
        \draw[-,thick][black] (1.2,0.17) -- (1.3,0.34);
        \draw[-,thick][black] (1.1,0) -- (1.2,-0.17);
        \draw[->,thick][black] (1.3,-0.34) -- (1.2,-0.17);
        \draw[->,thick][black] (1.3,0.34) arc (14.66:180:\radA);
        \draw[-,thick][black] (1.3,0.34) arc (14.66:360 - 14.66:\radA);

        \draw[->,thick][black] (0.8,0) -- (0.7,0.17);
        \draw[-,thick][black] (0.7,0.17) -- (0.6,0.34);
        \draw[-,thick][black] (0.8,0) -- (0.7,-0.17);
        \draw[->,thick][black] (0.6,-0.34) -- (0.7,-0.17);        
        \draw[->,thick][black] (0.6,0.34) arc (29.54:180:\radB);
        \draw[-,thick][black] (0.6,0.34) arc (29.54:360 - 29.54:\radB);

        \draw[black, fill = black] (0.9,0) circle (0.02);
        \draw (0.9,-0.1) node{$1$}; 
        \draw[black, fill = black] (1,0) circle (0.02);
        \draw (1,0.1) node{$c_N$};
        \draw[black, fill = black] (1.6,0) circle (0.02);
        \draw (1.6,-0.1) node{$q^{-1}$};
       
        \draw (-1.4,1.2) node{$\gamma_{N}^{+}(\ap^q_3, 1)$};
        \draw[->, >=stealth'] ( - 1.1, 1.2)  to[bend left] ( -0.9, 1);
        \draw (1.6,0.6) node{$\gamma_{N}^{+}(\ap^q_3, 0)$};
        \draw[->, >=stealth'] ( 1.5, 0.5)  to[bend left] ( 1.25, 0.23);
        \draw (-0.7,0.7) node{$\gamma_{N}^{-}(-\ap^q_2, 1)$};
        \draw[->, >=stealth'] ( -0.65, 0.6)  to[bend right] ( -0.48,0.5);
        \draw (0.35,-0.115) node{$\gamma_{N}^{-}(-\ap^q_2, 0)$};
        \draw[->, >=stealth'] ( 0.3, -0.03)  to[bend left] (0.65,0.27);

        \draw (1.1,-1.3) node{$1+ \ap_3^q N^{-1/3}$};
        \draw[-,very thin, dashed][gray] (1.1,-1.1) -- (1.1, 0);
        \draw (0.8,1.4) node{$1 - \ap^q_2 N^{-1/3}$};
        \draw[-,very thin, dashed][gray] (0.8,1.2) -- (0.8,0);

        \draw[-,very thin, dashed][gray] (0.6,0.34) -- (2.1,0.34);
        \draw[-,very thin, dashed][gray] (0.6,-0.3) -- (2.1,-0.34);
        \draw[->,very thin][black] (2.1,0) -- (2.1, -0.34);
        \draw[->,very thin][black] (2.1,0) -- (2.1, 0.34);
        \draw (2.5,0) node{$\sqrt{3}N^{-1/12}$};

    \end{tikzpicture} 
}
\caption{The figure depicts the contours $\gamma^+_N(\ap^q_3) = \gamma^+_N(\ap^q_3, 0) \cup \gamma^+_N(\ap^q_3,1)$ and $\gamma^-_N(-\ap^q_2) = \gamma^-_N(-\ap^q_2, 0) \cup \gamma^-_N(-\ap^q_2,1)$ from Lemma \ref{lem:PrelimitKernel}.} 
    \label{S61}
\end{figure}
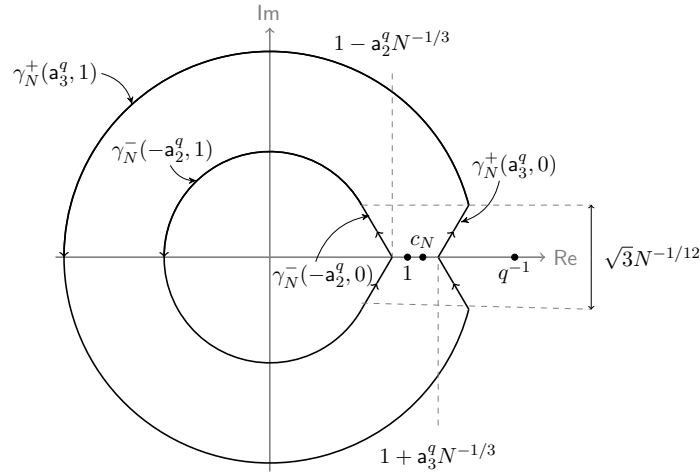

The next definition summarizes some of the functions and parameters we use.
\begin{definition}\label{def:funct} Fix $q \in (0,1)$. For $z \in \mathbb{C} \setminus \{0,q , q^{-1}\}$ we introduce the functions 
\begin{equation}\label{eq:DefS}
\begin{split}
&S(z) = \log (1 - q/z) - \log(1 - qz) - \frac{2q}{1-q} \cdot \log z \mbox{ and }\\
&G(z) = \log (1 - q/z) - \log (z) \cdot \frac{q}{1-q} - \log (1- q),
\end{split}
\end{equation}
where we take the principal branch of the logarithm. 

If $\varpi \in \mathbb{R}$, we fix the following real parameters
\begin{equation}\label{eq:PivotIneq}
\ap_i = |\varpi| + 3i \mbox{ and } \ap_i^q = \sigma_q^{-1} \cdot \ap_i \mbox{ for } i = 1,2,3.
\end{equation}
Recalling from Definition \ref{def:ParScale} that $c_N = 1 - \varpi \sigma_q^{-1} N^{-1/3} + o(N^{-1/3})$, we see that there exists $N_0 \in \mathbb{N}$ such that for $N \geq N_0$ we have
\begin{enumerate}
\item $\gamma_N^+(\ap_{3}^q)$ encloses the unit circle $C_1$;
\item $\gamma_N^-(\ap_{1}^q)$ encloses $c_N, c_N^{-1}$ and $q$, and $\gamma_N^+(\ap_{3}^q)$ encloses $\gamma_N^-(\ap_{1}^q)$;
\item $\gamma_N^{-}(-\ap_{2}^q)$ is contained in the unit circle $C_1$, and excludes the points $c_N, c_N^{-1}$;
\item if $w \in \gamma_N^{-}(-\ap_2^q)$, then $w^{-1}$ is outside of $\gamma_N^-(-\ap_2^q)$.
\end{enumerate}
\end{definition}

With the above notation in place we can state the main result of this section.
\begin{lemma}\label{lem:PrelimitKernel} Assume the same notation as in Definitions \ref{def:ParScale}, \ref{def:contours} and \ref{def:funct}. Let $M^N$ be the point process on $\mathbb{R}^2$, formed by $\{(t_j, X_i^{j,N}): i \geq 1, j \in \llbracket 1, m\rrbracket \}$. Then, for $N \geq N_0$ the $M^N$ is a Pfaffian point process with reference measure $\mu_{\mathcal{T},\nu(N)}$ and correlation kernel $K^N$ that are defined as follows. 

The measure $\mu_{\mathcal{T},\nu(N)}$ is as in Definition \ref{def:setup slices} for $\nu(N) = (\nu_{t_1}(N), \dots, \nu_{t_m}(N))$, where $\nu_{t}(N)$ is $\sigma_q^{-1} N^{-1/3}$ times the counting measure on $\Lambda_{t}(N)$. 

The correlation kernel $K^N: (\mathcal{T} \times \mathbb{R}) \times (\mathcal{T} \times \mathbb{R}) \rightarrow\operatorname{Mat}_2(\mathbb{C})$ takes the form
\begin{equation}\label{eq:S6Kdecomp}
\begin{split}
&K^N(s,x; t,y) = \begin{bmatrix}
    K^N_{11}(s,x;t,y) & K^N_{12}(s,x;t,y)\\
    K^N_{21}(s,x;t,y) & K^N_{22}(s,x;t,y) 
\end{bmatrix} \\
&= \begin{bmatrix}
    I^N_{11}(s,x;t,y) & I^N_{12}(s,x;t,y) + R^N_{12}(s,x;t,y) \\
    -I^N_{12}(t,y;s,x) - R^N_{12}(t,y;s,x) & I^N_{22}(s,x;t,y) + R^N_{22}(s,x;t,y) 
\end{bmatrix} ,
\end{split}
\end{equation}
where $I^N(s,x;t,y), R^N(s,x;t,y)$ are defined as follows. The kernels $I^N_{ij}$ are given by
\begin{equation}\label{eq:DefIN11}
\begin{split}
&I^N_{11}(s,x;t,y) = \frac{1}{(2\pi \im)^{2}}\oint_{\gamma_N^+(\ap^q_3)} dz \oint_{\gamma_N^+(\ap^q_{3})} dw F_{11}^N(z,w) H_{11}^N(z,w) \mbox{, where }\\
& F^N_{11}(z,w) = e^{NS(z) + NS(w)} \cdot e^{T_s G(z) + T_t G(w)} \cdot e^{- \sigma_q x N^{1/3} \log (z) - \sigma_q y N^{1/3} \log(w)  }, \\
&H^N_{11}(z,w) = 4\sigma_q^2 N^{2/3} \cdot  \frac{z-w}{(z^{2}-1)(w^{2}-1)(zw-1)} \cdot ( 1 - c_N/z) (1 - c_N/w);
\end{split}
\end{equation}
\begin{equation}\label{eq:DefIN12}
\begin{split}
&I^N_{12}(s,x;t,y) = \frac{1}{(2\pi \im)^{2}}\oint_{\gamma_N^+(\ap^q_{3})} dz \oint_{\gamma_N^-(\ap^q_{1})} dw F_{12}^N(z,w) H_{12}^N(z,w) \mbox{, where }\\
& F^N_{12}(z,w) = e^{NS(z) - NS(w)} \cdot e^{T_s G(z) - T_t G(w)} \cdot e^{- \sigma_q x N^{1/3} \log (z) + \sigma_q y N^{1/3} \log(w)  }, \\
&H^N_{12}(z,w) =  \sigma_q N^{1/3} \cdot \frac{zw - 1}{z (z-w)(z^2 - 1)} \cdot \frac{z - c_N}{w - c_N};
\end{split}
\end{equation}
\begin{equation}\label{eq:DefIN22}
\begin{split}
&I^N_{22}(s,x;t,y) = \frac{1}{(2\pi \im)^{2}}\oint_{\gamma_N^-(-\ap^q_{2})} dz \oint_{\gamma_N^-(-\ap^q_{2})} dw F_{22}^N(z,w) H_{22}^N(z,w) \mbox{, where }\\
& F^N_{22}(z,w) = e^{-NS(z) - NS(w)} \cdot e^{-T_s G(z) - T_t G(w)} \cdot e^{ \sigma_q x N^{1/3} \log (z) + \sigma_q y N^{1/3} \log(w)  }, \\
&H^N_{22}(z,w) =  \frac{1}{4} \cdot \frac{z-w}{zw - 1} \cdot \frac{1}{(z- c_N)(w - c_N)}.
\end{split}
\end{equation}
The kernels $R^N_{ij}$ are given by
\begin{equation}\label{eq:DefRN12}
\begin{split}
&R^N_{12}(s,x;t,y) = \frac{-{\bf 1}\{s < t \} \cdot \sigma_q N^{1/3} }{2 \pi \im} \oint_{\gamma_{N}^+(\ap^q_{1})}dz e^{(T_s - T_t) G(z)} \cdot e^{ (\sigma_q y N^{1/3}  - \sigma_q x N^{1/3} - 1) \log (z)};
\end{split}
\end{equation}
\begin{equation}\label{eq: DefRN22}
\begin{split}
&R^N_{22}(s,x;t,y) = \frac{1}{2\pi \im} \oint_{\gamma^-_N(\ap^q_{1})} dz \frac{F_{22}^N(z,c_N)}{4(c_N z - 1)} - \frac{1}{2\pi \im} \oint_{\gamma^-_N(-\ap^q_{2})} dw \frac{F_{22}^N(c_N,w)}{4(c_N  w - 1)} \\
& + \frac{1}{2\pi \im}\oint_{\gamma^-_N(-\ap^q_{2})} dw \cdot \frac{(1-w^2)}{4(1-c_N w)(w-c_N)} \cdot e^{ ( \sigma_q y N^{1/3} - \sigma_q x N^{1/3} - 1) \log (w) - T_s G(w^{-1}) - T_t G(w)}.
\end{split}
\end{equation}
\end{lemma}
\begin{proof} Let $f: \mathbb{R} \rightarrow \mathbb{R}$ be a piece-wise linear increasing bijection such that $f(i) = t_i$ for $i \in \llbracket 1, m \rrbracket$. Define $\phi_N: \mathbb{R}^2 \rightarrow \mathbb{R}^2$ through 
$$\phi_N(s, x) = \left(f(s), \sigma_q^{-1} N^{-1/3} \cdot \left( x- \frac{2qN}{1-q} - \frac{q \lfloor f(s) N^{2/3} \rfloor}{1-q} \right) \right),$$   
and observe that $M^N = \mathfrak{S}(\lambda) \phi_N^{-1}$, where $\mathfrak{S}(\lambda)$ is as in Lemma \ref{lem:PSP as PPP}. It follows from Lemma \ref{lem:PSP as PPP} and Proposition \ref{prop:basic properties Pfaffian point process} part (5) with the above $\phi_N$, part (4) with $f(s,x) = (1-q)^{-T_s}$ and part (6) with $c_1 = 2\sigma_q N^{1/3}$ and $c_2 = 1/2$ that $M^N$ is a Pfaffian point process with reference measure $\mu_{\mathcal{T},\nu(N)}$ and correlation kernel $\tilde{K}^N: (\mathcal{T} \times \mathbb{R}) \times (\mathcal{T} \times \mathbb{R}) \rightarrow\operatorname{Mat}_2(\mathbb{C})$, given by
\begin{equation*}
\tilde{K}^N(s,x;t,y) = \begin{bmatrix} 4\sigma_q^2 N^{2/3} (1-q)^{-T_s - T_t} \kgeo_{11}(\tilde{s},\tilde{x}; \tilde{t},\tilde{y}) &  \sigma_q N^{1/3} (1-q)^{-T_s + T_t} \kgeo_{12}(\tilde{s},\tilde{x}; \tilde{t},\tilde{y}) \\ \sigma_q N^{1/3} (1-q)^{T_s - T_t} \kgeo_{21}(\tilde{s},\tilde{x}; \tilde{t},\tilde{y}) & (1/4) (1-q)^{T_s + T_t} \kgeo_{22}(\tilde{s},\tilde{x}; \tilde{t},\tilde{y}) \end{bmatrix},
\end{equation*}
where $\kgeo$ is as in Lemma \ref{lem:PSP as PPP} with $c = c_N$, $\tilde{s} = T_s = \lfloor s N^{2/3} \rfloor$, $\tilde{t} = T_t = \lfloor t N^{2/3} \rfloor$ and
\begin{equation*}
\tilde{x} = \frac{2qN}{1-q} + \frac{q T_s }{1-q} + \sigma_q N^{1/3} x, \hspace{2mm}\tilde{y} = \frac{2qN}{1-q} + \frac{q T_t }{1-q} + \sigma_q N^{1/3} y.
\end{equation*}
All that remains is to show that $\tilde{K}^N$ agrees with $K^N$ as in the statement of the lemma.\\

We note that the have the following identities
\begin{equation}\label{eq:ChangeOfVar}
\begin{split}
&z^{\mp \tilde{x}} (1-q/z)^{\pm (T_s +N)}(1-qz)^{\mp N} (1-q)^{\mp T_s} = e^{\pm NS(z) \pm T_s G(z) \mp \sigma_q x N^{1/3} \log(z) }, \\
&w^{\mp \tilde{y}} (1-q/w)^{\pm (T_t +N)}(1-qw)^{\mp N} (1-q)^{\mp T_t} = e^{ \pm NS(w) \pm T_t G(w) \mp \sigma_q y N^{1/3} \log(w)}.
\end{split}
\end{equation}

{\bf \raggedleft Matching $K^N_{11}$.} If $N \geq N_0$, We may deform both contours $C_{r_1}$ in the definition of $\kgeo_{11}$ in Lemma \ref{lem:PSP as PPP} to $\gamma_N^+(\ap_3^q)$ without crossing any of the poles of the integrand and hence without affecting the value of the integral by Cauchy's theorem. The reason we do not cross any poles is because $\gamma_N^+(\ap_3^q)$ encloses the unit circle $C_1$, see point (1) below (\ref{eq:PivotIneq}). After we perform the deformation, apply (\ref{eq:ChangeOfVar}) and multiply by $4\sigma_q^2 N^{2/3} (1-q)^{-T_s - T_t} $ we obtain $\tilde{K}^N_{11}( s,x; t,y) = I^N_{11}(s,x; t,y)$.\\

{\bf \raggedleft Matching  $K^N_{12}$ and $K^N_{21}$.} Since $\tilde{K}^N$ and $K^N$ are both skew-symmetric it suffices to match $K^N_{12}$. If $s < t$ we deform $C_{r_{12}^w}$ inside the contour $C_{r_{12}^z}$ and in the process of deformation we pick up a residue from the simple pole at $w = z$. We thus obtain the formula 
\begin{equation}\label{eq:K12Res}
\begin{split}
&\kgeo_{12}(\tilde{s},\tilde{x}; \tilde{t},\tilde{y}) = \frac{1}{(2\pi \im)^{2}}\oint_{C_{r^z_{12}}} d  z \oint_{C_{r^w_{12}}}  dw  \frac{zw-1}{z(z-w)(z^{2}-1)} \cdot \frac{z-c_N}{w-c_N} \cdot  z^{-\tilde{x} } w^{\tilde{y} }  \\
&  \times     (1-q/z)^{T_s+N} (1-q/w)^{-T_t-N}(1-qz)^{-N}   (1-qw)^{N} \\
& + \frac{-{\bf 1}\{s < t\} }{2\pi \im} \oint_{C_{r^z_{12}}} dz   (1-q/z)^{T_s} (1-q/z)^{-T_t}z^{\tilde{y}-\tilde{x} - 1}  ,
\end{split}
\end{equation}
where $\max(1, c_N) < r_{12}^w < r_{12}^z < q^{-1}$. We now deform $C_{r^z_{12}}$ and $C_{r^w_{12}}$ in the first line of (\ref{eq:K12Res}) to $\gamma_N^{+}(\ap_3^q)$ and $\gamma_N^-(\ap_1^q)$, respectively. Note that from the second point below (\ref{eq:PivotIneq}) we do not cross any poles in the process of deformation if $N \geq N_0$. After the deformation we apply (\ref{eq:ChangeOfVar}) and recognize the first two lines in (\ref{eq:K12Res}), multiplied by $ \sigma_q N^{1/3} (1-q)^{-T_s + T_t} $, as $I^N_{12}(s,x; t,y)$. We can also deform the contour $C_{r^z_{12}}$ on the third line in (\ref{eq:K12Res}) to $\gamma_{N}^+(\ap_1^q)$ without crossing any poles and the resulting expression, multiplied by $ \sigma_q N^{1/3} (1-q)^{-T_s + T_t}$, agrees with $R^N_{12}(s,x;t,y)$.\\

{\bf \raggedleft Matching  $K^N_{22}$.} Starting from the formula for $\kgeo_{22}$ in Lemma \ref{lem:PSP as PPP} with $r_2 $ large (say $r_2 \geq 2$), we deform the $w$ contour to $\gamma_{N}^-(-\ap_2^q)$. In the process of deformation we pick up a residue from the simple pole at $w = c_N$, see point (3) below (\ref{eq:PivotIneq}). We thus obtain the formula
\begin{equation}\label{eq:K22Res1}
\begin{split}
&\kgeo_{22}(\tilde{s},\tilde{x}; \tilde{t},\tilde{y}) = \frac{1}{(2\pi \im)^{2}}\oint_{C_{r_2}} d  z \oint_{\gamma_N^-(-\ap_2^q)}  dw \frac{z-w}{zw-1} \cdot \frac{1}{(z-c_N)(w-c_N)} \cdot z^{\tilde{x}}w^{\tilde{y}} \\
&\times  (1-q/z)^{-T_s-N}(1-q/w)^{-T_t-N} (1-qz)^N (1-qw)^N \\
& + \frac{1}{2\pi \im} \oint_{C_{r_2}} d  z  \frac{1}{zc_N-1} \cdot  z^{\tilde{x}}c_N^{\tilde{y}}  \cdot  (1-q/z)^{-T_s-N}(1-q/c_N)^{-T_t-N} (1-qz)^N (1-qc_N)^N.
\end{split}
\end{equation}
We may now deform the contour $C_{r_2}$ on the third line of (\ref{eq:K22Res1}) to $\gamma_N^-(\ap_1^q)$ without crossing any poles, since by point (2) below (\ref{eq:PivotIneq}) we have that $\gamma_N^-(\ap_1^q)$ encloses $c_N^{-1}$. Using (\ref{eq:ChangeOfVar}) we obtain
\begin{equation}\label{eq:K22Match1}
(1/4)(1-q)^{T_s + T_t} \times [\mbox{line 3 in (\ref{eq:K22Res1})}] = \frac{1}{2\pi \im} \oint_{\gamma^-_N(\ap^q_{1})} dz \frac{F_{22}^N(z,c_N)}{4(c_N z - 1)}.
\end{equation}

We next deform $C_{r_2}$ in the first line of (\ref{eq:K22Res1}) to $\gamma_N^-(-\ap_2^q)$. In the process of deformation we cross the simple poles at $z = c_N$ and $z = w^{-1}$, see points (3) and (4). We conclude
\begin{equation}\label{eq:K22Res2}
\begin{split}
&[\mbox{lines 1 and 2 in (\ref{eq:K22Res1})}] = \frac{1}{(2\pi \im)^{2}} \oint_{\gamma_N^-(-\ap_2^q)} d  z \oint_{\gamma_N^-(-\ap_2^q)}  dw \frac{z-w}{zw-1}\cdot \frac{1}{(z-c_N)(w-c_N)} \cdot z^{\tilde{x}}w^{\tilde{y}} \\
&\times  (1-q/z)^{-T_s-N}(1-q/w)^{-T_t-N} (1-qz)^N (1-qw)^N \\
& -  \frac{1}{2\pi \im} \oint_{\gamma_N^-(-\ap_2^q)}  dw \frac{c_N^{\tilde{x}}w^{\tilde{y}}}{c_Nw-1} \cdot  (1-q/c_N)^{-T_s-N}(1-q/w)^{-T_t-N} (1-qc_N)^N (1-qw)^N \\
&  + \frac{1}{2\pi \im}  \oint_{\gamma_N^-(-\ap_2^q)}  dw \frac{(1-w^2)}{(1-c_Nw)(w-c_N)} \cdot  w^{\tilde{y} -\tilde{x} - 1} \cdot  (1-wq)^{-T_s}(1-q/w)^{-T_t}.
\end{split}
\end{equation}
Using (\ref{eq:ChangeOfVar}) we obtain
\begin{equation}\label{eq:K22Res3}
\begin{split}
&(1/4)(1-q)^{T_s + T_t} \times [\mbox{lines 1 and 2 in (\ref{eq:K22Res2})}] = I_{22}^N(s,x;t,y), \\
&(1/4)(1-q)^{T_s + T_t} \times [\mbox{line 3 in (\ref{eq:K22Res2})}] = - \frac{1}{2\pi \im} \oint_{\gamma^-_N(-\ap^q_{2})} dw \frac{F_{22}^N(c_N,w)}{4(c_N  w - 1)}, \\
&(1/4)(1-q)^{T_s + T_t} \times [\mbox{line 4 in (\ref{eq:K22Res2})}] = \frac{1}{2\pi \im}\oint_{\gamma^-_N(-\ap^q_{2})} dw  \frac{(1-w^2)}{4(1-c_N w)(w-c_N)}  \\
&\times e^{ ( \sigma_q y N^{1/3} - \sigma_q x N^{1/3} - 1) \log (w) - T_s G(w^{-1}) - T_t G(w)}.
\end{split}
\end{equation}
Combining (\ref{eq:K22Res1}), (\ref{eq:K22Match1}), (\ref{eq:K22Res2}) and (\ref{eq:K22Res3}) we conclude $\tilde{K}^N_{22}(s,x;t,y) = K^N_{22}(s,x;t,y)$.
\end{proof}

%
%
\subsection{Point process convergence}\label{Section6.2} In this section we state two technical lemmas regarding the asymtptotic behavior of $K^N$ from Lemma \ref{lem:PrelimitKernel}. The first, Lemma \ref{lem:kernelLimits} below, shows that these kernels have a pointwise limit and its proof is given in Section \ref{Section7.2}. The second, Lemma \ref{lem:kernelUpperTail} below, obtains upper tail estimates for the kernels $K^N$ and its proof is given in Section \ref{Section7.3}. After these two results we prove two additional lemmas used in the proof of Proposition \ref{prop: FinDimConv} in Section \ref{Section6.4}.

\begin{lemma}\label{lem:kernelLimits} Assume the same notation as in Lemma \ref{lem:PrelimitKernel}, and recall the contours $\mathcal{C}_z^{\varphi}$ from Definition \ref{S1Contours}. Fix $x_{\infty},y_{\infty} \in \mathbb{R}$, $s,t \in \mathcal{T}$ and sequences $x_N \in \Lambda_s(N), y_N \in \Lambda_{t}(N)$ such that $\lim_N x_N = x_{\infty}$ and $\lim_N y_N = y_{\infty}$. Then, we have that the following limits all hold.
\begin{equation}\label{eq:I11Lim}
\begin{split}
&\lim_N I^N_{11}(s,x_N;t,y_N) = I^{\infty}_{11}(s,x_{\infty};t,y_{\infty}) \mbox{, where } I^{\infty}_{11}(s,x;t,y) \\
& = \frac{1}{(2\pi \im)^{2}}\int_{\mathcal{C}^{\pi/3}_{\ap_{3}}} dz\int_{\mathcal{C}^{\pi/3}_{\ap_{3}}} dw e^{z^3/3 + w^3/3 - f_q s z^2 - f_q t w^2 - xz - yw  } \frac{(z-w) (z + \varpi) (w + \varpi)}{zw (z + w)},
\end{split}
\end{equation}
\begin{equation}\label{eq:I12Lim}
\begin{split}
&\lim_N I^N_{12}(s,x_N;t,y_N) = I^{\infty}_{12}(s,x_{\infty};t,y_{\infty}) , \mbox{ where }  I^{\infty}_{12}(s,x;t,y) \\
&= \frac{1}{(2\pi \im)^{2}}\int_{\mathcal{C}^{\pi/3}_{\ap_{3}}} dz\int_{\mathcal{C}^{2\pi/3}_{\ap_{1}}} dw e^{z^3/3 - w^3/3 - f_q s z^2 + f_q t w^2 - xz + yw  } \frac{(z+w)(z+\varpi)}{2z(z-w)(w + \varpi)}.
\end{split}
\end{equation}
\begin{equation}\label{eq:I22Lim}
\begin{split}
&\lim_N I^N_{22}(s,x_N;t,y_N) =  I^{\infty}_{22}(s,x_{\infty};t,y_{\infty}) , \mbox{ where }  I^{\infty}_{22}(s,x;t,y)  \\
& = \frac{1}{(2\pi \im)^{2}}\int_{\mathcal{C}^{2\pi/3}_{-\ap_{2}}} dz\int_{\mathcal{C}^{2\pi/3}_{-\ap_{2}}} dw e^{-z^3/3 - w^3/3 + f_q s z^2 + f_q t w^2 + xz + y w  } \frac{(z-w)}{4(z+w)(z+ \varpi)(w+\varpi)},
\end{split}
\end{equation}
\begin{equation}\label{eq:R12Lim}
\begin{split}
\lim_N R^N_{12}(s,x_N;t,y_N) = R^{\infty}_{12}(s,x_{\infty};t,y_{\infty}) ,\mbox{ where } R^{\infty}_{12}(s,x;t,y) =\frac{- {\bf 1}\{s < t\}}{\sqrt{4\pi f_q (t-s)}} \cdot e^{-\frac{(y - x)^2}{4 f_q (t-s)}},
\end{split}
\end{equation}
In addition, if $y_{\infty} \neq x_{\infty}$ we have 
\begin{equation}\label{eq:R22Lim}
\begin{split}
\lim_N R^N_{22}(s,x_N;t,y_N) =  R^{\infty}_{22}(s,x_{\infty};t,y_{\infty}) ,
\end{split}
\end{equation}
where $R^{\infty}_{22}(s,x;t,y)$ is a skew-symmetric kernel such that for $y > x$ we have
\begin{equation}\label{eq:R22LimDef}
\begin{split}
&R^{\infty}_{22}(s,x;t,y)= \frac{1}{2\pi \im} \int_{\mathcal{C}^{2\pi/3}_{\ap_{1}}} dz \frac{e^{-z^3/3 + \varpi^3/3 + f_q s z^2 + f_q t \varpi^2 + xz - y \varpi  }}{4(z - \varpi)}\\
& - \frac{1}{2\pi \im} \int_{\mathcal{C}^{2\pi/3}_{-\ap_{2}}}\hspace{-1.5mm} dw\frac{e^{\varpi^3/3 - w^3/3 + f_q s \varpi^2 + f_q t w^2 - x\varpi + y w  }}{4(w - \varpi)} +  \frac{{\bf 1}\{s + t > 0\} }{2\pi \im} \int_{\mathcal{C}^{2\pi/3}_{-\ap_{2}}} \hspace{-1.5mm}dw  \frac{we^{  f_q (s + t) w^2  + w (y-x)  }}{2(w- \varpi)(w + \varpi) }.
\end{split}
\end{equation}
\end{lemma}
\begin{remark}\label{rem:PtLimit} We mention that Lemma \ref{lem:kernelLimits} and its proof are similar to \cite[Proposition 4.5]{BBNV}, where the authors worked with a different Pfaffian Schur process, resulting in different integrands than ours. One of the things we do differently is we work with considerably simpler contours, which allow us to obtain fairly detailed estimates for the kernels $K^N$ that are uniform in $N$. The latter are useful for getting the upper tail estimates in Lemma \ref{lem:kernelUpperTail} below.
\end{remark}

\begin{lemma}\label{lem:kernelUpperTail}Assume the same notation as in Lemma \ref{lem:PrelimitKernel} and fix $t \in [0,\infty)$, $A > 0$. Then, we can find $N_1 \in \mathbb{N}$, depending on $q, \varpi, t, A$ and the sequence $c_N$, and a constant $D$, depending on $q, \varpi, t, A$, such that for $N \geq N_1$ and $x_N, y_N \in \Lambda_t(N)$ with $x_N, y_N \geq - A$ we have
\begin{equation}\label{eq:kernelUpperTail}
\begin{split}
\left|K_{11}^N(t,x_N;t,y_N) \right|& \leq D e^{-ax_N - ay_N}, \hspace{2mm}\left|K_{12}^N(t,x_N;t,y_N) \right| \leq D e^{-ax_N + by_N}, \\
&\hspace{2mm} \left|K_{22}^N(t,x_N;t,y_N) \right| \leq D e^{bx_N + by_N},
\end{split}
\end{equation}
where $a = \ap_3 - 1 = |\varpi| + 8$ and $b =  \ap_1 + 1 = |\varpi| + 4$.
\end{lemma}
\begin{remark}\label{rem:kernelUpperTail} We mention that the constants $a,b$ in Lemma \ref{lem:kernelUpperTail} are specific for our choice of contours in the definition of $K^N$. If one picks different contours, one could obtain (\ref{eq:kernelUpperTail}) for any $a > b > |\varpi|$ at the expense of having different constants $D$ and $N_1$.
\end{remark}

In the remainder of this section we prove two useful lemmas that follow from Lemmas \ref{lem:kernelLimits} and \ref{lem:kernelUpperTail}, and will be used in the proof of Proposition \ref{prop: FinDimConv}. We also require the following definition.
\begin{definition}\label{def:LimKernel} For $q \in (0,1)$, we define the kernel $K^{\infty}: ([0, \infty) \times \mathbb{R})^2 \rightarrow\operatorname{Mat}_2(\mathbb{C})$ via
\begin{equation}\label{eq:LimKernel}
\begin{split}
&K^{\infty}(s,x;t,y) = \begin{bmatrix}
    K^{\infty}_{11}(s,x;t,y) & K^{\infty}_{12}(s,x;t,y)\\
    K^{\infty}_{21}(s,x;t,y) & K^{\infty}_{22}(s,x;t,y) 
\end{bmatrix} \\
&= \begin{bmatrix}
    I^{\infty}_{11}(s,x;t,y) & I^{\infty}_{12}(s,x;t,y) + R^{\infty}_{12}(s,x;t,y) \\
    -I^{\infty}_{12}(t,y;s,x) - R^{\infty}_{12}(t,y;s,x) & I^{\infty}_{22}(s,x;t,y) + R^{\infty}_{22}(s,x;t,y) 
\end{bmatrix} ,
\end{split}
\end{equation}
where $I_{11}^{\infty}, I_{12}^{\infty}, I_{22}^{\infty}, R_{12}^{\infty}$ and $R_{22}^{\infty}$ are as in Lemma \ref{lem:kernelLimits}. We also define for $t \in [0, \infty)$ the kernel $K^{t, \infty}: \mathbb{R}^2 \rightarrow  \operatorname{Mat}_2(\mathbb{C})$ via $K^{t,\infty}(x,y) = K^{\infty}(t,x;t,y)$.
\end{definition}

With the above results in place we can show that the point processes $M^N$ from Lemma \ref{lem:PrelimitKernel} converge weakly -- this is the content of the following lemma.
\begin{lemma}\label{lem:WeakConvPP} Assume the same notation as in Lemma \ref{lem:PrelimitKernel}. Then, $M^N$ converge weakly to $M^{\infty}$, which is a Pfaffian point process on $\mathbb{R}^2$ with reference measure $\mu_{\mathcal{T}} \times \lambda$ and correlation kernel $K^{\infty}$ as in Definition \ref{def:LimKernel}. If we define the random measure on $\mathbb{R}$
$$M^{t_j,N}(A) = M^N(\{t_j\} \times A) = \sum_{ i\geq 1} {\bf 1}\{X_i^{j,N} \in A\},$$
then $M^{t_j,N}$ is a Pfaffian point process with correlation kernel $K^{t_j,N}(x,y ) = K^{N}(t_j,x; t_j,y)$ and reference measure $\nu_{t_j}(N)$. In addition $M^{t_j, N}$ converge weakly to a Pfaffian point process $M^{t_j,\infty}$ that has reference measure $\lambda$ and correlation kernel $K^{t_j, \infty}$ as in Definition \ref{def:LimKernel}.
\end{lemma}
\begin{proof} The fact that $M^{t_j,N}$ is a Pfaffian point process follows from Lemma \ref{lem:LemmaSlice}. The weak convergence of $M^N$ follows by applying Proposition \ref{prop: conv of Pfaffian point processes 1} with $a_t(N), b_t(N)$ as in (\ref{eq:Lattices}). Indeed, by Lemma \ref{lem:PrelimitKernel} the $M^N$ are Pfaffian point processes with kernels $K^N$ as in (\ref{eq:S6Kdecomp}). The kernels $K^N$ satisfy (\ref{eq:limitProp}) with $K = K^{\infty}$ as in (\ref{eq:LimKernel}) and $U = \{(x,y) \in \mathbb{R}^2: x\neq y \}$ from Lemma \ref{lem:kernelLimits}, and they satisfy (\ref{eq:limitProp2}) from Lemma \ref{lem:kernelUpperTail}. The fact that $M^{t_j, N}$ converge weakly analogously follows by applying Lemma \ref{lem:LemmaSlice} with $\mathcal{T} = \{t_j\}$, Lemma \ref{lem:LemmaSlice} and using the natural isomorphism $\{t_j\} \times \mathbb{R} \cong \mathbb{R}$. 
\end{proof}

We end this section by showing that for each $j \in \llbracket 1, m \rrbracket$, the sequence $(X^{j,N}_1)^+ = \max(0, X_1^{j,N})$ from Definition \ref{def:ParScale} is tight. The latter will be used to verify one of the conditions in Lemma \ref{lem:technical lemma fdd 1} in the following section.
\begin{lemma}\label{lem:TightFromAbove} Assume the same notation as in Lemma \ref{lem:PrelimitKernel}. Then, for each $j \in \llbracket 1, m \rrbracket$, we have that the sequence $(X^{j,N}_1)^+ = \max(0, X_1^{j,N})$ is tight.
\end{lemma}
\begin{proof} Fix $j \in \llbracket 1, m \rrbracket$, $K^{t_j,N}(x,y) = K^{N}(t_j,x; t_j,y)$,  $A > 0$, $u \geq -A$, $\Delta = \sigma_q^{-1} N^{-1/3}$ and $\Lambda = \Lambda_{t_j}(N) \cap (u, \infty)$. From Lemmas \ref{lem:HadamardPf} and \ref{lem:kernelUpperTail} there is a constant $D > 0$ such that for all large $N$
\begin{equation}\label{eq:upperT}
\begin{split}
&\sum_{n = 1}^{\infty} \frac{1}{n!}\sum_{x_1, \dots, x_n \in \Lambda} \Delta^n \left| \operatorname{Pf} [K^{t_j,N}(x_i,x_j)]_{i,j = 1}^n \right| \leq \sum_{n = 1}^{\infty} \frac{1}{n!} \sum_{x_1, \dots, x_n \in \Lambda} \Delta^n D^n (2n)^{n/2} \prod_{i = 1}^n e^{-(a-b) x_i}\\
&\leq \sum_{n = 1}^{\infty} \frac{D^n (2n)^{n/2} e^{|a-b| n \Delta}}{n!}  \sum_{x_1, \dots, x_n \in \Lambda} \prod_{i = 1}^n\int_{x_i}^{x_i + \Delta} e^{-(a- b)y_i} dy_i\\
& \leq \sum_{n = 1}^{\infty} \frac{D^n (2n)^{n/2} e^{|a-b| n \Delta}}{n!} \prod_{i = 1}^n \int_{u}^{\infty} e^{-(a- b)y_i} dy_i =   \sum_{n = 1}^{\infty} \frac{D^n (2n)^{n/2} e^{|a-b| n \Delta} e^{-nu(a- b)}}{n!(a-b)^n}  < \infty,
\end{split}
\end{equation}
where in going from the first to the second line we used that for any $c,x \in \mathbb{R}$ we have
\begin{equation*}
\Delta \cdot e^{cx} \leq e^{|c| \Delta} \cdot \int_{x}^{x + \Delta} e^{cy} dy.
\end{equation*}
Proposition \ref{prop:Last Particle Cdf}, the fact that $M^{t_j,N}$ is Pfaffian from Lemma \ref{lem:WeakConvPP}, and (\ref{eq:upperT}), imply for all $u \geq 0$ and large $N$ so that $\Delta |a-b| \leq 1$:
$$\mathbb{P}(X_1^{j,N} \geq u) \leq \sum_{n = 1}^{\infty} \frac{D^n (2n)^{n/2} e^{|a-b| n \Delta} e^{-nu(a- b)}}{n!(a-b)^n} \leq e^{-u(a-b)}\sum_{n = 1}^{\infty} \frac{D^n (2n)^{n/2} e^{n}}{n!(a-b)^n} \leq \cb e^{-u(a-b)},$$ 
where $\cb$ is a large constant that does not depend on $N$. The last inequality proves the lemma.
\end{proof}

%
%
\subsection{Infinite atoms in the limit}\label{Section6.3} The goal of this section is to establish the following result.

\begin{lemma}\label{lem:InfiniteAtoms} Let $M^{t,\infty}$ be as in Lemma \ref{lem:WeakConvPP}. Then, $\mathbb{P}(M^{t,\infty}(\mathbb{R}) = \infty) = 1$.
\end{lemma}
\begin{remark}\label{rem:InfiniteAtoms} In words, Lemma \ref{lem:InfiniteAtoms} says that $M^{t,\infty}$ contains infinitely many atoms almost surely.
\end{remark}

In order to establish Lemma \ref{lem:InfiniteAtoms} we require the following statement, whose proof is given in Section \ref{Section7.4}.

\begin{lemma}\label{lem:ConvToAiryKernel} Let $K^{t,\infty}(x,y)$ be as in Definition \ref{def:LimKernel}. Then, the following limits all hold uniformly as $(x,y)$ vary over bounded sets in $\mathbb{R}^2$.
\begin{equation}\label{eq:Limkcr11}
\begin{split}
\lim_{t\rightarrow\infty} t e^{-\frac{2}{3}t^3} e^{t(x+y)} {K}_{11}^{f_q^{-1}t,\infty} (x-t^2,y-t^2)= \frac{1}{(2\pi \im)^{2}} \hspace{-1mm}\int_{\mathcal{C}^{\pi/3}_{1}} \hspace{-1mm}dz\int_{\mathcal{C}^{ \pi/3}_{1}} \hspace{-1mm} dw \frac{(z-w)}{2}  e^{z^3/3+w^3/3-zx-wy}.
\end{split}
\end{equation}
\begin{equation}\label{eq:Limkcr12}
\lim_{t \rightarrow \infty} e^{t(x-y)}{K}^{f_q^{-1}t, \infty}_{12}(x - t^2,y - t^2) = \frac{1}{(2\pi \im)^{2}}\int_{\mathcal{C}^{\pi/3}_{\ap_3}} dz\int_{\mathcal{C}^{2\pi/3}_{\ap_1}} dw \frac{e^{z^3/3 - w^3/3 - xz + yw } }{z-w}.
\end{equation}
\begin{equation}\label{eq:Limkcr22}
\lim_{t \rightarrow \infty} t^3 e^{\frac{2}{3}t^3} e^{-t(x+y)} {K}_{22}^{f_q^{-1}t,\infty} (x-t^2,y-t^2)= \frac{1}{(2\pi \im)^{2}}\hspace{-1mm}\int_{\mathcal{C}^{\pi/2}_{-1}} \hspace{-1mm}dz\int_{\mathcal{C}^{\pi/2}_{-1}} \hspace{-1mm}dw \frac{z-w}{8} e^{-z^3/3 - w^3/3 +x z + y w  }.
\end{equation}
\end{lemma}

In the remainder of this section we prove Lemma \ref{lem:InfiniteAtoms}.
\begin{proof}[Proof of Lemma \ref{lem:InfiniteAtoms}] Fix $t \in [0, \infty)$, and let $s_n = t + n$ for $n \in \mathbb{Z}_{\geq 0}$. In addition, define $\phi_n(x) = x + f_q^2 s_n^2$. In what follows we seek to apply Lemma \ref{lem:technical lemma fdd 2} to the sequence $\{M^{s_n,\infty} \phi_n^{-1}\}_{ n \geq 0}$, and for clarity we split the proof into three steps.\\

{\bf \raggedleft Step 1.} We claim that for any fixed $a, n \in \mathbb{Z}_{\geq 0}$ we have that
\begin{equation}\label{UOP1}
p_n^a \geq p_{n+1}^a, \mbox{ where } p_n^a := \mathbb{P}(M^{s_n, \infty} \phi_n^{-1}(\mathbb{R}) \geq a) = \mathbb{P}(M^{s_n, \infty}(\mathbb{R}) \geq a).
\end{equation}
We establish (\ref{UOP1}) in the steps below. Here, we assume its validity and prove the lemma.\\

Note that from Lemma \ref{lem:WeakConvPP} and parts (4) and (5) of Proposition \ref{prop:basic properties Pfaffian point process} we know that $M^{s_n, \infty} \phi_n^{-1}$ is a Pfaffian point process on $\mathbb{R}$ with correlation kernel
$$\begin{bmatrix}
    f_n(x) f_n(y) K^{s_n,\infty}_{11}(x - f_q^2s_n^2, y- f_q^2s_n^2) & \frac{f_n(x)}{f_n(y)}K^{s_n, \infty}_{12}(x - f_q^2s_n^2, y- f_q^2s_n^2)\\
    \frac{f_n(y)}{f_n(x)} K^{t_n,\infty}_{21}(x- f_q^2s_n^2, y- f_q^2s_n^2) &  \frac{1}{f_n(x) f_n(y)}K^{t_n,\infty}_{22}(x- f_q^2s_n^2,y- f_q^2s_n^2) 
\end{bmatrix}, $$
where $f_n(x) = e^{f_q s_n x} \cdot e^{-f_q^3s_n^3/3}$. From Lemma \ref{lem:ConvToAiryKernel} we have that the above kernel converges uniformly over bounded sets to 
$$\begin{bmatrix}
    0 & K^{\mathrm{Airy}}(x,y)\\
    -K^{\mathrm{Airy}}(y,x) &  0
\end{bmatrix}, $$
where we used that the right side of (\ref{eq:Limkcr12}) is precisely the {\em Airy kernel}, i.e. the kernel from (\ref{S1AiryKer}) with $t_1 = t_2$. From Lemma \ref{lem:determinantal point process as Pfaffian} and Proposition \ref{prop: conv of Pfaffian point processes 0} we conclude that $M^{s_n, \infty} \phi_n^{-1}$ converge to the Airy point process $\mathcal{A}$ as $n \rightarrow \infty$. It is well-known that the Airy point process has infinitely many atoms almost surely, i.e.
\begin{equation}\label{S72E2}
\mathbb{P}(\mathcal{A}(\mathbb{R}) = \infty) = 1,
\end{equation}
see e.g. \cite[Equation (7.11)]{dimitrov2024airy}. Combining the above observations with (\ref{UOP1}) allows us to conclude that $M^{s_n, \infty} \phi_n^{-1}$ satisfy the conditions of Lemma \ref{lem:technical lemma fdd 2}, and so $\mathbb{P}(M^{t, \infty}(\mathbb{R}) = \infty) = \mathbb{P}(M^{s_0, \infty}\phi_0^{-1}(\mathbb{R}) = \infty) = 1$ as desired.\\

{\bf \raggedleft Step 2.} In the sequel we fix $n \in \mathbb{Z}_{\geq 0}$ and proceed to show (\ref{UOP1}). Since $p_n^0 = 1$, we may assume $a \geq 1$. Finally, as (\ref{UOP1}) clearly holds if $p^a_{n+1} = 0$, we may also assume that $p^a_{n+1} > 0$.\\

We first observe from Lemma \ref{lem:technical lemma fdd 1} that for each $x \in \mathbb{R}$ and $u, v \in \mathbb{Z}_{\geq 0}$ we have
\begin{equation}\label{UOP2}
\lim_{N \rightarrow \infty} \mathbb{P}( M^{s_v, N}([x, \infty)) \leq u) = \mathbb{P}( M^{s_v, \infty}([x, \infty)) \leq u) .
\end{equation}
Indeed, condition (1) of the lemma is satisfied by Lemma \ref{lem:TightFromAbove} , condition (2) by Lemma \ref{lem:WeakConvPP} and condition (3) by the fact that for each $x \in \mathbb{R}$ 
$$\mathbb{E}[M^{s_v,\infty}(\{x\})] = \int_{\{x\}} K^{s_v,\infty}_{12}(y,y) \lambda(dy) = 0,$$
where we used the definition of correlation functions (\ref{eq:RND determinantal}) and that $\lambda$ is the Lebesgue measure.\\

Let $\varepsilon \in (0,1/2)$ be given. Using that $p^a_{n+1} > 0$ and the monotone convergence theorem, we can find $x_0 \in \mathbb{R}$ such that
\begin{equation}\label{UOP3}
\mathbb{P}(M^{s_{n+1}, \infty}([x_0,\infty)) \geq a) \geq (1- \varepsilon)p^a_{n+1} .
\end{equation}
In addition, using (\ref{UOP2}) we have for all large $N$ that
\begin{equation}\label{UOP4}
\mathbb{P}(M^{s_{n+1}, N}([x_0,\infty)) \geq a) \geq (1- 2\varepsilon)p^a_{n+1} .
\end{equation}
We now claim that we can find a constant $R \geq 0$ sufficiently large, depending on $q, \varpi, t, n, \varepsilon$, such that for all large $N$
\begin{equation}\label{UOP5}
\mathbb{P}\left(M^{s_{n}, N}([x_0 -R ,\infty)) \geq a \Big\vert M^{s_{n+1}, N}([x_0  ,\infty)) \geq a \right) \geq 1- \varepsilon .
\end{equation}
We will establish (\ref{UOP5}) in the next step. Here, we assume its validity and prove (\ref{UOP1}).\\

Combining (\ref{UOP4}) and (\ref{UOP5}) we conclude for all large $N$
\begin{equation*}
\begin{split}
\mathbb{P}(M^{s_{n}, N}([x_0 -R ,\infty)) \geq a)  \geq (1- \varepsilon)  (1- 2\varepsilon)p^a_{n+1}.
\end{split}
\end{equation*}
Taking the $N \rightarrow \infty$ limit above, and using (\ref{UOP2}) we conclude
\begin{equation*}
\begin{split}
&p_n^a = \mathbb{P}(M^{s_{n}, \infty}(\mathbb{R}) \geq a) \geq \mathbb{P}(M^{s_{n}, \infty}([x_0 -R ,\infty)) \geq a) \\
&= \lim_{N \rightarrow \infty}  \mathbb{P}(M^{s_{n}, N}([x_0 -R ,\infty)) \geq a) \geq (1- \varepsilon) (1- 2\varepsilon)p^a_{n+1}.
\end{split}
\end{equation*}
Letting $\varepsilon \rightarrow 0$ in the last expression gives (\ref{UOP1}).\\

{\bf \raggedleft Step 3.} In this final step we show (\ref{UOP5}), which by the definition of $M^{s_{n}, N}$ is equivalent to 
\begin{equation}\label{UOP6}
\mathbb{P}\left( \lambda_a^{T_{s_n}} \geq A_N(s_n, x_0 -R  ) \Big\vert  \lambda_a^{T_{s_{n+1}}} \geq A_N(s_{n+1}, x_0 )  \right)  \geq 1- \varepsilon ,
\end{equation}
where  we recall that $T_s = T_s(N) = \lfloor s N^{2/3} \rfloor$ and 
$$A_N(s,x) = \sigma_q N^{1/3} x_0 + \frac{2qN}{1-q} + \frac{qT_s}{1-q} + a.$$
The way we prove (\ref{UOP6}) is by applying appropriately Lemma \ref{lem:lower bound on curve canal} as we explain below.

From Lemma \ref{lem:SchurGibbs} we know that the line ensemble $\mathfrak{L}^N = \{L^N_i \}_{i  \geq 1}$ defined by $L_i^N(s) = \lambda_i^{s}$ for $i \geq 1$ and $s \in \llbracket 0, M \rrbracket$ with $(\lambda^1, \dots, \lambda^M)$ having law $\mathbb{P}_N$ as in Definition \ref{def:ParScale} satisfies the half-space interlacing Gibbs property with parameters $q_i = q_i^N = c_N^{(-1)^{i}} \cdot q = q + O(N^{-1/3})$, where the constant in the big $O$ notation depends on $\varpi$ and $q$ alone. In particular, we see that if we denote 
$$\mathcal{F}_{n,N} = \sigma \left( \{ L_i^N(T_{s_{n+1}})   \mbox{ for } i \in \llbracket 1,a \rrbracket \mbox{ and } L^N_{a+1}(s) \mbox{ for }s \in \llbracket 0, T_{s_{n+1}}  \rrbracket \} \right),$$
then we have for any $x_{i}(s) \in \mathbb{R}$ that $\mathbb{P}$-almost surely
\begin{equation}\label{UOP7}
\begin{split}
&\mathbb{P}\left( L_i^N(s) = x_{i}(s) \mbox{ for }i \in \llbracket 1, a \rrbracket, s \in \llbracket 0, T_{s_{n+1}} \rrbracket \vert \mathcal{F}_{n,N}  \right)  \\
&= \mathbb{P}_{\ice, \operatorname{Geom}}^{T_{s_{n+1}}, \vec{y}, \vec{q}^N_a, g} \left(Q_i(s) = x_{i}(s) \mbox{ for }i \in \llbracket 1, a \rrbracket, s \in \llbracket 0, T_{s_{n+1}}  \llbracket \right),
\end{split}
\end{equation}
 where $\vec{q}^N_a = (q_1^N, \dots, q_a^N)$, $\vec{y} = (y_1, \dots, y_a) = (L_1^N(T_{s_{n+1}}), \dots L_a^N(T_{s_{n+1}}))$ and $g(s) = L^N_{a+1}(s)$ for $s \in \llbracket 0, T_{s_{n+1}} \rrbracket$. 

We now let $\delta \in (0,1/2)$ be small enough and $N_1$ be large enough (depending on $q, \varpi, t, n$) so that $T_{s_{n+1}} \geq \mathsf{T}_3$ from Lemma \ref{lem:lower bound on curve canal} and $q_i^N \in [\delta, 1- \delta]$ for $N \geq N_1$. Applying Lemma \ref{lem:lower bound on curve canal} with $T = T_{s_{n+1}}$, $\varepsilon, \delta$ as in the present setting, $k = a$, $q_i = q_i^N$ and $t = T_{s_n}$, we conclude that we can find a large enough $M$, depending on $q, \varpi, t, n, \varepsilon$, such that
\begin{equation}\label{UOP8}
\begin{split}
\mathbb{P}_{\ice, \operatorname{Geom}}^{T_{s_{n+1}}, \vec{y}, \vec{q}^N_a, g}\left(Q_a(T_{s_n})\geq y_a + \frac{q (T_{s_n}-T_{s_{n+1}})}{1 - q} -M N^{1/3} \right)\geq 1-\varepsilon.
\end{split}
\end{equation}
We mention that in the last inequality we also used $q_i^N = q + O(N^{-1/3})$ and $T_{s_{n+1}} =  \lfloor (t+n) N^{2/3}  \rfloor $. 

If $E_{n, N} = \{ \lambda_a^{T_{s_{n+1}}} \geq A_N(s_{n+1}, x_0 ) \}$ and $R = \sigma_q^{-1} M$, we get the following tower of inequalities
\begin{equation}\label{UOP9}
\begin{split}
&\mathbb{P} \left( \lambda_a^{T_{s_n}} \geq A_N(s_n, x_0 -R  ) \cap E_{n, N}  \right) =\mathbb{E} \left[ {\bf 1}_{E_{n, N}} \cdot \mathbb{E}  \left[ {\bf 1}\{ \lambda_a^{T_{s_n}} \geq A_N(s_n, x_0 -R  ) \} \vert \mathcal{F}_{n,N} \right] \right] \\
& = \mathbb{E} \left[ {\bf 1}_{E_{n, N}} \cdot \mathbb{P}_{\ice, \operatorname{Geom}}^{T_{s_{n+1}}, \vec{y}, \vec{q}^N_a, g}\left(Q_a(T_{s_n}) \geq  A_N(s_n, x_0 -R  )  \right) \right] \\
& \geq  \mathbb{E} \left[ {\bf 1}_{E_{n, N}} \cdot \mathbb{P}_{\ice, \operatorname{Geom}}^{T_{s_{n+1}}, \vec{y}, \vec{q}^N_a, g}\left(Q_a(T_{s_n}) \geq y_a + A_N(s_n, x_0 -R  ) - A_N(s_{n+1}, x_0 )   \right) \right] \\
& =  \mathbb{E} \left[ {\bf 1}_{E_{n, N}} \cdot \mathbb{P}_{\ice, \operatorname{Geom}}^{T_{s_{n+1}}, \vec{y}, \vec{q}^N_a, g}\left(Q_a(T_{s_n}) \geq y_a  + \frac{q (T_{s_n}-T_{s_{n+1}})}{1 - q} - \sigma_q RN^{1/3}  \right) \right] \\
&\geq (1-\varepsilon) \cdot \mathbb{P}(E_{n,N}).
\end{split}
\end{equation}
We mention that the equality on the first line follows from the tower property for conditional expectation, and $E_{n, N}  \in \mathcal{F}_{n,N} $, while in going from the first to the second line we used (\ref{UOP7}). The inequality on the third line follows from the fact that on $E_{n,N}$ we have $y_a = L_a(T_{s_{n+1}}) \geq A_N(s_{n+1}, x_0 ) $. In going from the third to the fourth line we used the formula for $A_{N}(s,x)$ and the inequality on the fifth line used $R = \sigma_q^{-1} M$ and (\ref{UOP8}).

Equation (\ref{UOP9}) now implies (\ref{UOP6}), which concludes the proof of the lemma.
\end{proof}

%
%
\subsection{Proof of Proposition \ref{prop: FinDimConv}}\label{Section6.4} We first seek to apply \cite[Proposition 2.21]{dimitrov2024airy} in our present setup. The measures $M^{t_j, N}$ from Lemma \ref{lem:WeakConvPP} converge weakly to $M^{t_j, \infty}$, which verifies condition (1) in \cite[Proposition 2.21]{dimitrov2024airy}. In addition, by Lemma \ref{lem:InfiniteAtoms} we have $\mathbb{P}(M^{t_j, \infty}(\mathbb{R}) = \infty) = 1$, which verifies condition (2) in \cite[Proposition 2.21]{dimitrov2024airy}. Lastly, from Lemma \ref{lem:TightFromAbove}, we have 
$$\lim_{a \rightarrow \infty} \limsup_{N \rightarrow \infty} \mathbb{P}(X_1^{j,N} \geq a) = 0,$$
verifying condition (3) in \cite[Proposition 2.21]{dimitrov2024airy}. From \cite[Proposition 2.21]{dimitrov2024airy} we conclude that for each $i \in \mathbb{N}$ and $j \in \llbracket 1, m \rrbracket$ the sequence of random variables $\{X_i^{j, N}\}_{N \geq 1}$ is tight.

Combining the latter with Lemma \ref{lem:WeakConvPP}, we see that the conditions of \cite[Proposition 2.19]{dimitrov2024airy} are satisfied by $M^N$, from which we conclude that $(X_i^{j,N}: i \geq 1, j \in \llbracket 1, m \rrbracket)$ converges in the finite-dimensional sense to a vector $(X_i^{j,\infty}: i \geq 1, j \in \llbracket 1, m \rrbracket)$, and moreover the random measure $\hat{M}^{\infty}$ from (\ref{eq:Point process formed}) is a Pfaffian point process on $\mathbb{R}^2$ with reference measure $\mu_{\mathcal{T}} \times \lambda$ and correlation kernel $K^{\infty}$ as in (\ref{eq:LimKernel}). From part (4) of Proposition \ref{prop:basic properties Pfaffian point process} we also have that $\hat{M}^{\infty}$ is a Pfaffian point process with correlation kernel
$$\begin{bmatrix}
    f(s,x) f(t,y) K^{\infty}_{11}(s,x;t,y) & \frac{f(s,x)}{f(t,y)}K^{\infty}_{12}(s,x;t,y)\\
    \frac{f(t,y)}{f(s,x)}K^{\infty}_{21}(s,x;t,y) &  \frac{1}{f(s,x) f(t,y)} K^{\infty}_{22}(s,x;t,y) 
\end{bmatrix} , \mbox{ where } f(s,x) = e^{ -\frac{s^3 f_q^3}{3} + (x + f_q^2s^2) f_qs }.$$
We observe that the latter kernel agrees with $\hat{K}^{\infty}(s,x;t,y)$ from (\ref{eq:KX}) after a straightforward (albeit tedious) computation. To see that the two kernels match one needs to apply the gauge transformation above, and change the integration variables to $\tilde{z} = z +s$ and $\tilde{w} = w+ t$ in (\ref{S1DefIcross}).

%
%
\section{Asymptotic analysis}\label{Section7} The goal of this section is to prove Lemmas  \ref{lem:kernelLimits}, \ref{lem:kernelUpperTail} and \ref{lem:ConvToAiryKernel} from Section \ref{Section6}. We accomplish this in Sections \ref{Section7.2}, \ref{Section7.3} and \ref{Section7.4} below after we derive some preliminary estimates in Section \ref{Section7.1}.

%
%
\subsection{Preliminary estimates}\label{Section7.1} In this section we derive certain estimates for the integrands in $K^N$ from Lemma \ref{lem:PrelimitKernel}, which will be useful for its proof. As there are several kinds of functions that appear in those integrands, we split our discussion into several parts. Throughout the section we fix $A > 0$, $B = \ap_3^q = \sigma_q^{-1} (|\varpi| + 9)$ and assume the same notation as in Definitions \ref{def:contours} and \ref{def:funct}. In the section there will be explicit constants and ones contained in big $O$ notations, which unless otherwise specified depend on $\varpi, q$ alone. In addition, certain inequalities will hold provided that $N$ is sufficiently large, depending on $\varpi, q$, but possible additional parameters that will be listed.

%
%
\subsubsection{Estimates for $S(z)$}\label{Section7.1.1} From \cite[Lemma 5.2]{dimitrov2024airy} we have the following result.
\begin{lemma}\label{TaylorS} There exist $\delta_0 \in (0, 1/2)$ and $C_0 > 0$ (depending on $q$) such that 
\begin{equation}\label{EqTayS}
\left| S(z) - \sigma_q^3 (z-1)^3/3 \right| \leq C_0 \cdot |z-1|^4 \mbox{ if } |z-1| \leq \delta_0.
\end{equation}
Furthermore, we have the following inequalities
\begin{equation}\label{RealS}
\begin{split}
&\frac{d}{d\theta} \Real S(r e^{ \pm \im \theta}) > 0 \mbox{ for } \theta \in (0, \pi) \mbox{ and } r \in (0,1) \\
& \frac{d}{d\theta} \Real S(r e^{ \pm \im \theta}) < 0 \mbox{ for } \theta \in (0, \pi) \mbox{ and } r > 1.
\end{split}
\end{equation}
\end{lemma}
We now use Lemma \ref{EqTayS} to derive several estimates for $S(z)$. \\

Let $b \in [-B, B]$ and $z \in \gamma_N^+(b,0)$. Using that $z -1 = bN^{-1/3} + |s|e^{\pm \pi \im/3}$ for $s \in [0, N^{-1/12}]$, and (\ref{EqTayS}) we conclude for some positive $a_S, b_S > 0$ and all large $N$
\begin{equation}\label{eq:SReal0}
\begin{split}
&\Real S(z) \leq \Real \left( \sigma_q^3(bN^{-1/3} + |s|e^{\pm \pi \im/3})^3 /3 \right) + 2 C_0 N^{-1/12} \cdot|z-1|^3  \\
&\leq a_S N^{-1} + b_S N^{-1/3} \cdot  |z-1|^2 - (\sigma_q^3/6) \cdot |z-1|^3.
\end{split}
\end{equation}
One analogously shows, for possibly bigger $a_S,b_S$ that if $w \in \gamma_N^-(b,0)$, then
\begin{equation}\label{eq:SReal1}
\Real S(w) \geq - a_S N^{-1} - b_S N^{-1/3} \cdot  |w-1|^2 + (\sigma_q^3/6) \cdot |w-1|^3.
\end{equation}
Using (\ref{RealS}), (\ref{eq:SReal0}) and (\ref{eq:SReal1}), we also see that for all large $N$
\begin{equation}\label{eq:SReal2}
\begin{split}
&\Real S(z) \leq \Real S(1 + b N^{-1/3} + N^{-1/12} e^{ \pi \im /3}) \leq - (\sigma_q^3/12) N^{-1/4} \mbox{ for } z \in \gamma_N^+(b,1), \mbox{ and }\\
&\Real S(w) \geq \Real S(1 + b N^{-1/3} + N^{-1/12} e^{ 2\pi \im /3}) \geq  (\sigma_q^3/12) N^{-1/4} \mbox{ for } w \in \gamma_N^-(b,1).
\end{split}
\end{equation}

%
%
\subsubsection{Estimates for $G(z)$}\label{Section7.1.2} We have the following result.
\begin{lemma}\label{TaylorG} There exist $\delta_0 \in (0, 1/2)$ and $C_0 > 0$ (depending on $q$) such that 
\begin{equation}\label{EqTayG}
\left| G(z^{\pm 1}) + \frac{q}{2(1-q)^2} \cdot (z-1)^2 \right| \leq C_0 \cdot |z-1|^3 \mbox{ if } |z-1| \leq \delta_0.
\end{equation}
Furthermore, we have the following inequalities
\begin{equation}\label{RealG}
\begin{split}
&\frac{d}{d\theta} \Real G(r e^{\pm \im \theta}) > 0 \mbox{ for } \theta \in (0, \pi) \mbox{ and } r \in (q, \infty).
\end{split}
\end{equation}
\end{lemma}
\begin{proof} One directly computes 
$$G(1) = 0, \hspace{2mm} \frac{d}{dz} G(z^{\pm 1})\vert_{z = 1} = 0, \mbox{ and } \frac{d^2}{dz^2} G(z^{\pm 1})\vert_{z = 1} = \frac{-q}{(1-q)^2},$$
which implies (\ref{EqTayG}). Observe that 
$$\Real G(r e^{\pm \im \theta}) = (1/2) \log ( 1 +(q/r)^2 - 2(q/r) \cos(\theta))  - \log r \cdot \frac{q}{1-q} - \log (1-q),$$
and so
$$\frac{d}{d\theta} \Real G(r e^{\pm \im \theta}) = \frac{(q/r) \sin(\theta) }{1 + (q/r)^2 - 2(q/r) \cos(\theta)},$$
which gives (\ref{RealG}).
\end{proof}
We now use Lemma \ref{TaylorG} to derive several estimates for $G(z)$. \\

Let $b \in [-B, B]$ and $z \in \gamma_N^{\pm}(b,0)$. Using that $z -1 = bN^{-1/3} + |s|e^{ \alpha \im}$ for $s \in [0, N^{-1/12}]$ and $\alpha \in \{\pm \pi/3,  \pm 2\pi/3\}$, and (\ref{EqTayG}), we conclude for some positive $a_G, b_G > 0$ and all large $N$
\begin{equation}\label{eq:GReal0}
\begin{split}
&\Real G(z^{\pm 1}) \geq \Real \left( - \frac{q}{2(1-q)^2} \cdot (bN^{-1/3} + |s|e^{ \alpha \im})^2 \right) - 2C_0 N^{-1/12} \cdot|z-1|^2  \\
&\geq - a_G N^{-2/3} - b_G N^{-1/3} \cdot  |z-1| + \frac{q}{8(1-q)^2} \cdot |z-1|^2.
\end{split}
\end{equation}
For possibly bigger $a_G,b_G$ we also have from (\ref{EqTayG}) for $z \in \gamma^{\pm}_N(b,0)$
\begin{equation}\label{eq:GReal1}
|\Real G(z^{\pm 1})|  \leq | G(z)| \leq a_G N^{-2/3} + b_G N^{-1/3} \cdot  |z-1| + \frac{q}{(1-q)^2} \cdot |z-1|^2.
\end{equation}
Using (\ref{EqTayG}), (\ref{RealG}) and (\ref{eq:GReal0}), we also see that for all large $N$
\begin{equation}\label{eq:GReal2}
\begin{split}
&\Real G(z^{\pm 1}) \geq \Real G\left( (1 + b N^{-1/3} + N^{-1/12} e^{ (\pi/2 \pm \pi/6) \cdot \im })^{\pm 1} \right) \geq \frac{q \cdot N^{-1/6}  }{16(1-q)^2} \mbox{ for } z \in \gamma_N^{\pm}(b,1), \\
& \mbox{and } |G(z^{\pm 1})| = O(1) \mbox{ for } z \in \gamma_N^{\pm}(b).
\end{split}
\end{equation}

%
%
\subsubsection{Estimates for $z^x$ and $w^y$}\label{Section7.1.3} In the definition of $K^N$ in Lemma \ref{lem:PrelimitKernel} we have expressions of the form $\pm \sigma_q x N^{1/3} \log(z)$ and $\pm \sigma_q y N^{1/3} \log(w)$, which correspond to $z^x$ and $w^y$ in $\kgeo$ from Lemma \ref{lem:PSP as PPP}. Here, we obtain various estimates for these terms when $x,y \geq - A$. We mention that for the purposes of Lemma \ref{lem:kernelLimits} it suffices to obtain estimates when $x,y \in [-A,A]$; however, for the upper-tail estimates in Lemma \ref{lem:kernelUpperTail} we will actually need bounds when $x,y \geq - A$, and so we take care of these two cases together.

Let $b \in [-B,B]$. Using the Taylor expansion of the logarithm, and the fact that $\gamma^{\pm}_N(b)$ is $O(N^{-1/6})$ away from the unit circle, we have for all large $N$ and $x \in [-A,A]$
\begin{equation}\label{eq:Power1}
\begin{split}
&\left|  \sigma_q x N^{1/3} \log(z) \right| \leq 2 \sigma_q A N^{1/3} \cdot |z - 1| \mbox{ for } z \in \gamma^{\pm}_N(b, 0) \mbox{, and } \\
&\left| \sigma_q x N^{1/3} \log(z) \right| \leq N^{1/4} \mbox{ for } z \in \gamma^{\pm}_N(b, 1).
\end{split}
\end{equation}
We next note that for all large $N$ we have
\begin{equation*}
\Real \log (z) \geq  (1/3)N^{-1/6} \mbox{ for } z \in \gamma_N^{+}(b,1) \mbox{ and } \Real \log(w) \leq - (1/3) N^{-1/6} \mbox{ for } w \in \gamma_N^{-}(b,1).
\end{equation*}
Combining the latter with the second line in (\ref{eq:Power1}), we conclude for all large $N$ 
\begin{equation}\label{eq:Power2}
\begin{split}
&\left| e^{- \sigma_q x N^{1/3} \log (z)} \right| \leq \exp \left( 2N^{1/4} - (1/3) \sigma_q N^{1/6} \cdot |x| \right) \mbox{ if } z \in \gamma_N^{+}(b,1) \mbox{ and } x \geq -A,\\
&\left| e^{\sigma_q y N^{1/3} \log (w)} \right| \leq \exp \left( 2N^{1/4} - (1/3) \sigma_q N^{1/6} \cdot |y|  \right) \mbox{ if } w \in \gamma_N^{-}(b,1) \mbox{ and } y \geq -A,\\
\end{split}
\end{equation}
Finally, we observe that for all large $N$, the point $1 + b N^{-1/3}$ is furthest in $\gamma_N^-(b,0)$ from the origin, and the closest in $\gamma_N^+(b,0)$ to the origin. The latter implies for all large $N$, $z \in \gamma_N^+(b,0)$, $w \in \gamma_N^-(b,0)$ and $x,y \geq 0$ that
\begin{equation*}
\begin{split}
&\left| e^{- \sigma_q x N^{1/3} \log (z)} \right| = e^{- \sigma_q x N^{1/3} \log |z|}  \leq e^{-\sigma_q x N^{1/3} \log(1 + b N^{-1/3})} \leq e^{-(b \sigma_q -1) \cdot x},  \\
& \left| e^{ \sigma_q y N^{1/3} \log (w)} \right| =  e^{ \sigma_q y N^{1/3} \log |w|}  = e^{\sigma_q y N^{1/3} \log(1 + b N^{-1/3})} \leq e^{(b \sigma_q + 1) \cdot y}.
\end{split}
\end{equation*}
Combining the latter with the first line in (\ref{eq:Power1}), we conclude for all large $N$, $z \in \gamma_N^+(b,0)$, $w \in \gamma_N^-(b,0)$ and $x,y \geq -A$ that 
\begin{equation}\label{eq:Power3}
\begin{split}
&\left| e^{- \sigma_q x N^{1/3} \log (z)} \right| \leq \exp \left( 2 \sigma_q AN^{1/3} \cdot |z-1| -(b \sigma_q -1) \cdot x +  A |b \sigma_q -1| \right),\\
&\left| e^{\sigma_q y N^{1/3} \log (w)} \right| \leq \exp \left( 2 \sigma_q AN^{1/3} \cdot|w-1| +(b \sigma_q +1) \cdot y +  A |b \sigma_q + 1| \right).
\end{split}
\end{equation}

%
%
\subsection{Pointwise kernel convergence}\label{Section7.2} In this section we give the proof of Lemma \ref{lem:kernelLimits}.  For clarity we split the proof into four steps. In the first step we show that we can truncate the contours $\gamma^{\pm}_N(b)$ in the definitions of $I_{ij}^N$ and $R_{ij}^N$ to $\gamma^{\pm}_N(b,0)$ and estimate the effect of this truncation, which will be asymptotically negligible. The reason is that if $z \in \gamma_N^+(b,1)$ or $w \in \gamma_N^-(b,1)$, we have from (\ref{eq:SReal2}) 
$$\left| e^{NS(z)} \right| = e^{N \Real S(z)} \leq e^{-(\sigma_q^3/12) N^{3/4}} \mbox{ and } \left| e^{-NS(w)} \right| = e^{N \Real S(w)} \leq e^{-(\sigma_q^3/12) N^{3/4}},$$
while the other factors of the integrands are bounded by $\exp( \cb N^{2/3})$ for some large but fixed $\cb$. In the second step we prove (\ref{eq:I11Lim}), (\ref{eq:I12Lim}) and (\ref{eq:I22Lim}), and in the third we establish (\ref{eq:R12Lim}). Here, after we have truncated the contours we apply a change of variables that brings the contours $\gamma_{N}^{\pm}(b,0)$ to $\mathcal{C}_{\sigma_q^{-1}b}^{\pi/2 \mp \pi/6}$. Under this change of variables the integrands will have a clear pointwise limit, and we will be able to obtain estimates that allow us to apply the dominated convergence theorem and conclude the integral limits. In the fourth and final step we show (\ref{eq:R22Lim}). The first two terms in the formula for $R^N_{22}$ in (\ref{eq: DefRN22}) can be handled using similar arguments to those in Step 2. The third term in (\ref{eq: DefRN22}) needs to be handled specially when $s = t =0$, and can be taken care of like $R^N_{12}$ in Step 3 when $s + t > 0$. It is for this third term that we assumed $y_{\infty} \neq x_{\infty}$.

Throughout the proof we fix $A > 1$ such that $x_{\infty},y_{\infty} \in [- A + 1, A - 1]$, and assume that $N$ is large enough so that $x_N, y_N \in [-A,A]$. In addition, we write $\cb$ to mean a large positive generic constant that depends on $q, \varpi, A$, and $\mathcal{T} = \{t_1, \dots, t_m\}$. The values of these constants will change from line to line. The inequalities in the proof will hold provided that $N$ is sufficiently large depending on $ q, \varpi, A, \mathcal{T}$ and the sequence $c_N$ from Definition \ref{def:ParScale}, and we will not mention this dependence further. Finally, certain estimates will be made for $x,y \geq - A$ (as opposed to $x, y \in [-A, A]$) in order to be used later in the proof of Lemma \ref{lem:kernelUpperTail}.\\

{\bf \raggedleft Step 1. Truncation.} From the definition of $\gamma_N^{\pm}(a)$, we have for all large $N$
\begin{equation}\label{Q1}
\begin{split}
&\left| H_{11}^N(z,w) \right| \leq \cb \cdot N \mbox{ if } z,w \in \gamma_N^+(\ap_3^q), \\
& \left| H_{12}^N(z,w) \right| \leq \cb \cdot N^{2/3} \mbox{ if } z \in \gamma_N^+(\ap_3^q), w \in \gamma_N^{-}(\ap_1^q), \\
&\left| H_{22}^N(z,w) \right| \leq \cb \cdot N^{1/3} \mbox{ if } z,w \in \gamma_N^+(-\ap_2^q).
\end{split}
\end{equation}
Furthermore, we have for all large $N$ 
\begin{equation}\label{Q2}
\begin{split}
&\left| \frac{1}{4(c_N - z)}\right| \leq \cb \cdot N^{1/3} \mbox{ if } z \in \gamma_N^+(\ap_1^q) \cup \gamma_N^-(-\ap_2^q), \hspace{2mm}  \\
& \left| \frac{(1 - w^2)}{4 ( 1 - c_N w) (w - c_N)}\right| \leq \cb \cdot N^{2/3} \mbox{ if } w \in \gamma_N^+(\ap_3^q), w \in \gamma_N^{-}(-\ap_2^q).
\end{split}
\end{equation}

Combining (\ref{Q1}) with (\ref{eq:SReal0}), (\ref{eq:SReal1}), (\ref{eq:SReal2}), the second line in (\ref{eq:GReal2}), (\ref{eq:Power2}) and (\ref{eq:Power3}) we conclude that for $x,y \geq -A$, $s, t \in \mathcal{T}$, and all large $N$ we have
\begin{equation}\label{eq:I11Trunc}
\begin{split}
&\left| I^N_{11}(s,x; t, y) -  \frac{1}{(2\pi \im)^{2}}\int_{\gamma_N^+(\ap^q_3, 0)} dz \int_{\gamma_N^+(\ap^q_{3},0)} dw F_{11}^N(z,w) H_{11}^N(z,w) \right| \\
&\leq e^{\cb N^{2/3} - (\sigma_q^3/12) N^{3/4} } \cdot e^{-(\ap_3-1) x - (\ap_3-1) y },
\end{split}
\end{equation}
\begin{equation}\label{eq:I12Trunc}
\begin{split}
&\left| I^N_{12}(s,x; t, y) -  \frac{1}{(2\pi \im)^{2}}\int_{\gamma_N^+(\ap^q_3, 0)} dz \int_{\gamma_N^-(\ap^q_{1},0)} dw F_{12}^N(z,w) H_{12}^N(z,w) \right| \\
&\leq e^{\cb N^{2/3} - (\sigma_q^3/12) N^{3/4} } \cdot e^{-(\ap_3-1) x + (\ap_1+1) y },
\end{split}
\end{equation}
\begin{equation}\label{eq:I22Trunc}
\begin{split}
&\left| I^N_{22}(s,x; t, y) -  \frac{1}{(2\pi \im)^{2}}\int_{\gamma_N^+(-\ap^q_2, 0)} dz \int_{\gamma_N^-(-\ap^q_{2},0)} dw F_{22}^N(z,w) H_{22}^N(z,w) \right| \\
&\leq e^{\cb N^{2/3} - (\sigma_q^3/12) N^{3/4} } \cdot e^{(-\ap_2 + 1) x + (-\ap_2 + 1) y }.
\end{split}
\end{equation}

We next have from the first line in (\ref{eq:GReal2}) and (\ref{eq:Power1}) that for $x, y\in [-A, A]$, $s, t \in \mathcal{T}$
\begin{equation}\label{eq:R12Trunc}
\begin{split}
&\left| R^N_{12}(s,x; t, y) - \frac{-{\bf 1}\{s < t \} \cdot \sigma_q N^{1/3} }{2 \pi \im} \int_{\gamma_{N}^+(\ap^q_{1},0)}dz e^{(T_s - T_t) G(z)} \cdot e^{ (\sigma_q y N^{1/3}  - \sigma_q x N^{1/3} - 1) \log (z)} \right| \\
&\leq \cb \cdot  \exp \left(2N^{1/4}  - \frac{q \cdot N^{1/2} \cdot (t -s ) }{16(1-q)^2}  \right).
\end{split}
\end{equation}

Using $c_N = 1 - \varpi \sigma_q^{-1} N^{-1/3} + o(N^{-1/3})$, with (\ref{EqTayS}) and (\ref{EqTayG}), and the second line in (\ref{eq:Power3}) with $w = c_N$ and $b = N^{1/3}(c_N-1)$, we get for $x,y \geq -A$ and all large $N$
\begin{equation}\label{Q3}
\begin{split}
&\left| e^{-N S(c_N) - T_s G(c_N) + \sigma_q x N^{1/3} \log (c_N)} \right| \leq \cb \cdot e^{(N^{1/3}(c_N-1)\sigma_q + 1) \cdot x}, \\
&\left| e^{-N S(c_N) - T_t G(c_N) + \sigma_q y N^{1/3} \log (c_N)} \right| \leq \cb \cdot e^{(N^{1/3}(c_N-1)\sigma_q + 1) \cdot y}.
\end{split}
\end{equation}
Combining (\ref{eq:SReal2}), the first line in (\ref{eq:GReal2}), the second line in (\ref{eq:Power2}), the first line in (\ref{Q2}) and (\ref{Q3}), we obtain for $x,y \geq - A$
\begin{equation}\label{eq:R22Trunc1}
\begin{split}
&\left|  \frac{1}{2\pi \im} \int_{\gamma^-_N(\ap^q_{1},1)} dz \frac{F_{22}^N(z,c_N)}{4(c_N z - 1)} \right| \leq \cb \cdot e^{2N^{1/4} - (\sigma_q^3/12)N^{3/4} } \cdot e^{(N^{1/3}(c_N-1)\sigma_q + 1) \cdot y} \\
&\left|\frac{1}{2\pi \im} \int_{\gamma^-_N(-\ap^q_{2},1)} dw \frac{F_{22}^N(c_N,w)}{4(c_N  w - 1)} \right| \leq \cb \cdot e^{2N^{1/4} - (\sigma_q^3/12)N^{3/4} } \cdot e^{(N^{1/3}(c_N-1)\sigma_q + 1) \cdot x}.
\end{split}
\end{equation}
Finally, if $s + t > 0$, we have from the first line in (\ref{eq:GReal2}), the second line in (\ref{eq:Power1}) and the second line in (\ref{Q2}) that for $x,y \in [-A,A]$ 
\begin{equation}\label{eq:R22Trunc2}
\begin{split}
&\left| \frac{1}{2\pi \im}\int_{\gamma^-_N(-\ap^q_{2},1)} dw \cdot \frac{(1-w^2)}{4(1-c_N w)(w-c_N)} \cdot e^{ ( \sigma_q y N^{1/3} - \sigma_q x N^{1/3} - 1) \log (w) - T_s G(w^{-1}) - T_t G(w)} \right|    \\
& \leq \exp \left( \cb \cdot N^{1/4}  - \frac{q \cdot N^{1/2} \cdot (t + s ) }{16(1-q)^2}  \right).
\end{split}
\end{equation}

{\bf \raggedleft Step 2. The limits of $I^N_{11}, I^N_{12}$ and $I^N_{22}$.} We proceed to change variables $z = 1 + \sigma_q^{-1} N^{-1/3} \tilde{z}$ and $w = 1 + \sigma_q^{-1} N^{-1/3} \tilde{w}$, and note that by the definition of $\gamma_N^+(\ap_3^q)$, we have that 
\begin{equation}\label{COV1}
\begin{split}
&\frac{1}{(2\pi \im)^{2}}\int_{\gamma_N^+(\ap^q_3, 0)} dz \int_{\gamma_N^+(\ap^q_{3},0)} dw F_{11}^N(z,w) H_{11}^N(z,w) =  \frac{1}{(2\pi \im)^{2}}\int_{\mathcal{C}_{\ap_3}^{\pi/3}} d\tilde{z} \int_{\mathcal{C}_{\ap_3}^{\pi/3}} d\tilde{w} \\
& {\bf 1}\{|\Imag (\tilde{z})|, |\Imag (\tilde{w})| \leq (\sqrt{3}/2) \sigma_q^{-1} N^{1/4} \} \cdot \sigma_q^{-2} N^{-2/3} F_{11}^N(z,w) H_{11}^N(z,w).
\end{split}
\end{equation}
Using the definitions of $F^N_{11}$ and $H^N_{11}$ from (\ref{eq:DefIN11}) with $x = x_N$, $y = y_N$, the fact that $c_N = 1 - \varpi \sigma_q^{-1} N^{-1/3} + o(N^{-1/3})$ and the Taylor expansions of $S,G$ from (\ref{TaylorS}) and (\ref{TaylorG}), we see that
\begin{equation}\label{COV2}
\begin{split}
&\lim_N F^N_{11}(z ,w) = e^{\tilde{z}^3/3 + \tilde{w}^3/3 - \frac{q \sigma_q^{-2} s}{2(1-q)^2} \cdot \tilde{z}^2 - \frac{q \sigma_q^{-2} t}{2(1-q)^2} \cdot \tilde{w}^2 - x_{\infty} \tilde{z} - y_{\infty} \tilde{w} }, \\
&\lim_N \sigma_q^{-2} N^{-2/3} H^N_{11}(1 + \sigma_q^{-1} N^{-1/3} \tilde{z} ,1 + \sigma_q^{-1} N^{-1/3} \tilde{w}) = \frac{(\tilde{z} - \tilde{w})(\tilde{z} + \varpi) (\tilde{w}+ \varpi)}{\tilde{z} \tilde{w} (\tilde{z} + \tilde{w})}.
\end{split}
\end{equation}
In what follows we find bounds for the two functions in (\ref{COV2}) that would allow us to apply the dominated convergence theorem.

From (\ref{eq:SReal0}), (\ref{eq:GReal1}) and the first line in (\ref{eq:Power3}) we have for $x, y \geq - A$, $\tilde{z}, \tilde{w} \in \mathcal{C}_{\ap_3}^{\pi/3}$ with $|\Imag (\tilde{z})|, |\Imag (\tilde{w})| \leq (\sqrt{3}/2) \sigma_q^{-1} N^{1/4}$ and all large $N$
\begin{equation}\label{eq:DomI11}
\begin{split}
&\left|F^N_{11}(z ,w) \cdot \sigma_q^{-2} N^{-2/3} H^N_{11}(z,w)\right| \\
& \leq \exp \left( \cb \cdot (1 + |\tilde{z}|^2 + |\tilde{w}|^2) - (1/6)|\tilde{w}|^3 - (1/6) |\tilde{z}|^3 \right) \cdot \exp \left( - (\ap_3-1) x - (\ap_3 -1) y \right).
\end{split}
\end{equation}
Using (\ref{eq:I11Trunc}), (\ref{COV1}) and (\ref{COV2}), and the dominated convergence theorem with dominating function as in the second line of (\ref{eq:DomI11}) with ``$\exp \left( - (\ap_3-1) x - (\ap_3 -1) y \right)$'' dropped (it gets absorbed into the $\cb$ when $x = x_N, y = y_N$ are in $[-A,A]$), we conclude (\ref{eq:I11Lim}). We mention that to identify the limit in (\ref{COV2}) with the one in (\ref{eq:I11Lim}) one needs to remove the tildes, and use $f_q = \frac{q \sigma_q^{-2} }{2(1-q)^2}$.\\

The analysis of $I_{12}^N$ and $I_{22}^N$ is quite similar so we will be brief. We apply the same change of variables $z = 1 + \sigma_q^{-1} N^{-1/3} \tilde{z}$ and $w = 1 + \sigma_q^{-1} N^{-1/3} \tilde{w}$ and note
\begin{equation}\label{COV3}
\begin{split}
&\lim_N \sigma_q^{-2} N^{-2/3} F_{12}^N(z,w) H_{12}^N(z,w) =  \frac{e^{\tilde{z}^3/3 - \tilde{w}^3/3 - f_q s \tilde{z}^2 + f_q t \tilde{w}^2 - x_{\infty} \tilde{z} + y_{\infty} \tilde{w}  }  \cdot (\tilde{z}+\tilde{w})(\tilde{z}+\varpi)}{2\tilde{z}(\tilde{z}-\tilde{w})(\tilde{w} + \varpi)} \\
& \lim_N \sigma_q^{-2} N^{-2/3} F_{22}^N(z,w) H_{22}^N(z,w) =   \frac{e^{-\tilde{z}^3/3 - \tilde{w}^3/3 - f_q s \tilde{z}^2 - f_q t \tilde{w}^2 + x_{\infty} \tilde{z} + y_{\infty} \tilde{w}  } \cdot (\tilde{z}-\tilde{w})}{4(\tilde{z}+\tilde{w})(\tilde{z}+ \varpi)(\tilde{w}+\varpi)},
\end{split}
\end{equation}
From (\ref{eq:SReal1}), (\ref{eq:SReal2}), (\ref{eq:GReal1}) and (\ref{eq:Power3}) we also have the following bounds when $x, y \geq - A$, $|\Imag (\tilde{z})|, |\Imag (\tilde{w})| \leq (\sqrt{3}/2) \sigma_q^{-1} N^{1/4}$ and all large $N$.
\begin{equation}\label{eq:DomI12}
\begin{split}
&\mbox{ For $\tilde{z} \in \mathcal{C}_{\ap_3}^{\pi/3}$, $\tilde{w} \in \mathcal{C}_{\ap_1}^{2\pi/3}$ we have } \left|F^N_{12}(z ,w) \cdot \sigma_q^{-2} N^{-2/3} H^N_{12}(z ,w)\right| \\
& \leq \exp \left( \cb \cdot (1 + |\tilde{z}|^2 + |\tilde{w}|^2) - (1/6)|\tilde{w}|^3 - (1/6) |\tilde{z}|^3 \right) \cdot \exp \left( - (\ap_3-1) x + (\ap_1 + 1) y \right).
\end{split}
\end{equation}
\begin{equation}\label{eq:DomI22}
\begin{split}
&\mbox{ For $\tilde{z} \in \mathcal{C}_{-\ap_2}^{2\pi/3}$, $\tilde{w} \in \mathcal{C}_{-\ap_2}^{2\pi/3}$ we have } \left|F^N_{22}(z ,w) \cdot \sigma_q^{-2} N^{-2/3} H^N_{22}(z ,w)\right| \\
& \leq \exp \left( \cb \cdot (1 + |\tilde{z}|^2 + |\tilde{w}|^2) - (1/6)|\tilde{w}|^3 - (1/6) |\tilde{z}|^3 \right) \cdot \exp \left(  (-\ap_2 + 1) x + (-\ap_2 + 1) y \right).
\end{split}
\end{equation}
From (\ref{eq:I12Trunc}), the first line in (\ref{COV3}), and the dominated convergence theorem with dominating function as in (\ref{eq:DomI12}) with ``$\exp \left( - (\ap_3-1) x + (\ap_1 + 1) y \right)$'' dropped we obtain (\ref{eq:I12Lim}). Also, from (\ref{eq:I22Trunc}), the second line in (\ref{COV3}), and the dominated convergence theorem with dominating function as in (\ref{eq:DomI22}) with ``$\exp \left( (-\ap_2 + 1) x + (-\ap_2 + 1) y \right)$'' dropped we obtain (\ref{eq:I22Lim}). \\

{\bf \raggedleft Step 3. The limit of $R^N_{12}$.} Equation (\ref{eq:R12Lim}) is trivial when $s \geq t$ and so we assume $s < t$ in the remainder of this step. We again change variables $z = 1 + \sigma_q^{-1} N^{-1/3} \tilde{z}$, and note that by the definition of $\gamma_N^+(\ap_1^q)$, we have that 
\begin{equation}\label{COV4}
\begin{split}
& \frac{-  \sigma_q N^{1/3} }{2 \pi \im} \int_{\gamma_{N}^+(\ap^q_{1},0)}dz e^{(T_s - T_t) G(z)} \cdot e^{ (\sigma_q y_N N^{1/3}  - \sigma_q x_N N^{1/3} - 1) \log (z)}\\
&=  \frac{-1}{2\pi \im}\int_{\mathcal{C}_{\ap_1}^{\pi/3}} d\tilde{z} {\bf 1}\{|\Imag (\tilde{z})| \leq (\sqrt{3}/2) \sigma_q^{-1} N^{1/4} \} \cdot e^{(T_s - T_t) G(z)} \cdot e^{ (\sigma_q y_N N^{1/3}  - \sigma_q x_N N^{1/3} - 1) \log (z)}.
\end{split}
\end{equation}
Using the Taylor expansion of $G$ from (\ref{TaylorG}) and $f_q = \frac{q \sigma_q^{-2} }{2(1-q)^2}$ we get
\begin{equation}\label{COV5}
\begin{split}
\lim_N e^{(T_s - T_t) G(z)} \cdot e^{ (\sigma_q y_N N^{1/3}  - \sigma_q x_N N^{1/3} - 1) \log (z)} = e^{f_q(t-s) z^2 + z (y_{\infty} - x_{\infty})}.
\end{split}
\end{equation}
In addition, from (\ref{eq:GReal0}) and the first line in (\ref{eq:Power1}) we have 
\begin{equation}\label{eq:DomR12}
\begin{split}
&\left|e^{(T_s - T_t) G(z)} \cdot e^{ (\sigma_q y_N N^{1/3}  - \sigma_q x_N N^{1/3} - 1) \log (z)}  \right| \leq \exp \left( \cb \cdot (|\tilde{z}| + 1) - \frac{q(t-s)}{8(1-q)^2} |\tilde{z}|^2 \right).
\end{split}
\end{equation}
Combining (\ref{eq:R12Trunc}), (\ref{COV4}), (\ref{COV5}) and the dominated convergence theorem with dominating function as in (\ref{eq:DomR12}) we conclude 
\begin{equation}\label{eq:R12Lim2}
\begin{split}
&\lim_{N} R_{12}(s,x_N; t, y_N) = \frac{-{\bf 1}\{s < t \} }{2\pi \im}\int_{\mathcal{C}^{\pi/3}_{\ap_{1}}} d\tilde{z} e^{f_q(t-s)\tilde{z}^2 + (y_{\infty} - x_{\infty}) \tilde{z}}.
\end{split}
\end{equation}
At this point we can deform $\mathcal{C}^{\pi/3}_{\ap_{1}}$ to $\im \mathbb{R}$, change variables $\tilde{z} = iz$, at which point we recognize the right side of (\ref{eq:R12Lim2}) as the characteristic function of a Gaussian random variable, which precisely evaluates to the right side of (\ref{eq:R12Lim}).\\

{\bf \raggedleft Step 4. The limit of $R_{22}^N$.} In this final step we establish (\ref{eq:R22Lim}). Since both $R^N_{22}$ and $R^{\infty}_{22}$ are skew-symmetric it suffices to prove the statement when $y_{\infty} > x_{\infty}$, which we assume in the sequel. We analyze the limits of the three integrals in (\ref{eq: DefRN22}) one at a time. As done previously, we change variables $z = 1 + \sigma_q^{-1} N^{-1/3} \tilde{z}$, $w = 1 + \sigma_q^{-1} N^{-1/3} \tilde{w}$ and observe 
\begin{equation}\label{COV6}
\begin{split}
& \int_{\gamma^-_N(\ap^q_{1},0)} dz \frac{F_{22}^N(z,c_N) }{4(c_N z - 1)}  = \frac{1}{2\pi \im} \int_{\mathcal{C}^{2\pi/3}_{\ap^q_{1}}} d\tilde{z}\frac{F_{22}^N(z,c_N) {\bf 1}\{|\Imag (\tilde{z})| \leq (\sqrt{3}/2) \sigma_q^{-1} N^{1/4} \}}{4 \sigma_q N^{1/3} (c_N z - 1)}, \\
&\int_{\gamma^-_N(-\ap^q_{2},0)} dw \frac{F_{22}^N(c_N,w)  }{4(c_N  w - 1)} =  \frac{1}{2\pi \im} \int_{\mathcal{C}^{2\pi/3}_{-\ap^q_{2}}} d \tilde{w} \frac{F_{22}^N(c_N,w) {\bf 1}\{|\Imag (\tilde{w})| \leq (\sqrt{3}/2) \sigma_q^{-1} N^{1/4} \} }{4\sigma_q N^{1/3} (c_N  w - 1)}  \\
&\int_{\gamma^-_N(-\ap^q_{2},0)} dw \cdot \frac{(1-w^2)e^{ ( \sigma_q y N^{1/3} - \sigma_q x N^{1/3} - 1) \log (w) - T_s G(w^{-1}) - T_t G(w)}}{4(1-c_N w)(w-c_N)} = \int_{\mathcal{C}^{2\pi/3}_{-\ap^q_{2}}}d\tilde{w}   \\
& \frac{{\bf 1}\{|\Imag (\tilde{w})| \leq (\sqrt{3}/2) \sigma_q^{-1} N^{1/4} \} (1-w^2) e^{ ( \sigma_q y N^{1/3} - \sigma_q x N^{1/3} - 1) \log (w) - T_s G(w^{-1}) - T_t G(w)} }{4\sigma_q N^{1/3} (1-c_N w)(w-c_N)}.
\end{split}
\end{equation}
Using the definition of $F^N_{22}$ from (\ref{eq:DefIN22}) with $x = x_N$, $y = y_N$, the fact that $c_N = 1 - \varpi \sigma_q^{-1} N^{-1/3} + o(N^{-1/3})$ and the Taylor expansions of $S,G$ from (\ref{TaylorS}) and (\ref{TaylorG}), we see that
\begin{equation}\label{COV7}
\begin{split}
&\lim_N  \frac{F_{22}^N(z,c_N)}{4 \sigma_q N^{1/3} (c_N z - 1)}  =  \frac{e^{-\tilde{z}^3/3 + \varpi^3/3 - f_q s \tilde{z}^2 - f_q t \varpi^2 + x_{\infty}\tilde{z} - y_{\infty}\varpi  }}{4(\tilde{z} - \varpi)} \\
&\lim_N  \frac{F_{22}^N(c_N,w)}{4 \sigma_q N^{1/3} (c_N w - 1)}  =   \frac{e^{\varpi^3/3 - \tilde{w}^3/3 - f_q s \varpi^2 - f_q t \tilde{w}^2 - x_{\infty}\varpi + y_{\infty}\tilde{w}  }}{4(\tilde{w} - \varpi)}\\
& \lim_N \frac{(1-w^2)e^{ ( \sigma_q y_N N^{1/3} - \sigma_q x_N N^{1/3} - 1) \log (w) - T_s G(w^{-1}) - T_t G(w)}}{4\sigma_q N^{1/3} (1-c_N w)(w-c_N)} = \frac{\tilde{w} e^{ \tilde{w}(y_{\infty} - x_{\infty}) - f_q (s + t)\tilde{w}^2 } }{2(\tilde{w} - \varpi)(\tilde{w} + \varpi)}.
\end{split}
\end{equation}
From (\ref{eq:SReal1}), (\ref{eq:GReal1}), the second line in (\ref{eq:Power3}) and (\ref{Q3}) we have for $x, y \geq - A$, $\tilde{z} \in \mathcal{C}^{2\pi/3}_{\ap_{1}}$, $\tilde{w} \in \mathcal{C}^{2\pi/3}_{-\ap_{2}}$ with $|\Imag (\tilde{z})|, |\Imag (\tilde{w})| \leq (\sqrt{3}/2) \sigma_q^{-1} N^{1/4}$, and all large $N$
\begin{equation}\label{eq:DomF22}
\begin{split}
&\left| \frac{F_{22}^N(z,c_N)}{4 \sigma_q N^{1/3} (c_N z - 1)} \right| \leq \exp \left( \cb \cdot (1 + |\tilde{z}|^2 ) - (1/6) |\tilde{z}|^3 \right) \cdot e^{ (\ap_1 + 1) x + (N^{1/3}(c_N-1)\sigma_q + 1) y },\\
&\left| \frac{F_{22}^N(c_N,w)}{4 \sigma_q N^{1/3} (c_N w - 1)} \right| \leq \exp \left( \cb \cdot (1 + |\tilde{w}|^2 ) - (1/6) |\tilde{w}|^3 \right) \cdot e^{ (N^{1/3}(c_N-1)\sigma_q + 1) x + (-\ap_2 + 1) y }.
\end{split}
\end{equation}

From (\ref{eq:R22Trunc1}), the first two lines in (\ref{COV7}), and the dominated convergence theorem with dominating functions as in (\ref{eq:DomF22}) with ``$\exp \left( (\ap_1 + 1) x + (N^{1/3}(c_N-1)\sigma_q + 1) y \right)$'' dropped from the first line and ``$\exp \left( (N^{1/3}(c_N-1)\sigma_q + 1) x + (-\ap_2 + 1) y \right)$'' dropped from the second, we obtain 
\begin{equation}\label{eq:F22Lim}
\begin{split}
&\lim_N \frac{1}{2\pi \im} \oint_{\gamma^-_N(\ap^q_{1})} dz \frac{F_{22}^N(z,c_N)}{4(c_N z - 1)} = \frac{1}{2\pi \im} \int_{\mathcal{C}^{2\pi/3}_{\ap_{1}}} d\tilde{z} \frac{e^{-\tilde{z}^3/3 + \varpi^3/3 - f_q s \tilde{z}^2 - f_q t \varpi^2 + x_{\infty}\tilde{z} - y_{\infty} \varpi  }}{4(\tilde{z} - \varpi)},\\
&\lim_N \frac{1}{2\pi \im} \oint_{\gamma^-_N(-\ap^q_{2})} dw \frac{F_{22}^N(c_N,w)}{4(c_N  w - 1)} = \frac{1}{2\pi \im} \int_{\mathcal{C}^{2\pi/3}_{-\ap_{2}}} d\tilde{w}\frac{e^{\varpi^3/3 - \tilde{w}^3/3 - f_q s \varpi^2 - f_q t \tilde{w}^2 - x_{\infty}\varpi + y_{\infty} \tilde{w}  }}{4(\tilde{w} - \varpi)}.
\end{split}
\end{equation}

If $s + t > 0$, we have from (\ref{eq:GReal0}) and the first line in (\ref{eq:Power1}) that for $x,y \in [- A, A]$, $\tilde{w} \in \mathcal{C}^{2\pi/3}_{-\ap_{2}}$ with $|\Imag (\tilde{w})| \leq (\sqrt{3}/2) \sigma_q^{-1} N^{1/4}$
\begin{equation}\label{eq:DomS1}
\begin{split}
&\left|  \frac{(1-w^2)e^{ ( \sigma_q y N^{1/3} - \sigma_q x N^{1/3} - 1) \log (w) - T_s G(w^{-1}) - T_t G(w)}}{4\sigma_q N^{1/3} (1-c_N w)(w-c_N)}  \right| \leq \exp \left( \cb (1 + |\tilde{w}| ) - \frac{q (t+s)|\tilde{w}|^2}{8 (1 -q)^2}   \right).
\end{split}
\end{equation}
From (\ref{eq:R22Trunc2}), the third line in (\ref{COV7}), and the dominated convergence theorem with dominating functions as in (\ref{eq:DomS1}) we obtain when $s + t > 0$
\begin{equation}\label{eq:SLim}
\begin{split}
&\lim_N  \frac{1}{2\pi \im}\oint_{\gamma^-_N(-\ap^q_{2})} dw  \frac{(1-w^2)}{4(1-c_N w)(w-c_N)} \cdot e^{ ( \sigma_q y_N N^{1/3} - \sigma_q x_N N^{1/3} - 1) \log (w) - T_s G(w^{-1}) - T_t G(w)} \\
& = \frac{{\bf 1}\{s + t > 0\} }{2\pi \im} \int_{\mathcal{C}^{2\pi/3}_{-\ap_{2}}} dw e^{ - f_q (s + t) w^2  + w (y_{\infty}-x_{\infty})  } \cdot \frac{w}{2(w- \varpi)(w + \varpi) }.
\end{split}
\end{equation}
If we can show that (\ref{eq:SLim}) holds when $s + t = 0$ (i.e. $s = t = 0$), then the definition of $R^N_{22}$ from (\ref{eq: DefRN22}),  (\ref{eq:F22Lim}) and (\ref{eq:SLim}) would imply (\ref{eq:R22Lim}) and conclude the proof of the lemma. \\

In the remainder of this step we establish (\ref{eq:SLim}) when $s = t = 0$, and it is here that we will finally use that $y_{\infty} > x_{\infty}$. If we set 
\begin{equation}\label{FG1}
\tilde{x}_N = \frac{2qN}{1-q} + \sigma_q N^{1/3} x_N \mbox{ and } \tilde{y}_N = \frac{2qN}{1-q} + \sigma_q N^{1/3} y_N,
\end{equation}
we obtain
\begin{equation}\label{FG2}
\begin{split}
&\frac{1}{ 2\pi \im}\oint_{\gamma^-_N(-\ap^q_{2})} dw  \frac{(1-w^2)e^{ ( \sigma_q y_N N^{1/3} - \sigma_q x_N N^{1/3} - 1) \log (w) - T_s G(w^{-1}) - T_t G(w)}}{4(1-c_N w)(w-c_N)} \\
& = \frac{1}{ 2\pi \im}\oint_{\gamma^-_N(-\ap^q_{2})} dw \frac{(1-w^2) w^{\tilde{y}_N - \tilde{x}_N - 1}}{4(1-c_N w)(w-c_N)}.
\end{split}
\end{equation}
We note that if $y_{\infty} > x_{\infty}$, then for all large $N$ we have $\tilde{y}_N - \tilde{x}_N - 1 \geq 0 $ and the above integral is equal to zero by Cauchy's theorem. This establishes (\ref{eq:SLim}) when $s + t = 0$.

%
%
\subsection{Upper tail estimates}\label{Section7.3} In this section we give the proof of Lemma \ref{lem:kernelUpperTail}, where we recall that $a = \ap_3-1$ and $b = \ap_1 + 1$. We have already done most of the work in the proof of Lemma \ref{lem:kernelLimits} and we just need to collect some of the estimates there. We adopt the convention from the proof of Lemma \ref{lem:kernelLimits} that $\cb$ stands for a large positive generic constant that depends on $q, \varpi, A, t$, whose value changes from line to line. In addition, all statements are for $N$ sufficiently large depending on $q, \varpi, A, t$ but also the sequence $c_N$. We do not mention this further. Below we write $|d\tilde{z}|, |d\tilde{w}|$ for integration with respect to arc-length.\\

Starting with $K^N_{11}$, we have from (\ref{eq:I11Trunc}) and (\ref{eq:DomI11}) that for all large $N$ and $x_N, y_N \geq -A$
\begin{equation*}
\begin{split}
&\left|K_{11}^N(t, x_N; t, y_N)\right| \leq  \exp \left( - (\ap_3-1) x_N - (\ap_3 -1) y_N \right) \\
& \times \left( 1 + \int_{\mathcal{C}^{\pi/3}_{\ap_3}} |d\tilde{z}| \int_{\mathcal{C}^{\pi/3}_{\ap_3}} |d\tilde{w}| \exp \left( \cb \cdot (1 + |\tilde{z}|^2 + |\tilde{w}|^2) - (1/6)|\tilde{w}|^3 - (1/6) |\tilde{z}|^3 \right) \right).
\end{split}
\end{equation*}
As the above integral is finite we obtain the first inequality in (\ref{eq:kernelUpperTail}). We now consider $K^N_{12}$ and note $K^N_{12}(t,x;t,y) = I_{12}^N(t,x; t,y)$. From (\ref{eq:I12Trunc}) and (\ref{eq:DomI12}) we have for all large $N$ and $x_N, y_N \geq -A$
\begin{equation*}
\begin{split}
&\left|K_{12}^N(t, x_N; t, y_N)\right| \leq \exp \left( - (\ap_3-1) x_N + (\ap_1 +1) y_N \right) \\
& \times \left( 1 + \int_{\mathcal{C}^{\pi/3}_{\ap_3}} |d\tilde{z}| \int_{\mathcal{C}^{2\pi/3}_{\ap_1}} |d\tilde{w}| \exp \left( \cb \cdot (1 + |\tilde{z}|^2 + |\tilde{w}|^2) - (1/6)|\tilde{w}|^3 - (1/6) |\tilde{z}|^3 \right) \right),
\end{split}
\end{equation*}
which gives the second inequality in (\ref{eq:kernelUpperTail}). \\

In the remainder we focus on $K^N_{22}$ and show that we can find a large enough $D$ such that 
\begin{equation}\label{UT1}
|I^N_{22}(t,x_N;t,y_N)| \leq D e^{(\ap_1 + 1)(x_N + y_N)} \mbox{ and } |R^N_{22}(t,x_N;t,y_N)| \leq D e^{(\ap_1 + 1)(x_N + y_N)},
\end{equation}
which if true would prove the third inequality in (\ref{eq:kernelUpperTail}) and hence the lemma. Since the kernels $I_{22}^N, R_{22}^N$ are skew-symmetric, it suffices to show (\ref{UT1}) under the assumption that $y_N \geq x_N$ (we only need this for the integral in the second line of (\ref{eq: DefRN22}) where  $R_{22}^N$ is defined).

From (\ref{eq:I22Trunc}) and (\ref{eq:DomI22}) we have for all large $N$ and $x_N, y_N \geq -A$
\begin{equation*}
\begin{split}
&\left|I_{22}^N(t, x_N; t, y_N)\right| \leq \exp \left( (-\ap_2 + 1) x_N + (-\ap_2 + 1) y_N \right) \\
& \times \left( 1 + \int_{\mathcal{C}^{2\pi/3}_{-\ap_2}} |d\tilde{z}| \int_{\mathcal{C}^{2\pi/3}_{-\ap_2}} |d\tilde{w}| \exp \left( \cb \cdot (1 + |\tilde{z}|^2 + |\tilde{w}|^2) - (1/6)|\tilde{w}|^3 - (1/6) |\tilde{z}|^3 \right) \right),
\end{split}
\end{equation*}
which implies the first inequality in (\ref{UT1}). 

We also have from (\ref{eq:R22Trunc1}) and (\ref{eq:DomF22}) that 
\begin{equation}\label{UT2}
\begin{split}
&\left|\frac{1}{2\pi \im} \oint_{\gamma^-_N(\ap^q_{1})} dz \frac{F_{22}^N(z,c_N)}{4(c_N z - 1)} \right| \leq e^{(N^{1/3}(c_N-1)\sigma_q + 1) \cdot y_N} + e^{(\ap_1 + 1) x_N + (N^{1/3}(c_N-1)\sigma_q + 1) \cdot y_N} \\
&\times  \int_{\mathcal{C}^{2\pi/3}_{\ap_1}} |d\tilde{z}| \exp \left( \cb \cdot (1 + |\tilde{z}|^2) - (1/6) |\tilde{z}|^3 \right),
\end{split}
\end{equation}
\begin{equation}\label{UT3}
\begin{split}
&\left|\frac{1}{2\pi \im} \oint_{\gamma^-_N(-\ap^q_{2})} dw \frac{F_{22}^N(c_N,w)}{4(c_N w - 1)} \right| \leq e^{(N^{1/3}(c_N-1)\sigma_q + 1) \cdot x_N} + e^{(N^{1/3}(c_N-1)\sigma_q + 1) \cdot x_N + (-\ap_2 + 1) y_N } \\
&\times  \int_{\mathcal{C}^{2\pi/3}_{-\ap_2}} |d\tilde{w}| \exp \left( \cb \cdot (1 + |\tilde{w}|^2) - (1/6) |\tilde{w}|^3 \right).
\end{split}
\end{equation}

Recall from point (3) below (\ref{eq:PivotIneq}) that $\gamma_N^-(-\ap_2^q)$ is contained in the unit circle $C_1$, and since $y_N \geq x_N$, we have for $w \in \gamma_N^-(-\ap_2^q)$
\begin{equation}\label{UT4}
\begin{split}
&\left|e^{ ( \sigma_q y_N N^{1/3} - \sigma_q x_N N^{1/3}) \log (w)} \right| \leq 1.
\end{split}
\end{equation}
If $s + t > 0$, we conclude from (\ref{eq:GReal0}), (\ref{eq:GReal1}), the third line in (\ref{COV6}) and (\ref{UT4}) that
\begin{equation}\label{UT5}
\begin{split}
& \left|  \frac{1}{2\pi \im}\oint_{\gamma^-_N(-\ap^q_{2})} dw  \frac{(1-w^2) e^{ ( \sigma_q y_N N^{1/3} - \sigma_q x_N N^{1/3} - 1) \log (w) - T_s G(w^{-1}) - T_t G(w)} }{4(1-c_N w)(w-c_N)}   \right| \\
& \leq \cb \cdot e^{ - \frac{q N^{1/2} }{16(1-q)^2} } +  \int_{\mathcal{C}^{2\pi/3}_{-\ap_2}} |d\tilde{w}| \exp \left( \cb \cdot (1 + |\tilde{w}|)  - \frac{q (t + s) |\tilde{w}|^2 }{8(1-q)^2}   \right).
\end{split}
\end{equation}
If $s + t = 0$, then from (\ref{FG2}) we conclude that the left side of (\ref{UT5}) equals zero if $y_N > x_N$ and if $y_N = x_N$ it is equal to the absolute value of the residue of (\ref{FG2}) at $w = 0$, namely $(1/4) c_N^{-2}$. Combining the latter with (\ref{UT5}) we conclude 
\begin{equation}\label{UT6}
\begin{split}
& \left|  \frac{1}{2\pi \im}\oint_{\gamma^-_N(-\ap^q_{2})} dw \frac{(1-w^2) e^{ ( \sigma_q y_N N^{1/3} - \sigma_q x_N N^{1/3} - 1) \log (w) - T_s G(w^{-1}) - T_t G(w)}}{4(1-c_N w)(w-c_N)}  \right|\leq \cb.
\end{split}
\end{equation}
We now observe that the right sides of (\ref{UT2}), (\ref{UT3}) and (\ref{UT6}) are all bounded by $D e^{(\ap_1 + 1)(x_N + y_N)}$, uniformly for all large $N$ and $y_N \geq x_N \geq - A$, provided we pick $D$ large enough depending on $q , \varpi, t, A$. This proves the second inequality in (\ref{UT1}).

%
%
\subsection{Large time limits of $K^{\infty}$}\label{Section7.4} In this section we give the proof of Lemma \ref{lem:ConvToAiryKernel}. For clarity we split the proof into three steps.\\

{\bf \raggedleft Step 1.} In this step we show (\ref{eq:Limkcr11}). From Definition \ref{def:LimKernel}, see (\ref{eq:I11Lim}), we have
\begin{equation}\label{eq:AB1}
    \begin{split}
        &t e^{-\frac{2}{3}t^3} e^{t(x+y)}  K_{11}^{f_q^{-1}t, \infty} (x-t^2,y-t^2) \\
        &= \frac{1}{(2\pi \im)^{2}}\int_{\mathcal{C}^{\pi/3}_{\ap_{3}}} dz\int_{\mathcal{C}^{\pi/3}_{\ap_{3}}} dw e^{(z-t)^3/3 + (w-t)^3/3  - x(z-t) - y(w-t)  } \frac{t(z-w) (z + \varpi) (w + \varpi)}{zw (z + w)} \\
        &= \frac{1}{(2\pi \im)^{2}}\int_{\mathcal{C}^{\pi/3}_{1}} dz\int_{\mathcal{C}^{\pi/3}_{1}} dw e^{z^3/3 + w^3/3  - xz - yw  } \frac{t(z-w) (z + t + \varpi) (w + t + \varpi)}{(z+t)(w+ t) (z + w + 2t)},
    \end{split}
\end{equation}
where in the last line we changed variables $z \rightarrow z+t$, $w \rightarrow w+t$ and we shifted the contours without affecting the value of the integral by Cauchy's theorem. We mention that the deformation of the contours near infinity is justified, since the cubic terms in the exponential provide sufficient decay. Specifically, if $z = a + re^{\pm \pi \im /3}$ with $a \in [-A,A]$ and $r \geq 0$, then there is a constant $C> 0$, depending on $A$, such that
$$|e^{z^3/3}| \leq e^{C(1 + |z|^2)-|z|^3/3}.$$

We can now take the limit as $t \rightarrow \infty$ in (\ref{eq:AB1}) and conclude (\ref{eq:Limkcr11}) by the dominated convergence theorem with dominating function
$$(1 + |\varpi|)^2 \cdot |z-w| e^{-|z|^3/3 - |w|^3/3  + A|z| + A|w| },$$
where $A > 0$ is large enough so that $A \geq |x|, |y|$ and we used that for large $t$
$$\left| \frac{ t(z + t + \varpi) (w + t + \varpi)}{(z+t)(w+ t)(z+w +2t) } \right| \leq  (1 + |\varpi|)^2 .$$

{\bf \raggedleft Step 2.} In this step we show (\ref{eq:Limkcr12}). From Definition \ref{def:LimKernel}, see (\ref{eq:I12Lim}) and (\ref{eq:R12Lim}), we have
\begin{equation}\label{eq:AB2}
    \begin{split}
        & e^{t(x-y)} \cdot K_{12}^{f_q^{-1}t, \infty} (x-t^2,y-t^2) \\
        & = \frac{1}{(2\pi \im)^{2}}\int_{\mathcal{C}^{\pi/3}_{\ap_{3}}} dz\int_{\mathcal{C}^{2\pi/3}_{\ap_{1}}} dw e^{(z-t)^3/3 - (w-t)^3/3 - x(z-t) + y(w-t) } \frac{(z+w)(z+\varpi)}{2z(z-w)(w + \varpi)} \\
        & = \frac{1}{(2\pi \im)^{2}}\int_{\mathcal{C}^{\pi/3}_{\ap_3}} dz\int_{\mathcal{C}^{2\pi/3}_{\ap_1}} dw e^{z^3/3 - w^3/3 - xz + yw } \frac{(z+w + 2t)(z+\varpi + t)}{(2z + 2t)(z-w)(w + \varpi + t)},
    \end{split}
\end{equation}
where in the last line we changed variables $z \rightarrow z+t$, $w \rightarrow w+t$ and we shifted the contours without affecting the value of the integral by Cauchy's theorem.

We can now take the limit as $t \rightarrow \infty$ in (\ref{eq:AB2}) and conclude (\ref{eq:Limkcr12}) by the dominated convergence theorem with dominating function
$$(|z| + |w| + 1)^2 e^{-|z|^3/3 - |w|^3/3  + A|z| + A|w|},$$
where $A > 0$ is large enough so that $A \geq |x|, |y|$ and we used that for large $t$
$$\left| \frac{(z+w + 2t)(z+\varpi + t)}{(2z + 2t)(z-w)(w + \varpi + t)} \right| \leq (1/6)(|z| + |w| + 1) \cdot \left| \frac{z+\varpi + t}{w + \varpi + t}  \right| \leq (|z| + |w| + 1)^2.$$

{\bf \raggedleft Step 3.} In this step we show (\ref{eq:Limkcr22}). From (\ref{eq:I22Lim}) and (\ref{eq:R22Lim}) we have for $ t > 0$
\begin{equation}\label{eq:AB3}
\begin{split}
& t^3 e^{-t(x+y)} e^{\frac{2}{3} t^3}  \cdot K_{22}^{f_q^{-1}t, \infty} (x-t^2,y-t^2) \\
& = \frac{1}{(2\pi \im)^{2}}\int_{\mathcal{C}^{2\pi/3}_{-\ap_{2}}} dz\int_{\mathcal{C}^{2\pi/3}_{-\ap_{2}}} dw e^{-(z-t)^3/3 - (w-t)^3/3 +x (z-t) + y (w-t)  } \frac{t^3(z-w)}{4(z+w)(z+ \varpi)(w+\varpi)} \\
& +\frac{t^3}{2\pi \im} \int_{\mathcal{C}^{2\pi/3}_{\ap_{1}}} dz \frac{e^{-(z-t)^3/3 - (-\varpi-t)^3/3 + x(z-t) + y(-\varpi -t)}}{4(z - \varpi)}\\
& - \frac{t^3}{2\pi \im} \int_{\mathcal{C}^{2\pi/3}_{-\ap_{2}}} dw \frac{e^{ - (-\varpi-t)^3/3 -(w-t)^3/3 +x (-\varpi -t ) + y(w- t)}}{4(w - \varpi)} \\
& +  \frac{t^3 e^{-t(x+y)} e^{\frac{2}{3} t^3} }{2\pi \im} \int_{\mathcal{C}^{2\pi/3}_{-\ap_{2}}} dw  \frac{we^{  2t w^2  + w (y-x)  }}{2(w- \varpi)(w + \varpi)}.
\end{split}
\end{equation}
If $t > \ap_1$, then we may deform $\mathcal{C}^{2\pi/3}_{-\ap_{2}}$ and $\mathcal{C}^{2\pi/3}_{\ap_{1}}$ above to $\mathcal{C}^{\pi/2}_{-\ap_{2}}$ and $\mathcal{C}^{\pi/2}_{\ap_{1}}$, respectively, without affecting the value of the integrals by Cauchy's theorem. We require that $t > \ap_1$ to ensure that the integrand has enough decay in the sectors $\Arg(z), \Arg(w) \in [\pi/2, 3\pi/4] \cup [7\pi/4, 3\pi/2]$ to perform the deformation near infinity. We can then deform both contours on the second line of (\ref{eq:AB3}) to $\mathcal{C}^{\pi/2}_{\ap_{1}}$. When we deform the $z$-contour we cross the simple pole at $z = -\varpi$, which cancels with the third line in (\ref{eq:AB3}). Afterwards, when we deform the $w$ contour, we cross the simple pole at $w = -\varpi$, which cancels with the fourth line of (\ref{eq:AB3}), and the simple pole at $w = -z$, which cancels with the fifth line in (\ref{eq:AB3}). Overall, we conclude that for $t > \ap_1$ we have
\begin{equation}\label{eq:AB4}
\begin{split}
& t^3 e^{-t(x+y)} e^{\frac{2}{3} t^3} \cdot K_{22}^{f_q^{-1}t, \infty} (x-t^2,y-t^2) \\
& = \frac{1}{(2\pi \im)^{2}}\int_{\mathcal{C}^{\pi/2}_{\ap_{1}}} dz\int_{\mathcal{C}^{\pi/2}_{\ap_{1}}} dw e^{-(z-t)^3/3 - (w-t)^3/3 +x (z-t) + y (w-t)  } \frac{t^3(z-w)}{4(z+w)(z+ \varpi)(w+\varpi)} \\
& = \frac{1}{(2\pi \im)^{2}}\int_{\mathcal{C}^{\pi/2}_{-1}} dz\int_{\mathcal{C}^{\pi/2}_{-1}} dw e^{-z^3/3 - w^3/3 +x z + y w  } \frac{t^3(z-w)}{4(z+w + 2t)(z+ \varpi + t)(w+\varpi + t)},
\end{split}
\end{equation}
where in the last line we changed variables $z \rightarrow z+t$, $w \rightarrow w+t$ and we shifted the contours without affecting the value of the integral by Cauchy's theorem.

We can now take the limit as $t \rightarrow \infty$ in (\ref{eq:AB4}) and conclude (\ref{eq:Limkcr22}) by the dominated convergence theorem with dominating function
$$|z-w| e^{-|z|^3/3 - |w|^3/3  + A|z| + A|w|},$$
where $A > 0$ is large enough so that $A \geq |x|, |y|$ and hence for large $t$
$$\left| \frac{t^3}{4(z+w + 2t)(z+ \varpi + t)(w+\varpi + t)} \right| \leq 1.$$

\begin{appendix}
%
%
\section{Results for line ensembles} \label{SectionA} In this section we give the proofs of various results from Section \ref{Section2} after introducing some notation and statements in Section \ref{SectionA1}. We continue with the same notation as in Section \ref{Section2}.

%
%
\subsection{Auxiliary results}\label{SectionA1} We start with the following version of the monotone class theorem \cite[Theorem 5.2.2]{Durrett}, which we already used in Section \ref{Section4.3} of the main text, and also need below.
\begin{lemma}\label{MCA} Fix a measurable space $(\Omega, \mathcal{F})$ and suppose that $\{X_i\}_{i \in I}$ is an arbitrary collection of random variables defined on this space. Suppose that $\mathcal{H}$ is a family of bounded measurable functions $h: \Omega \rightarrow \mathbb{R}$ such that 
\begin{enumerate}
\item $\mathcal{H}$ contains all functions of the form $\prod_{r = 1}^n f_{r}(X_{i_r})$, where $f_r: \mathbb{R} \rightarrow \mathbb{R}$ are bounded continuous functions and $i_r \in I$ for $r \in \llbracket 1, n\rrbracket$;
\item if $f,g \in \mathcal{H}$ and $c \in \mathbb{R}$, then $cf + g \in \mathcal{H}$;
\item if $f_n \geq 0$, $f_n \in \mathcal{H}$ and $f_n \uparrow f$ for a bounded function $f$, then $f \in \mathcal{H}$.
\end{enumerate}
Then, $\mathcal{H}$ contains all bounded measurable functions with respect to $\sigma\left\{ X_i: i \in I\right\}$.
\end{lemma}
\begin{proof} The argument is quite standard, so we will be brief. Fix $n \in \mathbb{N}$ and $i_1, \dots, i_n \in I$. Define 
$$\chi_N(x,u) = \begin{cases} 0, &x < u \\ N(x-u), & x \in [u,u + 1/N], \\ 1, & x > u + 1/N. \end{cases}$$
From condition (1) we have $\prod_{r = 1}^n \chi_N(X_{i_r}, a_r) \in \mathcal{H}$ and so by condition (3) $\prod_{r = 1}^n {\bf 1}\{X_{i_r} > a_r \} \in \mathcal{H}$ for any $a_r \in \mathbb{R}$. Sets of the form $\{ X_{i_r} > a_r \mbox{ for }r \in \llbracket 1, n \rrbracket \}$ form a $\pi$-system that generates $\sigma\left\{ X_i: i \in I\right\}$, hence the result follows from  \cite[Theorem 5.2.2]{Durrett}.
\end{proof}

We next give a simple connection between Brownian bridges and reverse Brownian motions, which will be used in the proofs of Lemmas \ref{lem: BB touch intersect} and \ref{lem: BB positive measure} in Sections \ref{SectionA2} and \ref{SectionA3}, respectively. If $W_t$ denotes a standard one-dimensional Brownian motion, then the process
$$\hat{B}^{\operatorname{std}}(t) =  W_t - t W_1, \hspace{5mm} 0 \leq t \leq 1,$$
is called a {\em standard Brownian bridge}. Given $a, b,x,y \in \mathbb{R}$ with $a < b$, we define a random element in $(C([a,b]), \mathcal{C})$ through
\begin{equation}\label{BBDef}
\hat{B}(t) = (b-a)^{1/2} \cdot \hat{B}^{\operatorname{std}} \left( \frac{t - a}{b-a} \right) + \left(\frac{b-t}{b-a} \right) \cdot x + \left( \frac{t- a}{b-a}\right) \cdot y, 
\end{equation}
and refer to it as a {\em Brownian bridge from $\hat{B}(a) = x$ to $\hat{B}(b) = y$.} Given $k \in \mathbb{N}$ and $\vec{x}, \vec{y} \in \mathbb{R}^k$, we let $\mathbb{P}^{a,b, \vec{x},\vec{y}}_{\operatorname{free}}$ denote the law of $k$ independent Brownian bridges $\{\hat{B}_i: [a,b] \rightarrow \mathbb{R} \}_{i = 1}^k$ from $\hat{B}_i(a) = x_i$ to $\hat{B}_i(b) = y_i$, and write $\mathbb{E}^{a,b, \vec{x},\vec{y}}_{\operatorname{free}}$ for the expectation with respect to this measure.

Let $y, \mu \in \mathbb{R}$, $b\in(0,\infty)$ and $\cev{B}(t)$ be the reverse Brownian motion from (\ref{eq:RevDefBrownianMotionDrift}), i.e.
\begin{equation}\label{eq:revBM again}
\cev{B}(t)=y+W_{b-t}+\mu(b-t)\mbox{ for }0\leq t\leq b.
\end{equation}
We consider the spaces
\[
C_{0,0}([0,b])=\{g\in C([0,b]):g(0)=g(b)=0\} \quad\mbox{and}\quad
C_{y}([0,b])=\{g\in C([0,b]):g(b)=y\},
\]
equipped with the uniform topology, and define a map $H_y:C_{0,0}([0,b])\times\mathbb{R}\rightarrow C_{y}([0,b])$ by
\begin{equation}\label{eq:def of homeomorphism F}
H_y(g,z)(t)= y + g(b-t) + z \cdot (b-t) \mbox{ for }0\leq t\leq b.
\end{equation}
Endowing $C_{0,0}([0,b])\times\mathbb{R}$ with the product topology, we see that that $H_y$ is a homeomorphism between topological spaces with inverse map
$$H_y^{-1}(h) = (g,z) \mbox{, with } z = b^{-1}(h(0)-y) \mbox{ and } g(t) = h(b-t) - y - t b^{-1}(h(0)-y) \mbox{ for } 0 \leq t \leq b.$$
We next observe that if $\hat{B}$ is a Brownian bridge from $\hat{B}(0) = 0$ to $\hat{B}(b) = 0$, and $X$ is an independent normal variable with $\mathbb{E}[X] = \mu$ and $\operatorname{Var}(X) = b^{-1}$, then  $\cev{B}' = H_{y}(\hat{B}, X)$ has the same law as $\cev{B}$. To see the latter, note that by construction $\cev{B}'$ is a continuous Gaussian processes with 
\begin{equation}\label{matchMeanCov}
\begin{split}
&\mathbb{E}[\cev{B}'(t)] = y + \mu(b-t) \mbox{ and } \operatorname{Cov}(\cev{B}'(s), \cev{B}'(t)) = \min(b-t, b-s), 
\end{split}
\end{equation}
which agrees with (\ref{eq:revBM again}).

%
%
\subsection{Proof of Lemma \ref{lem: BB touch intersect}}\label{SectionA2} We continue with the same notation as in Section \ref{SectionA1}. Let $\hat{B}$ be a Brownian bridge from $\hat{B}(0) = 0 $ to $\hat{B}(b) = 0$, $X$ an independent normal variable with $\mathbb{E}[X] =\mu$ and $\operatorname{Var}(X) = b^{-1}$. From our discussion above (\ref{matchMeanCov}), we have that 
\begin{equation*}
\begin{split}
&\mathbb{P}(D \cap C^c) = \mathbb{P}(\hat{D} \cap \hat{C}^c), \mbox{ where } \hat{D} = \{y + \hat{B}(b-t) + X (b-t) = f(t) \mbox{ for some } t\in [0,b]\}, \mbox{ and } \\
&\hat{C} = \{ y + \hat{B}(b-t) + X (b-t) < f(t) \mbox{ for some $t \in [0,b]$ } \}.
\end{split}
\end{equation*}
From \cite[Corollary 2.9]{CorHamA} we have $\mathbb{P}(\hat{D} \cap \hat{C}^c |X = x) = 0$ for all $x \in \mathbb{R}$, which implies the lemma.

%
%
\subsection{Proof of Lemma \ref{lem: BB positive measure}}\label{SectionA3} We continue with the same notation as in Section \ref{SectionA1}. Let $\hat{B}$ be a Brownian bridge from $\hat{B}(0) = 0 $ to $\hat{B}(b) = 0$, $X$ an independent normal variable with $\mathbb{E}[X] =\mu$ and $\operatorname{Var}(X) = b^{-1}$. From our discussion above (\ref{matchMeanCov}), we have that 
$$\mathbb{P}(\cev{B} \in U) = \mathbb{P}\left((\hat{B}, X) \in H_y^{-1}(U \cap C_y([0,b]))\right).$$
Since $U\cap C_y([0,b])$ is a non-empty open subset of $C_y([0,b])$, and $H_y$ is a homeomorphism, we know that $H_y^{-1}(U\cap C_y([0,b]))$ is an open subset of $C_{0,0}([0,b])\times\mathbb{R}$, which must contains a set of the form $U_0\times I_0$, where $U_0$ is a non-empty open subset of $C_{0,0}([0,b])$ and $I_0$ is a non-empty open interval in $\mathbb{R}$. The latter shows that 
$$\mathbb{P}(\cev{B} \in U)  = \mathbb{P}\left((\hat{B}, X) \in H_y^{-1}(U \cap C_y([0,b]))\right) \geq \mathbb{P}(\hat{B} \in U_0) \cdot \mathbb{P}(X \in I_0) > 0,$$
where in the last inequality we used that $\mathbb{P}(\hat{B} \in U_0) > 0 $ from \cite[Corollary 2.10]{CorHamA}.

%
%
\subsection{Well-posedness of the half-space Brownian Gibbs property}\label{SectionA35} 
We continue with the same notation as in Section \ref{SectionA1}. In this section we establish the following lemma, which explains why equation (\ref{eq:HSBGP}) makes sense, cf. Remark \ref{RemMeas}.

I moved Lemma \ref{LemmaMeasExp} to this section.
\begin{lemma}\label{LemmaMeasExp} Assume the same notation as in Definition \ref{def: avoidBLE} and suppose that $F : C( \llbracket 1, k \rrbracket \times [0,b]) \rightarrow \mathbb{R}$ is a bounded Borel-measurable function. Let 
\begin{equation*}
\begin{split}
S_{\mathsf{b}} &= \{ (\vec{y}, g) \in \weyl_k \times C([0,b]): \mbox{ $g(b) < y_k$}\} \mbox{ and }S =\weyl_k ,
\end{split}
\end{equation*}
where $S_{\mathsf{b}}, S$ are endowed with the subspace topology coming from the product topology and corresponding Borel $\sigma$-algebra. Then, the functions $G_F: S \rightarrow \mathbb{R}$ and $G_F^{\mathsf{b}}: S_{\mathsf{b}} \rightarrow \mathbb{R}$, given by
\begin{equation}\label{MeasExpFun}
\begin{split}
G^{\mathsf{b}}_F(\vec{y},g) &= \eabm^{b, \vec{y}, \vec{\mu}, g}[F(\mathcal{Q})], \hspace{7mm} G_F( \vec{y}) = \eabm^{b, \vec{y}, \vec{\mu}, -\infty}[F(\mathcal{Q})],
\end{split}
\end{equation}
are bounded and measurable.
\end{lemma}
\begin{proof} By Definition \ref{def: avoidBLE} we have 
\begin{equation}\label{eq:in proof measurability}
G^{\mathsf{b}}_F(\vec{y},g) =\frac{\efbm^{b, \vec{y}, \vec{\mu}}\left[F(\mathcal{Q})\cdot\mathbf{1}_{E_{\operatorname{avoid}}^{[0,b],\infty,g}}(\mathcal{Q}) \right]}{\efbm^{b, \vec{y}, \vec{\mu}}\left[\mathbf{1}_{E_{\operatorname{avoid}}^{[0,b],\infty,g}}(\mathcal{Q}) \right]}, \mbox{ and }G_F( \vec{y})  = \frac{\efbm^{b, \vec{y}, \vec{\mu}}\left[F(\mathcal{Q})\cdot\mathbf{1}_{E_{\operatorname{avoid}}^{[0,b],\infty,-\infty}}(\mathcal{Q})\right]}{\efbm^{b, \vec{y}, \vec{\mu}}\left[\mathbf{1}_{E_{\operatorname{avoid}}^{[0,b],\infty,-\infty}}(\mathcal{Q})\right]},
\end{equation}
where 
\begin{equation}\label{A1AvoidSet}
E^{S,f,g}_{\operatorname{avoid}}  = \left\{ f(r) > \mathcal{Q}_1(r) > \mathcal{Q}_2(r) > \cdots > \mathcal{Q}_k(r) > g(r) \mbox{ for all $r \in S$} \right\}.
\end{equation}
Recall from Remark \ref{rem:WellDAvoid} that the denominators in (\ref{eq:in proof measurability}) are positive. \\

Let $W^1,\dots,W^k$ be independent standard Brownian motions, defined on a probability space $(\Omega,\mathcal{F},\mathbb{P})$. We define the random line ensembles $\mathcal{Q}^{\vec{y}}=(\mathcal{Q}_1^{\vec{y}},\dots,\mathcal{Q}_k^{\vec{y}})$ in $C(\llbracket1,k\rrbracket\times[0,b])$
by $\mathcal{Q}_i^{\vec{y}}(t)=y_i+W_{b-t}^i+\mu_i(b-t)$ for $i\in\llbracket1,k\rrbracket$, and observe that
\begin{equation}\label{eq:vgerwaarf}
\efbm^{b, \vec{y}, \vec{\mu}}\left[F(\mathcal{Q})\cdot\mathbf{1}_{E_{\operatorname{avoid}}^{[0,b],\infty,g}}(\mathcal{Q}) \right]=\mathbb{E}\left[F(\mathcal{Q}^{\vec{y}})\cdot\mathbf{1}_{E_{\operatorname{avoid}}^{[0,b],\infty,g}}(\mathcal{Q}^{\vec{y}})\right].
\end{equation}
We seek to show that if $F : C( \llbracket 1, k \rrbracket \times [0,b]) \rightarrow \mathbb{R}$ is a bounded continuous function, and $(\vec{y}_n, g_n)\rightarrow(\vec{y} , g )$ in $S_{\mathsf{b}}$, then
\begin{equation}\label{eq:vearbatr}
\mathbb{E}\left[F(\mathcal{Q}^{\vec{y}_n})\cdot\mathbf{1}_{E_{\operatorname{avoid}}^{[0,b],\infty,g_n}}(\mathcal{Q}^{\vec{y}_n})\right]\rightarrow\mathbb{E}\left[F(\mathcal{Q}^{\vec{y}})\cdot\mathbf{1}_{E_{\operatorname{avoid}}^{[0,b],\infty,g}}(\mathcal{Q}^{\vec{y}})\right] \mbox{ as }n\rightarrow\infty.
\end{equation}
Using Lemma \ref{lem: BB touch intersect} we have $\mathbb{P}\left(\mathcal{Q}^{\vec{y}}\in \bar{E}_{\operatorname{avoid}}^{[0,b],\infty,g}\setminus E_{\operatorname{avoid}}^{[0,b],\infty,g} \right)=0$, where 
$$\bar{E}_{\operatorname{avoid}}^{[0,b],\infty,g}=\left\{\mathcal{Q}_1 (t)\geq\mathcal{Q}_2(t)\geq\cdots\geq\mathcal{Q}_k(t)\geq g(t)\mbox{ for all } t\in[0,b] \right\}.$$
From the last observation, the continuity of $F$, and $(\vec{y}_n, g_n)\rightarrow(\vec{y} , g )$ in $S_{\mathsf{b}}$, we conclude that $\mathbb{P}$-a.s. 
$$\lim_n F(\mathcal{Q}^{\vec{y}_n})\cdot\mathbf{1}_{E_{\operatorname{avoid}}^{[0,b],\infty,g_n}}(\mathcal{Q}^{\vec{y}_n}) =  F(\mathcal{Q}^{\vec{y}})\cdot\mathbf{1}_{E_{\operatorname{avoid}}^{[0,b],\infty,g}}(\mathcal{Q}^{\vec{y}}),$$
which implies (\ref{eq:vearbatr}) by the bounded convergence theorem.

Combining (\ref{eq:in proof measurability}), (\ref{eq:vgerwaarf}) and (\ref{eq:vearbatr}), we see that if $F : C( \llbracket 1, k \rrbracket \times [0,b]) \rightarrow \mathbb{R}$ is bounded and continuous, then $G^{\mathsf{b}}_F$ is a bounded continuous function on $S_{\mathsf{b}}$ (as the ratio of two continuous functions with a positive denominator). One analogously shows that $G_F$ is bounded and continuous on $\weyl_k$. If we now let $\mathcal{H}$ denote the class of bounded measurable functions $F$, such that $G^{\mathsf{b}}_F$ and $G_F$ are measurable, we see that $\mathcal{H}$ contains all bounded continuous $F$, which by the monotone class argument in Lemma \ref{MCA} implies that $\mathcal{H}$ contains all bounded measurable functions as desired.
\end{proof}

%
%
\subsection{Proof of Lemma \ref{LemmaConsistent}}\label{SectionA4} We first recall a bit of notation from \cite{DimMat}. The following definition, which is \cite[Definition 2.4]{DimMat}, introduces the notion of an $(f,g)$-avoiding Brownian line ensemble.
\begin{definition}\label{DefAvoidingLaw}
Let $k \in \mathbb{N}$ and $\weyl_k$ denote the open Weyl chamber in $\mathbb{R}^k$ from (\ref{DefWeyl}). Let $\vec{x}, \vec{y} \in \weyl_k$, $a,b \in \mathbb{R}$ with $a < b$, and $f: [a,b] \rightarrow (-\infty, \infty]$ and $g: [a,b] \rightarrow [-\infty, \infty)$ be two continuous functions. The latter means that either $f: [a,b] \rightarrow \mathbb{R}$ is continuous or $f = \infty$ everywhere, and similarly for $g$. We also assume that $f(t) > g(t)$ for all $t \in[a,b]$, $f(a) > x_1, f(b) > y_1$ and $g(a) < x_k, g(b) < y_k.$

With the above data we define the {\em $(f,g)$-avoiding Brownian line ensemble on the interval $[a,b]$ with entrance data $\vec{x}$ and exit data $\vec{y}$} to be the $\Sigma$-indexed line ensemble $\mathcal{Q}$ with $\Sigma = \llbracket 1, k\rrbracket$ on $\Lambda = [a,b]$ and with the law of $\mathcal{Q}$ equal to $\mathbb{P}^{a,b, \vec{x},\vec{y}}_{\operatorname{free}}$ (the law of $k$ independent Brownian bridges $\{\hat{B}_i: [a,b] \rightarrow \mathbb{R} \}_{i = 1}^k$ from $\hat{B}_i(a) = x_i$ to $\hat{B}_i(b) = y_i$), conditioned on the event $E^{[a,b],f,g}_{\operatorname{avoid}}$ from (\ref{A1AvoidSet}). We denote the probability distribution of $\mathcal{Q}$ as $\mathbb{P}_{\operatorname{avoid}}^{a,b, \vec{x}, \vec{y}, f, g}$ and write $\mathbb{E}_{\operatorname{avoid}}^{a,b, \vec{x}, \vec{y}, f, g}$ for the expectation with respect to this measure.
\end{definition}

The following definition introduces the partial Brownian Gibbs property from \cite[Definition 2.7]{DimMat}.
\begin{definition}\label{DefPBGP}
Fix a set $\Sigma = \llbracket 1 , N \rrbracket$ with $N \in \mathbb{N}$ or $N  = \infty$ and an interval $\Lambda \subseteq \mathbb{R}$.  A $\Sigma$-indexed line ensemble $\mathcal{L}$ on $\Lambda$ is said to satisfy the {\em partial Brownian Gibbs property} if and only if it is non-intersecting and for any finite $K = \{k_1, k_1 + 1, \dots, k_2 \} \subset \Sigma$ with $k_2 \leq N - 1$, $[a,b] \subset \Lambda$ and any bounded Borel-measurable function $F: C(K \times [a,b]) \rightarrow \mathbb{R}$ we have $\mathbb{P}$-almost surely
\begin{equation}\label{PBGPTower}
\mathbb{E} \left[ F(\mathcal{L}|_{K \times [a,b]}) {\big \vert} \mathcal{F}_{\operatorname{ext}} (K \times (a,b))  \right] =\mathbb{E}_{\operatorname{avoid}}^{a,b, \vec{x}, \vec{y}, f, g} \bigl[ F(\mathcal{Q}) \bigr],
\end{equation}
where $\mathcal{L}|_{K \times [a,b]}$ is the restriction of $\mathcal{L}$ to $K \times [a,b]$, $\vec{x} = (\mathcal{L}_{k_1}(a), \dots, \mathcal{L}_{k_2}(a))$, $\vec{y} = (\mathcal{L}_{k_1}(b), \dots, \mathcal{L}_{k_2}(b))$, $f = \mathcal{L}_{k_1 - 1}[a,b]$ (with the convention that $f = \infty$ if $k_1 = 1$), $g = \mathcal{L}_{k_2 +1}[a,b]$, and 
$$\mathcal{F}_{\operatorname{ext}} (K \times (a,b)) = \sigma \left \{ \mathcal{L}_i(s): (i,s) \in D_{K,a,b}^c \right\},$$
with $D_{K,a,b} = K \times (a,b)$ and $D_{K,a,b}^c = (\Sigma \times \Lambda) \setminus D_{K,a,b}$. On the right side of (\ref{PBGPTower}) we have that $\mathcal{Q}$ has law $\mathbb{P}_{\operatorname{avoid}}^{a,b, \vec{x}, \vec{y}, f, g}$, and we think of $\mathcal{Q} = (\mathcal{Q}_1, \dots, \mathcal{Q}_{k_2 - k_1 + 1})$ as being in $C(K \times [a,b])$ as opposed to $C(\llbracket 1, k_2 - k_1 + 1 \rrbracket \times [a,b])$ by re-indexing the curves.
\end{definition}

We next establish the following useful lemma, which will be used in the proof of Lemma \ref{LemmaConsistent}. In plain words, the lemma says that the $g$-avoiding reverse Brownian line ensembles from Definition \ref{def: avoidBLE} satisfy both the half-space Brownian Gibbs property and the partial Brownian Gibbs property.

\begin{lemma}\label{LGPAll} Assume the same notation as in Definition \ref{def: avoidBLE} and suppose $\mathcal{Q} = \{\mathcal{Q}_i\}_{i = 1}^k$ is a $\llbracket 1, k \rrbracket$-indexed line ensemble on $[0,b]$ with law $\pabm^{b, \vec{y}, \vec{\mu}, g}$, where $g \in C([0,b])$. Suppose $\mathcal{L} = \{\mathcal{L}_i\}_{i = 1}^{k+1}$ is the $\llbracket 1, k+1\rrbracket$-indexed line ensemble on $[0,b]$ with $\mathcal{L}_i = \mathcal{Q}_i$ for $i \in \llbracket 1, k \rrbracket$ and $\mathcal{L}_{k+1} = g$. Then, $\mathcal{L}$ satisfies the half-space Brownian Gibbs property from Definition \ref{def:BGP} with the same $\mu_1, \dots, \mu_k$ and arbitrary $\mu_{k+1} \in \mathbb{R}$. In addition, $\mathcal{L}$ satisfies the partial Brownian Gibbs property from Definition \ref{DefPBGP}. 
\end{lemma}
\begin{proof} The idea of the proof is to appropriately apply Lemma \ref{lem:RW} and obtain the two Brownian Gibbs properties as diffuse scaling limits of the interlacing Gibbs property. For clarity, we split the proof into two steps.\\

{\bf \raggedleft Step 1.} In this step we show that $\mathcal{L}$ satisfies the half-space Brownian Gibbs property. Fix $p \in (0,1)$, $u = p/(1-p)$, $\sigma = \sqrt{p}/(1-p)$ and the sequences $B_n = n$, $d_n = n/b$, $g_n = g$, $\vec{Y}^{n} \in \mathfrak{W}_k$, where 
$$Y_i^n = \lfloor \sigma d_n^{1/2} y_i+ u  B_n \rfloor \mbox{ for } i \in \llbracket 1, k \rrbracket,$$
$G_n(t) = \sigma d_n^{1/2} \cdot g(t/d_n) + u t$, and $\vec{q}^{\,n} = (q_1^n, \dots, q_k^n) \in (0,1)^k$, such that 
$$q_i^n = p - \mu_{i} \sqrt{p} (1-p) \cdot d_n^{-1/2} + o(d_n^{-1/2}) \mbox{ for } i \in \llbracket 1, k \rrbracket.$$
From Lemma \ref{lem:RW}, we know that for all large $n$ the laws $\mathbb{P}_{\ice, \operatorname{Geom}}^{B_n, \vec{Y}^n, \vec{q}^{\,n}, G_n}$ are well defined. Moreover, if $\mathfrak{Q}^n = \{Q_i^n\}_{i = 1}^k$ has law $\mathbb{P}_{\ice, \operatorname{Geom}}^{B_n, \vec{Y}^n, \vec{q}^{\,n}, G_n}$, then the $\llbracket 1,k  \rrbracket$-indexed line ensembles $\mathcal{Q}^n$ on $[0,b]$, defined by
\begin{equation}\label{YQ1}
\mathcal{Q}_i^n(t) = \sigma^{-1} d_n^{-1/2} \cdot \left( Q^n_i(t d_n) - utd_n\right) \mbox{ for }  t \in [0,b] \mbox{ and } i \in \llbracket 1, k \rrbracket,
\end{equation}
converge weakly to $\pabm^{b, \vec{y}, \vec{\mu}, g}$ from Definition \ref{def: avoidBLE} as $n \rightarrow \infty$.

We consider $\mathcal{L}^n = \{\mathcal{L}^n_i\}_{i = 1}^{k+1}$ with $\mathcal{L}^n_i = \mathcal{Q}_i^n$ for $i \in \llbracket 1, k \rrbracket$ and $\mathcal{L}^n_{k+1} = g$, and note that $\mathcal{L}^n \Rightarrow \mathcal{L}^{\infty}$, where $\mathcal{L}^{\infty}$ has the same law as $\mathcal{L}$ in the statement of the present lemma. By Skorohod's Representation Theorem, see \cite[Theorem 6.7]{Billing}, we may assume that the sequence $\{\mathcal{L}^n\}_{n \geq 1}$ and $\mathcal{L}^{\infty}$ are all defined on the same probability space $(\Omega, \mathcal{F}, \mathbb{P})$ and $\lim_n \mathcal{L}^n(\omega) = \mathcal{L}^{\infty}(\omega)$ for each $\omega \in \Omega$.

Fix $v \in \llbracket 1, k \rrbracket$, $m \in \mathbb{N}$, $n_1, \dots, n_m \in \llbracket 1 , k+1 \rrbracket$, $w, t_1, \dots, t_m \in [0,b]$ and bounded continuous $h_1, \dots, h_m : \mathbb{R} \rightarrow \mathbb{R}$. Define $R = \{i \in \llbracket 1, m \rrbracket: n_i \in \llbracket 1 , v \rrbracket, t_i \in [0,w]\}$. We claim that 
\begin{equation}\label{YQ2}
\mathbb{E}\left[ \prod_{i = 1}^m h_i(\mathcal{L}^{\infty}_{n_i}(t_i)) \right] = \mathbb{E}\left[ \prod_{i \not \in R} h_i(\mathcal{L}^{\infty}_{n_i}(t_i))  \cdot \mathbb{E}_{\operatorname{avoid}}^{w,\vec{y}\,', \vec{\mu}_{v},g'} \left[ \prod_{i  \in R} h_i(\mathcal{Q}_{n_i}(t_i))   \right] \right], 
\end{equation}
where $\vec{y}\,' = (\mathcal{L}^{\infty}_{1} (w), \dots, \mathcal{L}^{\infty}_{v} (w))$, $\vec{\mu}_v = (\mu_1, \dots, \mu_v)$, $g' = \mathcal{L}_{v+1}^{\infty}[0,w]$. Note that by the monotone class argument in Lemma \ref{MCA} and (\ref{YQ2}) we would conclude for any bounded Borel-measurable $F: C(\llbracket 1, v \rrbracket \times [0,w]) \rightarrow \mathbb{R}$ and any bounded $\mathcal{F}_{\operatorname{ext}} (\llbracket 1,v\rrbracket \times [0,w))$-measurable $Y$ (this $\sigma$-algebra was introduced in Definition \ref{def:BGP}) that
\begin{equation}\label{YQ3}
\mathbb{E}\left[ F\left( \mathcal{L}^{\infty} \vert_{\llbracket 1, v \rrbracket \times [0,w]} \right) \cdot Y \right] = \mathbb{E}\left[ \mathbb{E}_{\operatorname{avoid}}^{w,\vec{y}\,', \vec{\mu}_{v}, g'}  \left[ F(\mathcal{Q}) \right]\cdot Y \right].
\end{equation}
Equation (\ref{YQ3}) shows that $\mathcal{L}^{\infty}$ (and hence $\mathcal{L}$) satisfies the half-space Brownian Gibbs property. We have thus reduced our proof to establishing (\ref{YQ2}).\\

We set $W_n = \lceil w d_n \rceil$, and observe directly from the definition of $\mathbb{P}_{\ice, \operatorname{Geom}}^{B_n, \vec{Y}^n, \vec{q}^{\,n}, G_n}$ in Definition \ref{def:interlaceGeom}, that for all large $n$
\begin{equation}\label{YQ4}
\mathbb{E}\left[ \prod_{i = 1}^m h_i(\mathcal{L}^{n}_{n_i}(t_i)) \right] = \mathbb{E}\left[ \prod_{i \not \in R} h_i(\mathcal{L}^{n}_{n_i}(t_i))  \cdot \mathbb{E}_{\ice, \operatorname{Geom}}^{W_n,\vec{Z}^n, \vec{q}^{\,n}_v, G_n'} \left[ \prod_{i  \in R} h_i(\mathcal{Q}'_{n_i}(t_i))   \right] \right], 
\end{equation}
where $\vec{Z}^n = (Q_1^n(W_n), \dots, Q_v^n(W_n))$, $\vec{q}_v^{\,n} = (q_1^n, \dots, q_v^n)$, $G_n' = Q^n_{v+1}\llbracket 0 , W_n \rrbracket$ (with the convention $Q^n_{k+1} = G_n$), and $\mathcal{Q}_i'(t) = \sigma^{-1}d_n^{-1/2} (Q_i'(td_n) - utd_n)$ for $t \in [0,w]$ with $\mathfrak{Q}' = \{Q_i'\}_{i = 1}^v$ having law $\mathbb{P}_{\ice, \operatorname{Geom}}^{W_n,\vec{Z}^n, \vec{q}^{\,n}_v, G_n'}$. 
We now deduce (\ref{YQ2}) by taking the $n \rightarrow \infty$ limit of (\ref{YQ4}) and using the bounded convergence theorem. We mention that the convergence of $h_i(\mathcal{L}^{n}_{n_i}(t_i))$ to $h_i(\mathcal{L}^{\infty}_{n_i}(t_i))$ uses the convergence of $\mathcal{L}^n$ to $\mathcal{L}^{\infty}$ and the continuity of $h_i$. In addition, the convergence of the expectations on the right side of (\ref{YQ4}) to the ones on the right of (\ref{YQ2}) follows from a second application of Lemma \ref{lem:RW} with $B_n = W_n$.\\

{\bf \raggedleft Step 2.} In this step we show that $\mathcal{L}$ satisfies the partial Brownian Gibbs property. We adopt the same notation as in the first two paragraphs of Step 1. 

Fix $1 \leq k_1 \leq k_2 \leq k$, an interval $[c,d] \subseteq [0,b]$, $m \in \mathbb{N}$, $n_1, \dots, n_m \in \llbracket 1 , k+1 \rrbracket$, $t_1, \dots, t_m \in [0,b]$ and bounded continuous $h_1, \dots, h_m : \mathbb{R} \rightarrow \mathbb{R}$. Define $R = \{i \in \llbracket 1, m \rrbracket: n_i \in \llbracket k_1 , k_2 \rrbracket, t_i \in [c,d]\}$. We claim that 
\begin{equation}\label{YR2}
\mathbb{E}\left[ \prod_{i = 1}^m h_i(\mathcal{L}^{\infty}_{n_i}(t_i)) \right] = \mathbb{E}\left[ \prod_{i \not \in R} h_i(\mathcal{L}^{\infty}_{n_i}(t_i))  \cdot \mathbb{E}_{\operatorname{avoid}}^{c,d, \vec{x}\,',\vec{y}\,', f',g'} \left[ \prod_{i  \in R} h_i(\mathcal{Q}_{n_i - k_1 + 1}(t_i))   \right] \right], 
\end{equation}
where $\vec{x}\,' = (\mathcal{L}^{\infty}_{k_1} (c), \dots, \mathcal{L}^{\infty}_{k_2} (c))$, $\vec{y}\,' = (\mathcal{L}^{\infty}_{k_1} (d), \dots, \mathcal{L}^{\infty}_{k_2} (d))$, $f'=\mathcal{L}^{\infty}_{k_1 - 1}[c,d]$ with $f' = \infty$ when $k_1 = 1$,  $g' = \mathcal{L}^{\infty}_{k_2+1}[c,d]$, and we recall that $\mathbb{E}_{\operatorname{avoid}}^{c,d, \vec{x}\,',\vec{y}\,', f',g'}$ is as in Definition \ref{DefAvoidingLaw}. By the monotone class argument in Lemma \ref{MCA} and (\ref{YR2}) we would conclude for any bounded Borel-measurable $F: C(\llbracket k_1, k_2 \rrbracket \times [c,d]) \rightarrow \mathbb{R}$ and any bounded $\mathcal{F}_{\operatorname{ext}} (\llbracket k_1,k_2\rrbracket \times (c,d))$-measurable $Y$ (this $\sigma$-algebra was introduced in Definition \ref{DefPBGP}) that
\begin{equation}\label{YR3}
\mathbb{E}\left[ F\left( \mathcal{L}^{\infty} \vert_{\llbracket k_1, k_2 \rrbracket \times [c,d]} \right) \cdot Y \right] = \mathbb{E}\left[ \mathbb{E}_{\operatorname{avoid}}^{c,d, \vec{x}\,',\vec{y}\,', f',g'} \left[ F(\mathcal{Q}) \right]\cdot Y \right].
\end{equation}
Equation (\ref{YR3}) shows that $\mathcal{L}^{\infty}$ (and hence $\mathcal{L}$) satisfies the partial Brownian Gibbs property from Definition \ref{DefPBGP}. We have thus reduced our proof to establishing (\ref{YR2}).\\

We set $C_n = \lfloor c d_n \rfloor$, $D_n = \lceil d d_n \rceil$, and observe directly from the definition of $\mathbb{P}_{\ice, \operatorname{Geom}}^{B_n, \vec{Y}^n, \vec{q}^{\,n}, G_n}$ in Definition \ref{def:interlaceGeom}, that for all large $n$
\begin{equation}\label{YR4}
\mathbb{E}\left[ \prod_{i = 1}^m h_i(\mathcal{L}^{n}_{n_i}(t_i)) \right] = \mathbb{E}\left[ \prod_{i \not \in R} h_i(\mathcal{L}^{n}_{n_i}(t_i))  \cdot \mathbb{E}_{\ice, \operatorname{Geom}}^{C_n, D_n, \vec{X}^n, \vec{Z}^n, F_n, G_n'} \left[ \prod_{i  \in R} h_i(\mathcal{Q}'_{n_i - k_1 + 1}(t_i))   \right] \right], 
\end{equation}
where $\vec{X}^n = (Q_{k_1}^n(C_n), \dots, Q_{k_2}^n(C_n))$ $\vec{Z}^n = (Q_{k_1}^n(D_n), \dots, Q_{k_2}^n(D_n))$), $F_n = Q_{k_1-1}^n \llbracket C_n, D_n \rrbracket$ (with the convention that $Q^n_{0} = \infty$, and $G_n' = Q^n_{k_2+1}\llbracket C_n , D_n \rrbracket$ (with the convention $Q^n_{k+1} = G_n$). In addition, $\mathcal{Q}_i'(t) = \sigma^{-1}d_n^{-1/2} (Q_i'(td_n) - utd_n)$ for $t \in [c,d]$ with $\mathfrak{Q}' = \{Q_i'\}_{i = 1}^{k_2 - k_1  +1}$ having law $\mathbb{P}_{\ice, \operatorname{Geom}}^{C_n, D_n, \vec{X}^n, \vec{Z}^n, F_n, G_n'}$, where we recall that the latter was defined in Lemma \ref{LemmaConsistentGeom}.

We now deduce (\ref{YR2}) by taking the $n \rightarrow \infty$ limit of (\ref{YR4}) and using the bounded convergence theorem. The convergence of $h_i(\mathcal{L}^{n}_{n_i}(t_i))$ to $h_i(\mathcal{L}^{\infty}_{n_i}(t_i))$ is the same as in Step 1, while the convergence of the expectations on the right side of (\ref{YR4}) to the ones on the right of (\ref{YR2}) follows from Lemma \cite[Lemma 2.16]{dimitrov2024tightness} with $d_n$ as in the present setup, $p$ in that paper is set to $p/(1-p)$, $a = c$, $b = d$, $f = \mathcal{L}^{\infty}_{k_1-1}[c,d]$ (with the convention that $\mathcal{L}_0^{\infty} = \infty$), $g = \mathcal{L}^{\infty}_{k_2+1}[c,d]$, $\vec{x} = (\mathcal{L}^{\infty}_{k_1}(c), \dots, \mathcal{L}^{\infty}_{k_2}(c))$, $\vec{y} = (\mathcal{L}^{\infty}_{k_1}(d), \dots, \mathcal{L}^{\infty}_{k_2}(d))$.
\end{proof}

In the remainder of this section we present the proof of Lemma \ref{LemmaConsistent}.
\begin{proof}[Proof of Lemma \ref{LemmaConsistent}] Fix $1 \leq k_1 \leq k_2 \leq N-1$, and $[a,b] \subset \Lambda$. We seek to show that for any bounded Borel-measurable $F: C(\llbracket k_1, k_2 \rrbracket \times [a,b]) \rightarrow \mathbb{R}$ and any bounded $\mathcal{F}_{\operatorname{ext}} (\llbracket k_1,k_2\rrbracket \times (a,b))$-measurable $Y$ (this $\sigma$-algebra was introduced in Definition \ref{DefPBGP}) we have 
\begin{equation}\label{YS1}
\mathbb{E}\left[ F\left( \mathcal{L} \vert_{\llbracket k_1, k_2 \rrbracket \times [a,b]} \right) \cdot Y \right] = \mathbb{E}\left[ \mathbb{E}_{\operatorname{avoid}}^{a,b, \vec{x},\vec{y}, f,g}\left[ F(\mathcal{Q}) \right]\cdot Y \right],
\end{equation}
where $\vec{x} = (\mathcal{L}_{k_1}(a), \dots, \mathcal{L}_{k_2}(a))$, $\vec{y} = (\mathcal{L}_{k_1}(b), \dots, \mathcal{L}_{k_2}(b))$, $f = \mathcal{L}_{k_1-1}[a,b]$ (with the convention $\mathcal{L}_0[a,b] = \infty$), $g = \mathcal{L}_{k_2 + 1}[a,b]$. By the monotone class argument in Lemma \ref{MCA}, to prove (\ref{YS1}) it suffices to show that for any $m \in \mathbb{N}$, $n_1, \dots, n_m \in \Sigma$, $t_1, \dots, t_m \in \Lambda$ and bounded continuous $h_1, \dots, h_m : \mathbb{R} \rightarrow \mathbb{R}$
\begin{equation}\label{YS2}
\mathbb{E}\left[ \prod_{i = 1}^m h_i(\mathcal{L}_{n_i}(t_i))\right] = \mathbb{E}\left[ \prod_{i \not \in R} h_i(\mathcal{L}_{n_i}(t_i))  \cdot \mathbb{E}_{\operatorname{avoid}}^{a,b, \vec{x},\vec{y}, f,g} \left[ \prod_{i  \in R} h_i(\mathcal{Q}_{n_i-k_1+1}(t_i))   \right] \right], 
\end{equation}
where $R = \{i \in \llbracket 1, m \rrbracket: n_i \in \llbracket k_1 , k_2 \rrbracket, t_i \in [a,b]\}$. We have thus reduced our proof to showing (\ref{YS2}).\\

Since $\mathcal{L}$ satisfies the half-space Brownian Gibbs property, we have for any bounded Borel-measurable $G_1: C(\llbracket 1, k_2 \rrbracket \times [0,b]) \rightarrow \mathbb{R}$ and bounded $\mathcal{F}_{\operatorname{ext}} (\llbracket 1,k_2\rrbracket \times [0,b))$-measurable $Y$ (recall that this $\sigma$-algebra was given in Definition \ref{def:BGP})  
\begin{equation*}
\mathbb{E}\left[ G_1\left( \mathcal{L} \vert_{\llbracket 1, k_2 \rrbracket \times [0,b]} \right) \cdot Y \right] = \mathbb{E}\left[ \mathbb{E}_{\operatorname{avoid}}^{b,\vec{y}\,', \vec{\mu}_{k_2}, g'}  \left[ G_1(\mathcal{Q}_1, \dots, \mathcal{Q}_{k_2}) \right]\cdot Y \right],
\end{equation*}
where $\vec{y}\,' = (\mathcal{L}_{1}(b),\dots, \mathcal{L}_{k_2}(b))$, $\vec{\mu}_{k_2} = (\mu_1, \dots, \mu_{k_2})$, $g' = \mathcal{L}_{k_2+1}[0,b]$. Using the last equation and the monotone class argument in Lemma \ref{MCA}, we see that for any bounded Borel-measurable $G: C(\llbracket 1, k_2 + 1 \rrbracket \times [0,b]) \rightarrow \mathbb{R}$ and bounded $\mathcal{F}_{\operatorname{ext}} (\llbracket 1,k_2\rrbracket \times [0,b))$-measurable $Y$
\begin{equation}\label{YS3}
\mathbb{E}\left[ G\left( \mathcal{L}\vert_{\llbracket 1, k_2 + 1 \rrbracket \times [0,b]} \right) \cdot Y \right] = \mathbb{E}\left[ \mathbb{E}_{\operatorname{avoid}}^{b,\vec{y}\,', \vec{\mu}_{k_2}, g'}  \left[ G(\mathcal{Q}_1, \dots, \mathcal{Q}_{k_2}, g') \right]\cdot Y \right].
\end{equation}

We set $R' = \{i \in \llbracket 1, m \rrbracket: n_i \in \llbracket 1 , k_2 \rrbracket, t_i \in [0,b]\}$, and apply (\ref{YS3}) with 
$$G(\mathcal{Q}_{1}, \mathcal{Q}_2, \dots, \mathcal{Q}_{k_2+1}) =  \prod_{i  \in R'} h_i(\mathcal{Q}_{n_i}(t_i)) \mbox{, and } Y =  \prod_{i \not \in R'} h_i(\mathcal{L}_{n_i}(t_i)) $$
to get 
\begin{equation}\label{YS4}
\mathbb{E}\left[ \prod_{i = 1}^m h_i(\mathcal{L}_{n_i}(t_i)) \right] = \mathbb{E}\left[ \prod_{i \not \in R'} h_i(\mathcal{L}_{n_i}(t_i))  \cdot \mathbb{E}_{\operatorname{avoid}}^{b,\vec{y}\,', \vec{\mu}_{k_2},g'} \left[ \prod_{i  \in R'} h_i(\mathcal{Q}_{n_i}(t_i))   \right] \right].
\end{equation}
From Lemma \ref{LGPAll} we know that $\mathbb{P}_{\operatorname{avoid}}^{b,\vec{y}\,', \vec{\mu}_{k_2},g'}$ satisfies the partial Brownian Gibbs property, and so 
\begin{equation}\label{YS5}
\begin{split}
&\mathbb{E}_{\operatorname{avoid}}^{b,\vec{y}\,', \vec{\mu}_{k_2},g'} \left[ \prod_{i  \in R'} h_i(\mathcal{Q}_{n_i}(t_i))   \right] = \\
&\mathbb{E}_{\operatorname{avoid}}^{b,\vec{y}\,', \vec{\mu}_{k_2},g'} \left[ \prod_{i  \in R' \setminus R} h_i(\mathcal{Q}_{n_i}(t_i)) \cdot  \mathbb{E}_{\operatorname{avoid}}^{a,b, \vec{x}\,'',\vec{y}\,'', f'',g''} \left[ \prod_{i  \in R} h_i(\mathcal{Q}''_{n_i -k_1 + 1}(t_i)) \right] \right],
\end{split}
\end{equation}
where $R=\{i\in\llbracket1,m\rrbracket: n_i\in\llbracket k_1,k_2\rrbracket, t_i\in[a,b]\}$, and we have that $\vec{x}\,'' = (\mathcal{Q}_{k_1}(a), \dots, \mathcal{Q}_{k_2}(a))$, $\vec{y}\,'' = (\mathcal{Q}_{k_1}(b), \dots, \mathcal{Q}_{k_2}(b))$, $f'' = \mathcal{Q}_{k_1-1}[a,b]$ (with the convention $\mathcal{Q}_0 = \infty$), $g'' = \mathcal{Q}_{k_2 + 1}[a,b]$. Substituting (\ref{YS5}) into (\ref{YS4}) and applying (\ref{YS3}) with $Y$ as above and
$$G(\mathcal{Q}_{1}, \mathcal{Q}_2, \dots, \mathcal{Q}_{k_2+1}) =  \prod_{i  \in R' \setminus R} h_i(\mathcal{Q}_{n_i}(t_i)) \cdot  \mathbb{E}_{\operatorname{avoid}}^{a,b, \vec{x}\,'',\vec{y}\,'', f'',g''} \left[ \prod_{i  \in R} h_i(\mathcal{Q}''_{n_i -k_1 + 1}(t_i)) \right],$$
we arrive at (\ref{YS2}).
\end{proof}

%
%
\subsection{Proof of Lemma \ref{lem:monotone coupling}}\label{SectionA5}
Our proof is very similar to that of \cite[Lemma 3.1]{S}, which in turn uses ideas from \cite[Section 6]{CorHamA} and \cite[Section 6]{Wu23}.
Throughout the proof we write $\Omega^{\mathrm{b}}$ for $\Omega_{\ice}(T,\vec{y}\,^{\mathrm{b}},g^{\mathrm{b}})$ and $\Omega^{\mathrm{t}}$ for $\Omega_{\ice}(T,\vec{y}\,^{\mathrm{t}},g^{\mathrm{t}})$. 
We split the proof into two steps. In the first step, under the extra assumption $g^{\mathrm{b}}(0)\in\mathbb{Z}$, we construct two Markov chains, ordered with respect to one another, which have invariant measures $\mathbb{P}_{\ice,\operatorname{Geom} }^{ T, \vec{y}\,^{\mathrm{b}},\vec{q}, g^{\mathrm{b}}}$ and $\mathbb{P}_{\ice,\operatorname{Geom} }^{ T, \vec{y}\,^{\mathrm{t}},\vec{q}, g^{\mathrm{t}}}$. We also show that their laws converge weakly to these two measures. In the second step we use Skorohod's Representation Theorem to construct the monotone coupling in the statement of the lemma, first when $g^{\mathrm{b}}(0)\in\mathbb{Z}$ and then for general $g^{\mathrm{b}}$ by taking a weak limit of our monotone couplings for $g^{\mathrm{b},M}:=\max(-M,g^{\mathrm{b}})$ and  $g^{\mathrm{t},M}:=\max(-M,g^{\mathrm{t}})$ as $M\rightarrow\infty$.\\
	
\noindent\textbf{Step 1.} In this step we assume $g^{\mathrm{b}}(0)\in\mathbb{Z}$, so that $g^{\mathrm{b}}$ and $g^{\mathrm{t}}$ are increasing functions on $\llbracket 0, T \rrbracket$ taking values in $\mathbb{Z}$. We construct a Markov chain $(X^n,Y^n)_{n\geq 0}$ on $\Omega^{\mathrm{b}}\times\Omega^{\mathrm{t}}$ with the following properties:
\begin{enumerate}[label=(\arabic*)]
\item $(X^n)_{n\geq 0}$ and $(Y^n)_{n\geq 0}$ are both Markov chains with respect to their own filtrations;
\item $X_i^n(t) \leq Y_i^n(t)$ for all $t\in\llbracket 0,T\rrbracket$ and $i\in\llbracket 1,k\rrbracket$.	
\item $(X^n)$ is irreducible and aperiodic, with invariant measure  $\mathbb{P}_{\ice,\operatorname{Geom} }^{ T, \vec{y}\,^{\mathrm{b}},\vec{q}, g^{\mathrm{b}}}$;
\item $(Y^n)$ is irreducible and aperiodic, with invariant measure  $\mathbb{P}_{\ice,\operatorname{Geom} }^{ T, \vec{y}\,^{\mathrm{t}},\vec{q}, g^{\mathrm{t}}}$;
\end{enumerate}

We first define the initial distribution  
$$X^0_i(t) = y^{\mathrm{b}}_i, \quad Y^0_i(t) = y^{\mathrm{t}}_i\quad\mbox{for }t\in\llbracket0,T\rrbracket,\mbox{ }i\in\llbracket1,k\rrbracket.$$
It is easy to observe that $X^0 \in \Omega^{\mathrm{b}}$, $Y^0 \in \Omega^{\mathrm{t}}$ and $X_i^0(t) \leq Y_i^0(t)$ for all $t\in\llbracket 0,T\rrbracket$ and $i\in\llbracket 1,k\rrbracket$.	

We next define the dynamics of $(X^n,Y^n)_{n\geq0}$. We let 
$$U^{(i,t,\zeta)}\quad\mbox{for}\quad(i,t,\zeta)\in \llbracket 1,k \rrbracket \times \llbracket 0,T\rrbracket \times \{-1,1\}$$
be a family of i.i.d. random variables that are uniformly distributed on $[0,1]$. At time $n$, we first sample uniformly the triplet $(i,t,\zeta)$ in the finite set $\llbracket 1,k \rrbracket \times \llbracket 0,T\rrbracket \times \{-1,1\}$. We update $X^n$ to $X^{n+1}$ as follows. Define a candidate $\tilde{X}^n$ for $X^{n+1}$ by 
$$\tilde{X}^n_j(s) =\begin{cases} 
    \displaystyle X^n_j(s)+\zeta\quad &\mbox{if }j=i\mbox{ and }s=t,\\
    \displaystyle X^n_j(s)\quad &\mbox{otherwise}.
\end{cases}$$
In the case that $\tilde{X}^n \in \Omega^{\mathrm{b}}$ and
\begin{equation}\label{eq:condition update X}
R^{(i,t,\zeta)}_{X^n} := \frac{\mathbb{P}^{T, \vec{y}\,^{\mathrm{b}},\vec{q}}_{ \operatorname{Geom}}(\tilde{X}^n)}{\mathbb{P}^{T, \vec{y}\,^{\mathrm{b}},\vec{q}}_{ \operatorname{Geom}}(X^n)} \geq U^{(i,t,\zeta)},
\end{equation}
we set $X^{n+1} = \tilde{X}^n$. Otherwise we leave $X^{n+1} = X^n$ unchanged. We update $Y^n$ to $Y^{n+1}$ according to the analogous rule, with $R^{(i,t,\zeta)}_{X^n}$ replaced by the corresponding quantity $R^{(i,t,\zeta)}_{Y^n}$. 

It is easy to see that $(X^n,Y^n)_{n\geq0}$ defined above is indeed a Markov chain on the state space $\Omega^{\mathrm{b}}\times\Omega^{\mathrm{t}}$. Also, since $X^{n+1}$ depends only on $X^n$, we have that $(X^n)_{n\geq0}$ is a Markov chain with respect to its own filtration, and likewise for $(Y^n)_{n\geq0}$. This proves property (1) above.

Next, we show that $X_i^n \leq Y_i^n$ on $\llbracket 0,T\rrbracket$ for all $i\in\llbracket 1,k\rrbracket$ and $n\geq 0$. When $n=0$ this follows from our choice of $X^0$ and $Y^0$. Note that the update rule cannot move curves of $X^n$ and $Y^n$ in opposite directions, hence the only way the ordering can be violated at time $n+1$ is if we sample $(i,t,\zeta)$ and $X_i^n(t) = Y_i^n(t)$. Observe that if $X^n,\tilde{X}^n\in\Omega^{\mathrm{b}}$, we have $R^{(i,t,\zeta)}_{X^n}\geq1$ except for the special case $t=0$ and $\zeta=-1$, for which $R^{(i,t,\zeta)}_{X^n}=q_i$. An analogous statement holds for $Y^n$. In view of our update rule, the ordering is preserved at time $n+1$. This completes the proof of property (2).

To see that $(X^n)$ is irreducible (and likewise for $(Y^n)$), consider any $W \in \Omega^{\mathrm{b}}$. We define a procedure to update all the paths in $X^0$ downwards to $W$. Starting with the bottom path $X_k^0$, we move it downwards to $W_k$ from left to right on $\llbracket 0, T \rrbracket$. Specifically, in each step we move $X_k(t)$ downwards by $1$ until it reaches $W_k(t)$, starting from $t = 0$ to $t = T$. Next, we move the $(k-1)$-th path to $W_{k-1}$, and finally, we move the first path to $W_1$. By our order of updates, one can see that the interlacing property gets preserved. This procedure terminates after finitely many steps, and the probability of each step is positive; hence, the probability of reaching $W$ is positive.

To see that the Markov chains $(X^n)$ and $(Y^n)$ are aperiodic, observe that with positive probability, the triplet $(i, t, \zeta)$ is sampled as $(k, T, 1)$, and this update leaves both $X^n$ and $Y^n$ unchanged.

We next show that $(X^n)$ has invariant measure $\mathbb{P}_{\ice,\operatorname{Geom} }^{ T, \vec{y}\,^{\mathrm{b}},\vec{q}, g^{\mathrm{b}}} $, and a similar proof works for $(Y^n)$. For any $\omega,\tau\in\Omega^{\mathrm{b}}$, observe that $\mathbb{P}(X^{n+1} = \omega \,|\, X^n = \tau) > 0$ if and only if $\mathbb{P}(X^{n+1} = \tau \,|\, X^n = \omega) > 0$, since both inequalities are equivalent to $\omega$ and $\tau$ being differ only at one point $(i,t)$ by $1$. By the update rule \eqref{eq:condition update X}, one of these two probabilities must equal to $1$, hence 
$$\frac{\mathbb{P}(X^{n+1} = \omega \,|\, X^n = \tau)}{\mathbb{P}(X^{n+1} = \tau \,|\, X^n = \omega)} = \frac{\mathbb{P}^{T, \vec{y}\,^{\mathrm{b}},\vec{q}}_{ \operatorname{Geom}}(\omega)}{\mathbb{P}^{T, \vec{y}\,^{\mathrm{b}},\vec{q}}_{ \operatorname{Geom}}(\tau)}.$$
From the latter we get
\begin{equation*}
\begin{split}
&\sum_{\tau\in\Omega^{\mathrm{b}}} \mathbb{P}_{\ice,\operatorname{Geom} }^{ T, \vec{y}\,^{\mathrm{b}},\vec{q}, g^{\mathrm{b}}} (\tau)\cdot\mathbb{P}(X^{n+1} = \omega \,|\, X^n = \tau) 
\\
& =  \sum_{\tau\in\Omega^{\mathrm{b}}} \mathbb{P}_{\ice,\operatorname{Geom} }^{ T, \vec{y}\,^{\mathrm{b}},\vec{q}, g^{\mathrm{b}}} (\tau)\cdot \frac{\mathbb{P}^{T, \vec{y}\,^{\mathrm{b}},\vec{q}}_{ \operatorname{Geom}}(\omega)}{\mathbb{P}^{T, \vec{y}\,^{\mathrm{b}},\vec{q}}_{ \operatorname{Geom}}(\tau)}\cdot\mathbb{P}(X^{n+1} = \tau \,|\, X^n = \omega)\\
& =  \mathbb{P}_{\ice,\operatorname{Geom} }^{ T, \vec{y}\,^{\mathrm{b}},\vec{q}, g^{\mathrm{b}}} (\omega) \sum_{\tau\in\Omega^{\mathrm{b}}}\mathbb{P}(X^{n+1} = \tau \,|\, X^n = \omega) = \mathbb{P}_{\ice,\operatorname{Geom} }^{ T, \vec{y}\,^{\mathrm{b}},\vec{q}, g^{\mathrm{b}}} (\omega),
\end{split}
\end{equation*}
confirming that $(X_n)$ has invariant measure $\mathbb{P}_{\ice,\operatorname{Geom} }^{ T, \vec{y}\,^{\mathrm{b}},\vec{q}, g^{\mathrm{b}}} $. We have established properties (3) and (4).\\

\noindent\textbf{Step 2.} In this step we construct the probability space $(\Omega,\mathcal{F},\mathbb{P})$ and the random elements $\mathfrak{L}^{\mathrm{b}}$ and $\mathfrak{L}^{\mathrm{t}}$ as in the statement of the lemma. Assume first that $g^{\mathrm{b}}(0)\in\mathbb{Z}$. Using properties (3) and (4) in Step 1, and \cite[Theorem 1.8.3]{Norris}, we have
\begin{equation}\label{eq:monotone coupling weak convergence} X^n \Rightarrow \mathbb{P}_{\ice,\operatorname{Geom} }^{ T, \vec{y}\,^{\mathrm{b}}, \vec{q}, g^{\mathrm{b}} } \quad \mathrm{and} \quad Y^n \Rightarrow \mathbb{P}_{\ice,\operatorname{Geom} }^{ T, \vec{y}\,^{\mathrm{t}},\vec{q}, g^{\mathrm{t}}}.
\end{equation}
We conclude that $(X^n)$ and $(Y^n)$ are tight sequences, hence $(X^n,Y^n)$ is also tight, and there exists a sequence $n_m\uparrow \infty$ such that $(X^{n_m},Y^{n_m})$ converges weakly. By Skorohod's Representation Theorem, \cite[Theorem 6.7]{Billing}, there exists a probability space $(\Omega,\mathcal{F},\mathbb{P})$ supporting random elements $(\mathfrak{X}^{n_m}$, $\mathfrak{Y}^{n_m})$ and $(\mathfrak{L}^{\mathrm{b}},\mathfrak{L}^{\mathrm{t}})$ taking values in $\Omega^{\mathrm{b}}\times\Omega^{\mathrm{t}}$ respectively, such that
the law of $(\mathfrak{X}^{n_m},\mathfrak{Y}^{n_m})$ under $\mathbb{P}$ is the same as that of $(X^{n_m},Y^{n_m})$, and $(\mathfrak{X}^{n_m},\mathfrak{Y}^{n_m})$ converges $\mathbb{P}$-a.s. to $(\mathfrak{L}^{\mathrm{b}},\mathfrak{L}^{\mathrm{t}})$.
In view of \eqref{eq:monotone coupling weak convergence}, $\mathfrak{L}^{\mathrm{b}}$ has law $\mathbb{P}_{\ice,\operatorname{Geom} }^{ T, \vec{y}\,^{\mathrm{b}},\vec{q}, g^{\mathrm{b}}}$ and $\mathfrak{L}^{\mathrm{t}}$ has law $\mathbb{P}_{\ice,\operatorname{Geom} }^{ T, \vec{y}\,^{\mathrm{t}},\vec{q}, g^{\mathrm{t}}}$. Moreover,  property (2) in Step 1 implies  $\mathfrak{X}^{n_m}_i \leq \mathfrak{Y}^{n_m}_i$, $\mathbb{P}$-a.s., so $\mathfrak{L}^{\mathrm{b}}_i \leq \mathfrak{L}^{\mathrm{t}}_i$ for $i\in\llbracket 1,k\rrbracket$, $\mathbb{P}$-a.s, as desired.\\
 
In the remainder we assume $g^{\mathrm{b}}(0) = -\infty$. Let $\vec{y}\in\mathfrak{W}_k$ and $g:\llbracket0,T\rrbracket\rightarrow\mathbb{Z}\cup\{-\infty\}$ be an increasing function such that $g(T)\leq y_k$. For $M\in\mathbb{N}$, we consider the functions $g^M=\max(g,-M):\llbracket0,T\rrbracket\rightarrow\mathbb{Z}$ and observe that
\begin{equation}\label{eq:weak conv last step MCoupling}
    \mathbb{P}_{\ice,\operatorname{Geom}}^{T,\vec{y},\vec{q},g^M}\Rightarrow
    \mathbb{P}_{\ice,\operatorname{Geom}}^{T,\vec{y},\vec{q},g}\quad\mbox{as }M\rightarrow\infty.
\end{equation}
Indeed, if we fix $L^0_i \in \Omega(T, y_i)$ for $i \in \llbracket 1, k \rrbracket$, and $\mathfrak{L}^0 = \{L^0_i\}_{i = 1}^k$, we observe that 
\begin{equation*}
\begin{split}
&\lim_{M \rightarrow \infty} \mathbb{P}_{\ice,\operatorname{Geom}}^{T,\vec{y},\vec{q},g^M}(Q_i = L_i^0 \mbox{ for $i \in \llbracket 1, k \rrbracket$}) = \lim_{M \rightarrow \infty} \frac{\prod_{i = 1}^k q^{- L_i^0(0)} \cdot {\bf 1} \{\mathfrak{L}^0 \in \Omega_{\ice}(T, \vec{y}, g^M) \}}{\sum_{\mathfrak{L} \in \Omega_{\ice}(T, \vec{y}, g^M)} \prod_{i = 1}^k q^{- L_i(0)}}\\
& = \frac{\prod_{i = 1}^k q^{- L_i^0(0)} \cdot {\bf 1} \{\mathfrak{L}^0 \in \Omega_{\ice}(T, \vec{y}, g) \}}{\sum_{\mathfrak{L} \in \Omega_{\ice}(T, \vec{y}, g)} \prod_{i = 1}^k q^{- L_i(0)}} =  \mathbb{P}_{\ice,\operatorname{Geom}}^{T,\vec{y},\vec{q},g}(Q_i = L_i^0 \mbox{ for $i \in \llbracket 1, k \rrbracket$}),
\end{split}
\end{equation*}
where in going from the first to the second line we used the monotone convergence theorem.

Using (\ref{eq:weak conv last step MCoupling}), we have that as $M\rightarrow\infty$,
\begin{equation}\label{eq:bregregwe}
\mathbb{P}_{\ice,\operatorname{Geom}}^{T,\vec{y}\,^{\mathrm{b}},\vec{q},g^{\mathrm{b},M}}\Rightarrow
    \mathbb{P}_{\ice,\operatorname{Geom}}^{T,\vec{y}\,^{\mathrm{b}},\vec{q},g^{\mathrm{b}}}, \mbox{ and }
\mathbb{P}_{\ice,\operatorname{Geom}}^{T,\vec{y}\,^{\mathrm{t}},\vec{q},g^{\mathrm{t},M}}\Rightarrow
    \mathbb{P}_{\ice,\operatorname{Geom}}^{T,\vec{y}\,^{\mathrm{t}},\vec{q},g^{\mathrm{t}}},
\end{equation}
where $g^{\mathrm{b},M}=\max(g^{\mathrm{b}},-M)$ and $g^{\mathrm{t},M}=\max(g^{\mathrm{t}},-M)$ are $\mathbb{Z}$-valued functions. From the beginning of the step, there exist probability spaces $(\Omega^M,\mathcal{F}^M,\mathbb{P}^M)$ supporting $\mathfrak{L}^{\mathrm{b},M}$, $\mathfrak{L}^{\mathrm{t},M}$, such that 
$$\mbox{ $\mathfrak{L}^{\mathrm{b},M}\sim \mathbb{P}_{\ice,\operatorname{Geom}}^{T,\vec{y}\,^{\mathrm{b}},\vec{q},g^{\mathrm{b},M}}$,
$\mathfrak{L}^{\mathrm{t},M}\sim \mathbb{P}_{\ice,\operatorname{Geom}}^{T,\vec{y}\,^{\mathrm{t}},\vec{q},g^{\mathrm{t},M}}$
and $\mathbb{P}^M$-a.s. we have $\mathfrak{L}^{\mathrm{b},M}_i \leq \mathfrak{L}^{\mathrm{t},M}_i $ for all $i\in\llbracket 1,k\rrbracket$. }$$
From \eqref{eq:bregregwe} and Skorohod's representation theorem, we can find a subsequence $M_n \uparrow \infty$ and a probability space $(\Omega, \mathcal{F}, \mathbb{P})$ that supports a sequence of random elements $(\mathfrak{X}^{M_n}$, $\mathfrak{Y}^{M_n})$, that have the same laws as $(\mathfrak{L}^{\mathrm{b},M_n},\mathfrak{L}^{\mathrm{t},M_n})$,
and $(\mathfrak{L}^{\mathrm{b}},\mathfrak{L}^{\mathrm{t}})$, such that $(\mathfrak{X}^{M_n}$, $\mathfrak{Y}^{M_n})$ converges $\mathbb{P}$-a.s. to $(\mathfrak{L}^{\mathrm{b}},\mathfrak{L}^{\mathrm{t}})$. It follows from \eqref{eq:bregregwe} that 
$\mathfrak{L}^{\mathrm{b}}$ has law $\mathbb{P}_{\ice,\operatorname{Geom}}^{T,\vec{y}\,^{\mathrm{b}},\vec{q},g^{\mathrm{b}}}$ and 
 $\mathfrak{L}^{\mathrm{t}}$ has law $\mathbb{P}_{\ice,\operatorname{Geom}}^{T,\vec{y}\,^{\mathrm{t}},\vec{q},g^{\mathrm{t}}}$. Since $\mathfrak{L}^{\mathrm{b},M}_i \leq \mathfrak{L}^{\mathrm{t},M}_i $ for all $i\in\llbracket 1,k\rrbracket$, we conclude $\mathfrak{L}^{\mathrm{b}}_i\leq\mathfrak{L}^{\mathrm{t}}_i$ for all $i\in\llbracket 1,k\rrbracket$, $\mathbb{P}$-a.s. as well. This suffices for the proof.

%
%
\subsection{Proof of Lemma \ref{lem:RW}}\label{SectionA6} The proof we present is similar to that of \cite[Lemma 2.16]{dimitrov2024tightness}, and for clarity is split into two steps.\\
	
\noindent\textbf{Step 1.} In this step we prove that if $\ell_i^n$ have laws $\mathbb{P}^{B_n, Y_i^n,q_i^n}_{\operatorname{Geom}}$ for $i\in\llbracket 1,k\rrbracket$, and are independent, then the $C([0,B_n/d_n])$-valued random elements $Z_i^n$, defined by
\[
Z_i^n(t) = \sigma^{-1}d_n^{-1/2}(\ell_i^n(td_n)-utd_n), \quad t\in[0,B_n/d_n],
\] 
converge weakly to $\mathbb{P}^{b,y_i,\mu_i}_{\operatorname{free}}$, where, by an abuse of notation, we use $Z_i^n$ to denote also the $C([0,b])$-valued random elements by restriction. By independence, we also get that $Z^n:=(Z_1^n,\dots,Z_k^n)$ converges weakly to $\mathbb{P}^{b,\vec{y},\vec{\mu}}_{\operatorname{free}}$, i.e., to $k$ independent reverse Brownian motions with drifts. 

In the sequel we drop the subscript $i$, and suppose that $\ell^n$ have laws $\mathbb{P}^{B_n, Y^n,q^n}_{\operatorname{Geom}}$ with 
\begin{equation}\label{A61}
q^n = p - \mu \sqrt{p}(1-p) d_n^{-1/2} + o(d_n^{-1/2}) \mbox{ and } \lim_n \sigma^{-1}d_n^{-1/2} (Y^n - uB_n) = y,
\end{equation}
and set $Z^n(t) = \sigma^{-1}d_n^{-1/2}(\ell^n(td_n)-utd_n)$. By the definition of the probability law $\mathbb{P}^{B_n, Y^n,q^n}_{\operatorname{Geom}}$, there exist i.i.d. geometric random variables $H_{j}^n$ for $j\in\llbracket 1,B_n\rrbracket$ with parameter $q^n$ such that
\[
\ell^n(s)=Y^n-\sum_{r=s+1}^{B_n}H_{r}^n\quad\mbox{for }s\in\llbracket0,B_n\rrbracket.
\]
Therefore, for $s\in\llbracket0,B_n\rrbracket$ we have
\begin{equation}\label{eq:interpolate Y}
\begin{split}
        &Z^n\left(sd_n^{-1}\right)= \sigma^{-1}d_n^{-1/2}(\ell^n(s) - us)  
        = \sigma^{-1}d_n^{-1/2}\left(Y^n-\sum_{r=s+1}^{B_n}H_{r}^n-us \right)\\
        &=\sigma^{-1}d_n^{-1/2}\left(Y^n-uB_n-\sum_{r=s+1}^{B_n}\mathring{H}_{r}^n+(B_n-s)\left(\frac{p}{1-p} -\frac{q^n}{1- q^n}\right)\right)\\
        &=y-\sigma^{-1}d_n^{-1/2}\left(\sum_{r=s+1}^{B_n}\mathring{H}_{r}^n\right)+ \mu\left(b-s/d_n\right) +o(1),
    \end{split}    
\end{equation}
where in going from the first to the second line, we set $\mathring{H}_r^n = H^n_r - \mathbb{E}[H^n_r] = H^n_r - \frac{q^n}{1 - q^n}$, and in the last equality we used (\ref{A61}).

We next use the strong approximation result to approximate the third line above by a reverse Brownian motion. Using Lemma \ref{prop:ThmA Shao}, there is a probability space that supports $H_{j}^n$ for $j\in\llbracket 1,B_n\rrbracket$ and a sequence of i.i.d. normal random variables $N_{j}^n$ for $j\in\llbracket 1,B_n\rrbracket$ with 
\begin{equation}\label{eq:mean and variance of H}
\mathbb{E}\left[N_{1}^n\right]=0,\quad
\operatorname{Var}\left(N_{1}^n\right)=\frac{q^n}{\left(1-q^n\right)^2},
\end{equation}
such that for all large $n$
\begin{equation}\label{A62}
\mathbb{E}\left[\exp\left(\lambda A\max_{s\in\llbracket0,B_n\rrbracket}\left|\sum_{r=s+1}^{B_n}\mathring{H}_{r}^n-\sum_{r=s+1}^{B_n}N_{r}^n\right|\right)\right]\leq1+\frac{\lambda B_n q^n}{\left(1-q^n\right)^2},
\end{equation}
where $\lambda,A>0$ are constants that depend on $p$. By possibly enlarging our probability space, we may assume that it supports a standard Brownian motion $W_t$, so that 
\begin{equation}\label{A63}
 \frac{ \sqrt{q^n} }{1 - q^n} \cdot \sigma^{-1} \cdot  W_{(B_n-s)/d_n}=- \sigma^{-1}d_n^{-1/2}\left(\sum_{r=s+1}^{B_n}N_{r}^n\right) \quad\mbox{for }s\in\llbracket0,B_n\rrbracket.
\end{equation}

Define the $C([0,B_n/n])$-valued random elements $V^n$ by
\begin{equation}\label{eq:interpolate V}
   V^n\left(sd_n^{-1}\right) = y-\sigma^{-1}d_n^{-1/2}\left(\sum_{r=s+1}^{B_n} N_{r}^n\right)+ \mu\left(b-s/d_n\right)  
\end{equation}
for $s\in\llbracket0,B_n\rrbracket$ and by linear interpolation otherwise. Using (\ref{eq:interpolate Y}), (\ref{A62}), (\ref{eq:interpolate V}) and Chebyshev's inequality we conclude
\begin{equation}\label{A64}
\sup_{t\in[0,b]}\left|Z^n(t)-V^n(t)\right|\Rightarrow 0 \quad\mbox{as }n\rightarrow\infty.
\end{equation}
On the other hand, using (\ref{A63}) and the uniform continuity of the sample paths of a Brownian motion we have almost surely
\begin{equation}\label{A65}
\sup_{t\in[0,b]}\left|V^n(t)- \left(y + W_{b-t} + \mu (b-t) \right)\right|\rightarrow 0 \quad\mbox{as }n\rightarrow\infty.
\end{equation}
Equations (\ref{A64}), (\ref{A65}) and the convergence together theorem, see \cite[Theorem 3.1]{Billing}, imply that $Z^n \Rightarrow \mathbb{P}^{b,y,\mu}_{\operatorname{free}}$ as desired. We mention that in \cite[Theorem 3.1]{Billing} one needs to set $X_n = V^n$ and $Y_n = Z^n$.\\

\noindent\textbf{Step 2.} In this step we finish the proof of the lemma. By \eqref{EdgeLim}, \eqref{SideLim}, the continuity of $g$, and the fact that $g(b) < y_k$, we conclude that there exists $N_0 \in \mathbb{N}$ such that for $n \geq N_0$ we have $G_n(s) \leq Y_k^n$ for $s \in \llbracket 0, B_n \rrbracket$. In particular, the collection of constant paths $\{L_i: \llbracket 0,B_n \rrbracket \rightarrow \mathbb{Z} \}_{i = 1}^k$, defined by $L_i(s)=Y_i^n$ for $s\in\llbracket0,B_n\rrbracket$ and $i\in\llbracket1,k\rrbracket$, belongs to $\Omega_{\ice}(B_n,\vec{Y}^n,G_n)$ and so $\mathbb{P}_{ \operatorname{Geom}}^{B_n, \vec{Y}^n,\vec{q}^n}$ is well-defined for all $n \geq N_0$.

 Write $\Sigma=\llbracket1,k\rrbracket$ and $\Lambda=[0,b]$. In view of \cite[Theorem 2.1]{Billing} we only need to show that for any bounded continuous function $F : C(\Sigma\times\Lambda)\to\mathbb{R}$, we have
\begin{equation}\label{eq:avoid weak}
	\lim_{n\to\infty} \ex[F(\mathcal{Q}^n)] = \ex[F(\mathcal{Q})].
\end{equation}
We define the functions $\chi,\chi^n:C(\Sigma\times\Lambda)\to\mathbb{R}$ by
\begin{align*}
\chi(\mathcal{L}) &= \mathbf{1}\{ \mathcal{L}_1(t) > \cdots > \mathcal{L}_k(t) > g(t) \mbox{ for } t\in\Lambda\},\\
\chi^n(\mathcal{L}) &= \mathbf{1}\left\{ \mathcal{L}_i\left((r-1)/d_n\right)\geq \mathcal{L}_{i+1}\left(r /d_n \right)+ u\sigma^{-1}d_n^{-1/2}  \mbox{ for }i\in\llbracket1,k\rrbracket\mbox{ and }r\in\llbracket1,B_n\rrbracket\right\},
\end{align*}
where in the definition of $\chi^n(\mathcal{L})$ we use the convention $\mathcal{L}_{k+1}(t)=g_n(t)$ for $t\in[0,B_n/d_n]$. 
Additionally, in the definition of $\chi^n(\mathcal{L})$ above, since $\mathcal{L}_{i}(B_n/d_n)$ are undefined if $B_n/d_n>b$ ($\mathcal{L}_i$ are continuous functions on $[0,b]$), we use the convention by linearity
\[
\mathcal{L}_{i}\left( B_n/d_n \right):=\frac{\mathcal{L}_i(b)-\mathcal{L}_i\left( (B_n-1)/d_n\right)(B_n-bd_n)}{bd_n-B_n+1},\quad i\in\llbracket1,k+1\rrbracket.
\]
By our definition of $Z^n$ in Step 1 and $N_0$ above, we have for $n\geq N_0$,
\begin{equation}\label{eq:avoid condition}
\ex[F(\mathcal{Q}^n)] = \frac{\ex[F(Z^n)\chi^n(Z^n)]}{\ex[\chi^n(Z^n)]}.
\end{equation}
If $\mathcal{L}$ has law $\mathbb{P}_{\operatorname{free}}^{b,\vec{y},\vec{\mu}}$, we also have
\begin{equation}\label{eq:backward Brownian avoid condition}
\ex[F(\mathcal{Q})] = \frac{\ex[F(\mathcal{L})\chi(\mathcal{L})]}{\ex[\chi(\mathcal{L})]}.
\end{equation}
In view of \eqref{eq:avoid condition} and \eqref{eq:backward Brownian avoid condition}, to prove the convergence \eqref{eq:avoid weak} we only need to show that for any bounded continuous function $F:C(\Sigma\times\Lambda)\to\mathbb{R}$, we have
\begin{equation}\label{eq:to prove end conv}
\lim_{n\to\infty}\ex[F(Z^n)\chi^n(Z^n)] = \ex[F(\mathcal{L})\chi(\mathcal{L})].
\end{equation}
By Step 1, we have $Z^n \Rightarrow \mathcal{L}$ as $n\to\infty$. In view of Skorohod's Representation Theorem, \cite[Theorem 6.7]{Billing}, there exists a probability space $(\Omega,\mathcal{F},\mathbb{P})$ supporting $C(\Sigma\times\Lambda)$-valued random elements $\mathcal{Z}^n$ with the same laws as $Z^n$, together with a $C(\Sigma\times\Lambda)$-valued random element $\mathcal{L}$ with law $\mathbb{P}^{b,\vec{y}, \vec{\mu}}_{\operatorname{free}}$, such that $\mathcal{Z}^n(\omega) \to \mathcal{L}(\omega)$ uniformly as $n\to\infty$ for all $\omega\in\Omega$. 

Consider the events
\begin{align*}
A &= \{\omega\in\Omega : \mathcal{L}_1(\omega) > \cdots > \mathcal{L}_k(\omega) > g \mbox{ on } [0,b]\},\\
C &= \{\omega\in\Omega : \mathcal{L}_i(\omega)(r) < \mathcal{L}_{i+1}(\omega)(r) \mbox{ for some } i\in\llbracket 1,k\rrbracket \mbox{ and } r\in[0,b]\},
\end{align*}
where in the definition of $C$ we use the convention $\mathcal{L}_{k+1} = g$. In view of the continuity of $F$ and that $g_n\rightarrow g$ uniformly on $[0,b]$, we have $F(\mathcal{Z}^n)\chi^n(\mathcal{Z}^n) \to F(\mathcal{L})$ on the event $A$, and $F(\mathcal{Z}^n)\chi^n(\mathcal{Z}^n)\to 0$ on the event $C$. 

We claim that
\begin{equation}\label{Eq.AC1}
\mathbb{P}(A\cup C) = 1.
\end{equation}
If true, then $\mathbb{P}$-a.s. we have $F(\mathcal{Z}^n)\chi^n(\mathcal{Z}^n) \to F(\mathcal{L})\chi(\mathcal{L})$. The bounded convergence theorem implies \eqref{eq:to prove end conv}, which completes the proof of \eqref{eq:avoid weak}.\\

Let us quickly verify (\ref{Eq.AC1}). By definition, we have $A^c \cap C^c \subseteq \cup_{j = 1}^{k}E_j$, where 
$$E_j = \{\omega\in\Omega : \mathcal{L}_j(\omega) \geq \mathcal{L}_{j+1}(\omega) \mbox{ on } [0,b] \mbox{ and  } \mathcal{L}_j(\omega)(r) = \mathcal{L}_{j+1}(\omega)(r) \mbox{ for some $r \in [0,b]$} \}.$$
If we condition on $\mathcal{L}_{j+1}$, we can apply Lemma \ref{lem: BB touch intersect} (with $\mathcal{L}_{j+1}$ playing the role of $f$) to conclude that $\mathbb{P}(E_j\vert \mathcal{L}_{j+1}) = 0$. By the tower property, $\mathbb{P}(E_j) = 0$, and so $\mathbb{P}(A^c \cap C^c) = 0$, implying (\ref{Eq.AC1}).

%
%
\section{Results for point processes} \label{SectionB} In this section we give the proofs of various results from Section \ref{Section5} after introducing some notation and statements in Section \ref{SectionB1}. We continue with the same notation as in Sections \ref{Section5.1} and \ref{Section5.2}.

%
%
\subsection{Auxiliary results}\label{SectionB1} We recall from Section \ref{Section5.1} that $(E,\mathcal{E}) = (\mathbb{R}^k, \mathcal{R}^k)$, $M$ is a point process on $(E,\mathcal{E})$, $M_n$ are as in \eqref{eq: Mn} and their means are $\mu_n$. From \cite[(2.3)]{dimitrov2024airy} we have for each $n_1, \dots, n_m \in \mathbb{N}$ and bounded pairwise disjoint $A_1, \dots, A_m \in \mathcal{E}$,
\begin{equation} \label{eq: factorial moment correlation function}
\mu_n(A_1^{n_1}\times\cdots\times A_m^{n_m})=\mathbb{E}\left[\prod_{j=1}^m\frac{M(A_j)!}{(M(A_j)-n_j)!}\right],
\end{equation}
where both sides are allowed to be infinite, and we use the convention that $\frac{1}{r!} = 0$ for $r < 0$. We mention that if $Z$ is a random variable taking value in $\mathbb{Z}_{\geq 0}$, then 
\[ 
\mathbb{E}\left[\frac{Z!}{(Z-n)!}\right]=\mathbb{E}\left[Z(Z-1)\cdots(Z-n+1)\right]
\]
is called the $n$-th {\em factorial moment}. In particular, \eqref{eq: factorial moment correlation function} shows that the joint (factorial) moments of random variables $M(A_1),\dots,M(A_m)$ are completely described by the mean measures $\mu_n$. If $M$ is a Pfaffian point process with correlation kernel $K$ and reference measure $\lambda$ as in Definition \ref{def:def of Pfaffian point process}, and $n = n_1 + \cdots +n_m$, then (\ref{eq: factorial moment correlation function}) shows that 
\begin{equation} \label{eq:FactPfaff}
\mathbb{E}\left[\prod_{j=1}^m\frac{M(A_j)!}{(M(A_j)-n_j)!}\right] = \int_{A_1^{n_1}\times\cdots\times A_m^{n_m}}\mathrm{Pf}\left[K(x_i,x_j) \right]_{i,j=1}^n \lambda^n(dx).
\end{equation}

We now seek to establish a converse to the above statement, for which we require a bit of notation. As in \cite{kallenberg2017random}, a family of sets $\mathcal{C}$ is called {\em dissecting} if
\begin{itemize}
    \item every open $G \subseteq E$ is a countable union of sets in $\mathcal{C}$,
    \item every bounded Borel $B \in \mathcal{E}$ is covered by finitely many sets in $\mathcal{C}$,
\end{itemize}
and it is a {\em semi-ring} if
\begin{itemize}
    \item it is closed under finite intersections,
    \item any proper difference between sets in $\mathcal{C}$ is a finite disjoint union of sets in $\mathcal{C}$.
\end{itemize}
For a countable set $T \subset \mathbb{R}$, we define
\begin{equation}\label{eq:rectangle}
\mathcal{I}_T = \{ (a_1, b_1] \times \cdots \times (a_k, b_k]: a_i, b_i \not \in T \mbox{ for } i = 1, \dots, k \},
\end{equation}
which is the set of rectangles in $E$ whose corner coordinates avoid the set $T$. One readily observes that $\mathcal{I}_T$ is a dissecting semi-ring.

\begin{lemma}
\label{lem:correlation function} Let $M$ be a point process on $E$, $\mathcal{I} \subset \mathcal{E}$ a dissecting semi-ring, $\lambda$ a locally finite measure on $(E,\mathcal{E})$, and $K:E^2\rightarrow\mathrm{Mat}_2(\mathbb{C})$ a skew-symmetric locally bounded measurable function. Suppose that for each $n_1, \dots, n_m \in \mathbb{N}$, $n = n_1 + \cdots + n_m$ and bounded pairwise disjoint $A_1, \dots, A_m \in \mathcal{I}$
\begin{equation}\label{RN3}
 \mathbb{E} \left[ \prod_{j = 1}^m \frac{M(A_j)!}{(M(A_j) - n_j)!} \right] = \int_{A_1^{n_1} \times \cdots \times A_m^{n_m}}  \operatorname{Pf} \left[ K(x_i, x_j) \right]_{i, j = 1}^n  \lambda^n(dx) < \infty.
\end{equation}
Then, $M$ is a Pfaffian point process with correlation kernel $K$ and reference measure $\lambda$.
\end{lemma}
\begin{proof} This is a special case of \cite[Lemma 2.8]{dimitrov2024airy}, applied to $u_n(x_1, \dots, x_n) = \operatorname{Pf} \left[ K(x_i, x_j) \right]_{i, j = 1}^n$. We mention that equation (\ref{RN3}) implies that $u_n$ is $\lambda^n$-a.e. non-negative.
\end{proof}

We end this section with a few known properties and estimates for Pfaffians. From \cite[(B.3)]{OQR17} 
\begin{equation}\label{eq:conjugation of Pfaffian}
    \mathrm{Pf}(RAR^T)=\det(R)\mathrm{Pf}(A),
\end{equation}
for any $2n\times 2n$ skew-symmetric matrix $A$ and $2n \times 2n$ matrix $R$. We also recall {\em Hadamard's inequality}, see \cite[Corollary 33.2.1.1.]{Prasolov}, which states that if $B$ is an $n \times n$ matrix with complex entries and columns $v_1, \dots, v_n$, then
\begin{equation}\label{eq: Hadamard}
 |\det B| \leq \prod_{i = 1}^n \|v_i\|,
\end{equation}
where $\|x\| = (|x_1|^2 + \cdots + |x_n|^2)^{1/2}$ for $x = (x_1, \dots, x_n) \in \mathbb{C}^n$. Combining Hadamard's inequality with $\mathrm{Pf}(A)^2 = \det(A)$, see (\ref{eq: det to Pfaf}), we conclude that if $|A_{ij}| \leq C$ for some constant $C \geq 0$, then
\begin{equation}\label{eq: HadamardPfaffian}
\left|\mathrm{Pf}(A) \right| \leq C^{n} \cdot (2n)^{n/2}.
\end{equation}

%
%
\subsection{Proof of Lemma \ref{lem:technical lemma fdd 1}}\label{SectionB2} To prove that for all $t\in\mathbb{R}$ and $a \in\mathbb{Z}_{\geq 0}$
\begin{equation}\label{eq:conv lem1 fdd again}
\lim_{N\rightarrow\infty}\mathbb{P}\left(M^N([t,\infty))\leq a \right)=\mathbb{P}\left(M^{\infty}([t,\infty))\leq a\right),
\end{equation}
we only need to show that for any strictly increasing sequence $N_m \in \mathbb{N}$, there exists a subsequence $N_{m_n}$ such that this convergence holds. Let us fix an arbitrary strictly increasing sequence $N_m$. Since $M^N$ converge weakly to $M^{\infty}$, the sequence $M^{N_m}$ is tight. Notice that $(X_1^{N_m})^+$ is also tight, hence $(M^{N_m},(X_1^{N_m})^+)$ is tight, and there exists a subsequence $N_{m_n}$ such that $(M^{N_{m_n}},(X_1^{N_{m_n}})^+)$ converge weakly. For simplicity of the notation, we replace $N_{m_n}$ with $N$, assume that $(M^N,(X_1^N)^+)$ converge weakly, and proceed to prove (\ref{eq:conv lem1 fdd again}) for fixed $t \in \mathbb{R}$ and $a \in \mathbb{Z}_{\geq 0}$. 

By the Skorohod representation theorem, see \cite[Theorem 6.7]{Billing}, one can define the random variables $(M^N,(X_1^N)^+)$ and $M^{\infty}$ on a richer probability space $(\Omega,\mathcal{F},\mathbb{P})$, such that $(M^N,(X_1^N)^+)$ converges to $(M^{\infty},X^{\infty})$ $\mathbb{P}$-a.s., where $X^{\infty}$ is a real-valued random variable on $(\Omega,\mathcal{F},\mathbb{P})$. We denote by $E_t$ the set of $\omega\in\Omega$ such that the following conditions hold:
\begin{enumerate}
    \item [$\bullet$] $\lim_N(X_1^N)^+(\omega) = X^{\infty}(\omega)$; 
    \item [$\bullet$] for each $N\in\mathbb{N}$, the support of the measure $M^N(\omega)$ is contained in $(-\infty,(X_1^N)^+(\omega)]$; 
    \item [$\bullet$] $\lim_N M^N(\omega) = M^{\infty}(\omega)$, i.e. the sequence of measures $M^N(\omega)$ converges vaguely to $M^{\infty}(\omega)$;
    \item [$\bullet$] $M^{\infty}(\omega)(\{t\})=0$.
\end{enumerate}
By our assumptions, $E_t$ is measurable and $\mathbb{P}(E_t)=1$. To prove \eqref{eq:conv lem1 fdd again}, it suffices to show that for any $t\in\mathbb{R}$ and $a\in\mathbb{Z}_{\geq0}$
\begin{equation}\label{eq:vsevferge}
\{\omega:M^N(\omega)([t,\infty))\leq a\}\cap E_t\rightarrow\{\omega:M^{\infty}(\omega)([t,\infty))\leq a\}\cap E_t\quad\mbox{ as }N\rightarrow\infty.
\end{equation}
We split the proof of this convergence (of sets) into two steps.\\
  
\noindent\textbf{Step 1.}  Suppose that $\omega\in E_t$ satisfies   $M^{\infty}(\omega)([t,\infty))>a$.
We aim to show that $M^{N}(\omega)([t,\infty))>a$ for sufficiently large $N$. 

In view of $M^{\infty}(\omega)([t,\infty))\in\mathbb{Z}_{\geq0}\cup\{\infty\}$, we have $M^{\infty}(\omega)([t,\infty))\geq a+1$.  
We next choose a sequence of functions $f_n\in C_c(\mathbb{R})$ satisfying $\mathbf{1}_{[t+2^{-n},t+n]}\leq f_n\leq \mathbf{1}_{[t+2^{-n-1},t+n+2^{-n}]}$. In particular, $f_1\leq f_2\leq\cdots$ and $f_n\uparrow\mathbf{1}_{(t,\infty)}$. By the monotone convergence theorem, we have $M^{\infty}(\omega)f_n\uparrow M^{\infty}(\omega)\mathbf{1}_{(t,\infty)}=M^{\infty}(\omega)((t,\infty))=M^{\infty}(\omega)([t,\infty))\geq a+1$. Therefore, there exists $n_0\in\mathbb{N}$ such that  $M^{\infty}(\omega)f_{n_0}\geq a+2/3$. Since $M^N(\omega)$ converges vaguely to $M^{\infty}(\omega)$, we have $M^N(\omega) f_{n_0} \rightarrow M^{\infty}(\omega) f_{n_0} \geq a+2/3$. Hence, there exists $N_0\in\mathbb{N}$ such that for all $N\geq N_0$, 
$$M^N(\omega)[t,\infty) \geq M^N(\omega)f_{n_0}\geq a+1/3>a.$$
\noindent\textbf{Step 2.} Suppose that $\omega\in E_t$ satisfies   $M^{\infty}(\omega)([t,\infty))\leq a$. We aim to show that $M^{N}(\omega)([t,\infty))\leq a$ for sufficiently large $N$. 

We choose a sequence of functions $f_n\in C_c(\mathbb{R})$ satisfying 
$$\mathbf{1}_{[t-2^{-n-1},X_{\infty}(\omega)+1+2^{-n-1}]}\leq f_n\leq \mathbf{1}_{[t-2^{-n},X_{\infty}(\omega)+1+2^{-n}]}.$$
In particular, $\mathbf{1}_{[t-1,X_{\infty}(\omega)+2]}\geq f_1\geq f_2\geq \cdots $ and $f_n\downarrow\mathbf{1}_{[t,X_{\infty}(\omega)+1]}$. As the measure $M^{\infty}(\omega)$ is locally bounded, we have $M^{\infty}(\omega)\mathbf{1}_{[t-1,X_{\infty}(\omega)+2]}<\infty$. By the dominated convergence theorem, we have
$M^{\infty}(\omega)f_n\rightarrow M^{\infty}(\omega)\mathbf{1}_{[t,X_{\infty}(\omega)+1]}\leq M^{\infty}(\omega)([t,\infty))\leq a$. Therefore, there exists $n_1\in\mathbb{N}$ such that $M^{\infty}(\omega)f_{n_1}\leq a+1/3$. Since $M^N(\omega)$ converges vaguely to $M^{\infty}(\omega)$, we have $M^N(\omega) f_{n_1} \rightarrow M^{\infty}(\omega) f_{n_1} \leq a+1/3$. Hence, there exists $N_1\in\mathbb{N}$ such that for all $N\geq N_1$, 
\begin{equation}\label{eq:vaervrvergte}
M^N(\omega)([t,X_{\infty}(\omega)+1])\leq M^N(\omega)\mathbf{1}_{[t-2^{-n_1-1},X_{\infty}(\omega)+1+2^{-n_1-1}]}\leq M^N(\omega)f_{n_1}\leq a+2/3.
\end{equation}
Note that $(X_1^N)^+(\omega)$ converges to $X^{\infty}(\omega)$, and so there exists $N_2\in\mathbb{N}$ such that for $N\geq N_2$, we have $(X_1^N)^+(\omega)\leq X^{\infty}(\omega)+1$. Combining the latter with the fact that the support of $M^N(\omega)$ is contained in $(-\infty,(X_1^N)^+(\omega)]$ and (\ref{eq:vaervrvergte}), we conclude that for any $N\geq\max(N_1,N_2)$
$$
M^N(\omega)([t,\infty))=M^N(\omega)([t,X_{\infty}(\omega)+1])\leq a+2/3.$$
Since $M^N(\omega)([t,\infty))\in\mathbb{Z}_{\geq0}\cup\{\infty\}$, we have $M^N(\omega)([t,\infty))\leq a$. This completes Step 2.

Combining Step 1 and Step 2 above, we conclude \eqref{eq:vsevferge} and hence the lemma.

%
%
\subsection{Proof of Lemma \ref{lem:technical lemma fdd 2}}\label{SectionB3}
We choose a sequence $f_n\in C_c(\mathbb{R})$, satisfying $0\leq f_1\leq f_2\leq\cdots$ and $f_n\uparrow\mathbf{1}_{\mathbb{R}}$. By the monotone convergence theorem,   we have $M^{\infty}(\omega)f_n\uparrow M^{\infty}(\omega)(\mathbb{R})=\infty$. Therefore, for any $a\in\mathbb{Z}_{\geq0}$ and $\varepsilon>0$, there exists $n_0\in\mathbb{N}$ such that $\mathbb{P}(M^{\infty}f_{n_0}\geq a+1)\geq1-\varepsilon$. Since the space $\mathcal{M}_{\mathbb{R}}$ of locally bounded measures on $\mathbb{R}$ is equipped with the vague topology, for any fixed $f\in C_c(\mathbb{R})$, the mapping $\mu \mapsto \mu f$ is continuous on $\mathcal{M}_{\mathbb{R}}$. By the continuous mapping theorem and our assumption that $M^N$ weakly converges to $M^{\infty}$, we conclude that $M^Nf_{n_0}$ converges weakly to $M^{\infty}f_{n_0}$. Hence, there exists $N_0\in\mathbb{N}$ such that for any $N\geq N_0$,
$$1-2\varepsilon\leq\mathbb{P}\left( M^Nf_{n_0}\geq a\right)\leq\mathbb{P}\left( M^N(\mathbb{R})\geq a\right)=p_N^a.$$
As $p_N^a$ is decreasing in $N$, we conclude $p_N^a\geq 1-2\varepsilon$ for all $N\in\mathbb{N}$. Taking $\varepsilon\rightarrow0+$ we conclude $p_N^a=1$ for each $N\in\mathbb{N}$ and $a\in\mathbb{Z}_{\geq0}$, which implies $\mathbb{P}(M^{N}(\mathbb{R}) = \infty) = 1$.

%
%
\subsection{Proof of Proposition \ref{prop:basic properties Pfaffian point process}}\label{SectionB4}
The proofs of parts (1), (2) and (5) are very similar to the proofs of the corresponding parts of \cite[Proposition 2.13]{dimitrov2024airy}, with the determinants replaced by Pfaffians, and so we omit them.

We now turn to part (3). Fix a bounded Borel set $B$. Using that $K$ and $\lambda$ are locally bounded, we have for some $C_1, C_2 > 0$
$$\lambda(B) \leq C_1 \mbox{ and } |K(x,y)| \leq C_2 \mbox{ for }x,y \in B.$$
From (\ref{eq: HadamardPfaffian}), we conclude there exists a constant $ C>0$, such that for $n\geq1$
\begin{equation}\label{eq:Hadamard for Pfaffian}
\int_{B^n}\left|\mathrm{Pf}[K(x_i, x_j)]_{i,j=1}^n\right|\lambda^n(dx)\leq \int_{B^n}C_2^n (2n)^{n/2}\lambda^n(dx) \leq C_1^n C_2^n (2n)^{n/2} \leq  C^n n^{n/2}.
\end{equation}
The result now follows from the last inequality, (\ref{eq:FactPfaff}) and \cite[Corollary 2.4]{dimitrov2024airy}.

For part (4), we use \eqref{eq:conjugation of Pfaffian} with 
$$A=[K(x_i,x_j)]_{i,j=1}^n \mbox{ and } R=\mathrm{diag}(f(x_1),1/f(x_1),\dots,f(x_n),1/f(x_n)) $$ 
and get $\mathrm{Pf}[K(x_i,x_j)]_{i,j=1}^n=\mathrm{Pf}[\tilde{K}(x_i,x_j)]_{i,j=1}^n$, which concludes the proof. 

For part (6) we use again \eqref{eq:conjugation of Pfaffian} with
$$A=[K(x_i,x_j)]_{i,j=1}^n \mbox{ and } R=\mathrm{diag}(c_1,c_2,\dots,c_1,c_2), $$ 
and obtain $\mathrm{Pf}[\tilde{K}(x_i,x_j)]_{i,j=1}^n=(c_1c_2)^n\mathrm{Pf}[K(x_i,x_j)]_{i,j=1}^n$. Using the latter, the definition of $\tilde{\lambda}$ and (\ref{eq:FactPfaff}) we conclude
$$\mathbb{E}\left[\prod_{i = 1}^m\frac{M(A_i)!}{(M(A_i)-n_i)!}\right]=\int_{A_1^{n_1}\times\cdots\times A_m^{n_m}}\mathrm{Pf}\left[\tilde{K}(x_i,x_j) \right]_{i,j=1}^n\tilde{\lambda}^n(dx),$$
for any pairwise disjoint bounded Borel sets  $A_1,\dots,A_m $ and $n_1, \dots, n_m \in \mathbb{N}$. We conclude the proof in view of Lemma \ref{lem:correlation function}, where $\mathcal{I}$ is the collection of bounded Borel sets in $E$.

%
%
\subsection{Proof of Proposition \ref{prop: conv of Pfaffian point processes 0}}\label{SectionB5}
The proof is very similar to \cite[Proposition 2.15]{dimitrov2024airy} and so we will be brief. Since $\lambda$ is the vague limit of $\lambda_n\in S$, we have $\lambda \in S$, i.e. $\lambda$ is locally bounded. For $i=1,\dots,k$ we let $\pi_i: E \rightarrow \mathbb{R}$ be given by $\pi_i(x) = x_i$ for $x=(x_1,\dots,x_k)\in E$ and
\[
H_i^n=\left\{y\in\mathbb{R}:\lambda\left(\pi_i^{-1}(y)\cap[-n, n ]^k\right)>1/n\right\}.
\]
Since $\lambda([-n,n]^k) < \infty$, each $H_i^n$ is finite for $i=1,\dots,k$ and $n\in\mathbb{N}$. Let $T=\cup_{i=1}^k\cup_{n\geq 1}H_i^n$ which is a countable subset of $\mathbb{R}$. For any $t\notin T$, we have $\lambda\left(\pi_i^{-1}(t)\right)=0$, since $\lambda\left(\pi_i^{-1}(t)\cap[-n,n]^k\right)\leq1/n$ for each $n\in\mathbb{N}$ and $i=1,\dots,k$. 
We next fix pairwise disjoint $A_1,\dots,A_m\in\mathcal{I}_T$ and $n_1,\dots,n_m\in\mathbb{N}$ with $n=n_1+\cdots+n_m$. Similarly to \cite[Equations (A.20), (A.22) and (A.23)]{dimitrov2024airy}, we have
\begin{equation}\label{GH1}
    \begin{split}
        \lim_{N\rightarrow\infty}\mathbb{E}\left[\prod_{i=1}^m\frac{M^N(A_i)!}{(M^N(A_i)-n_i)!}\right] 
        &=\lim_{N\rightarrow\infty}\int_{A_1^{n_1}\times\cdots\times A_m^{n_m}}\mathrm{Pf}\left[K^N(x_i, x_j)\right]_{i,j=1}^n\lambda_N^n(dx)\\
        &=\lim_{N\rightarrow\infty}\int_{A_1^{n_1}\times\cdots \times A_m^{n_m}}\mathrm{Pf}\left[K(x_i,x_j)\right]_{i,j = 1}^n\lambda_N^n(dx)\\
        &=\int_{A_1^{n_1}\times\cdots\times A_m^{n_m}}\mathrm{Pf}\left[K(x_i,x_j)\right]_{i,j=1}^n\lambda^n(dx).
    \end{split}
\end{equation}
The first equality in (\ref{GH1}) uses (\ref{eq:FactPfaff}) and that $M^N$ is Pfaffian with correlation kernel $K^N$ and reference measure $\lambda_N$. The second equality uses that $K^N(x,y)$ converge uniformly over bounded sets to $K(x,y)$, and that 
$$
\limsup_{N\rightarrow\infty}\lambda_N^n(A_1^{n_1}\times\cdots\times A_m^{n_m})\leq\prod_{i=1}^m\lambda(\overline{A_i})^{n_i}<\infty,$$
which comes from \cite[Lemma 4.1]{kallenberg2017random}. The third equality uses the continuity of $\mathrm{Pf}\left[K(x_i, x_j)\right]_{i,j = 1}^n$ (which was assumed in the proposition) combined with \cite[Lemma A.5 and Equation (A.24)]{dimitrov2024airy}. 

Equation (\ref{GH1}) verifies \cite[(2.6)]{dimitrov2024airy}. In addition,  since $K$ and $\lambda$ are locally bounded, for each $B\in\mathcal{I}_T$ there exists a constant $C\in(0,\infty)$ such that the bound \eqref{eq:Hadamard for Pfaffian} holds, which verifies \cite[(2.8)]{dimitrov2024airy}. From \cite[Proposition 2.2]{dimitrov2024airy} we conclude that there is a countable $T^{\infty}\subset\mathbb{R}$ with $T\subseteq T^{\infty}$, and a point process $M^{\infty}$, such that $M^N$ converge weakly to $M^{\infty}$ and for all pairwise disjoint $A_1,\dots,A_m\in\mathcal{I}_{T^{\infty}}$
$$
\mathbb{E}\left[\prod_{i=1}^m\frac{M^{\infty}(A_i)!}{(M^{\infty}(A_i)-n_i)!}\right]=\int_{A_1^{n_1}\times\cdots\times A_m^{n_m}}\mathrm{Pf}\left[K(x_i, x_j)\right]_{i,j=1}^n\lambda^n(dx).$$
From Lemma \ref{lem:correlation function} we conclude that $M^{\infty}$ satisfies the conditions of the proposition.

%
%
\subsection{Proof of Lemma \ref{lem:LemmaSlice}}\label{SectionB6}
Using the definition of $M^t$ and (\ref{eq:FactPfaff}), we have for all pairwise disjoint bounded Borel sets $A_1,\dots, A_m$ in $\mathbb{R}$ and $n_1,\dots,n_m\geq0$ with $n_1+\cdots+n_m=n$ that
$$\mathbb{E}\left[\prod_{j=1}^m\frac{M^t(A_j)!}{(M^t(A_j)-n_j)!}\right]=\int_{A_1^{n_1}\times\cdots\times A_m^{n_m}}\mathrm{Pf}\left[K(t,x_i;t,x_j)\right]_{i,j=1}^n  \nu_t^n(dx). $$
We conclude the proof by Lemma \ref{lem:correlation function}, where $\mathcal{I}$ is the collection of bounded Borel sets in $\mathbb{R}$.

%
%
\subsection{Proof of Proposition \ref{prop: conv of Pfaffian point processes 1}}\label{SectionB7}
The proof is similar to the proof of \cite[Proposition 2.18]{dimitrov2024airy}, although the assumptions are somewhat different. For example, we do not assume that the limiting kernel $K$ is continuous. For clarity we split the proof into two steps.\\

{\bf \raggedleft Step 1.} Let $\mathcal{I}$ be the set of all the rectangles in $E=\mathbb{R}^2$, i.e. $\mathcal{I}=\mathcal{I}_{T}$ as in \eqref{eq:rectangle} with $T=\emptyset$. Fix pairwise disjoint $A_1,\dots,A_m\in\mathcal{I}$ and $n_1,\dots,n_m\in\mathbb{N}$ with $n=n_1+\cdots+n_m$. We claim that 
\begin{equation}\label{eq:in proof B7}
\lim_{N\rightarrow\infty}\mathbb{E}\left[\prod_{i=1}^m\frac{M^N(A_i)!}{(M^N(A_i)-n_i)!}\right]=\int_{A_1^{n_1}\times\cdots\times A_m^{n_m}}\mathrm{Pf}\left[K(s_i,x_i; s_j, x_j)\right]_{i,j=1}^n(\mu_{\mathcal{T}}\times\lambda)^n(dsdx).
\end{equation}
We will prove \eqref{eq:in proof B7} in the second step. Here, we assume its validity and conclude the proof of the proposition.
Using the local boundedness of $K$ and $\lambda$ and (\ref{eq: HadamardPfaffian}), we conclude that for each $B\in\mathcal{I}$ there exists a constant $C\in(0, \infty)$ such that
\begin{equation*}
\left|\int_{B^n}\mathrm{Pf}\left[K(s_i,x_i;s_j,x_j)\right]_{i,j=1}^n(\mu_{\mathcal{T}}\times\lambda)^n(dsdx)\right|\leq C^n\cdot n^{n/2},
\end{equation*}
which shows
\begin{equation}\label{GK1}
\sum_{n = 1}^{\infty} \frac{1}{n!} \cdot \left|\int_{B^n}\mathrm{Pf}\left[K(s_i,x_i;s_j,x_j)\right]_{i,j=1}^n(\mu_{\mathcal{T}}\times\lambda)^n(dsdx)\right|\leq \sum_{n = 1}^{\infty} \frac{C^n\cdot n^{n/2}}{n!} < \infty.
\end{equation}
Equation (\ref{eq:in proof B7}) verifies \cite[(2.6)]{dimitrov2024airy}, and (\ref{GK1}) verifies \cite[(2.8)]{dimitrov2024airy} with $\epsilon = 1$. From \cite[Proposition 2.2]{dimitrov2024airy} we conclude that there exists a countable $T^{\infty}\subset\mathbb{R}$ and a point process $M^{\infty}$ such that $M^N$ converge weakly to $M^{\infty}$ and for all pairwise disjoint $A_1,\dots,A_m\in\mathcal{I}_{T^{\infty}}$
$$
\mathbb{E}\left[\prod_{i=1}^m\frac{M^{\infty}(A_i)!}{(M^{\infty}(A_i)-n_i)!}\right]=\int_{A_1^{n_1}\times\cdots\times A_m^{n_m}}\mathrm{Pf}\left[ K(s_i, x_i; s_j, x_j) \right]_{i,j=1}^n(\mu_{\mathcal{T}}\times\lambda)^n(dsdx). $$
From Lemma \ref{lem:correlation function} we conclude that $M^{\infty}$ satisfies the conditions of the proposition. \\

{\bf \raggedleft Step 2.} We know that $M^N$ is Pfaffian with correlation kernel $K^N$ and reference measure $\mu_{\mathcal{T},\nu(N)}$. 
For each $\vec{s}=(s_1,\dots,s_n)\in\mathcal{T}^n$ we write $\nu_{\vec{s}}(N)$ to denote the measure on $\mathbb{R}^n$ defined as the product of measures $\nu_{s_i}(N)$ on $\mathbb{R}$ for $i=1,\dots,n$. In view of (\ref{eq:FactPfaff}), to show \eqref{eq:in proof B7} it suffices to prove
\begin{equation}\label{eq:in proof B7 2}
\lim_{N\rightarrow\infty}\int_{I}\mathrm{Pf}\left[K^N(s_i, x_i; s_j, x_j)\right]_{i,j=1}^n\nu_{\vec{s}}(N)(dx)=\int_{I}\mathrm{Pf}\left[K(s_i, x_i; s_j, x_j)\right]_{i,j=1}^n\lambda^n(dx),
\end{equation}
where $I=(a_1, b_1]\times\cdots\times(a_n, b_n]$ and $\vec{s} \in \mathcal{T}^n$ are fixed.

For each $x\in\mathbb{R}^n$ and $N\in\mathbb{N}$, we write
$$f^N(x)=\mathbf{1}_I(x)\cdot\mathrm{Pf}\left[K^N(s_i, x_i; s_j, x_j)\right]_{i,j=1}^n \quad \mbox{and}\quad
f(x)=\mathbf{1}_I(x)\cdot\mathrm{Pf}\left[K(s_i, x_i; s_j, x_j)\right]_{i,j=1}^n.$$
For each $x\in\mathbb{R}^n$ and $N\in\mathbb{N}$, we also define $y^N(x)\in\mathbb{R}^n$ by
$$
y^N_i(x)=b(s_i,N)+a(s_i,N)\lfloor a(s_i,N)^{-1}(x_i-b(s_i,N))\rfloor,\quad i=1,\dots,n.$$
One can observe that \eqref{eq:in proof B7 2} is equivalent to 
\begin{equation}\label{eq:in proof B7 3}
   \lim_{N\rightarrow\infty} \int_{\mathbb{R}^n} f^N(y^N(x))\lambda^n(dx)=\int_{\mathbb{R}^n}f(x)\lambda^n(dx).
\end{equation}
Indeed, we have that $f^N(y^N(x))$ is a step function that is constant on rectangles of the form $[c_1, c_1 + a(s_1,N)) \times \cdots \times [c_n, c_n + a(s_n,N))$, where $c_i = b(s_i,N) + a(s_i,N) z_i$ and $z_i \in \mathbb{Z}$. Rewriting the integral of $f^N(y^N(x))$ over $\mathbb{R}^n$ as a sum, one arrives at the integral on the left of (\ref{eq:in proof B7 2}).

By our assumption $\lim_Na(s_i,N)=0$ for $i=1,\dots,n$ we have $y^N(x)\rightarrow x$ as $N\rightarrow\infty$ for each $x\in\mathbb{R}^n$.
From $y^N_i(x)\in a(s_i,N)\cdot\mathbb{Z}+b(s_i,N)$ and \eqref{eq:limitProp} we have that for $i,j\in\{1,\dots,n\}$
\[
\lim_{N \rightarrow \infty}K^N (s_i,y^N_i(x);s_j,y^N_j(x)) = K (s_i,x_i;s_j,x_j)\quad \mbox{when }(x_i,x_j)\in U.
\]
Taking Pfaffians on both sides, we have $f^N(y^N(x))\rightarrow f(x)$ as $N\rightarrow\infty$ for each $x\in\mathbb{R}^n$ in the subset
\[
\{x\in\mathbb{R}^n\setminus\partial I: (x_i,x_j)\in U \mbox{ for each }i,j\in\{1,\dots,n\}\}.
\]
Since $U^c$ has Lebesgue measure zero, the complement of the above set in $\mathbb{R}^n$ also has Lebesgue measure zero.
Using our assumption \eqref{eq:limitProp2}, the functions $|f^N(y^N(x))|$ can be uniformly (in $N$) bounded by a finite constant. Therefore, \eqref{eq:in proof B7 3} follows from the bounded convergence theorem.   
\end{appendix}

\bibliographystyle{amsplain} 
\bibliography{PD}

\end{document}